\documentclass[12pt]{article}
\usepackage[paperwidth=210mm, paperheight=297mm,bindingoffset=0mm]{geometry}
\usepackage[T1]{fontenc}
\usepackage{tgtermes}
\usepackage{bbm}
\usepackage{mathbbol}
\usepackage[]{enumitem}
\setlist[enumerate,1]{itemsep=0pt, parsep=0pt, listparindent=\parindent}
\setlist[enumerate,2]{ref=\theenumii, itemsep=0pt, parsep=0pt, listparindent=\parindent}
\setlist[itemize,1]{itemsep=0pt, parsep=0pt, listparindent=\parindent}
\setlist[itemize,2]{itemsep=0pt, parsep=0pt, listparindent=\parindent}
\usepackage{titling}
\usepackage{mathtools}
\usepackage[backend=biber,
style=numeric,
sorting=ynt
]{biblatex}
\usepackage[nottoc,numbib]{tocbibind}
\usepackage{listings}

\usepackage{xcolor}
\usepackage[bottom]{footmisc} 
\usepackage{amssymb}
\usepackage{blindtext}
\usepackage{etoolbox}
\usepackage{graphicx}              
\usepackage{amsmath}               
\usepackage{amsfonts}              
\usepackage{amsthm}                
\usepackage[pdfencoding=auto, psdextra]{hyperref}
\hypersetup{citecolor=red}
\usepackage{cleveref}
\numberwithin{equation}{section}

\hypersetup{
colorlinks=true,
linkcolor=blue
}

\title{Forcing with Language Fragments, Extending Namba Forcing, and Models of Theories with Constraints in Interpretation}
\author{Desmond Lau}

\begin{document}

\maketitle

\begin{abstract}
    We develop a forcing framework based on the idea of amalgamating language fragments into a theory with a canonical term model. We then demonstrate the usefulness of this framework by applying it to variants of the extended Namba problem, as well as to the analysis of models of certain \emph{theories with constraints in interpretation} (\emph{TCIs}). The foundations for a theory of TCIs and their models are laid in parallel to the development of our framework, and are of independent interest.
\end{abstract}

\newtheorem{thm}{Theorem}[section]
\newtheorem{lem}[thm]{Lemma}
\newtheorem{prop}[thm]{Proposition}
\newtheorem{cor}[thm]{Corollary}
\newtheorem{conj}[thm]{Conjecture}
\newtheorem{ques}[thm]{Question}
\newtheorem*{claim}{Claim}
\newtheorem{claim2}[thm]{Claim}
\theoremstyle{definition}
\newtheorem{defi}[thm]{Definition}
\theoremstyle{remark}
\newtheorem*{rem*}{Remark}
\newtheorem{rem}[thm]{Remark}
\newtheorem{ex}[thm]{Example}
\newtheorem{ob}[thm]{Observation}
\newtheorem{fact}[thm]{Fact}
\newtheorem{con}[thm]{Convention}

\theoremstyle{definition}
\newtheorem{innercustomthm}{Theorem}
\newenvironment{customthm}[1]
  {\renewcommand\theinnercustomthm{#1}\innercustomthm}
  {\endinnercustomthm}

\theoremstyle{definition}
\newtheorem{innercustomlem}{Lemma}
\newenvironment{customlem}[1]
  {\renewcommand\theinnercustomlem{#1}\innercustomlem}
  {\endinnercustomlem}

\theoremstyle{definition}
\newtheorem{innercustomdef}{Definition}
\newenvironment{customdef}[1]
  {\renewcommand\theinnercustomdef{#1}\innercustomdef}
  {\endinnercustomdef}

\theoremstyle{remark}
\newtheorem{innercustomfact}{Fact}
\newenvironment{customfact}[1]
  {\renewcommand\theinnercustomfact{#1}\innercustomfact}
  {\endinnercustomfact}

\newcommand{\bd}[1]{\mathbf{#1}}  
\newcommand{\RR}{\mathbb{R}}      
\newcommand{\ZZ}{\mathbb{Z}}      
\newcommand{\col}[1]{\left[\begin{matrix} #1 \end{matrix} \right]}
\newcommand{\comb}[2]{\binom{#1^2 + #2^2}{#1+#2}}
\newcommand{\eq}{=}

\newcommand{\blankpage}{
\newpage
\thispagestyle{empty}
\mbox{}
\newpage
}

{\let\clearpage\relax \tableofcontents} 

\section{Introduction}\label{sect1}

Forcing is a technique in mathematical logic, whereby a set is proven to exist through a non-constructive but coherent assembly of known components. It is most often used in relative consistency proofs, and has been so ever since Cohen invented syntactic forcing in 1963, when he proved that the continuum hypothesis is independent of $\mathsf{ZFC}$ in \cite{cohen1} and \cite{cohen2}.

An application of forcing involves coming up with a partial order (called a \emph{forcing notion}) and analysing a filter (called a \emph{generic filter}) intersecting suitably many dense subsets of the aforementioned partial order. A set of which existence we want to show (called a \emph{generic object}) typically manifests as an amalgamation of members of the generic filter. In set theoretic applications, a generic object generates an extension (called a \emph{generic extension}) of the original universe. 

Over the years, set theorists have discovered important relationships between forcing notions and their generic objects\slash extensions. These relationships usually associate higher-order properties of a partial order with its forcing consequences, to the extent that the study of such properties has become a good description of forcing theory. Now, in order to derive a generic extension we desire, we need to construct a forcing notion drawing from forcing theory with our generic object in mind, at the same time easing tensions between requirements. This balancing act can be extremely tough, in part because it sees little in the way of systematic support.

In this paper we develop a framework in which certain desiderata of a generic object can be naturally realised. Using this framework, all we need to do is to translate requirements on the generic extension into requirements on the generic object. We see via two non-trivial examples, how this structured approach can make things more convenient and intuitive in practice.

It turns out that the bare-bones framework we initially envisioned can be adorned with additional layers of syntactic sugar to capture analogous properties of a specific type of first-order structures. These structures are incidentally models of what we term first-order \emph{theories with constraints in interpretation} (\emph{TCIs}). The concept of a first-order TCI, in some sense, generalises that of a first-order theory in logic, and can be useful in expressing different kinds of objects logicians care about. We devote some space in this paper to the basic theory of TCIs and their models, before relating forcing to generic models of TCIs in a variety of ways that traverses set theory and computability theory. As an aside, we argue for a way to characterise the expressiveness of forcing as a technique, by the kind of truths it is able to impose on generic objects.

\subsection{Sectional Content and Dependencies}

Section \ref{sect15} lays the technical foundation for the rest of the paper. Potential philosophical and meta-theoretic concerns are addressed, prerequisite knowledge highlighted, and background readings recommended. Important definitions and conventions are made explicit, especially those that are more niched, and those that lack community consensus. Consequentially, the topics related to forcing and generic iterations (Subsections \ref{subs24} to \ref{ss27}) tend to get more comprehensive treatments than the others. All subsequent sections depend on Subsections \ref{subs21} to \ref{subs24}, whereas the materials in Subsections \ref{ss25} to \ref{ss27} are only referenced in Subsection \ref{ss43}.

Section \ref{setupsec} concerns itself with the development of our central framework for forcing. The technical machinery of this section concentrates around Lemma \ref{uni}, which itself is a generalisation of Lemma \ref{main2}. On the other hand, Lemma \ref{main2} is the accessible and more applicable backbone of the paper, a throughline tying subsequent sections together. Subsection \ref{subs51} is notable for not depending on Lemma \ref{main2}, nor in fact, on any of Section \ref{setupsec}'s technology.

The entirety of Section \ref{sect2} is devoted to applying the forcing framework developed in Section \ref{setupsec} to variants of the extended Namba problem. Said framework is used to construct specific forcing notions that give rise to generic extensions satisfying various sets of requirements. We start with a relatively simple construction in Theorem \ref{notion1}, before extending it to solve a more difficult problem in Theorem \ref{notion2}. No other section is dependent on what transpires here. (N.B. The proof of Lemma \ref{lem339} references Remark \ref{rempp} in parentheses, but there is no strict dependency --- the low level subtleties in the remark are unnecessary and might detract from the clarity of the proof.)

Section \ref{TCIsec} introduces the notions of (first-order) TCIs and models of TCIs, before relating them to forcing and genericity. In particular, Subsection \ref{subs51} develops the basic theory of TCIs and their models, and can be read right after Section \ref{sect15}. The other subsections depend in part on the results of Section \ref{setupsec}, and give applications of Section \ref{setupsec}'s forcing framework to more general contexts of genericity. Here, genericity is investigated in both set-theoretic and computability-theoretic senses of the word. Like Section \ref{sect2}, this section is not the object of any dependency.

\section{Preliminaries}\label{sect15}

\subsection{The Meta-theory}\label{subs21}

At the meta-level, it suffices to assume $\mathsf{ZFC}$. We frame relative consistency proofs involving additional assumptions as proofs of statements of the form $$\text{``}\mathrm{Con}(\mathsf{ZFC} + \phi) \implies \mathrm{Con}(\mathsf{ZFC} + \psi)\text{''}$$ over $\mathsf{ZFC}$, where $\phi$ is the conjunction of the relevant assumptions. Implicit are the invocations of G\"{o}del's completeness theorem at the meta-level, whenever we argue using models (a.k.a. \emph{universes}) of set theory.

Our meta-theoretic approach to forcing is only slightly more complicated. Conventionally we start with a countable transitive model of $\mathsf{ZFC}$, called a \emph{ground model}, which is not guaranteed to exist by our meta-theory. There are many ways to sidestep this problem and treat the statement ``countable transitive ground model'' as a convenient abuse of notation, a few of which are noted by Kunen in IV.5 of \cite{kunen}. We adopt the first approach detailed in IV.5.1 of \cite{kunen}, an approach that is pretty much standard in the community, and one we feel is most immediately and formally accessible.

\subsection{Basic Mathematical Logic}

Unless specified otherwise, we follow the standard definitions of concepts related to the syntax and semantics of first-order logic, as seen in e.g. \cite{enderton}.

\begin{con}\label{firstconv}
\leavevmode
\begin{enumerate}[label=(\arabic*)]
    \item We call any set of first order formulas a \emph{first-order language}.
    \item We assume the first-order languages we consider to contain only $\neg, \wedge, \vee, \implies, \iff$ as their zeroth-order logical symbols, interpreted semantically in the usual sense. 
    \item Each first-order logical symbol is identified with a unique member of $H(\omega) \setminus \omega$.
    \item Given a first-order language $\mathcal{L}$, let $\mathrm{Ter}(\mathcal{L})$ denote the set of all terms occurring in (some formula in) $\mathcal{L}$.
    \item A first-order structure $\mathfrak{A}$ is presented in the form $(A; \mathcal{I})$, where $A$ is the base set of $\mathfrak{A}$ and $\mathcal{I}$ is the interpretation of the signature of $\mathfrak{A}$. In this presentation, the signature of $\mathfrak{A}$ is simply $dom(\mathcal{I})$. 
    
    Sometimes, when the correspondence between a signature and its interpretation is clear, we might write $(A; \mathcal{I})$ as $(A; \Vec{S})$, where $\Vec{S}$ is some ordering of $ran(\mathcal{I})$.
    \item The signature of a first-order language or structure can contain only relation symbols of a non-zero finite arity, function symbols of a non-zero finite arity, and constant symbols, identified as follows:
    \begin{itemize}
        \item a $n$-ary relation symbol is a triple of the form $(X, 0, n)$,
        \item a $n$-ary function symbol is a triple of the form $(X, 1, n)$, and
        \item a constant symbol is a pair of the form $(X, 2)$.
    \end{itemize}
    We call any such symbol a \emph{signature-related symbol}.
    \item A \emph{form-preserving signature embedding} is an injective function from a set of signature-related symbols into the class of signature-related symbols, such that
    \begin{itemize}
        \item $n$-ary relation symbols are mapped to $n$-ary relation symbols,
        \item $n$-ary function symbols are mapped to $n$-ary function symbols, and
        \item constant symbols are mapped to constant symbols.
    \end{itemize}
    \item\label{227} We will assume that the class of signature-related symbols is disjoint from the set of first-order logical symbols.
    \item\label{224} The variables occurring in any first-order formula must come from a fixed countably infinite set $\mathrm{Var}$. We will assume that $\mathrm{Var}$ is disjoint from both the set of first-order logical symbols and the class of signature-related symbols.
    \item A string over a vocabulary set $\Sigma$ is a member of $\Sigma^{< \omega}$.
    \item We sometimes define procedures in which subformulas of a first-order formula are replaced by other formulas. In these cases, for convenience's sake, what follows will be adopted.
    \begin{quote}
        Let 
        \begin{itemize}
            \item $\phi$ be a first-order formula, 
            \item $\varphi$ be a subformula of $\phi$, and
            \item $\ulcorner x \urcorner$ be a variable occurring in $\phi$.
        \end{itemize}
        Suppose we are to replace $\varphi$ in $\phi$ with some formula $\varphi'$. Let $\ulcorner y \urcorner$ be a bound variable in $\varphi'$. Unless otherwise stated, we always assume $$\ulcorner x \urcorner \neq \ulcorner y \urcorner \text{.}$$        
    \end{quote}
\end{enumerate}
\end{con}

We write ``structure(s)'' as the abbreviation of ``first-order structure(s)'' henceforth. There should be no confusion as these are the only type of structures we will be dealing with.

\begin{defi}
Given any set $X$ and any signature $\sigma$, the \emph{language associated with} $(X; \sigma)$ is the set of first-order formulas over $\sigma$ with parameters from $X$. Similarly, given any structure $\mathfrak{A} = (A; \mathcal{I})$, the \emph{language associated with} $\mathfrak{A}$ is the set of first-order formulas over the signature of $\mathfrak{A}$ with parameters from $A$.
\end{defi}

\begin{defi}
For any structure $\mathfrak{A} = (A; \mathcal{I})$, a $\mathfrak{A}$\emph{-valuation} is a function from $\mathrm{Var}$ into $A$.
\end{defi}

\begin{defi}
The \emph{size of a structure} $\mathfrak{A} = (A; \mathcal{I})$ is equal to $$max\{|A|, |\mathcal{I}|\}.$$ We say $\mathfrak{A}$ is a \emph{countable structure} iff its size is a countable cardinal.
\end{defi}

\begin{defi}\label{def25}
Let $\phi$ be a first-order formula over a signature $\sigma$. We inductively define what it means for $\phi$ to be $\Pi_n$ or $\Sigma_n$ as $n$ ranges over the natural numbers.
\begin{enumerate}[label=(\arabic*)]
    \item If $n = 0$, then $\phi$ is $\Pi_0$ iff $\phi$ is $\Sigma_0$ iff $\phi$ is quantifier-free.
    \item\label{3292} If $n = m + 1$ for some $m < \omega$, then 
    \begin{enumerate}[label=(\alph*)]
        \item\label{3292a} $\phi$ is $\Pi_n$ iff there is a $\Sigma_m$ formula $\varphi$, a number $k < \omega$, and variable symbols $x_1, \dots, x_k$ not bound in $\varphi$, such that 
        \begin{equation*}
            \phi = \ulcorner \forall x_1 \dots \forall x_k \ \varphi \urcorner \text{, and}
        \end{equation*}
        \item\label{3292b} $\phi$ is $\Sigma_n$ iff there is a $\Pi_m$ formula $\varphi$, a number $k < \omega$, and variable symbols $x_1, \dots, x_k$ not bound in $\varphi$, such that 
        \begin{equation*}
            \phi = \ulcorner \exists x_1 \dots \exists x_k \ \varphi \urcorner \text{.}
        \end{equation*}
    \end{enumerate}
\end{enumerate}
Note that if $k = 0$ in \ref{3292}\ref{3292a} and \ref{3292}\ref{3292b}, then $\phi$ is $\Sigma_m$ and $\Pi_m$ respectively.
\end{defi}

\subsection{Basic Set Theory}

Unless specified otherwise, we follow the standard definitions of concepts typically encountered in a foundational set theory course, following e.g. \cite{jechhrbacek}.

\begin{con}
\leavevmode
\begin{enumerate}[label=(\arabic*)]
    \item Unless otherwise specified, $V$ always refers to the universe we are currently working in. For all practical purposes, we can assume it is a countable transitive model of $\mathsf{ZFC}$, so that it doubles as a ground model in case forcing arguments are to be run.
    \item We adopt the set-theoretic interpretation of functions as sets of ordered pairs satisfying certain properties. So when we say a function is definable, we actually mean its graph is definable as a set --- usually a subset of an ambient structure that should be clear in context, if not explicitly mentioned.
    \item We fix in $V$, a distinguished first-order relation symbol $\ulcorner \in \urcorner$.
    \item A first-order formula is \emph{in the language of set theory} iff it is a formula over the signature $\{\ulcorner \in \urcorner\}$.
    \item\label{stcon1} We say a structure $\mathfrak{A}$ is \emph{a structure in the language of set theory} iff 
    \begin{itemize}
        \item the signature of $\mathfrak{A}$ is the singleton set $\{\ulcorner \in \urcorner\}$, and
        \item $\mathfrak{A}$ interprets $\ulcorner \in \urcorner$ as the membership relation on $V$ restricted to the base set of $\mathfrak{A}$.
    \end{itemize}
    More formally, $$\mathfrak{A} = (A; \mathcal{I}) \text{ and } \mathcal{I} = \{(\ulcorner \in \urcorner, \in \cap A)\},$$ where $\in$ is the membership relation on $V$. In this case, we can just write $\mathfrak{A} = (A; \in)$.
    \item A first-order formula is \emph{in a possibly expanded language of set theory} iff it is a formula over some signature $\sigma$ satisfying $\ulcorner \in \urcorner \in \sigma$.
    \item We say a structure $\mathfrak{A}$ is \emph{a structure in a possibly expanded language of set theory} iff we omit the cardinality requirement on the signature of $\mathfrak{A}$ in \ref{stcon1}. More formally, $$\mathfrak{A} = (A; \mathcal{I}) \text{ and } \mathcal{I} = \{(\ulcorner \in \urcorner, \in \cap A)\} \cup \Vec{X},$$ where $\Vec{X}$ is some function and $\in$ is the membership relation on $V$. In this case, we can just write $\mathfrak{A} = (A; \in, \Vec{X})$.
    \item\label{idu} We identify a universe of set theory $W$ with the structure $(W; \in)$. This should not cause confusion in the circumstances we find ourselves in.
    \item A real is a subset of $\omega$. We identify a real with its characteristic function on $\omega$, as is standard in computability theory. As in \ref{idu}, this ambiguity should not cause any confusion.
\end{enumerate}
\end{con}

We overload and expand on Definition \ref{def25} when dealing with the special case of set-theoretic languages. 

\begin{defi}\label{def27}
Let $\phi$ be a first-order formula over a possibly expanded language of set theory. We inductively define what it means for $\phi$ to be $\Delta_n$, $\Pi_n$ or $\Sigma_n$ as $n$ ranges over the natural numbers.
\begin{enumerate}[label=(\arabic*)]
    \item\label{def271} If $n = 0$, then $\phi$ is $\Delta_n$ iff $\phi$ is $\Pi_n$ iff $\phi$ is $\Sigma_n$ iff every quantifier occurring in $\phi$ is bounded by $\in$.
    \item\label{2712} If $n = m + 1$ for some $m < \omega$, then 
    \begin{enumerate}[label=(\alph*)]
        \item\label{2712a} $\phi$ is $\Pi_n$ iff there is a $\Sigma_m$ formula $\varphi$, a number $k < \omega$, and variable symbols $x_1, \dots, x_k$ not bound in $\varphi$, such that 
        \begin{equation*}
            \phi = \ulcorner \forall x_1 \dots \forall x_k \ \varphi \urcorner \text{, and}
        \end{equation*}
        \item\label{2712b} $\phi$ is $\Sigma_n$ iff there is a $\Pi_m$ formula $\varphi$, a number $k < \omega$, and variable symbols $x_1, \dots, x_k$ not bound in $\varphi$, such that 
        \begin{equation*}
            \phi = \ulcorner \exists x_1 \dots \exists x_k \ \varphi \urcorner \text{.}
        \end{equation*}
        \item $\phi$ is $\Delta_n$ iff 
        \begin{itemize}[label=$\circ$]
            \item $\phi$ is $\Pi_n$, and
            \item for some $\Sigma_n$ formula $\varphi$,
            \begin{equation*}
                \mathsf{ZFC} \vdash \phi \iff \varphi \text{.}
            \end{equation*}
        \end{itemize}
    \end{enumerate}
\end{enumerate}
Note that if $k = 0$ respectively in \ref{2712}\ref{2712a} and \ref{2712}\ref{2712b}, then $\phi$ is $\Sigma_m$ and $\Pi_m$ respectively.
\end{defi}

Most of the time, the context should indicate clearly which interpretation to adopt. Nevertheless, we shall try as much as possible to disambiguate things in this respect, and highlight each time Definition \ref{def27} is used instead of Definition \ref{def25}.

\begin{defi}
Let $X$ be a set and $\mathfrak{A}$ be a structure in a possibly expanded language of set theory with base set $A$. We say $X$ is \emph{definable in the language associated with} $\mathfrak{A}$ iff $$X = \{y \in A : \mathfrak{A} \models \phi(y)\}$$ for some formula $\phi$ with one free variable in the language associated with $\mathfrak{A}$.
\end{defi}

\begin{defi}\label{def29}
Let 
\begin{itemize}
    \item $X$ be a set,
    \item $\mathfrak{A}$ be a structure in a possibly expanded language of set theory with base set $A$, and
    \item $n < \omega$. 
\end{itemize} 
We say $X$ is $\Sigma_n$\emph{-definable} (resp. $\Pi_n$\emph{-definable} and $\Delta_n$\emph{-definable}) \emph{in the language associated with} $\mathfrak{A}$ iff $$X = \{y \in A : \mathfrak{A} \models \phi(y)\}$$ for some $\Sigma_n$ (resp. $\Pi_n$ and $\Delta_n$) formula (in accordance with Definition \ref{def27}) $\phi$ with one free variable in the language associated with $\mathfrak{A}$.
\end{defi}

We can generalise Definition \ref{def29} by relaxing the language requirement.

\begin{defi}
Let 
\begin{itemize}
    \item $X$ be a class, 
    \item $\mathfrak{A}$ be a structure in a possibly expanded language of set theory with base class $A$ and signature $\sigma$, such that 
    \begin{itemize}[label=$\circ$]
        \item $A$ is transitive, and
        \item $\mathfrak{A} \models \mathsf{ZFC}$, and
    \end{itemize}
    \item $n < \omega$.
\end{itemize}
We say $X$ is $\Sigma_n$\emph{-definable} (resp. $\Pi_n$\emph{-definable} and $\Delta_n$\emph{-definable}) \emph{in} $\mathfrak{A}$ iff $$X = \{y \in A : \mathfrak{A} \models \phi(y)\}$$ for some $\Sigma_n$ (resp. $\Pi_n$ and $\Delta_n$) formula (in the sense of Definition \ref{def27}) $\phi$ with one free variable over $\sigma$.

If in addition,
\begin{itemize}
    \item $\mathfrak{A} = (A; \in)$, and
    \item $\sigma = \{\ulcorner \in \urcorner\}$,
\end{itemize}
we may identify $\mathfrak{A}$ with $A$ and say $X$ is $\Sigma_n$\emph{-definable} (resp. $\Pi_n$\emph{-definable} and $\Delta_n$\emph{-definable}) \emph{in} $A$ iff $X$ is $\Sigma_n$-definable (resp. $\Pi_n$-definable and $\Delta_n$-definable) in $\mathfrak{A}$. 

We say $X$ is $\Sigma_n$\emph{-definable} (resp. $\Pi_n$\emph{-definable} and $\Delta_n$\emph{-definable}) iff $X$ is $\Sigma_n$-definable (resp. $\Pi_n$-definable and $\Delta_n$-definable) in $\mathfrak{A}'$ for every structure $\mathfrak{A}'$ in a possibly expanded language of set theory such that
\begin{itemize}[label=$\circ$]
    \item $\mathfrak{A}'$ has a transitive base class, and
    \item $\mathfrak{A}' \models \mathsf{ZFC}$.
\end{itemize}
\end{defi}

\begin{defi}
Let $\mathfrak{A}$ be a structure in a possibly expanded language of set theory. When we say 
\begin{equation}\label{eq11}
    \text{``}\mathfrak{A} \textit{ is a model of a sufficiently strong set theory'',} 
\end{equation}
we mean to emphasise the low strength of the set theory $\mathfrak{A}$ satisfies. 

In more concrete terms, what we typically term a \emph{set theory}, is a set of first-order formulas in the language of set theory. Examples of set theories include
\begin{itemize}
    \item the set of axioms of $\mathsf{ZFC}$, and
    \item the set of axioms of \textit{Kripke-Platek set theory} (\emph{without infinity}), denoted $\mathsf{KP}$, which is a much weaker theory than $\mathsf{ZFC}$.
\end{itemize} 
For convenience's sake, one may always assume (\ref{eq11}) to mean 
\begin{quote}
    $\text{``}\mathfrak{A} \models T \text{ for some set theory } T \text{ such that}$ 
    \begin{itemize}
        \item $\mathsf{KP} \subset T$, and
        \item $T \vdash \ulcorner \forall x \ (\text{``} x \text{ is Dedekind-finite} \implies x \text{ is finite''}) \urcorner$''.
    \end{itemize}
\end{quote}
\end{defi}

\begin{defi}
Given (externally) a class $\mathbf{M}$ of classes, we say a definition $\varphi$ in $n$ variables --- for some finite number $n$ --- is \emph{absolute for} $\mathbf{M}$ iff for every two members $X, Y$ of $\mathbf{M}$ such that $X \subset Y$, and for every sequence $\Vec{x}$ of $n$ members of $X$,
\begin{equation*}
    X \models \varphi(\Vec{x}) \iff Y \models \varphi(\Vec{x}) \text{.}
\end{equation*}
\end{defi}

\begin{defi}\label{inoutmodels}
Given that $V'$ and $W$ are models of $\mathsf{ZFC}$, we say $V'$ is a \emph{weak inner model} of $W$ (or equivalently, $W$ is a \emph{weak outer model} of $V'$) iff
\begin{itemize}
    \item $V'$ and $W$ are both transitive, and
    \item $V' \subset W$.
\end{itemize}
If in addition, $V'$ and $W$ share the same ordinals i.e. $ORD^{V'} = ORD^W$, then $V'$ is an \emph{inner model} of $W$ (or equivalently, $W$ is an \emph{outer model} of $V'$).
\end{defi}

\begin{defi}
Let $Y$ be any set.

We say $X$ \emph{codes} $Y$ (or equivalently, $X$ \emph{is a code of} $Y$) iff 
\begin{itemize}
    \item $X$ is a set of ordinals, and
    \item every transitive model of $\mathsf{ZFC - Powerset}$ containing $X$ also contains $Y$.
\end{itemize}

$X$ \emph{codes} $Y$ \emph{as a real} iff $X$ codes $Y$ and $X \subset \omega$.

$Y$ \emph{has a real code} (or equivalently, $Y$ \emph{can be coded as a real}) iff $X$ codes $Y$ as a real for some $X$.
\end{defi}

\begin{lem}\label{setcode}
Let $X$ be a set with $|trcl(X)| = \kappa$. Then there is $A \subset \kappa$ such that $A$ codes $X$.

In particular, any set with a countable transitive closure has a real code.
\end{lem}

\begin{proof}
Let $Y := trcl(X) \cup \{X\}$. Note that 
\begin{equation}\label{trcluni}
    X \text{ is the unique } \! \in \!\text{-maximal member of } Y
\end{equation} 
in any transitive model of $\mathsf{ZFC - Powerset}$ containing $Y$. Choose any bijection $f$ from $Y$ into $\kappa$. Define $$R := \{(f(x), f(y)) \in \kappa \times \kappa : (x, y) \in Y \times Y \text{ and } x \in y\}.$$ Now $R$ can be thought of as a subset $A$ of $\kappa$ via the canonical G\"{o}del numbering of pairs. If $V'$ is a transitive model of $\mathsf{ZFC - Powerset}$ containing $A$, then we can recover $R$ in $V'$. The Mostowski collapse function works on $R$ in $V'$ to give us $Y \in V'$. This implies $X \in V'$ since $X$ definable from $Y$ via (\ref{trcluni}).
\end{proof}

\subsection{Forcing and Generic Objects}\label{subs24}

Following the meta-theoretic convention highlighted in Subsection \ref{subs21}, we start with a countable transitive ground model $V$. In the language of forcing, a forcing notion in $V$ is just any partial order in $V$. If $\mathbb{P}$ is a forcing notion in $V$, then a $\mathbb{P}$\emph{-generic filter over} $V$ is a filter on $\mathbb{P}$ intersecting all dense subsets of $\mathbb{P}$ in $V$. 

Given a forcing notion $\mathbb{P}$ in $V$, the class of $\mathbb{P}$-names in $V$ --- denoted $V^{\mathbb{P}}$ --- and the forcing relation $\Vdash_{\mathbb{P}}^V$ (which relates elements of $\mathbb{P}$ with formulas parametrised by $\mathbb{P}$-names in $V$), are both essential to a forcing argument involving $\mathbb{P}$ carried out in $V$. These two classes are uniformly definable in $V$ over the class of all forcing notions $\mathbb{P}$. $\mathbb{P}$-names in $V$ are ``evaluated at" a $\mathbb{P}$-generic filter $g$ over $V$ to obtain the $\mathbb{P}$-generic extension $V[g]$, which is necessarily countable and transitive. In other words, if $g$ is a $\mathbb{P}$-generic filter over $V$, then
\begin{equation*}
    V \subset V[g] = \{\dot{x}[g] : \dot{x} \in V^{\mathbb{P}}\},
\end{equation*}
where $\dot{x}[g]$ means ``$x$ evaluated at $g$". The evaluation procedure is done outside $V$ because $g$ typically (in order to be of use at all) does not exist in $V$. 

\begin{con}
When it is clear that the background universe is $V$, we suppress mention of $V$ when writing forcing relations in $V$. This means that given a forcing notion $\mathbb{P}$ in $V$, $\Vdash_{\mathbb{P}}$ is used interchangeably with $\Vdash_{\mathbb{P}}^V$.
\end{con}

\begin{defi}
We call $W$ a \emph{generic extension} (or a \emph{forcing extension}) of $V$ iff there exists a forcing notion $\mathbb{P}$ in $V$ and a $\mathbb{P}$-generic filter $g$ over $V$, such that $W = V[g]$.
\end{defi}

\begin{defi}
We write ``$\Vdash_{\mathbb{P}} \phi$" to mean $$\text{``} \forall p \ (p \in \mathbb{P} \implies p \Vdash_{\mathbb{P}} \phi) \text{''.}$$ 
\end{defi}

\begin{rem}
A theorem fundamental to the technique of forcing intricately connects the forcing relation $\Vdash_{\mathbb{P}}^V$ with truth in $\mathbb{P}$-generic extensions. It goes as follows:
\begin{quote}
If $\mathbb{P}$ is a forcing notion in $V$, $p \in \mathbb{P}$, $\phi$ is a formula with $n$ free variables, and $\dot{x}_1$, \dots, $\dot{x}_n$ are $\mathbb{P}$-names in $V$, then
\begin{itemize}
    \item 
    \!
    $\begin{aligned}[t]
    & p \Vdash_{\mathbb{P}} \phi(\dot{x}_1, \dots, \dot{x}_n) \iff \\
    & \forall g \ ((g \text{ is } \mathbb{P} \text{-generic over } V \text{ and } p \in g) \\
    & \implies V[g] \models \phi(\dot{x}_1[g], \dots, \dot{x}_n[g])), \text{ and}
    \end{aligned}$
    \item 
    \!
    $\begin{aligned}[t]
    \forall g \ ( & (g \text{ is } \mathbb{P} \text{-generic over } V \text{ and } V[g] \models \phi(\dot{x}_1[g], \dots, \dot{x}_n[g])) \\
    & \implies \exists q \ (q \Vdash_{\mathbb{P}} \phi(\dot{x}_1, \dots, \dot{x}_n) \text{ and } q \in g)) \text{.}
    \end{aligned}$
\end{itemize}
\end{quote}
This theorem, colloquially known as the \emph{forcing theorem}, enables us to reason about truth in generic extensions from within the ground model, and often reduces the argument from one about semantic entailment to one pertaining to combinatorial properties of partial orders. For a technical lowdown of forcing terminology and the proof of the forcing theorem, the reader is encouraged to read Chapter IV of \cite{kunen}. 
\end{rem}

\begin{defi}
Let $\mathbb{P} = (P, \leq_{\mathbb{P}})$ be a forcing notion and $X$ be any set. The \emph{upward closure of} $X$ \emph{in} $\mathbb{P}$, denoted $\mathrm{UC}(\mathbb{P}, X)$, is the set $$\{p \in P : \exists q \ (q \in X \text{ and } q \leq_{\mathbb{P}} p)\}.$$
\end{defi}

\begin{defi}
Let $\mathbb{P} = (P, \leq_{\mathbb{P}})$ be a forcing notion, $D \subset P$ and $A$ be any set. We say a subset $g$ of $P$ \emph{meets} $D$ \emph{in} $A$ iff $$g \cap \{p \in P : p \in D \text{ or } \forall q \ (q \leq_{\mathbb{P}} p \implies q \not\in D)\} \cap A \neq \emptyset.$$ We say $g$ \emph{meets} $D$ iff $g$ meets $D$ in $V$.
\end{defi}

\begin{defi}
Let $\mathbb{P} = (P, \leq_{\mathbb{P}})$ be a forcing notion and $\mathfrak{A} = (A; \in, \Vec{X})$ be a structure in a possibly expanded language of set theory. We say a subset $g$ of $P$ is $\mathbb{P}$\emph{-generic over} $\mathfrak{A}$ (or $g$ is a $\mathbb{P}$\emph{-generic subset over} $\mathfrak{A}$) iff $g$ meets $D$ in $A$ for all $D$ such that
\begin{itemize}
    \item $D \subset P$
    \item $D$ is dense in $\mathbb{P}$, and
    \item $D$ is definable in the language associated with $\mathfrak{A}$.
\end{itemize}
If in addition, $g$ is a filter on $\mathbb{P}$, then we call $g$ a $\mathbb{P}$\emph{-generic filter over} $\mathfrak{A}$.
\end{defi}

\begin{defi}
Let $\mathbb{P} = (P, \leq_{\mathbb{P}})$ be a forcing notion and $\mathfrak{A} = (A; \in, \Vec{X})$ be a structure in a possibly expanded language of set theory. Further, let $n < \omega$. We say a subset $g$ of $P$ is $\mathbb{P}$\emph{-}$\Sigma_n$\emph{-generic} (resp. $\mathbb{P}$\emph{-}$\Pi_n$\emph{-generic} and $\mathbb{P}$\emph{-}$\Delta_n$\emph{-generic}) \emph{over} $\mathfrak{A}$ iff $g$ meets $D$ in $A$ for all $D$ such that
\begin{itemize}
    \item $D \subset P$
    \item $D$ is dense in $\mathbb{P}$, and
    \item $D$ is $\Sigma_n$-definable (resp. $\Pi_n$-definable and $\Delta_n$-definable) in the language associated with $\mathfrak{A}$.
\end{itemize}
If in addition, $g$ is a filter on $\mathbb{P}$, then we call $g$ a $\mathbb{P}$\emph{-}$\Sigma_n$\emph{-generic} (resp. $\mathbb{P}$\emph{-}$\Pi_n$\emph{-generic} and $\mathbb{P}$\emph{-}$\Delta_n$\emph{-generic}) \emph{filter over} $\mathfrak{A}$.
\end{defi}

\begin{defi}
Let $\mathbb{P}$ be a forcing notion and $\mathfrak{A} = (A; \in, \Vec{X})$ be a \emph{transitive} structure in a possibly expanded language of set theory. A set $x$ is a $(\mathbb{P}, \mathfrak{A})$\emph{-generic object} iff there exists $g$ a $\mathbb{P}$-generic filter over $\mathfrak{A}$ such that
\begin{itemize}
    \item $x \in A[g]$, and
    \item $g$ is definable in the language associated with $(A \cup \{x\}; \in, \Vec{X})$,
\end{itemize}
in which case we say $g$ \emph{witnesses} $x$ is a $(\mathbb{P}, \mathfrak{A})$-generic object.
\end{defi}

\begin{defi}
Let $x$ be any set. We call $x$ a \emph{generic object} iff there is a pair $(\mathbb{P}, \mathfrak{A})$ for which $x$ is a $(\mathbb{P}, \mathfrak{A})$-generic object. We call $x$ a $V$\emph{-generic object} iff there is $\mathbb{P}$ for which $x$ is a $(\mathbb{P}, V)$-generic object.
\end{defi}

\begin{ob}\label{ob0}
Let $\mathbb{P} = (P, \leq_{\mathbb{P}})$ be a forcing notion and $X$ be any set. Then there is a structure $\mathfrak{A} = (A; \in) \in V$ such that in every weak outer model of $V$, 
\begin{align*}
    x \text{ is a } (\mathbb{P}, \mathfrak{A}) \text{-generic object} \iff x \text{ is a } (\mathbb{P}, V) \text{-generic object}
\end{align*}
for all $x \subset X$. In fact, we can choose $A$ to be $H(\kappa)$ for any $\kappa > |trcl(\{P, X\})|$.
\end{ob}

Observation \ref{ob0} allows us to refer to $(\mathbb{P}, V)$-generic objects for any forcing notion $\mathbb{P}$, without needing to quantify over all formulas.

There are looser definitions of a generic filter or a generic object in the literature. For example, we can require the filter to only meet subsets with definitions belonging to a certain complexity class, as is commonly seen in computability theory. Informally then, the study of genericity boils down to observing the effects of a filter meeting a bunch of subsets. 

Section \ref{sect1} hinted at a key difference between forcing theory and the study of partial orders in order theory, and that is the nature of the properties studied apropos of their common subjects. In order theory, only first-order properties of partial orders are considered, whereas forcing theory concerns itself with their higher-order properties. Now, another such differentiating factor is the overwhelming focus on generic objects in forcing theory. In fact, so much attention is paid to generic objects in forcing theory that one might as well call it \emph{genericity theory}.

\begin{defi}
If $\mathbb{B}$ is a boolean algebra with base set $B$ and least element $\mathbb{0}$, then $\mathbb{B}^+$ denotes the partial order reduct of $\mathbb{B}$ restricted to $B \setminus \{\mathbb{0}\}$.
\end{defi}

\begin{defi}
Let $\mathbb{P} = (P, \leq_{\mathbb{P}})$ and $\mathbb{Q} = (Q, \leq_{\mathbb{Q}})$ be preorders. We say $\pi$ is an \emph{embedding from} $\mathbb{P}$ \emph{into} $\mathbb{Q}$ iff $\pi$ is an injective function from $P$ into $Q$ satisfying
\begin{itemize}
    \item $p_1 \leq_{\mathbb{P}} p_2 \iff \pi(p_1) \leq_{\mathbb{Q}} \pi(p_2)$, and
    \item $p_1 \ \bot_{\mathbb{P}} \ p_2 \implies \pi(p_1) \ \bot_{\mathbb{Q}} \ \pi(p_2)$.
\end{itemize}
\end{defi}

\begin{defi}
Let $\mathbb{P} = (P, \leq_{\mathbb{P}})$ and $\mathbb{Q} = (Q, \leq_{\mathbb{Q}})$ be preorders. An embedding $\pi$ from $\mathbb{P}$ into $\mathbb{Q}$ is \emph{complete} iff for every maximal antichain $A$ of $\mathbb{P}$, $$\{\pi(p) : p \in A\}$$ is a maximal antichain of $\mathbb{Q}$.
\end{defi}

\begin{defi}
Let $\mathbb{P} = (P, \leq_{\mathbb{P}})$ and $\mathbb{Q} = (Q, \leq_{\mathbb{Q}})$ be preorders. An embedding $\pi$ from $\mathbb{P}$ into $\mathbb{Q}$ is \emph{dense} iff $ran(\pi)$ is dense in $\mathbb{Q}$.
\end{defi}

\begin{fact}\label{dic}
Let $\mathbb{P}$ and $\mathbb{Q}$ be preorders. Then every dense embedding from $\mathbb{P}$ into $\mathbb{Q}$ is a complete embedding from $\mathbb{P}$ into $\mathbb{Q}$.
\end{fact}

\begin{defi}
Let $\mathbb{P} = (P, \leq_{\mathbb{P}})$ be a preorder. Define \begin{align*}
    w(\leq_{\mathbb{P}}) := \ & \{(p, q) \in P \times P : \{q' : q' \leq_{\mathbb{P}} q\} \text{ is dense below } p\} \text{, and} \\
    w(\mathbb{P}) := \ & (P, w(\leq_{\mathbb{P}})) \text{.}
\end{align*}
\end{defi}

\begin{fact}\label{wfixedp}
For any preorder $\mathbb{P}$, $w(w(\mathbb{P})) = w(\mathbb{P})$.
\end{fact}

\begin{fact}\label{seppre}
If $\mathbb{P} = (P, \leq_{\mathbb{P}})$ is a preorder, then $w(\mathbb{P})$ is also a preorder.
\end{fact}

\begin{defi}
A preorder $\mathbb{P}$ is \emph{separative} iff $w(\mathbb{P}) = \mathbb{P}$.
\end{defi}

\begin{fact}\label{bcompletion}
If $\mathbb{P}$ is a separative forcing notion, then there is a unique (up to isomorphism) complete boolean algebra $B(\mathbb{P})$ such that a dense embedding exists from $\mathbb{P}$ into $B(\mathbb{P})^+$.
\end{fact}

Fix a preorder $\mathbb{P} = (P, \leq_{\mathbb{P}})$. Note that by Fact \ref{seppre}, $w(\leq_{\mathbb{P}})$ induces an equivalence relation on $P$. To wit, for any $p, q \in P$, let $$p \sim_{\mathbb{P}} q \text{ iff } (p, q) \in w(\leq_{\mathbb{P}}) \text{ and } (q, p) \in w(\leq_{\mathbb{P}}) \text{.}$$ Then $\sim_{\mathbb{P}}$ is an equivalence relation on $P$. 

\begin{defi}
Given a preorder $\mathbb{P}$, call $$s(\mathbb{P}) := w(\mathbb{P}) / \sim_{\mathbb{P}}$$ the \emph{separative quotient} of $\mathbb{P}$.
\end{defi}

\begin{rem}
By Fact \ref{wfixedp}, $s(\mathbb{P})$ is a separative forcing notion given any preorder $\mathbb{P}$.
\end{rem}

\begin{defi}
Given preorders $\mathbb{P}$ and $\mathbb{Q}$, we say $\mathbb{P} \lessdot \mathbb{Q}$ iff there is a complete embedding from $s(\mathbb{P})$ into $B(s(\mathbb{Q}))^+$.
\end{defi}

\begin{fact}\label{forpre}
The relation $\lessdot$ is a preordering of the class of all preorders. Hence it also pre-orders the class of all forcing notions.
\end{fact}

\begin{defi}\label{comdef}
Preorders $\mathbb{P}$ and $\mathbb{Q}$ are \emph{forcing equivalent} iff $\mathbb{P} \lessdot \mathbb{Q}$ and $\mathbb{Q} \lessdot \mathbb{P}$.
\end{defi}

\begin{rem}\label{sepeq}
By Facts \ref{dic}, \ref{wfixedp}, \ref{bcompletion} and Definition \ref{comdef}, for any preorder $\mathbb{P}$, $\mathbb{P}$ and $w(\mathbb{P})$ are forcing equivalent.
\end{rem}

\begin{fact}\label{defor}
Let $\mathbb{P}$ and $\mathbb{Q}$ be preorders. If there is a dense embedding from $\mathbb{P}$ into $\mathbb{Q}$, then $\mathbb{P}$ and $\mathbb{Q}$ are forcing equivalent.
\end{fact}

\begin{defi}
Let $\mathbb{P} = (P, \leq_{\mathbb{P}})$ and $\mathbb{Q} = (Q, \leq_{\mathbb{Q}})$ be preorders. We say $\pi$ is a \emph{weak embedding from} $\mathbb{P}$ \emph{into} $\mathbb{Q}$ iff $\pi$ is an embedding from $w(\mathbb{P})$ into $w(\mathbb{Q})$.
\end{defi}

\begin{defi}
Let $\mathbb{P} = (P, \leq_{\mathbb{P}})$ and $\mathbb{Q} = (Q, \leq_{\mathbb{Q}})$ be preorders. A weak embedding $\pi$ from $\mathbb{P}$ into $\mathbb{Q}$ is \emph{dense} iff $\pi$ is dense as an embedding from $w(\mathbb{P})$ into $w(\mathbb{Q})$.
\end{defi}

\begin{rem}\label{dwefor}
Let $\mathbb{P}$ and $\mathbb{Q}$ be preorders. By Remark \ref{sepeq} and Fact \ref{defor}, if there is a dense weak embedding from $\mathbb{P}$ into $\mathbb{Q}$, then $\mathbb{P}$ and $\mathbb{Q}$ are forcing equivalent.
\end{rem}

\begin{defi}
If $\mathbb{P} = (P, \leq_{\mathbb{P}})$ is a forcing notion and $p \in P$, we let $g_p (\mathbb{P})$ denote the set $$\{q \in P : p \not \! \! \bot_{\mathbb{P}} \ q\}.$$
\end{defi}

\begin{defi}
Let $\mathbb{P} = (P, \leq_{\mathbb{P}})$ be a forcing notion. A member $p$ of $P$ is an \emph{atom} of $\mathbb{P}$ iff $$\forall q_1 \ \forall q_2 \ ((q_1 \leq_{\mathbb{P}} p \text{ and } q_2 \leq_{\mathbb{P}} p) \implies q_1 \not \! \! \bot_{\mathbb{P}} \ q_2).$$
\end{defi}

\begin{lem}\label{gpgeneric}
If $\mathbb{P} = (P, \leq_{\mathbb{P}})$ is a forcing notion and $p$ is an atom of $\mathbb{P}$, then $g_p (\mathbb{P})$ is a $\mathbb{P}$-generic filter over $V$.
\end{lem}

\begin{proof}
If $D$ is dense in $\mathbb{P}$, then there is $q \in D$ with $q \leq_{\mathbb{P}} p$. Obviously, $q \in g_p (\mathbb{P})$. Therefore $g_p (\mathbb{P})$ is a $\mathbb{P}$-generic subset over $V$. To see that $g_p (\mathbb{P})$ is a filter, let $q_1$ and $q_2$ be members of $g_p (\mathbb{P})$. By the definition of $g_p (\mathbb{P})$, there are $r_1$ and $r_2$ such that 
\begin{itemize}
    \item $r_1 \leq_{\mathbb{P}} q_1$, 
    \item $r_1 \leq_{\mathbb{P}} p$,
    \item $r_2 \leq_{\mathbb{P}} q_2$,
    \item $r_2 \leq_{\mathbb{P}} p$.
\end{itemize}
As $p$ is an atom of $\mathbb{P}$, it must be the case that $r_1 \not \! \! \bot_{\mathbb{P}} \ r_2$, which means $q_1 \not \! \! \bot_{\mathbb{P}} \ q_2$.
\end{proof}

\begin{defi}
A forcing notion $\mathbb{P}$ is \emph{atomic} iff the set of atoms of $\mathbb{P}$ is dense in $\mathbb{P}$.
\end{defi}

\begin{defi}
A forcing notion $\mathbb{P} = (P, \leq_{\mathbb{P}})$ is \emph{atomless} iff no member of $P$ is an atom of $\mathbb{P}$.
\end{defi}

\begin{defi}
Define
\begin{align*}
    C := \ & 2^{< \omega} \text{, and} \\
    \leq_{\mathbb{C}} \ := \ & \{(p, q) : q \subset p\} \text{.}
\end{align*}
Call the forcing notion $\mathbb{C} := (C, \leq_{\mathbb{C}})$ \emph{Cohen forcing}.
\end{defi}

\begin{defi}
Let $\mathbb{P} = (P, \leq_{\mathbb{P}})$ and $\mathbb{Q} = (Q, \leq_{\mathbb{Q}})$ be forcing notions. We say $\mathbb{P}$ is a \emph{regular suborder} of $\mathbb{Q}$, denoted $\mathbb{P} \lessdot \mathbb{Q}$, iff 
\begin{itemize}
    \item $\mathbb{P}$ is a suborder of $\mathbb{Q}$, and
    \item for all $q \in Q$ there is $p \in P$ such that every $p' \leq_{\mathbb{P}} p$ is compatible with $q$ in $\mathbb{Q}$.
\end{itemize}
\end{defi}

\begin{fact}\label{regsubo}
If $\mathbb{P} = (P, \leq_{\mathbb{P}}) \lessdot \mathbb{Q}$, then for every $\mathbb{Q}$-generic filter $g$ over $V$, $g \cap P$ is a $\mathbb{P}$-generic filter over $V$.
\end{fact}

\subsection{Universally Baire Sets and Productive Classes}\label{ss25}

\begin{defi}[Feng-Magidor-Woodin]
Let 
\begin{itemize}
    \item $1 \leq k < \omega$, 
    \item $D \in \mathcal{P}(\mathbb{R}^k)$, and
    \item $T$ and $U$ be trees on ${^{k}{\omega}} \times \lambda$ for some cardinal $\lambda$.
\end{itemize} 
We say $T$ and $U$ \emph{witness} $D$ \emph{is universally Baire} iff 
\begin{itemize}
    \item $D = p[T]$, and
    \item $\Vdash^V_{\mathbb{P}} ``p[U] = \mathbb{R}^k \setminus p[T]"$ for all set forcing notions $\mathbb{P}$.
\end{itemize}
We say $D$ is \emph{universally Baire} iff there are trees $T$ and $U$ witnessing $D$ is universally Baire.
\end{defi}

The definition of universally Baire sets of reals first appeared in \cite[Section 2]{fmw}. 

If $T$ and $U$ witness $D$ is universally Baire, then we can read off $T$ an unambiguous version of $D$, which we denote $D^*$, in any generic extension of $V$. Essentially, we let $(D^*)^{V[g]} = (p[T])^{V[g]}$ for any poset $\mathbb{P} \in V$ and any $\mathbb{P}$-generic filter $g$ over $V$.

Note also that if
\begin{itemize}
    \item $T$ and $U$ witness $D$ is universally Baire, and
    \item $T'$ and $U'$ witness $D$ is universally Baire,
\end{itemize}
then for all set forcing notions $\mathbb{P}$, 
\begin{equation*}
    \Vdash^V_{\mathbb{P}} ``p[T] = p[T']",
\end{equation*}
so the evaluation of $D^*$ is independent of the witnesses for $D$ being universally Baire.

\begin{defi}
Let $\Gamma^{\infty}$ denote the set of all universally Baire sets of reals, i.e. 
\begin{equation*}
    \Gamma^{\infty} := \{D \in \bigcup_{1 \leq k < \omega} \mathcal{P}(\mathbb{R}^k) : D \text{ is universally Baire}\}.
\end{equation*}
\end{defi}

\begin{defi}
Let $\Gamma \subset \bigcup_{1 \leq k < \omega} \mathcal{P}(\mathbb{R}^k)$. We say $\Gamma$ is \emph{productive} iff
\begin{enumerate}[label=(\arabic*), leftmargin=40pt]
    \item $\Gamma \subset \Gamma^{\infty}$,
    \item $\Gamma$ is closed under complements, i.e. for all $k < \omega$, if $D \in \Gamma \cap \mathcal{P}(\mathbb{R}^{k+1})$, then $\mathbb{R}^{k+1} \setminus D \in \Gamma$,
    \item $\Gamma$ is closed under projections, i.e. for all $k < \omega$, if $D \in \Gamma \cap \mathcal{P}(\mathbb{R}^{k+2})$, then
    \begin{equation*}
        \exists^{\mathbb{R}} D := \{\Vec{x} \in \mathbb{R}^{k+1} : \exists y \ (\Vec{x}^{\frown}(y) \in D)\} \in \Gamma,
    \end{equation*}
    and
    \item the closure of $\Gamma$ under projections is preserved by set forcing notions in a strong way: for all $k < \omega$, if $D \in \Gamma \cap \mathcal{P}(\mathbb{R}^{k+2})$, then
    \begin{equation*}
        (\exists^{\mathbb{R}} D)^* = \{\Vec{x} \in \mathbb{R}^{k+1} : \exists y \ (\Vec{x}^{\frown}(y) \in D^*)\}
    \end{equation*}
    in all generic extensions of $V$.
\end{enumerate}
\end{defi}

\begin{lem}\label{prod}
Let $\Gamma  = \bigcup_{1 \leq k < \omega} P(\mathbb{R}^{k}) \cap L(\Gamma, \mathbb{R})$ be productive, $D \in \Gamma$, and $\phi$ be a projective formula. If $\Vec{s} \in {^{< \omega}{\mathbb{R}}}$ and $arity(\phi) = dom(\Vec{s}) + 1$, then 
\begin{equation*}
    V \models \phi(\Vec{s}, D) \iff \ \Vdash^V_{\mathbb{P}} \phi(\Vec{s}, D^*)
\end{equation*}
for all set forcing notions $\mathbb{P}$.
\end{lem}
\begin{proof}
By induction on the length of $\phi$.
\end{proof}

\subsection{Generic Iterations}

We want to first define a fragment of $\mathsf{ZFC}$ rich enough to 
\begin{itemize}
    \item allow for basic analyses of generic ultrapowers, and
    \item be preserved by the generic ultrapowers we will be using.
\end{itemize}
To this end, we follow \cite[Section 3.1]{woodin}.

\begin{defi}[Woodin]\label{nota1}
Let $\mathsf{ZFC}^*$ be the conjunction of $$\mathsf{ZFC} - \mathsf{Replacement} - \mathsf{Powerset}$$ and the following schema:
\begin{quote}
    ``Given any nonempty class of functions $\mathcal{R}$ with
    \begin{itemize}
        \item $dom(f) < \omega_1$ and
        \item $f \restriction \alpha \in \mathcal{R}$
    \end{itemize}
    for all $f \in \mathcal{R}$ and all $\alpha \in dom(f)$, there is $\beta \leq \omega_1$ and a function $g$ with domain $\beta$ such that
    \begin{itemize}
        \item $g \not \in \mathcal{R}$,
        \item for all $\gamma < \beta$, $g \restriction \gamma \in \mathcal{R}$, and
        \item if $\beta = \gamma + 1$, then $g \restriction \gamma$ is $\subset$-maximal in $\mathcal{R}$.''
    \end{itemize}
\end{quote}
\end{defi}

Informally, the schema found in the block quote in Definition~\ref{nota1} says that every tree of height at most $\omega_1$ has a path.

\begin{defi}\label{def257}
For $\alpha \leq \omega_1$, we say a class $$\mathfrak{C} = \langle \bar{N}_i = (N_i; \tilde{\in}_i, I_i), \sigma_{ij} : i \leq j < \alpha \rangle$$ is a \emph{generic iteration} iff 
\begin{itemize}
    \item $\tilde{\in}_0$ is a binary relation on $N_0$ interpreting $\ulcorner \in \urcorner$ in the language of set theory,
    \item $I_0$ is a unary relation on $N_0$,
    \item $(N_0; \tilde{\in}_0) \models \mathsf{ZFC}^*$,
    \item $\bar{N}_0 \models ``I_0$ is a normal uniform ideal on $\omega_1"$, 
    \item for all $i < \alpha$, $\sigma_{ii} = id_{N_i}$,
    \item for all $i < \alpha$ such that $i + 1 < \alpha$, there is $g_{i}$ a $\mathcal{P}(\omega_1)^{\bar{N}_{i}} / I$-generic filter over $N_{i}$ such that 
    \begin{itemize}[label=$\circ$]
        \item $N_{i + 1} = Ult(N_{i}, g_{i})$, and
        \item $\sigma_{i(i + 1)} : N_{i} \longrightarrow N_{i + 1}$ is the corresponding ultrapower embedding,
    \end{itemize}
    \item for all limit ordinals $k < \alpha$, $(\bar{N}_{k}, \langle \sigma_{ik} : i < k \rangle)$ is the direct limit of $\langle \bar{N}_i, \sigma_{ij} : i \leq j < k \rangle$, and
    \item for all $i \leq j \leq k < \alpha$, $\sigma_{ik} = \sigma_{jk} \circ \sigma_{ij}$.
\end{itemize}
In this case, we call $\alpha$ the \emph{length} of $\mathfrak{C}$.
\end{defi}

\begin{defi}
For $\alpha \leq \omega_1$ and any class $\bar{N}$, a class $$\langle \bar{N}_i, \sigma_{ij} : i \leq j < \alpha \rangle$$ is a \emph{generic iteration of} $\bar{N}$ iff 
\begin{itemize}
    \item $\langle \bar{N}_i, \sigma_{ij} : i \leq j < \alpha \rangle$ is a generic iteration, and
    \item $\bar{N} = \bar{N}_0$.
\end{itemize}
\end{defi}

\begin{defi}
A generic iteration $$\langle \bar{N}_i = (N_i; \tilde{\in}_i, I_i), \sigma_{ij} : i \leq j < \alpha \rangle$$ is \emph{well-founded} iff for all $i < \alpha$, $\tilde{\in}_i$ is a well-founded relation on $N_i$.
\end{defi}

Following convention, if $$\langle \bar{N}_i, \sigma_{ij} : i \leq j < \alpha \rangle$$ is a well-founded generic iteration, we shall identify 
\begin{itemize}
    \item each $\bar{N}_i$ with its transitive collapse, and
    \item each $\sigma_{ij}$ with the unique embedding that commutes with $\sigma_{ij}$ and the transitive collapse isomorphisms of $\bar{N}_i$ and $\bar{N}_j$.
\end{itemize}

\begin{defi}
For any class $\bar{N}$, $\bar{N}$ is \emph{generically iterable} iff  
\begin{itemize}
    \item for some $\sigma_{00}$, $\langle \bar{N}_0, \sigma_{00} \rangle$ is a generic iteration of $\bar{N}$, and
    \item every generic iteration of $\bar{N}$ is well-founded.
\end{itemize}
\end{defi}

\begin{defi}
Given a class $\bar{N} = (N; \in, I)$, we say 
\begin{equation*}
    \bar{N} \models ``I \text{ is a precipitous ideal on } \omega_1"
\end{equation*}
iff
\begin{itemize}
    \item for some $\sigma_{00}$, $\langle \bar{N}_0, \sigma_{00} \rangle$ is a generic iteration of $\bar{N}$, and
    \item every generic iteration of $\bar{N}$ of length $2$ is well-founded.
\end{itemize}
\end{defi}

\begin{fact}\label{satipre}
If $\bar{N} = (N; \in, I)$ is such that 
\begin{itemize}
    \item for some $\sigma_{00}$, $\langle \bar{N_0}, \sigma_{00} \rangle$ is a generic iteration of $\bar{N}$, and
    \item $\bar{N} \models ``I \text{ is a saturated ideal on } \omega_1"$,
\end{itemize}
then
\begin{equation*}
    \bar{N} \models ``I \text{ is a precipitous ideal on } \omega_1".
\end{equation*}
\end{fact}

\begin{lem}\label{gisubset}
If $\bar{N}_0$ is generically iterable, $J$ is a normal uniform ideal on $\omega_1$, and $$\langle \bar{N}_i = (N_i; \in, I_i), \sigma_{ij} : i \leq j \leq \omega_1 \rangle$$ is a generic iteration of $\bar{N}_0$, then $I_{\omega_1} \subset J$.
\end{lem}
\begin{proof}
As $\mathrm{NS}_{\omega_1}$ is the smallest normal uniform ideal on $\omega_1$, it suffices to show $I_{\omega_1} \subset \mathrm{NS}_{\omega_1}$. Note that $$C = \{\omega_1^{\bar{N}_i} : i < \omega_1\}$$ is a club in $\omega_1$. Now let $x \in I_{\omega_1}$, so that for some $i < \omega_1$, there is $x_i \in I_i$ for which $\sigma_{i\omega_1}(x_i) = x$. Since 
\begin{equation*}
    \bar{N}_{j+1} \cong Ult(\bar{N}_j, g_j := \{y \in \mathcal{P}(\omega_1)^{\bar{N}_j} : \omega_1^{\bar{N}_j} \in \sigma_{j(j+1)}(y)\})
\end{equation*}
and
\begin{equation*}
    g_j \cap I_j = \emptyset
\end{equation*}
for all $i \leq j < \omega_1$, we have 
\begin{equation*}
    \omega_1^{\bar{N}_j} \not\in \sigma_{i(j+1)}(x_i) = x \restriction (\omega_1^{\bar{N}_{j+1}}) 
\end{equation*}
for all $i \leq j < \omega_1$, whence $$x \cap (C \setminus \omega_1^{\bar{N}_i}) = \emptyset.$$ This means $x \in \mathrm{NS}_{\omega_1}$, and we are done.
\end{proof}

\begin{lem}\label{itercopy}
Suppose 
\begin{itemize}
    \item for some $\sigma_{00}$, $\langle \bar{N}_0, \sigma_{00} \rangle$ is a generic iteration of $\bar{N} = (N; \in, I)$,
    \item for some $\pi_{00}$, $\langle \bar{M}_0, \pi_{00} \rangle$ is a generic iteration of $\bar{M} = (M; \in, I)$,
    \item $\bar{N} \in M$,
    \item $\bar{N} \models ``\text{every } {\omega_1^{\bar{M}}} \text{-sequence is a set}"$,
    \item $N$ contains all maximal antichains of $(\mathcal{P}(\omega_1) / I)^{M}$.
\end{itemize}
Then for each generic iteration $$\langle \bar{N}_i = (N_i; \in, I_i), \sigma_{ij} : i \leq j \leq \alpha \rangle$$ of $\bar{N}$, there is a unique generic iteration $$\langle \bar{M}_i = (M_i; \in, I_i), \pi_{ij} : i \leq j \leq \alpha \rangle$$ of $\bar{M}$ such that for all $i \leq j \leq \alpha$,
\begin{itemize}
    \item $\pi_{0i}(N) = N_i$,
    \item $\bar{N}_i \models ``\text{every } {\omega_1^{\bar{M}_i}} \text{-sequence is a set}"$,
    \item $N_i$ contains all maximal antichains of $(\mathcal{P}(\omega_1) / I_i)^{M_i}$, and
    \item $\pi_{ij} \restriction N_i = \sigma_{ij}$.
\end{itemize}
\end{lem}
\begin{proof}
By induction on $\gamma$.
\end{proof}

Lemma~\ref{itercopy} is a modified version of \cite[Lemma 1.5]{larson}: instead of requiring $N$ to contain $\mathcal{P}(\mathcal{P}(\omega_1) / I)^{M}$, we only require it to contain all the maximal antichains of $(\mathcal{P}(\omega_1) / I)^{M}$. These two lemmas share the same proof.

\begin{lem}\label{iterlength}
Suppose
\begin{itemize}
    \item $M$ is a transitive model of $\mathsf{ZFC}$,
    \item $\bar{M} := (M; \in, J) \models ``J$ is a precipitous ideal on $\omega_1"$, and
    \item for some $\alpha \in \omega_1^V \cap M$, $$\mathfrak{C} = \langle \bar{M}_i, \pi_{ij} : i \leq j \leq \alpha \rangle$$ is a generic iteration of $\bar{M}$.
\end{itemize}
Then $\mathfrak{C}$ is well-founded.
\end{lem}
\begin{proof}
This follows from \cite[Lemma 1.6]{larson}.
\end{proof}

\begin{lem}\label{lift}
Let $V \subset W$ be transitive models of $\mathsf{ZFC}$, such that 
\begin{itemize}
    \item $V$ is definable in $W$,
    \item $V \models ``\mathrm{NS}_{\omega_1}$ is saturated'', and
    \item $\omega_1^W \in V$.
\end{itemize}
If 
\begin{itemize}
    \item $\lambda \leq \kappa \leq \omega_1^W$, and
    \item in $W$, $$\langle \bar{M}_i = (M_i; \in, J_i), \pi_{ij} : \lambda \leq i \leq j \leq \kappa \rangle$$ is a generic iteration of $(H(\omega_2)^V; \in, \mathrm{NS}_{\omega_1}^V)$
\end{itemize}
then $\pi_{\lambda \kappa}$ lifts to a generic ultrapower map $\pi^{+}_{\lambda \kappa} : V \longrightarrow M^+$, for some inner model $M^+$ of $W$.
\end{lem}
\begin{proof}
With Fact~\ref{satipre} in mind, the lemma follows immediately from applications of Lemmas \ref{itercopy} and~\ref{iterlength} in $W$, with $$(V; \in, \mathrm{NS}_{\omega_1}^V)$$ in place of $\bar{M}$.
\end{proof}

\subsection{\texorpdfstring{$\mathbb{P}_{max}$}{P-max} Forcing}\label{ss27}

We start by overloading what it means to be a generic iteration.

\begin{defi}\label{def269}
A class
\begin{equation*}
    \mathfrak{C} = \langle \bar{N}_i = (N_i; \tilde{\in}_i, I_i, a_i), \sigma_{ij} : i \leq j < \alpha \rangle
\end{equation*}
is a \emph{generic iteration} iff
\begin{itemize}
    \item $a_i \in N_i$ for all $i < \alpha$,
    \item $\langle (N_i; \tilde{\in}_i, I_i), \sigma_{ij} : i \leq j < \alpha \rangle$ is generic iteration in the sense of Definition \ref{def257}, and
    \item $\sigma_{ij}(a_i) = a_j$.
\end{itemize}
\end{defi}

Going forward, unless specified otherwise or under clear context, the term ``generic iteration'' will be used in the sense of Definition \ref{def269}.

\begin{defi}\label{pmax}
The conditions of $\mathbb{P}_{max}$ are exactly the structures $\bar{N} = (N; \in, I, a)$ such that
\begin{itemize}
    \item $N$ is countable and transitive,
    \item $\bar{N} \models \mathsf{ZFC}^* + \mathsf{MA}(\omega_1)$,
    \item $I \subset N$,
    \item $x \cap I \in N$ for all $x \in N$,
    \item $\bar{N} \models ``I$ is a normal uniform ideal on $\omega_1"$, 
    \item $a \in N$
    \item $\bar{N} \models ``a \subset \omega_1$ and $\omega_1 = \omega_1^{L[a, x]}$ for some real $x"$, and
    \item $(N; \in, I)$ is generically iterable.
\end{itemize}
Let $\leq_{\mathbb{P}_{max}}$ be a binary relation on the conditions of $\mathbb{P}_{max}$, such that $$\bar{M} = (M; \in, J, b) \leq_{\mathbb{P}_{max}} \bar{N} = (N; \in, I, a)$$ iff one of the following conditions hold:
\begin{itemize}
    \item $\bar{M} = \bar{N}$, or 
    \item $\bar{N} \in M$ and
    \begin{align*}
        \bar{M} \models `` & \text{there is a generic iteration } \\ 
        & \langle \bar{N}_i = (N_i; \in, I_i, a_i), \sigma_{ij} : i \leq j \leq \omega_1 \rangle \\
        & \text{of } \bar{N} \text{ such that} \\
        & a_{\omega_1} = b \text{ and } J \cap N_{\omega_1} = I_{\omega_1}".
    \end{align*}
\end{itemize}
We can easily check that $\leq_{\mathbb{P}_{max}}$ is a partial ordering.
\end{defi}

Although the theory of $\mathbb{P}_{max}$ forcing is replete with remarkable combinatorial arguments, one need not understand these arguments to appreciate our $\mathbb{P}_{max}$-related work in Subsection \ref{ss43}. As such, we present only the following curious lemma.

\begin{lem}\label{ho2}
Assume $\mathrm{NS}_{\omega_1}$ is saturated, $MA(\omega_1)$ holds, $2^{\omega_1} = \omega_2$, and $A \subset \omega_1$ is such that $\omega_1^{L[A]} = \omega_1$. Then $$\Vdash_{Col(\omega, \omega_2)} (H(\omega_2)^{\dot{V}}; \in, \mathrm{NS}_{\omega_1}^{\dot{V}}, A) \in \mathbb{P}_{max}.$$
\end{lem}
\begin{proof}
First, note that $|H(\omega_2)| = \aleph_2$, so $$\Vdash_{Col(\omega, \omega_2)} ``H(\omega_2)^{\dot{V}} \text{ is countable}".$$ Next, we can invoke Lemma~\ref{lift} with $V^{Col(\omega, \omega_2)}$ in place of $W$, to give us the generic iterability of $$(H(\omega_2)^{V}; \in, \mathrm{NS}_{\omega_1}^V)$$ in $V^{Col(\omega, \omega_2)}$. It is easy to check that $$(H(\omega_2)^{V}; \in, \mathrm{NS}_{\omega_1}^V, A)$$ satisfies the other prerequisites (as per Definition \ref{pmax}) of being a $\mathbb{P}_{max}$ condition in $V^{Col(\omega, \omega_2)}$. 
\end{proof}

Lemma \ref{ho2} is a first step towards deriving the $Col(\omega, \omega_2)$-name $\dot{p}$ in Fact \ref{fact434}. It is also why the structure 
\begin{equation*}
    (H(\omega_2); \in, \mathrm{NS}_{\omega_1}, A)
\end{equation*}
is instrumental in the proof of Theorem \ref{notion2}.

\section{Forcing with Language Fragments}\label{setupsec}

In model theory, a Henkin construction involves building a model of a theory over a language, from terms of that language. When such a construction is unequivocally guided by a given theory, we can safely identify the resulting term model with said theory. As such, we have the following viable means of proving the existence of a object with property $P$: 
\begin{enumerate}
    \item Translate $P$ into a specification $S$ for a theory, such that the unique term model of any theory satisfying $S$ has property $P$.
    \item Prove that a theory satisfying $S$ exists.
\end{enumerate}
In a similar fashion, we can force the existence of an object with property $P$ by forcing the existence of a theory satisfying $S$. Naturally, this leads to forcing notions with conditions being fragments of the language over which a theory satisfying $S$ is defined.

The idea of forcing models of a theory has been studied by model theorists --- notably, Robinson and Barwise --- since the 1970s (see e.g. \cite{keisler}), under the label \emph{model-theoretic forcing}. More recently, set theorists have leveraged on model-theoretic forcing to generate conditions of the forcing notions used in various relative consistency proofs. Some examples include \cite{doebler} and \cite{lforcing}. However, these forcing conditions involve are highly complicated and specialised structures, and it is not immediately clear how much of the analysis of one forcing notion can be recycled in the analysis of another. 

Asper\'{o} and Schindler are perhaps the first to present a construction with language fragments as forcing conditions, in their seminal work \cite{schindler}. It quickly became clear that this construction generalises well to extend model-theoretic forcing, allowing us to obtain models outside $V$ which are generic over $V$. Streamlining and modularising the analysis and construction of forcing notions similar to the Asper\'{o}-Schindler ones thus seems like a useful proposition.

This section details a framework for constructing forcing notions with fragments of a language $\mathcal{L}$ as conditions, based on specifications of a theory over $\mathcal{L}$. To state these specifications, a ``meta-language'' dependent on $\mathcal{L}$ is required. Our goal is to ensure that the generic filters of each forcing notion produced indeed give rise to theories satisfying the given specifications. We will make precise the relevant technical terms and concepts as we build our framework over the subsequent three subsections. 

The main result in this section is Lemma \ref{uni}, which is stated and proven in greater generality than is needed for our framework. From Lemma \ref{uni} we derive Lemma \ref{main2}, the primary workhorse of the entire paper.

\subsection{General Languages and Meta-languages}

The initial step in the development of our framework involves the ability to potentially interpret any set as a language. 

\begin{defi}\label{neg}
The canonical negation function $\neg$ on $V$ is defined as follows. 
\begin{align*}
    \neg x := \neg(x) = 
    \begin{cases}
        y & \text{if } x = \ulcorner \neg y \urcorner \text{ for some } y \\
        \ulcorner \neg x \urcorner & \text{otherwise}.
    \end{cases}
\end{align*}
\end{defi}

Basically, $\neg$ takes a member of $\mathcal{L}$ as input, and check whether it is a string with first (leftmost) character $\ulcorner \neg \urcorner$. If so, it removes the leading $\ulcorner \neg \urcorner$; otherwise, it ``casts'' the input as a string (mapping the input to a string of length 1 containing the input as the only character, if the input is not already a string) and prepend $\ulcorner \neg \urcorner$ to the result. For ease of argument, we identify the string containing a single character $x$ with $x$ itself. 

\begin{rem}\label{rem32}
Note that $\neg$ is $\Delta_0$-definable with a single parameter $\ulcorner \neg \urcorner$, which we assume is in every $\mathfrak{A}$ satisfying
\begin{itemize}
    \item $\mathfrak{A}$ is a structure in a possibly expanded language of set theory, and
    \item $\mathfrak{A}$ is a ``sufficiently transitive'' (see \ref{343} of Definition \ref{lsuitable} for a formal treatment of ``sufficiently transitive'') model of a sufficiently strong set theory.
\end{itemize} 
So the definition of $\neg$ is absolute for all such structures $\mathfrak{A}$.
\end{rem}

We sometimes abuse notation and use $\neg$ the function and $\ulcorner \neg \urcorner$ the first-order logical symbol interchangeably. However, we take special care to distinguish them wherever is crucial in our definitions and arguments.

\begin{defi}
A set $\mathcal{L}$ is \emph{closed under negation} iff for each $\phi \in \mathcal{L}$, $\neg \phi \in \mathcal{L}$.
\end{defi}

Before we proceed, fix a set $\mathcal{L}$ that is closed under negation and does not contain any variable symbol. We will stick to this $\mathcal{L}$ for the rest of this section.

\begin{defi}\label{lsuitable}
A structure $\mathfrak{A} = (A; \in, \Vec{R})$ in a possibly expanded language of set theory is $\mathcal{L}$\emph{-suitable} iff
\begin{enumerate}[label=(\alph*)]
    \item $\Vec{R}$ is a set of relations on $A$,
    \item\label{342} $\mathfrak{A}$ is a model of a sufficiently strong set theory,
    \item\label{343} $A$ is $\mathfrak{A}$\emph{-finitely transitive}: that is, 
    \begin{equation*}
        \mathfrak{A} \models \text{``} x \text{ is finite''} \implies x \subset A 
    \end{equation*}
    whenever $x \in A$, 
    \item $\mathcal{L} \subset A$, and
    \item $\mathcal{L}$ is $\Pi_1$-definable in the language associated with $\mathfrak{A}$.
\end{enumerate}
\end{defi}

We can think of $\mathfrak{A}$ as a first-order structure expanding on $(A; \in)$, for constants and functions interpreted over the base set $A$ can be represented by relations on $A$. In typical scenarios, $\mathfrak{A}$ is
\begin{itemize}
    \item \emph{either} an expansion of a transitive model of $\mathsf{ZFC - Powerset - Infinity}$,
    \item \emph{or} an elementary substructure of some expansion of a transitive model of \\ $\mathsf{ZFC - Powerset - Infinity}$.
\end{itemize} 
Any such $\mathfrak{A}$ immediately satisfies \ref{342} and \ref{343} of Definition \ref{lsuitable}. 

The requirement for $\mathcal{L}$ to be $\Pi_1$-definable in the language associated with $\mathfrak{A}$ is only there so that the proof of Lemma \ref{uni} can go through given its hypothesis.

\begin{rem}\label{rem35}
Let $\mathfrak{A} = (A; \in, \Vec{R})$ be a $\mathcal{L}$-suitable structure in a possibly expanded language of set theory. Then the following can be deduced from \ref{342} and \ref{343} of Definition \ref{lsuitable}.
\begin{enumerate}[label=(\arabic*)]
    \item\label{351} $H(\omega) \subset A$.
    \item\label{352} For each $x \in A$, it must be the case that
    \begin{equation*}
        \mathfrak{A} \models \text{``} x \text{ is non-empty''} \iff x \text{ is non-empty.}
    \end{equation*}
    \item\label{353} For each $x \in A$, it must be the case that
    \begin{equation*}
        \mathfrak{A} \models \text{``} x \text{ is Dedekind-finite''} \iff x \text{ is finite.}
    \end{equation*}
    \item\label{354} $A$ is \emph{finitely transitive}. That is, for each $x \in A$, it must be the case that
    \begin{equation*}
        x \text{ is finite} \implies x \subset A \text{.}
    \end{equation*}
\end{enumerate}
\end{rem}

Fix a $\mathcal{L}$-suitable $\mathfrak{A} = (A; \in, \Vec{R})$ for the rest of this section.

\begin{defi}
Define $\mathcal{L}^{*}_{\mathfrak{A}}$ to be the language associated with $(A; \{\in, \Vec{R}, E\})$, where $E$ is a unary predicate symbol not occurring in $\Vec{R}$. 
\end{defi}

We want to use $\mathcal{L}^{*}_{\mathfrak{A}}$ to reason about subsets of $\mathcal{L}$. Intuitively, a richer $\mathfrak{A}$ should allow us to formulate more statements about these subsets. Certain subsets of $\mathcal{L}$ are particularly interesting.

\begin{defi}\label{lnice}
A set $\Sigma$ is $\mathcal{L}$-\emph{nice} iff 
\begin{itemize}
    \item $\Sigma \subset \mathcal{L}$, 
    \item for all $\phi \in \mathcal{L}$, 
    \begin{itemize}[label=$\circ$]
        \item $\{\phi, \neg \phi\} \not\subset \Sigma$, and
        \item either $\phi \in \Sigma$ or $\neg \phi \in \Sigma$.
    \end{itemize} 
\end{itemize}
\end{defi}

\begin{defi}
Let $\phi\in \mathcal{L}^{*}_{\mathfrak{A}}$. Define $\mathsf{pos}(\phi)$ to be the formula in $\mathcal{L}^{*}_{\mathfrak{A}}$ resulting from the following operation:
\begin{itemize}
    \item for each subformula $\varphi$ of $\phi$, if $\varphi = \ulcorner \neg E(x) \urcorner$ for some $x \in \mathcal{L}$, then replace $\varphi$ with $\ulcorner E(\neg x) \urcorner$.
\end{itemize}
\end{defi}

\begin{defi}
Let 
\begin{itemize}
    \item $\phi\in \mathcal{L}^{*}_{\mathfrak{A}}$, and
    \item $\nu$ be any subset of a $\mathfrak{A}$-valuation.
\end{itemize}
Then $\nu^*(\phi)$ is defined to be the sentence in $\mathcal{L}^{*}_{\mathfrak{A}}$ resulting from the following operation:
\begin{itemize}
    \item for each $c \in dom(\nu)$, replace every free occurrence of $c$ in $\phi$ with $\nu(c)$.
\end{itemize}
\end{defi}

\begin{defi}\label{models}
Let 
\begin{itemize}
    \item $\phi \in \mathcal{L}^{*}_{\mathfrak{A}}$, and
    \item $\nu$ be a $\mathfrak{A}$-valuation.
\end{itemize} 
We say $X \models^*_{\mathfrak{A}, \nu} \phi$ iff $$(A; \in, \Vec{R}, X \cap A) \models \nu^*(\phi)(\Vec{x}; \in, \Vec{R}, E).$$ 
We say $X \models^*_{\mathfrak{A}} \phi$ iff for every $\mathfrak{A}$-valuation $\nu$, $X \models^*_{\mathfrak{A}, \nu} \phi$.

If $\Gamma \subset \mathcal{L}^*_{\mathfrak{A}}$ then we say
\begin{align*}
    X \models^*_{\mathfrak{A}, \nu} \Gamma \text{ iff } & X \models^*_{\mathfrak{A}, \nu} \phi \text{ for all } \phi \in \Gamma, \text{ and} \\ 
    X \models^*_{\mathfrak{A}} \Gamma \text{ iff } & X \models^*_{\mathfrak{A}} \phi \text{ for all } \phi \in \Gamma.
\end{align*}
\end{defi}

\begin{rem}\label{def2}
Let
\begin{itemize}
    \item $\phi \in \mathcal{L}^{*}_{\mathfrak{A}}$ be a sentence, and
    \item $p \in A \cap \mathcal{P}(A)$.
\end{itemize}
Then by Definition \ref{models}, $p \models^*_{\mathfrak{A}} \phi$ iff 
\begin{equation}\label{eq31}
    (A; \in, \Vec{R}, p) \models \phi(\Vec{x}; \in, \Vec{R}, E).
\end{equation}
Derive $\phi'$ from $\phi$ by replacing every subformula of $\phi$ of the form $\ulcorner E(x) \urcorner$ with $\ulcorner x \in p \urcorner$, so that the symbol $\ulcorner E \urcorner$ does not occur in $\phi'$. It is easy to see that 
\begin{itemize}
    \item the quantification structure of $\phi'$ is identical to that of $\phi$, and
    \item (\ref{eq31}) is semantically equivalent to $$\mathfrak{A} = (A; \in, \Vec{R}) \models \phi'(\Vec{x}; \in, \Vec{R}).$$ 
\end{itemize}
In particular, if $C$ is such that $C = A$ or $C$ occurs in $\Vec{R}$, then $$\{p \in C : p \models^*_{\mathfrak{A}} \phi\}$$ is a subset of $A$ that is definable in the language associated with $\mathfrak{A}$. This definition is absolute for transitive models of $\mathsf{ZFC - Powerset}$.
\end{rem}

Note that for any $\mathcal{L}$-nice $\Sigma$ and any $x \in \mathcal{L}$, $$x \not\in \Sigma \iff \neg x \in \Sigma,$$ so applying $\mathsf{pos}$ to a $\mathcal{L}^{*}_{\mathfrak{A}}$ formula does not alter its meaning with respect to $\mathcal{L}$-nice sets. More formally, we have the next lemma.

\begin{lem}\label{corr}
Let 
\begin{itemize}
    \item $\phi \in \mathcal{L}^{*}_{\mathfrak{A}}$, and
    \item $\nu$ be a $\mathfrak{A}$-valuation.
\end{itemize}
Then for every $\mathcal{L}$-nice $\Sigma$ in every weak outer model of $V$, $$\Sigma \models^*_{\mathfrak{A}, \nu} \phi \iff \Sigma \models^*_{\mathfrak{A}, \nu} \mathsf{pos}(\phi).$$
\end{lem}

\begin{proof}
By induction on the length of $\phi$, while taking note of the following.
\begin{itemize}
    \item $\Sigma$ being $\mathcal{L}$-nice means that for all $x \in \mathcal{L}$, $$\Sigma \models^*_{\mathfrak{A}} \ulcorner \neg E(x) \urcorner \iff \Sigma \models^*_{\mathfrak{A}} \ulcorner E(\neg x) \urcorner.$$ 
    \item If $\phi = \ulcorner \neg \varphi \urcorner$ and $\phi \neq \ulcorner \neg E(x) \urcorner$ for any $x$, then $\mathsf{pos}(\phi) = \ulcorner \neg \mathsf{pos}(\varphi) \urcorner$.
    \item If $\phi = \ulcorner \varphi_1 \wedge \varphi_2 \urcorner$, then $\mathsf{pos}(\phi) = \ulcorner \mathsf{pos}(\varphi_1) \wedge \mathsf{pos}(\varphi_2) \urcorner$.
    \item If $\phi = \ulcorner \exists y \ \varphi \urcorner$, then $\mathsf{pos}(\phi) = \ulcorner \exists y \ \mathsf{pos}(\varphi) \urcorner$.
\end{itemize}
The rest of the details are standard.
\end{proof}

\begin{defi}
If $\phi \in \mathcal{L}^{*}_{\mathfrak{A}}$, we say $\phi$ is $(\mathfrak{A}, \mathcal{L})$\emph{-satisfiable} iff there are $\nu$, $W$ and $X$ such that
\begin{itemize}
    \item $\nu$ is a $\mathfrak{A}$-valuation,
    \item $W$ is a weak outer model of $V$,
    \item $X \in W \cap \mathcal{P}(\mathcal{L)}$, and
    \item $X \models^*_{\mathfrak{A}, \nu} \phi$,
\end{itemize}
in which case the triple $(\nu, W, X)$ is said to \emph{witness the} $(\mathfrak{A}, \mathcal{L})$\emph{-satisfiability of} $\phi$.
\end{defi}

\begin{defi}
For any $\phi \in \mathcal{L}^{*}_{\mathfrak{A}}$, define $\mathsf{set}(\phi)$ to be the pair $(p, q)$ such that
\begin{itemize}
    \item $q = \{x \in \mathcal{L} : \ulcorner E(x) \urcorner \text{ is a subformula of } \phi\}$, and
    \item $p = \{x \in q : \ulcorner (\neg E(x)) \urcorner \text{ is not a subformula of } \phi\}$.
\end{itemize}
\end{defi}

\begin{defi}
Let 
\begin{itemize}
    \item $\phi \in \mathcal{L}^{*}_{\mathfrak{A}}$, and
    \item $\nu$ be a $\mathfrak{A}$-valuation.
\end{itemize}
For any sets $p$ and $q$, we say $\phi$ \emph{is} $\models^*_{\mathfrak{A}, \nu}$\emph{-true for} $(p, q)$ iff 
\begin{itemize}
    \item $p \subset q \subset \mathcal{L}$, and
    \item for all 
    \begin{itemize}[label=$\circ$]
        \item weak outer models $W$ of $V$, and
        \item $X \in W \cap \mathcal{P}(\mathcal{L)}$, 
    \end{itemize}
    $$X \cap q = p \implies X \models^*_{\mathfrak{A}, \nu} \phi.$$
\end{itemize}

We say $\phi$ \emph{is} $\models^*_{\mathfrak{A}}$ \emph{-true for} $(p, q)$ iff for every $\mathfrak{A}$-valuation $\nu$, $\phi$ is $\models^*_{\mathfrak{A}, \nu}$-true for $(p, q)$.

We say $\phi$ \emph{is} $\models^*_{\mathfrak{A}, \nu}$\emph{-true for} $p$ iff $\phi$ is $\models^*_{\mathfrak{A}, \nu}$-true for $(p, p)$.
\end{defi}

For our purposes, being $\models^*_{\mathfrak{A}, \nu}$-true can be too strong a requirement; it is often enough to narrow the scope of our ``test models'' to just $\mathcal{L}$-nice sets. This motivates the following definition.

\begin{defi}
Let 
\begin{itemize}
    \item $\phi \in \mathcal{L}^{*}_{\mathfrak{A}}$, and
    \item $\nu$ be a $\mathfrak{A}$-valuation.
\end{itemize}
For any sets $p$ and $q$, we say $\phi$ \emph{is} $\models^*_{\mathfrak{A}, \nu}$\emph{-nice for} $(p, q)$ iff 
\begin{itemize}
    \item $p \subset q \subset \mathcal{L}$, and
    \item for all 
    \begin{itemize}[label=$\circ$]
        \item weak outer models $W$ of $V$, and
        \item $X \in W \cap \mathcal{P}(\mathcal{L)}$, 
    \end{itemize}
    $$X \text{ is } \mathcal{L} \text{-nice and } X \cap q = p \implies X \models^*_{\mathfrak{A}, \nu} \phi.$$
\end{itemize}

We say $\phi$ \emph{is} $\models^*_{\mathfrak{A}}$ \emph{-nice for} $(p, q)$ iff for every $\mathfrak{A}$-valuation $\nu$, $\phi$ is $\models^*_{\mathfrak{A}, \nu}$-nice for $(p, q)$.

We say $\phi$ \emph{is} $\models^*_{\mathfrak{A}, \nu}$\emph{-nice for} $p$ iff $\phi$ is $\models^*_{\mathfrak{A}, \nu}$-nice for $(p, p)$.
\end{defi}

\begin{defi}\label{lsub}
Let $\mathcal{L}^{*}_{0, \mathfrak{A}}$ consist of all $\Delta_0$ formulas (in a accordance with Definition \ref{def27}) in the language associated with $\mathfrak{A}$. Obviously, $\mathcal{L}^{*}_{0, \mathfrak{A}} \subset \mathcal{L}^{*}_{\mathfrak{A}}$.

Let $\mathcal{L}^{*}_{1, \mathfrak{A}}$ be the smallest $\mathcal{L}'$ satisfying the following conditions:
\begin{itemize}
    \item $\mathcal{L}^{*}_{0, \mathfrak{A}} \subset \mathcal{L}'$,
    \item $\{\ulcorner E(x) \urcorner : x \in \mathrm{Ter}(\mathcal{L}^{*}_{0, \mathfrak{A}})\} \subset \mathcal{L}'$,
    \item if 
    \begin{itemize}[label=$\circ$]
        \item $\phi \in \mathcal{L}'$,
        \item $\ulcorner z \urcorner$ is a variable not bound in $\phi$, and
        \item $\ulcorner p \urcorner$ is either a constant or a variable not bound in $\phi$,
    \end{itemize} 
    then $$\ulcorner \forall z \ ((z \in p \wedge \text{``}\emptyset \neq p \text{ and } p \text{ is Dedekind-finite''}) \implies \phi) \urcorner \in \mathcal{L}',$$
    \item $\mathcal{L}'$ is closed under all zeroth-order logical operations.
\end{itemize}
A first-order formula $\phi$ is $\mathcal{L}^{*}_{\mathfrak{A}}$-$\Delta_0$ iff $\phi \in \mathcal{L}^{*}_{1, \mathfrak{A}}$. 
\end{defi}

We should perhaps highlight the following trivial observations. 
\begin{enumerate}[label=(\Roman*)]
    \item\label{s31I} The statement $$\text{``}\emptyset \neq p \text{ and } p \text{ is Dedekind-finite''}$$ can be expressed as a $\Pi_1$ formula (\emph{a la} Definition \ref{def27}) in the language associated with $\mathfrak{A}$. In fact, the clause $$\text{``} p \text{ is Dedekind-finite''}$$ can be omitted from said statement without loss of generality if 
    \begin{equation*}
        \mathfrak{A} \models \text{``every set is Dedekind-finite'',}
    \end{equation*}
    leaving us with a $\Delta_0$ formula.
    \item $\mathcal{L}^{*}_{1, \mathfrak{A}} \subset \mathcal{L}^{*}_{\mathfrak{A}}$.
\end{enumerate}

\begin{rem}
As $\mathfrak{A}$ is a first-order structure interpreting only relation symbols, the terms occurring in $\mathcal{L}^{*}_{0, \mathfrak{A}}$ (i.e. $\mathrm{Ter}(\mathcal{L}^{*}_{0, \mathfrak{A}})$) are either variables or constant symbols representing members of $A$. In the usual fashion, we 
\begin{itemize}
    \item identify each member of $A$ with its corresponding constant symbol, and
    \item have $\mathfrak{A}$ interpret each constant symbol as its corresponding member of $A$.
\end{itemize}
\end{rem}

\begin{defi}
We define the subset $\mathcal{D}$ of $\mathcal{L}^{*}_{1, \mathfrak{A}}$ to contain formulas of the form $$\ulcorner \bigvee_{i < m} (\bigwedge_{j < n_i} L_{ij}) \urcorner,$$ wherein for every $i < m$ and every $j < n_i$, there is $P_{ij}$ such that
\begin{itemize}
    \item $P_{ij} \in \mathcal{L}^{*}_{0, \mathfrak{A}}$ or $P_{ij} \in \mathcal{L}^{*}_{\mathfrak{A}}$ is of the form $\ulcorner E(x) \urcorner$, and 
    \item $L_{ij} = P_{ij}$ or $L_{ij} = \ulcorner \neg P_{ij} \urcorner$.
\end{itemize}
\end{defi}

\begin{defi}
Let $\phi$ be a $\mathcal{L}^{*}_{\mathfrak{A}}$ formula. Define $\mathrm{QA}(\phi)$ to be the set of all subformulas $\varphi$ of $\phi$ such that 
\begin{itemize}
    \item $\varphi$ starts with a quantifier, and
    \item no prefix of $\varphi$ is a subformula of $\phi$.
\end{itemize}
\end{defi}

\begin{defi}
Given any $\mathcal{L}^{*}_{\mathfrak{A}}$ formula $\phi$, a $\phi$\emph{-max peeling} is a maximal member of $\mathrm{QA}(\phi)$ with respect to set inclusion.
\end{defi}

\begin{rem}\label{rem320}
Notice that if $\phi$ is a $\mathcal{L}^{*}_{\mathfrak{A}}$ formula, then any two distinct $\phi$-max peelings must not overlap in $\phi$.
\end{rem} 

\begin{defi}
Given a $\mathcal{L}^{*}_{\mathfrak{A}}$-$\Delta_0$ formula $\phi$, we say $\phi$ is \emph{safe} iff every occurrence of $\ulcorner E \urcorner$ in $\phi$ lies outside the scope of any quantifier. In particular, every member of $\mathcal{D}$ is a safe $\mathcal{L}^{*}_{\mathfrak{A}}$-$\Delta_0$ formula.
\end{defi}

There are algorithms to convert arbitrary propositional formulas into disjunctive normal forms. Fix one such algorithm, call it $\mathsf{P_1}$. If $\phi$ is a $\mathcal{L}^{*}_{\mathfrak{A}}$-$\Delta_0$ formula, then we can apply $\mathsf{P_1}$ on $\phi$ by viewing each $\phi$-max peeling as a(n atomic) proposition. Have $\mathsf{DNF}_1$ denote the function that takes $\phi$ to the result of this application of $\mathsf{P_1}$. It is always possible to choose $\mathsf{P_1}$ in a way that guarantees
\begin{enumerate}[label=(\Roman*)]
    \setcounter{enumi}{2}
    \item\label{s323} $\mathsf{DNF}_1(\phi) = \mathsf{DNF}_1(\mathsf{DNF}_1(\phi))$ for all $\mathcal{L}^{*}_{\mathfrak{A}}$-$\Delta_0$ formulas $\phi$, and
    \item\label{s324} $\mathsf{DNF}_1$ commutes with substitution of literals modulo double-negation elimination.
\end{enumerate}
For convenience of analysis, we shall do so.

\begin{rem}\label{rem322}
Due to the nature of conversion algorithms such as $\mathsf{P_1}$, whenever $\phi$ is a $\mathcal{L}^{*}_{\mathfrak{A}}$-$\Delta_0$ formula, $\mathsf{DNF}_1(\phi)$ must be logically equivalent to $\phi$. If in addition, $\phi$ is safe, then $\mathsf{DNF}_1(\phi)$ is a formula in $\mathcal{D}$.
\end{rem}

Let $\mathsf{WNF}$ be the function with domain
\begin{equation*}
    \{\phi: \phi \text{ is a } \mathcal{L}^{*}_{\mathfrak{A}}\text{-}\Delta_0 \text{ formula}\} 
\end{equation*} 
defined by the following recursive procedure.

\begin{quote}
    \underline{$\mathsf{Procedure}$ $\mathsf{P_W}$}

    On input $\phi$:
    \begin{enumerate}[label=(\arabic*)]
        \item Set $\phi' := \mathsf{DNF}_1(\phi)$.
        \item If there is no $\phi'$-max peeling containing $\ulcorner E \urcorner$, return $\phi'$.
        \item For each $\phi'$-max peeling $\varphi$ containing $\ulcorner E \urcorner$ (the order does not matter because of Remark \ref{rem320}:
        \begin{enumerate}[label=(F\arabic*), leftmargin=30pt]
            \item Necessarily, for some $\ulcorner p \urcorner$ and $\varphi'$,
                \begin{align*}
                    \varphi = \ulcorner \forall z \ ( & (z \in p \wedge \text{``}\emptyset \neq p \text{ and } p \text{ is Dedekind-finite''}) \\
                    & \implies \varphi') \urcorner \text{.}
                \end{align*}
            \item Replace $\varphi'$ with $\mathsf{P_W}(\varphi')$ in $\phi'$.
        \end{enumerate}
        \item Return $\phi'$.
    \end{enumerate}
\end{quote}

\begin{rem}\label{rem321}
Obviously, $\mathsf{P_W}$ always terminates and returns a $\mathcal{L}^{*}_{\mathfrak{A}}$-$\Delta_0$ formula. Furthermore, routine code tracing with the aid of 
\begin{itemize}
    \item Remark \ref{rem322}, and the fact that
    \item $\mathsf{P_W}$ returns the result of a function call of $\mathsf{DNF}_1$ whenever its base case is fulfilled,
\end{itemize} 
allows us to ascertain that the output of $\mathsf{P_W}$ is always logically equivalent to its input. In other words, 
\begin{equation*}
    \mathsf{WNF}(\phi) \text{ is logically equivalent to } \phi
\end{equation*}
whenever $\phi$ is a $\mathcal{L}^{*}_{\mathfrak{A}}$-$\Delta_0$ formula. We can therefore, without loss of generality, assume every $\mathcal{L}^{*}_{\mathfrak{A}}$-$\Delta_0$ formula we encounter to be a member of $ran(\mathsf{WNF})$.
\end{rem}

Now let $\mathsf{DNF}$ be the function with domain
\begin{equation*}
    \mathrm{FV} := \{\phi: \phi \text{ is a } \mathcal{L}^{*}_{\mathfrak{A}}\text{-}\Delta_0 \text{ formula}\} \times \{\nu : \nu \text{ is a } \mathfrak{A}\text{-valuation}\} 
\end{equation*} 
defined by the following recursive procedure.
\begin{quote}
    \underline{$\mathsf{Procedure}$ $\mathsf{P_0}$}

    On input $(\phi, \nu)$:
    \begin{enumerate}[label=(\arabic*)]
        \item Set $\phi' := \mathsf{DNF}_1(\phi)$.
        \item If there is no $\phi'$-max peeling containing $\ulcorner E \urcorner$, return $\phi'$.
        \item For each $\phi'$-max peeling $\varphi$ containing $\ulcorner E \urcorner$ (the order does not matter because of Remark \ref{rem320}):
        \begin{enumerate}[label=(F\arabic*), leftmargin=30pt]
            \item Necessarily, for some $\ulcorner p \urcorner$ and $\varphi'$,
                \begin{align*}
                    \varphi = \ulcorner \forall z \ ( & (z \in p \wedge \text{``}\emptyset \neq p \text{ and } p \text{ is Dedekind-finite''}) \\
                    & \implies \varphi') \urcorner \text{.}
                \end{align*}
                There are only two possible cases.
                \begin{enumerate}[label=Case \arabic*:, leftmargin=50pt]
                    \item\label{f1c1} $\varphi$ is a conjunct of a disjunct of $\phi$'.
                    \item\label{f1c2} $\ulcorner \neg \varphi \urcorner$ is a conjunct of a disjunct of $\phi$'.
                \end{enumerate}
                For our next step, we consider these two cases separately.
            
            \item\label{step3} \underline{In the event of \ref{f1c1}}
            
            If $\ulcorner p \urcorner$ is a not a free variable, then it must be a constant symbol. In this case, we check if 
            \begin{equation*}
                \mathfrak{A} \models \text{``}\emptyset \neq p \text{ and } p \text{ is Dedekind-finite''.}
            \end{equation*}
            If so, then it follows from \ref{352} to \ref{354} of Remark \ref{rem35} that 
            \begin{itemize}
                \item $\ulcorner p \urcorner^{\mathfrak{A}}$, the interpretation of $\ulcorner p \urcorner$ by $\mathfrak{A}$, is indeed a non-empty finite set, and
                \item $\ulcorner p \urcorner^{\mathfrak{A}} \subset A$;
            \end{itemize}
            we can --- and shall --- thus replace $\varphi$ with 
            \begin{equation*}
                \ulcorner \bigwedge \{\varphi'[z \mapsto a] : a \in \ulcorner p \urcorner^{\mathfrak{A}}\} \urcorner
            \end{equation*}
            in $\phi'$. If not, replace $\varphi$ with $\ulcorner \text{``} \emptyset = \emptyset \text{''} \urcorner$ in $\phi'$. 
        
            Next, consider the case where $\ulcorner p \urcorner$ is a free variable. In this case, we check if
            \begin{equation*}
                \mathfrak{A} \models \text{``}\emptyset \neq \nu(p) \text{ and } \nu(p) \text{ is Dedekind-finite''.}
            \end{equation*}
            If so, then it follows from \ref{352} to \ref{354} of Remark \ref{rem35} that 
            \begin{itemize}
                \item $\nu(p)$ is indeed a non-empty finite set, and
                \item $\nu(p) \subset A$;
            \end{itemize} 
            we shall thus replace $\varphi$ with 
            \begin{equation*}
                \ulcorner p = \nu(p) \wedge \bigwedge \{\varphi'[z \mapsto a] : a \in \nu(p)\} \urcorner
            \end{equation*}
            in $\phi'$. Otherwise, replace $\varphi$ with 
            \begin{equation*}
                \ulcorner \neg \text{``}\emptyset \neq p \text{ and } p \text{ is Dedekind-finite''} \urcorner
            \end{equation*}
            in $\phi'$.
            
            \underline{In the event of \ref{f1c2}}
            
            If $\ulcorner p \urcorner$ is a not a free variable, then it must be a constant symbol. In this case, we check if 
            \begin{equation*}
                \mathfrak{A} \models \text{``}\emptyset \neq p \text{ and } p \text{ is Dedekind-finite''.}
            \end{equation*}
            If so, then it follows from \ref{352} to \ref{354} of Remark \ref{rem35} that 
            \begin{itemize}
                \item $\ulcorner p \urcorner^{\mathfrak{A}}$, the interpretation of $\ulcorner p \urcorner$ by $\mathfrak{A}$, is indeed a non-empty finite set, and
                \item $\ulcorner p \urcorner^{\mathfrak{A}} \subset A$;
            \end{itemize}
            we can --- and shall --- thus replace $\varphi$ with 
            \begin{equation*}
                \ulcorner \bigvee \{\neg \varphi'[z \mapsto a] : a \in \ulcorner p \urcorner^{\mathfrak{A}}\} \urcorner
            \end{equation*}
            in $\phi'$. If not, replace $\varphi$ with $\ulcorner \text{``} \emptyset \neq \emptyset \text{''} \urcorner$ in $\phi'$. 
        
            Next, consider the case where $\ulcorner p \urcorner$ is a free variable. In this case, we check if
            \begin{equation*}
                \mathfrak{A} \models \text{``}\emptyset \neq \nu(p) \text{ and } \nu(p) \text{ is Dedekind-finite''.}
            \end{equation*}
            If so, then it follows from \ref{352} to \ref{354} of Remark \ref{rem35} that 
            \begin{itemize}
                \item $\nu(p)$ is indeed a non-empty finite set, and
                \item $\nu(p) \subset A$; 
            \end{itemize}we shall thus replace $\varphi$ with
            \begin{equation*}
                \ulcorner p = \nu(p) \wedge \bigvee \{\neg \varphi'[z \mapsto a] : a \in \nu(p)\} \urcorner
            \end{equation*}
            in $\phi'$. Otherwise, replace $\varphi$ with 
            \begin{equation*}
                \ulcorner \text{``} \emptyset \neq \emptyset \text{''} \urcorner
            \end{equation*}
            in $\phi'$.
        \end{enumerate}
        
        \item With the value of $\phi'$ updated as per Step \ref{step3}, call $\mathsf{P_0}(\phi', \nu)$.
    \end{enumerate}
\end{quote}

\begin{rem}\label{rem324}
That $\mathsf{P_0}$ always terminates and returns a safe $\mathcal{L}^{*}_{\mathfrak{A}}$-$\Delta_0$ formula can be easily verified. Furthermore, since 
\begin{itemize}
    \item Remark \ref{rem322} holds,
    \item $\mathsf{P_0}$ returns the result of a function call of $\mathsf{DNF}_1$ right before termination, and
    \item the atomic nature of each $\phi$-max peeling in the running of $\mathsf{P_1}$ on input $\phi$ means $\mathsf{DNF}_1(\phi)$ cannot be safe if $\phi$ is not,
\end{itemize}
$\mathsf{P_0}$ must return a member of $\mathcal{D}$. We can thus conclude that $\mathsf{DNF}$ is a function from $\mathrm{FV}$ into $\mathcal{D}$.
\end{rem}

The next three propositions can be verified by routine --- if tedious --- applications of mathematical induction, with \ref{s323} and \ref{s324} in mind.

\begin{prop}\label{prop321}
Let $\phi \in ran(\mathsf{WNF})$. Then for any $\mathfrak{A}$-valuation $\nu$,
\begin{equation*}
    \mathsf{pos}(\mathsf{DNF}(\phi, \nu)) = \mathsf{DNF}(\mathsf{pos}(\phi), \nu) \text{.}
\end{equation*}
\end{prop}

\begin{prop}\label{p320}
Let
\begin{itemize}
    \item $X$ be a set in some weak outer model of $V$,
    \item $(\phi, \nu) \in \mathrm{FV}$, and
    \item $\nu'$ be a $\mathfrak{A}$-valuation. 
\end{itemize}
Then
\begin{enumerate}[label=(\arabic*)]
    \item\label{p3201} $\mathsf{DNF}(\nu^*(\phi), \nu')$ is a sentence and $$X \models^*_{\mathfrak{A}, \nu} \phi \iff X \models^*_{\mathfrak{A}} \mathsf{DNF}(\nu^*(\phi), \nu') \text{,}$$
    \item\label{p3202} $X \models^*_{\mathfrak{A}, \nu} \phi \implies X \models^*_{\mathfrak{A}, \nu} \mathsf{DNF}(\phi, \nu)$,
    \item\label{p3203} $X \models^*_{\mathfrak{A}, \nu'} \mathsf{DNF}(\phi, \nu) \implies X \models^*_{\mathfrak{A}, \nu'} \phi$,
\end{enumerate}
\end{prop}

\begin{defi}\label{LSigma1}
A first-order formula $\phi$ is $\mathcal{L}^{*}_{\mathfrak{A}}$-$\Sigma_1$ iff it is of the form $$\ulcorner \exists y_1 \dots \exists y_j \ \phi \urcorner,$$
where 
\begin{itemize}
    \item $j < \omega$,
    \item $\phi$ is $\mathcal{L}^{*}_{\mathfrak{A}}$-$\Delta_0$,
    \item $y_1, \dots, y_j$ are variables not bound in $\phi$.
\end{itemize}
\end{defi}

\begin{defi}\label{LPi2}
A first-order formula $\phi$ is $\mathcal{L}^{*}_{\mathfrak{A}}$-$\Pi_2$ iff it is of the form $$\ulcorner \forall x_1 \dots \forall x_i \ \phi \urcorner,$$
where 
\begin{itemize}
    \item $i < \omega$,
    \item $\phi$ is $\mathcal{L}^{*}_{\mathfrak{A}}$-$\Sigma_1$,
    \item $x_1, \dots, x_i$ are variables not bound in $\phi$.
\end{itemize}
\end{defi}

\begin{rem}\label{snpn}
Analogous to what the classification of general first-order formulas in prenex normal form, we can very naturally build on Definitions \ref{lsub}, \ref{LSigma1} and \ref{LPi2}, and inductively define $\mathcal{L}^{*}_{\mathfrak{A}}$-$\Sigma_n$ and $\mathcal{L}^{*}_{\mathfrak{A}}$-$\Pi_n$ sentences for all $n < \omega$. The only reason we did not is because our theorems and analyses neither mention nor require formulas outside of $\mathcal{L}^{*}_{\mathfrak{A}}$-$\Pi_2$.  
\end{rem}

Note also that in the definition of $\mathcal{L}^{*}_{\mathfrak{A}}$-$\Delta_0$ formulas, members of $\mathcal{L}^{*}_{0, \mathfrak{A}}$ are regarded, for all practical purposes, as atomic formulas. Further, bounded quantification is limited to finite sets --- more in the spirit of arithmetical bounded quantification than the usual set-theoretic one. This is not merely a cosmetic choice, for extending bounded quantification to countable sets would render Lemma \ref{uni} false, as we shall show in Lemma \ref{lem339}.

\begin{lem}\label{exist}
Let 
\begin{itemize}
    \item $\phi$ be a $\mathcal{L}^{*}_{\mathfrak{A}}$-$\Sigma_1$ formula, and
    \item $(\nu, W, X)$ witness the $(\mathfrak{A}, \mathcal{L})$-satisfiability of $\phi$.
\end{itemize}
Then there are finite sets $p \in A \cap \mathcal{P}(\mathcal{L})$ and $q \in A \cap \mathcal{P}(\mathcal{L})$ such that 
\begin{itemize}
    \item $X \cap q = p$, and
    \item $\phi$ is $\models^*_{\mathfrak{A}, \nu}$-true for $(p, q)$.
\end{itemize}
Moreover, if $\phi$ is a $\mathcal{L}^{*}_{\mathfrak{A}}$-$\Delta_0$ sentence, then $\phi$ is $\models^*_{\mathfrak{A}}$-true for $(p, q)$.
\end{lem}

\begin{proof}
By induction on the length of $\phi$. We work in $W$ throughout.
\begin{enumerate}[label=Case \arabic*:, leftmargin=50pt]
    \item $\phi$ is $\mathcal{L}^{*}_{\mathfrak{A}}$-$\Delta_0$. By \ref{p3201} of Proposition \ref{p320}, $$X \models^*_{\mathfrak{A}} \mathsf{DNF}(\nu^*(\phi), \nu')$$ for some (in fact, any) $\mathfrak{A}$-valuation $\nu'$. Next, since $\mathsf{DNF}(\nu^*(\phi), \nu') \in \mathcal{D}$ according to Remark \ref{rem324}, there is a disjunct $\varphi$ of $\mathsf{DNF}(\nu^*(\phi), \nu')$ for which $X \models^*_{\mathfrak{A}} \varphi$. 
    
    Let $(p, q)$ be $\mathsf{set}(\varphi)$. Then $p$ and $q$ are finite sets with $p \subset q \subset \mathcal{L}$. Since $(A; \in)$ models enough set theory and $\mathcal{L} \subset A$, we too have $\{p, q\} \subset A$. As $\varphi$ has all occurrences of literals over $\{E\}$ being conjuncts, due to 
    \begin{itemize}
        \item $\varphi$ being a disjunct of $\mathsf{DNF}(\nu^*(\phi), \nu')$,
        \item the definition of $\mathcal{D}$, and 
        \item the fact that $\mathsf{DNF}(\nu^*(\phi), \nu') \in \mathcal{D}$, 
    \end{itemize} 
    we must have $$X \cap q = p.$$ Fix $X' \subset \mathcal{L}$ in any weak outer model of $V$. It is now clear that $$X' \cap q = p \implies X' \models^*_{\mathfrak{A}} \varphi \implies X' \models^*_{\mathfrak{A}} \mathsf{DNF}(\nu^*(\phi), \nu').$$ By \ref{p3201} of Proposition \ref{p320} again, $$X' \models^*_{\mathfrak{A}, \nu} \phi \text{.}$$ We have thus shown that $\phi$ is $\models^*_{\mathfrak{A}, \nu}$-true for $(p, q)$.

    If $\phi$ is a sentence, then whenever $\nu''$ is a $\mathfrak{A}$-valuation,
    \begin{equation*}
        \nu^*(\phi) = (\nu'')^*(\phi) \text{,}
    \end{equation*}
    and so
    \begin{equation*}
        \mathsf{DNF}(\nu^*(\phi), \nu') = \mathsf{DNF}((\nu'')^*(\phi), \nu') \text{.}
    \end{equation*}
    But this means $\phi$ is $\models^*_{\mathfrak{A}, \nu''}$-true for $(p, q)$ for all $\mathfrak{A}$-valuations $\nu''$, or equivalently, $\phi$ is $\models^*_{\mathfrak{A}}$-true for $(p, q)$.
    \item $\phi = \ulcorner \exists y \ \phi' \urcorner$ for some $y$ and $\phi'$. Then there must be a $\mathfrak{A}$-valuation $\nu'$ that agrees with $\nu$ on the free variables of $\phi$, for which $$X \models^{*}_{\mathfrak{A}, \nu'} \phi'.$$ By the induction hypothesis, there are finite sets $p \in A \cap \mathcal{P}(\mathcal{L})$ and $q \in A \cap \mathcal{P}(\mathcal{L})$ satisfying 
    \begin{itemize}
        \item $X \cap q = p$ and
        \item $\phi'$ is $\models^*_{\mathfrak{A}, \nu'}$-true for $(p, q)$. 
    \end{itemize}
    Since $(\nu')^*(\phi')$ logically implies $\nu^*(\phi)$, it must also be that $\phi$ is $\models^{*}_{\mathfrak{A}, \nu}$-true for $(p, q)$. \qedhere
\end{enumerate}
\end{proof}

For any free variable $x$ and any $\mathcal{L}^{*}_{\mathfrak{A}}$ formula $\phi$, it is often a desideratum (if not an imperative) in practice to have $\phi$ explicitly ``guarantee $E(\neg x)$ insofar as $x$ is a member of $\mathcal{L}$'', whenever $\ulcorner \neg E(x) \urcorner$ occurs in $\phi$. Towards this end, we are incentivised to augment $\phi$ with a suitable gadget.

\begin{defi}\label{check}
Let $\phi$ be a $\mathcal{L}^{*}_{\mathfrak{A}}$-$\Delta_0$ formula. Define $\mathsf{check}(\phi)$ to be the unique result of replacing every atomic subformula $\varphi$ of $\phi$ satisfying
\begin{equation*}
    \varphi = \ulcorner \neg E(x) \urcorner \text{ for some variable } x 
\end{equation*}
with 
\begin{equation*}
    \ulcorner (\varphi \wedge (\text{``} x \in \mathcal{L} \text{''} \implies \text{``} E(\neg x) \text{''})) \urcorner
\end{equation*}
in $\phi$.
\end{defi}

\begin{rem}\label{shcs}
In the definition of $\mathsf{check}$, $\text{``} E(\neg x) \text{''}$ is a shorthand for both $$\ulcorner \exists z \ (\text{``} z = \neg (x) \text{''} \wedge E(z)) \urcorner$$ and $$\ulcorner \forall z \ (\text{``} z = \neg (x) \text{''} \implies E(z)) \urcorner,$$ where $\neg (\cdot)$ is the negation function on $V$ (see Definition \ref{neg}), so that ``$z = \neg (x)$'' is expressible as a formula in $\mathcal{L}^{*}_{0, \mathfrak{A}}$ by Remark \ref{rem32}. Particularly, since 
\begin{itemize}
    \item $\text{``} E(\neg x) \text{''}$ means exactly $$\ulcorner \forall z \ (\text{``} z = \neg (x) \text{''} \implies E(z)) \urcorner,$$ and
    \item $\mathfrak{A}$ is a model of a sufficiently strong set theory,
\end{itemize} 
we have that
\begin{align*}
    & X \models^*_{\mathfrak{A}, \nu} \text{``} E(\neg x) \text{''} \iff \\
    & X \models^*_{\mathfrak{A}, \nu} \ulcorner \exists p \ \forall z \ ((z \in p \wedge \text{``}\emptyset \neq p \text{ and } p \text{ is Dedekind-finite''}) \\
    & \mspace{125mu} \implies (\text{``} z = \neg (x) \text{''} \implies E(z))) \urcorner
\end{align*}
in case 
\begin{itemize}
    \item $X$ is a set in some weak outer model of $V$,
    \item $\nu$ is a $\mathfrak{A}$-valuation, and
    \item $\ulcorner p \urcorner$ is a variable outside $\{\ulcorner x \urcorner, \ulcorner z \urcorner\}$ bound in neither (the formal statement of)
    \begin{equation*}
        \text{``} z = \neg (x) \text{''}
    \end{equation*}
    nor (the formal statement of)
    \begin{equation*}
        \text{``}\emptyset \neq p \text{ and } p \text{ is Dedekind-finite''.}
    \end{equation*}
\end{itemize} 
As an implication, if $\ulcorner \exists x_1 \dots \exists x_n \ \phi \urcorner$ is a $\mathcal{L}^{*}_{\mathfrak{A}}$-$\Sigma_1$ formula with $\phi \in ran(\mathsf{WNF})$, then we can show by induction on the complexity of $\phi$, noting
\begin{itemize}
    \item the recursive structure of $\mathsf{P_W}$, 
    \item how $\mathsf{DNF_1}$ is used in $\mathsf{P_W}$, and
    \item Remark \ref{rem322}, 
\end{itemize}
that there exists a $\mathcal{L}^{*}_{\mathfrak{A}}$-$\Sigma_1$ formula $\varphi$ satisfying $$X \models^*_{\mathfrak{A}, \nu} \ulcorner \exists x_1 \dots \exists x_n \ \mathsf{check}(\mathsf{pos}(\phi)) \urcorner \iff X \models^*_{\mathfrak{A}, \nu} \varphi$$ for every 
\begin{itemize}
    \item $\mathfrak{A}$-valuation $\nu$, and
    \item set $X$ found in a weak outer model of $V$.
\end{itemize} 
We may thus assume, without loss of generality, that formulas of the form $$\ulcorner \exists x_1 \dots \exists x_n \ \mathsf{check}(\mathsf{pos}(\phi)) \urcorner$$ are $\mathcal{L}^{*}_{\mathfrak{A}}$-$\Sigma_1$, as long as
\begin{itemize}
    \item $n < \omega$, and 
    \item $\phi \in ran(\mathsf{WNF})$.
\end{itemize}
\end{rem}

The fact below can be derived from definitions through straightforward variable tracing.

\begin{fact}\label{fact330}
The functions $\mathsf{check}$ and $\mathsf{DNF}$ commute. To be precise, let
\begin{itemize}
    \item $\phi$ be a $\mathcal{L}^{*}_{\mathfrak{A}}$-$\Delta_0$ formula, and
    \item $\nu$ be a $\mathfrak{A}$-valuation.
\end{itemize}
Then $\mathsf{check}(\mathsf{DNF}(\phi, \nu))$ and $\mathsf{DNF}(\mathsf{check}(\phi)$ are logically equivalent. As such, without loss of generality, we can assume
\begin{equation}\label{eq31'}
    \mathsf{check}(\mathsf{DNF}(\phi, \nu)) = \mathsf{DNF}(\mathsf{check}(\phi), \nu) \text{.}
\end{equation}
In fact, (\ref{eq31'}) already holds if $\phi \in ran(\mathsf{WNF})$.
\end{fact}

\begin{lem}\label{postrue}
Let 
\begin{itemize}
    \item $\phi = \ulcorner \bigwedge \mathrm{S} \wedge \bigwedge \mathrm{T} \urcorner$ for some 
    \begin{itemize}[label=$\circ$]
        \item $S$ is a finite subset of $\mathcal{L}^{*}_{0, \mathfrak{A}}$, and
        \item $T$ is a finite subset of $\mathcal{L}^{*}_{\mathfrak{A}}$ containing only formulas either of the form $\ulcorner E(x) \urcorner$ or of the form $\ulcorner \neg E(x) \urcorner$,
    \end{itemize}
    \item $\nu$ be a $\mathfrak{A}$-valuation, and
    \item $\Sigma \subset \mathcal{L}$ in some weak outer model of $V$.
\end{itemize}
If $\Sigma \models^*_{\mathfrak{A}, \nu} \mathsf{check}(\mathsf{pos}(\phi))$, then $\mathsf{check}(\mathsf{pos}(\phi))$ is $\models^*_{\mathfrak{A}, \nu}$-nice for $\Sigma$.
\end{lem}

\begin{proof}
By the form of $\mathsf{check}(\mathsf{pos}(\phi))$, it is sufficient to prove that whenever 
\begin{itemize}
    \item $\Sigma' \supset \Sigma$,
    \item $\Sigma'$ is $\mathcal{L}$-nice, and
    \item $\varphi$ is a subformula of $\nu^*(\mathsf{check}(\mathsf{pos}(\phi)))$ of the form $\ulcorner E(x) \urcorner$,
\end{itemize}
 $\Sigma \models^*_{\mathfrak{A}} \varphi \iff \Sigma' \models^*_{\mathfrak{A}} \varphi$. 
 
That $\Sigma \subset \Sigma'$ means $$\Sigma \models^*_{\mathfrak{A}} \varphi \implies \Sigma' \models^*_{\mathfrak{A}} \varphi,$$ so it is sufficient to prove $$\Sigma \models^*_{\mathfrak{A}} \ulcorner \neg \varphi \urcorner \implies \Sigma' \models^*_{\mathfrak{A}} \ulcorner \neg \varphi \urcorner.$$ We examine the possible cases below.
\begin{enumerate}[label=Case \arabic*:, leftmargin=50pt]
    \item $x \not\in \mathcal{L}$. Then $x \not\in \Sigma'$, so $\Sigma' \models^*_{\mathfrak{A}} \ulcorner \neg \varphi \urcorner$.
    \item $x \in \mathcal{L}$ and $\varphi$ occurs in $\mathsf{pos}(\phi)$. Then $x$ cannot be a variable symbol. Since $\mathsf{check}(\mathsf{pos}(\phi))$ logically implies $\mathsf{pos}(\phi)$ and $$\Sigma \models^*_{\mathfrak{A}, \nu} \mathsf{check}(\mathsf{pos}(\phi)),$$ also $$\Sigma \models^*_{\mathfrak{A}, \nu} \mathsf{pos}(\phi).$$ By the definition of $\mathsf{pos}$, $\ulcorner \neg \varphi \urcorner$ must not occur in $\mathsf{pos}(\phi)$, so $\varphi$ is a conjunct of $\mathsf{pos}(\phi)$, and $\Sigma \models^*_{\mathfrak{A}} \varphi$. As a result, $$\Sigma \models^*_{\mathfrak{A}} \ulcorner \neg \varphi \urcorner \implies \Sigma' \models^*_{\mathfrak{A}} \ulcorner \neg \varphi \urcorner$$ trivially holds.
    \item $x \in \mathcal{L}$, $\varphi$ does not occur in $\mathsf{pos}(\phi)$ and moreover, $\ulcorner \neg \varphi \urcorner$ does not occur in $\nu^*(\mathsf{check}(\mathsf{pos}(\phi)))$. 
    
    If $\varphi$ occurs in $\nu^*(\mathsf{pos}(\phi))$ then by the same argument as in Case 2, $\Sigma \models^*_{\mathfrak{A}} \varphi$ and we have our desired conclusion. Otherwise, $\varphi$ occurs as a subformula of $\nu^*(\varphi')$ for some $\varphi'$ of the form $$\ulcorner (\text{``} x \in \mathcal{L} \text{''} \implies \text{``} E(\neg x) \text{''}) \urcorner \text{,}$$ where $x$ is a variable. By the fact that $$\Sigma \models^*_{\mathfrak{A}, \nu} \mathsf{check}(\mathsf{pos}(\phi)),$$ $\Sigma \models^*_{\mathfrak{A}} \ulcorner \neg \varphi \urcorner$ means $x \not \in \mathcal{L}$, in which case also $\Sigma' \models^*_{\mathfrak{A}} \ulcorner \neg \varphi \urcorner$.
    \item $x \in \mathcal{L}$, $\varphi$ does not occur in $\mathsf{pos}(\phi)$ and moreover, $\ulcorner \neg \varphi \urcorner$ occurs in \\ $\nu^*(\mathsf{check}(\mathsf{pos}(\phi)))$. 
    
    Here $\ulcorner \neg \varphi \urcorner$ must occur in $\nu^*(\mathsf{pos}(\phi))$, so by the definition of $\mathsf{check}$, $$\ulcorner (\text{``} x \in \mathcal{L} \text{''} \implies E(\neg x)) \urcorner$$ occurs in $\nu^*(\mathsf{check}(\mathsf{pos}(\phi)))$. Now $$\Sigma \models^*_{\mathfrak{A}, \nu} \mathsf{check}(\mathsf{pos}(\phi))$$ implies either $x \not \in \mathcal{L}$ or $\neg x \in \Sigma$. 
    \begin{enumerate}[label=Subcase \arabic*:, leftmargin=60pt]
        \item $x \not \in \mathcal{L}$. Then clearly $x \not \in \Sigma'$, so $\Sigma' \models^*_{\mathfrak{A}} \ulcorner \neg \varphi \urcorner$.
        \item $\neg x \in \Sigma$. Then $\neg x \in \Sigma'$ because $\Sigma \subset \Sigma'$. Since $\Sigma'$ is $\mathcal{L}$-nice, it must be that $x \not \in \Sigma'$. We thus also have $\Sigma' \models^*_{\mathfrak{A}} \ulcorner \neg \varphi \urcorner$. \qedhere
    \end{enumerate}
\end{enumerate}
\end{proof}

\begin{rem}\label{rem328}
It will be useful to note that if $\varphi$ is a member of $\mathcal{D}$, then every disjunct $\phi$ of $\varphi$ takes the form required of $\phi$ in (the hypothesis of) Lemma \ref{postrue}.
\end{rem}

\subsection{Forcing Notions and Universal Sentences}

Consider a forcing notion with conditions fragments of $\mathcal{L}$, ordered by reverse inclusion. Then genericity over $\mathbb{P}$ naturally gives us a subset of $\mathcal{L}$. We want to analyse this subset using $\mathcal{L}^{*}_{\mathfrak{A}}$.

\begin{defi}
A pair $(\mathfrak{A} = (A; \in, \Vec{R}), \mathbb{P} = (P, \leq_P))$ is \emph{good for} $\mathcal{L}$ iff 
\begin{itemize}
    \item $\mathfrak{A}$ is $\mathcal{L}$-suitable,
    \item $\emptyset \neq P \subset \mathcal{P}(\mathcal{L}) \cap A$, and
    \item for all $\{p, q\} \subset P$, $p \leq_P q$ iff $q \subset p$.
\end{itemize}
\end{defi}

For the rest of this subsection, we fix a forcing notion $\mathbb{P} = (P, \leq_P))$ for which
\begin{itemize}
    \item $\mathbb{P}$ is $\Sigma_1$-definable in the language associated with $\mathfrak{A}$, and
    \item $(\mathfrak{A}, \mathbb{P})$ is good for $\mathcal{L}$.
\end{itemize}
Similar in motivation to the final bullet point in Definition \ref{lsuitable}, the requirement for $\mathbb{P}$ to be $\Sigma_1$-definable in the language associated with $\mathfrak{A}$ is only there so that the proof of Lemma \ref{uni} can go through given its hypothesis.

\begin{defi}
For $p \in P$, a $p$\emph{-candidate for} $(\mathfrak{A}, \mathbb{P}, \mathcal{L})$\emph{-universality} is a set $\Sigma$ for which
\begin{itemize}
   \item $p \subset \Sigma$,
   \item for each $x \in [\Sigma]^{<\omega}$, there is $q \in P$ with $p \cup x \subset q$, and 
   \item $\Sigma$ is $\mathcal{L}$-nice.
\end{itemize}
\end{defi}

\begin{defi}
Let $\phi \in \mathcal{L}^{*}_{\mathfrak{A}}$ be a sentence and $p \in P$. We say $\phi$ is $(\mathfrak{A}, \mathbb{P}, \mathcal{L})$\emph{-universal for $p$} iff for all $q \leq_P p$ there is a set $\Sigma$ in some weak outer model of $V$ such that 
\begin{itemize}
    \item $\Sigma$ is a $q$-candidate for $(\mathfrak{A}, \mathbb{P}, \mathcal{L})$-universality, and
    \item $\Sigma \models^*_{\mathfrak{A}} \phi$.
\end{itemize}
We say $\phi$ is $(\mathfrak{A}, \mathbb{P}, \mathcal{L})$\emph{-universal} iff for all $p \in P$, $\phi$ is $(\mathfrak{A}, \mathbb{P}, \mathcal{L})$-universal $p$. 

For $\Gamma \subset \mathcal{L}^{*}_{\mathfrak{A}}$, we say $\Gamma$ is $(\mathfrak{A}, \mathbb{P}, \mathcal{L})$\emph{-universal} (\emph{for $p$}) iff $\phi$ is $(\mathfrak{A}, \mathbb{P}, \mathcal{L})$-universal (for $p$) for all $\phi \in \Gamma$.
\end{defi}

For notational convenience in the proofs to follow, we introduce the following definition.

\begin{defi}
For $p \in P$, let $\mathcal{F}_p$ denote the set 
\begin{equation*}
    \{\Sigma : \Sigma \text{ is a } p \text{-candidate for } (\mathfrak{A}, \mathbb{P}, \mathcal{L}) \text{-universality in some weak outer model of } V\}.
\end{equation*}
\end{defi}

\begin{lem}\label{uni}
Let 
\begin{itemize} 
    \item $W$ be a weak outer model of $V$,
    \item $g \in W$ be a $\mathbb{P}$-$\Sigma_1$-generic filter over $\mathfrak{A}$,
    \item $p \in g$, and
    \item $\phi$ be a $\mathcal{L}^{*}_{\mathfrak{A}}$-$\Pi_2$ sentence which is $(\mathfrak{A}, \mathbb{P}, \mathcal{L})$-universal for $p$.
\end{itemize} 
Then $\bigcup g$ is $\mathcal{L}$-nice and $\bigcup g \models^{*}_{\mathfrak{A}} \phi$.
\end{lem}
\begin{proof}
We prove $\bigcup g \models^{*}_{\mathfrak{A}} \phi$ by induction on the length of $\phi$. The proof that $\bigcup g$ is $\mathcal{L}$-nice will surface as a part of the induction argument.
\begin{enumerate}[label=Case \arabic*:, leftmargin=50pt]
    \item $\phi$ is $\mathcal{L}^{*}_{\mathfrak{A}}$-$\Delta_0$. By way of contradiction, assume $\bigcup g \models^{*}_{\mathfrak{A}} \neg \phi$. Since $\neg \phi$ is also a $\mathcal{L}^{*}_{\mathfrak{A}}$-$\Delta_0$ sentence, Lemma~\ref{exist} tells us there are finite sets $p^{\dagger} \in A \cap \mathcal{P}(\mathcal{L})$ and $q^{\dagger} \in A \cap \mathcal{P}(\mathcal{L})$ such that 
    \begin{itemize}
        \item $\bigcup g \cap q^{\dagger} = p^{\dagger}$, and
        \item $\neg \phi$ is $\models^{*}_{\mathfrak{A}}$-true for $(p^{\dagger}, q^{\dagger})$.
    \end{itemize}
    For each $z \in p^{\dagger} \subset \bigcup g$, pick a $p_z \in g$ such that $z \in p_z$. Since $g$ is a filter, there is some $p^* \in g$ for which $$p^{\dagger} \subset (\bigcup_{z \in p^{\dagger}} p_z) \cup p \subset p^*.$$ If $p' \leq_{P} p^*$, then $p' \leq_{P} p$, so by $\phi$ being $(\mathfrak{A}, \mathbb{P}, \mathcal{L})$-universal for $p$, there is $\Sigma \in \mathcal{F}_{p'}$ with $\Sigma \models^*_{\mathfrak{A}} \phi$. Now, necessarily 
    \begin{itemize}
        \item $p^{\dagger} \subset \Sigma$ and
        \item $\Sigma \cap q^{\dagger} \neq p^{\dagger}$,
    \end{itemize}
    whence $p^{\dagger} \subsetneq \Sigma \cap q^{\dagger}$. By the fact that $\Sigma$ is a $p'$-candidate for $(\mathfrak{A}, \mathbb{P}, \mathcal{L})$-universality, we can find $q \in P$ with $$p' \cup (\Sigma \cap q^{\dagger}) \subset q.$$ As a consequence, $q \leq_P p'$ and $p^{\dagger} \subsetneq q \cap q^{\dagger}$.
    
    We have thus shown that the set $$D_1 := \{q \in P : p^{\dagger} \subsetneq q \cap q^{\dagger}\}$$ is dense below $p^*$ in $\mathbb{P}$. Given the fact that $\mathbb{P}$ is $\Sigma_1$-definable in the language associated with $\mathfrak{A}$, $D_1$ obviously has the same property, so $g$ must meet $D_1$. As $p^* \in g$, we can conclude $g \cap D_1 \neq \emptyset$, and let $q^* \in g \cap D_1$. But then $p^{\dagger} \subsetneq q^* \cap q^{\dagger}$, which implies $$p^{\dagger} \subsetneq q^* \cap q^{\dagger} \subset \bigcup g \cap q^{\dagger} = p^{\dagger},$$ a contradiction. 
    
    \item $\phi$ is $\mathcal{L}^{*}_{\mathfrak{A}}$-$\Sigma_1$ but not $\mathcal{L}^{*}_{\mathfrak{A}}$-$\Delta_0$. Then $\phi$ is of the form $\ulcorner \exists x_1 \dots \exists x_n \ \phi^* \urcorner$ for some 
    \begin{itemize}
        \item $\mathcal{L}^{*}_{\mathfrak{A}}$-$\Delta_0$ formula $\phi^*$,
        \item $n$ such that $1 \leq n < \omega$, and
        \item $\{x_1, \dots, x_n\}$ the set of free variables of $\phi^*$.
    \end{itemize} 
    
    We first show that $\bigcup g$ is $\mathcal{L}$-nice. Obviously, $\bigcup g \subset \mathcal{L}$, so we need only consider the other two conditions of being $\mathcal{L}$-nice. To that end, define $$\varphi_x := \ulcorner (E(x) \vee E(\neg x)) \wedge (\neg E(x) \vee \neg E(\neg x)) \urcorner$$ for each $x \in \mathcal{L}$. Note that the $\varphi_x$'s are $\mathcal{L}^{*}_{\mathfrak{A}}$-$\Delta_0$ sentences. Moreover, $X \models^{*}_{\mathfrak{A}} \varphi_x$ for every $\mathcal{L}$-nice set $X$ and every $x \in \mathcal{L}$. 
    
    Let $x \in \mathcal{L}$ and $p' \leq_{P} p$. That $\phi$ is $(\mathfrak{A}, \mathbb{P}, \mathcal{L})$-universal for $p$ means $\mathcal{F}_{p'} \neq \emptyset$, so choose any $\Sigma \in \mathcal{F}_{p'}$. We must have $\Sigma \models^{*}_{\mathfrak{A}} \varphi_x$ because $\Sigma$ is $\mathcal{L}$-nice. This allows us to conclude that for all $x \in \mathcal{L}$, $\varphi_x$ is $(\mathfrak{A}, \mathbb{P}, \mathcal{L})$-universal for $p$. As Case 1 has been proven, we can apply it to yield $$\bigcup g \models^{*}_{\mathfrak{A}} \varphi_x \text{ for all } x \in \mathcal{L},$$ which is just another way of saying $\bigcup g$ fulfils the last two conditions of Definition \ref{lnice}. We have thus shown that $\bigcup g$ is $\mathcal{L}$-nice. 
    
    Once more, let $p' \leq_{P} p$, so that there is $\Sigma \in \mathcal{F}_{p'}$ for which $\Sigma \models^{*}_{\mathfrak{A}} \phi$. By Remark \ref{rem321}, we can safely assume $\phi^* \in ran(\mathsf{WNF})$. Then there is $\nu$ such that
    \begin{itemize}
        \item $\nu$ is a $\mathfrak{A}$-valuation, and
        \item $\Sigma \models^{*}_{\mathfrak{A}, \nu} \phi^*$.
    \end{itemize}
    According to $\ref{p3202}$ of Proposition \ref{p320}, it must be the case that
    \begin{equation}\label{eq32'}
        \Sigma \models^{*}_{\mathfrak{A}, \nu} \phi'\text{,}
    \end{equation}
    where
    \begin{align*}
        \phi' := \mathsf{DNF}(\phi^*, \nu) \text{.} 
    \end{align*}
    Define two other $\mathcal{L}^{*}_{\mathfrak{A}}$ sentences as follows: 
    \begin{align*}
        \phi'' := \ & \mathsf{check}(\mathsf{pos}(\phi^*)) \\
        \phi''' := \ & \ulcorner \bigvee \{\mathsf{check}(\mathsf{pos}(\psi)) : \psi \text{ is a disjunct of } \phi'\} \urcorner \text{.}
    \end{align*}
    Clearly,  
    \begin{equation*}
        \phi''' = \mathsf{check}(\mathsf{pos}(\phi')) \text{,}
    \end{equation*}
    from which, citing Proposition \ref{prop321} and Fact \ref{fact330}, we can conclude
    \begin{equation}\label{eq32f}
        \phi''' = \mathsf{DNF}(\phi'', \nu) \text{.}
    \end{equation}
    Finally, set
    \begin{equation*}
        \phi^{\dagger} := \ulcorner \exists x_1 \dots \exists x_n \ \phi'' \urcorner \text{.}
    \end{equation*}
    It is imperative to highlight that $\phi^{\dagger}$ depends on $\phi^*$, and thus on $\phi$, but not on any $\mathfrak{A}$-valuation.

    \begin{rem}\label{rem339}
    Note that we have just described a constructive procedure which converts an arbitrary $\mathcal{L}^{*}_{\mathfrak{A}}$-$\Sigma_1$ sentence $\phi$ into another $\mathcal{L}^{*}_{\mathfrak{A}}$-$\Sigma_1$ sentence $\phi^{\dagger}$. Indeed, this procedure makes sense even if $n = 0$ in the expansion of $\phi$, i.e. even if $\phi \in \mathcal{L}^{*}_{\mathfrak{A}}$-$\Delta_0$. From now on, call the function associated with this procedure $\mathsf{Conv}$, so that both the domain and codomain of $\mathsf{Conv}$ are equal to
    \begin{equation*}
        \{\phi: \phi \text{ is a } \mathcal{L}^{*}_{\mathfrak{A}}\text{-}\Sigma_1 \text{ sentence}\} \text{.}
    \end{equation*}
    \end{rem}

    \begin{claim2}\label{claim335}
    There exists $q \in P$ for which
    \begin{itemize}
        \item $q \leq_P p'$, and
        \item $q \models^{*}_{\mathfrak{A}} \phi^{\dagger}$.
    \end{itemize}
    \end{claim2}

    \begin{proof}
    Beginning with (\ref{eq32'}), we observe that there ought to be some $\varphi$ for which
    \begin{itemize}
        \item $\varphi$ is a disjunct of $\phi'$, and
        \item $\Sigma \models^{*}_{\mathfrak{A}} \nu^*(\varphi)$.
    \end{itemize}
    $\Sigma$ being $\mathcal{L}$-nice and Lemma \ref{corr} gives us $\Sigma \models^{*}_{\mathfrak{A}} \mathsf{pos}(\nu^*(\varphi))$. 
    Since literals of the form $$\ulcorner \neg E(x) \urcorner \text{ for some } x \in \mathcal{L}$$ do not occur in $\mathsf{pos}(\nu^*(\varphi))$, following the proof of Case 1 of Lemma \ref{exist}, there must be a finite set $p^{\dagger} \in A \cap \mathcal{P}(\Sigma)$ such that 
    \begin{itemize}
        \item $\mathsf{set}(\mathsf{pos}(\nu^*(\varphi))) = (p^{\dagger}, p^{\dagger})$, and 
        \item $\mathsf{pos}(\nu^*(\varphi))$ is $\models^{*}_{\mathfrak{A}}$-true for $p^{\dagger}$. 
    \end{itemize}
    By the fact that $\Sigma$ is a $p'$-candidate for $(\mathfrak{A}, \mathbb{P}, \mathcal{L})$-universality, we can find $q \in P$ with $$p' \cup p^{\dagger} \subset q.$$ In particular, $q \leq_P p'$ and $q \models^{*}_{\mathfrak{A}} \mathsf{pos}(\nu^*(\varphi))$.
    
    That $\phi$ is $(\mathfrak{A}, \mathbb{P}, \mathcal{L})$-universal for $p$ tells us that $q$ can be extended to a $\mathcal{L}$-nice set. Necessarily, $q$ must be ``internally consistent'' in the following sense:
    \begin{align*}
        q \models^{*}_{\mathfrak{A}} \ulcorner \neg E(x) \urcorner \text{ whenever } & x \in \mathcal{L} \text{ and } \\
        & \ulcorner \neg E(x) \urcorner \text{ occurs in } \nu^*(\varphi).
    \end{align*} 
    If $x \in \mathcal{L}$ and $\ulcorner \neg E(x) \urcorner$ occurs in $\nu^*(\mathsf{pos}(\varphi))$, then $\ulcorner \neg E(x) \urcorner$ already occurs in $\nu^*(\varphi)$. Consequently, $q \models^{*}_{\mathfrak{A}} \ulcorner \neg E(x) \urcorner$. We have thus established 
    \begin{equation}\label{eq32}
        q \models^{*}_{\mathfrak{A}} \nu^*(\mathsf{pos}(\varphi)).
    \end{equation} 
    Let $\varphi'$ be of the form $$\ulcorner (\text{``}x \in \mathcal{L} \text{''} \implies \text{``} E(\neg x) \text{''}) \urcorner \text{,}$$ where $x$ is a variable such that $\ulcorner \neg E(x) \urcorner$ occurs in $\mathsf{pos}(\varphi)$. Then the fact that $q \models^{*}_{\mathfrak{A}} \mathsf{pos}(\nu^*(\varphi))$ must imply $q \models^{*}_{\mathfrak{A}} \nu^*(\varphi')$. Bearing the definition of $\mathsf{check}$ in mind, this allows us to ascertain that $$q \models^{*}_{\mathfrak{A}} \nu^*(\mathsf{check}(\mathsf{pos}(\varphi)))$$ follows from (\ref{eq32}). As $\nu^*(\mathsf{check}(\mathsf{pos}(\varphi)))$ logically implies $\nu^*(\phi''')$, we have $q \models^{*}_{\mathfrak{A}} \nu^*(\phi''')$ as well. Invoking (\ref{eq32f}) and \ref{p3203} of Proposition \ref{p320} now would yield $q \models^{*}_{\mathfrak{A}, \nu} \phi''$, from which
    \begin{equation*}
        q \models^{*}_{\mathfrak{A}} \phi^{\dagger}
    \end{equation*}
    logically follows.
    \end{proof}

    \begin{claim2}\label{claim341}
    For every pair $(\Sigma', \Sigma'')$ in any weak outer model of $V$ such that 
    \begin{itemize}
        \item $\Sigma''$ is $\mathcal{L}$-nice, and 
        \item $\Sigma' \subset \Sigma''$,
    \end{itemize}
    we must observe
    \begin{equation*}
        \Sigma' \models^{*}_{\mathfrak{A}} \phi^{\dagger} \implies \Sigma'' \models^{*}_{\mathfrak{A}} \phi \text{.}
    \end{equation*}
    \end{claim2}

    \begin{proof}
    Assume 
    \begin{itemize}
        \item $(\Sigma', \Sigma'')$ is in some weak outer model of $V$,
        \item $\Sigma''$ is $\mathcal{L}$-nice,
        \item $\Sigma' \subset \Sigma''$, and 
        \item $\Sigma' \models^{*}_{\mathfrak{A}} \phi^{\dagger}$.
    \end{itemize} 
    By \ref{p3202} of Proposition \ref{p320}, this means we can find $\nu'$ and $\varphi^*$ such that
    \begin{itemize}
        \item $\nu'$ is a $\mathfrak{A}$-valuation, 
        \item $\varphi^*$ is a disjunct of $\mathsf{DNF}(\phi'', \nu')$, and
        \item $\Sigma' \models^{*}_{\mathfrak{A}, \nu'} \varphi^*$.
    \end{itemize}
    By an argument similar to that which led us to (\ref{eq32f}), $\varphi^*$ must be a disjunct of 
    \begin{equation*}
        \ulcorner \bigvee \{\mathsf{check}(\mathsf{pos}(\psi)) : \psi \text{ is a disjunct of } \mathsf{DNF}(\phi^*, \nu')\} \urcorner \text{,}
    \end{equation*}
    and so 
    \begin{equation*}
        \varphi^* = \mathsf{check}(\mathsf{pos}(\psi))
    \end{equation*}
    for some disjunct $\psi$ of $\mathsf{DNF}(\phi^*, \nu')$. Seeing that $\Sigma''$ is $\mathcal{L}$-nice and recalling Remark \ref{rem328}, we are permitted to apply Lemma \ref{postrue} to obtain $$\Sigma'' \models^{*}_{\mathfrak{A}, \nu'} \mathsf{check}(\mathsf{pos}(\psi)),$$ or equivalently, $$\Sigma'' \models^{*}_{\mathfrak{A}} (\nu')^*(\mathsf{check}(\mathsf{pos}(\psi))).$$
    
    According to the definition of $\mathsf{check}$ in Definition \ref{check}, it is immediate that $(\nu')^*(\mathsf{check}(\mathsf{pos}(\psi)))$ logically implies $(\nu')^*(\mathsf{pos}(\psi))$, so we also have $$\Sigma'' \models^{*}_{\mathfrak{A}} (\nu')^*(\mathsf{pos}(\psi)).$$ By Lemma \ref{corr} and the fact that $\Sigma''$ is $\mathcal{L}$-nice, $$\Sigma'' \models^{*}_{\mathfrak{A}} (\nu')^*(\psi).$$ That $(\nu')^*(\psi)$ logically implies $(\nu')^*(\mathsf{DNF}(\phi^*, \nu'))$ then yields $$\Sigma'' \models^{*}_{\mathfrak{A}, \nu'} \mathsf{DNF}(\phi^*, \nu').$$ Now we can invoke \ref{p3203} of Proposition \ref{p320} to arrive at 
    \begin{equation*}
        \Sigma'' \models^{*}_{\mathfrak{A}, \nu'} \phi^* \text{,}
    \end{equation*}
    from which 
    \begin{equation*}
        \Sigma'' \models^{*}_{\mathfrak{A}} \phi
    \end{equation*}
    logically follows.
    \end{proof}

    Claim \ref{claim335} informs us that the set $$D_2 := \{q \in P : q \models^{*}_{\mathfrak{A}} \phi^{\dagger}\}$$ is dense below $p$ in $\mathbb{P}$. With reference to \ref{s31I} and Remarks \ref{def2} and \ref{shcs}, as well as Definition \ref{def27} for what it means to be a $\Sigma_1$ formula, $D_2$ is $\Sigma_1$-definable in the language associated with $\mathfrak{A}$, since
    \begin{itemize}
        \item $\mathbb{P}$ is $\Sigma_1$-definable in the language associated with $\mathfrak{A}$,
        \item we may safely assume $\phi^{\dagger}$ to be $\mathcal{L}^{*}_{\mathfrak{A}}$-$\Sigma_1$,
        \item the act of replacing every subformula of $\phi^{\dagger}$ of the form $\ulcorner E(x) \urcorner$ with $\ulcorner x \in q \urcorner$ does not alter the quantification structure of $\phi^{\dagger}$,
        \item $\mathfrak{A}$ being a model of a sufficiently strong set theory means that for each 
        \begin{itemize}[label=$\circ$]
            \item $\Sigma_1$ formula $\varphi^+$, and
            \item pair of variables $\{\ulcorner z \urcorner, \ulcorner p \urcorner\}$ with both $\ulcorner z \urcorner$ and $\ulcorner p \urcorner$ not being bound in $\varphi^+$,
        \end{itemize}
        there exists a $\Sigma_1$ formula $\phi^+$ satisfying
        \begin{align*}
            X \models^*_{\mathfrak{A}, \nu} \ulcorner \forall z \ (z \in p \implies \varphi^+) \urcorner \iff X \models^*_{\mathfrak{A}, \nu} \phi^+
        \end{align*}
        in case 
        \begin{itemize}[label=$\circ$]
            \item $X$ is a set in some weak outer model of $V$, and
            \item $\nu$ is a $\mathfrak{A}$-valuation, 
        \end{itemize}
        and thus,
        \item one can assume without loss of generality, that the result of replacing every subformula of $\phi^{\dagger}$ of the form $\ulcorner E(x) \urcorner$ with $\ulcorner x \in q \urcorner$, is a $\Sigma_1$ formula in the language associated with $\mathfrak{A}$.
    \end{itemize}
    Thus $g$ must meet $D_2$. As $p \in g$, $g \cap D_2 \neq \emptyset$. Choose any $q^* \in g \cap D_2$. Obviously, 
    \begin{itemize}
        \item $q^* \subset \bigcup g$, and
        \item $(q^*, \bigcup g) \in V[g]$.
    \end{itemize}
    We have also shown that $\bigcup g$ is $\mathcal{L}$-nice. As a result, the pair $(q^*, \bigcup g)$ satisfies the hypothesis of Claim \ref{claim341}. Applying Claim \ref{claim341} to $(q^*, \bigcup g)$ then gives us
    \begin{equation*}
        q^* \models^{*}_{\mathfrak{A}} \phi^{\dagger} \implies \bigcup g \models^{*}_{\mathfrak{A}} \phi \text{.}
    \end{equation*}
    Now $q^* \in D_2$ just means $$q^* \models^{*}_{\mathfrak{A}} \phi^{\dagger},$$ so necessarily, 
    \begin{equation*}
        \bigcup g \models^{*}_{\mathfrak{A}} \phi \text{.}
    \end{equation*}
    
    \item $\phi = \ulcorner \forall x \ \varphi(x, \Vec{d}) \urcorner$ for some $x$, $\varphi$ and $\Vec{d}$. Then for each $a \in A$, $\varphi(a, \Vec{d})$ is $(\mathfrak{A}, \mathbb{P}, \mathcal{L})$-universal for $p$. By the induction hypothesis, $$\bigcup g \models^{*}_{\mathfrak{A}} \varphi(a, \Vec{d})$$ for all $a \in A$, so also $$\bigcup g \models^{*}_{\mathfrak{A}} \phi.$$ \qedhere
\end{enumerate}
\end{proof}

The upshot of Remark \ref{rem339} and Claims \ref{claim335} and \ref{claim341}, given the choice of parameters therein, is the general fact below. 

\begin{fact}\label{fact336}
Let 
\begin{itemize}
    \item $\phi$ be a $\mathcal{L}^{*}_{\mathfrak{A}}$-$\Sigma_1$ formula, and
    \item $\Sigma$ be a $\mathcal{L}$-nice set in some weak outer model of $V$.
\end{itemize}
Assume $\Sigma \models^{*}_{\mathfrak{A}} \phi$. Then
\begin{enumerate}[label=(\arabic*)]
    \item\label{3363} for every pair $(\Sigma', \Sigma'')$ in any weak outer model of $V$ such that $\Sigma''$ is $\mathcal{L}$-nice and $\Sigma' \subset \Sigma''$,
    \begin{equation*}
        \Sigma' \models^{*}_{\mathfrak{A}} \mathsf{Conv}(\phi) \implies \Sigma'' \models^{*}_{\mathfrak{A}} \phi \text{,}
    \end{equation*}
    and
    \item\label{3361} for every $p \in P$ with $\Sigma \in \mathcal{F}_p$, we can find $q \in P$ such that
    \begin{enumerate}[label=(\alph*)]
        \item $q \leq_P p$, and
        \item\label{3362} $q \models^{*}_{\mathfrak{A}} \mathsf{Conv}(\phi)$. 
    \end{enumerate}
\end{enumerate}
\end{fact}

\begin{rem}\label{endremmain}
Let $p$ be an arbitrary member of $P$ and define $$\mathbb{P}_{\leq p} := (\{q \in P : q \leq_{\mathbb{P}} p\}, \leq_{\mathbb{P}}).$$ Then 
\begin{enumerate}[label=(\arabic*)]
    \item $\mathbb{P}_{\leq p}$ is $\Sigma_1$-definable in the language associated with $\mathfrak{A}$, 
    \item $(\mathfrak{A}, \mathbb{P}_{\leq p})$ is good for $\mathcal{L}$, 
    \item any $\mathcal{L}^{*}_{\mathfrak{A}}$-$\Pi_2$ sentence which is $(\mathfrak{A}, \mathbb{P}, \mathcal{L})$-universal for $p$ is also $(\mathfrak{A}, \mathbb{P}_{\leq p}, \mathcal{L})$-universal, and
    \item whenever $g$ is a $\mathbb{P}$-$\Sigma_1$-generic filter over $\mathfrak{A}$ containing $p$, $$g \cap \{q \in P : q \leq_{\mathbb{P}} p\}$$ is a $\mathbb{P}_{\leq p}$-$\Sigma_1$-generic filter over $\mathfrak{A}$.
\end{enumerate}
Consequently, Lemma \ref{uni} is equivalent to, and can be restated as:
\begin{customlem}{3.29$'$}\label{lem329p}
Let 
\begin{itemize} 
    \item $W$ be a weak outer model of $V$,
    \item $g \in W$ be a $\mathbb{P}$-$\Sigma_1$-generic filter over $\mathfrak{A}$, and
    \item $\phi$ be a $\mathcal{L}^{*}_{\mathfrak{A}}$-$\Pi_2$ sentence which is $(\mathfrak{A}, \mathbb{P}, \mathcal{L})$-universal.
\end{itemize} 
Then $\bigcup g$ is $\mathcal{L}$-nice and $\bigcup g \models^{*}_{\mathfrak{A}} \phi$.
\end{customlem}
\end{rem}

\begin{rem}\label{rem331}
The proof of Lemma \ref{uni} can be reused to prove the following variation of said lemma.

\begin{lem}\label{lem332}
Assume 
\begin{itemize}
    \item $\mathcal{L}$ is just definable (instead of $\Pi_1$-definable) in the language associated with $\mathfrak{A}$, and
    \item $\mathbb{P}$ is just definable (instead of $\Sigma_1$-definable) in the language associated with $\mathfrak{A}$.
\end{itemize}
Let 
\begin{itemize} 
    \item $W$ be a weak outer model of $V$,
    \item $g \in W$ be a $\mathbb{P}$-generic filter over $\mathfrak{A}$, 
    \item $p \in g$, and
    \item $\phi$ be a $\mathcal{L}^{*}_{\mathfrak{A}}$-$\Pi_2$ sentence which is $(\mathfrak{A}, \mathbb{P}, \mathcal{L})$-universal for $p$.
\end{itemize} 
Then $\bigcup g$ is $\mathcal{L}$-nice and $\bigcup g \models^{*}_{\mathfrak{A}} \phi$.
\end{lem}

By recycling the argument in Remark \ref{endremmain} with 
\begin{itemize}
    \item ``definable'' in place of ``$\Sigma_1$-definable'', 
    \item ``$\mathbb{P}$-generic'' in place of ``$\mathbb{P}$-$\Sigma_1$-generic'', and
    \item ``$\mathbb{P}_{\leq_p}$-generic'' in place of ``$\mathbb{P}_{\leq_p}$-$\Sigma_1$-generic'',
\end{itemize}
we can conclude that Lemma \ref{lem332} is equivalent to, and can be restated as:

\begin{customlem}{3.32$'$}\label{lem332p}
Assume 
\begin{itemize}
    \item $\mathcal{L}$ is just definable (instead of $\Pi_1$-definable) in the language associated with $\mathfrak{A}$, and
    \item $\mathbb{P}$ is just definable (instead of $\Sigma_1$-definable) in the language associated with $\mathfrak{A}$.
\end{itemize}
Let 
\begin{itemize} 
    \item $W$ be a weak outer model of $V$,
    \item $g \in W$ be a $\mathbb{P}$-generic filter over $\mathfrak{A}$, and
    \item $\phi$ be a $\mathcal{L}^{*}_{\mathfrak{A}}$-$\Pi_2$ sentence which is $(\mathfrak{A}, \mathbb{P}, \mathcal{L})$-universal.
\end{itemize} 
Then $\bigcup g$ is $\mathcal{L}$-nice and $\bigcup g \models^{*}_{\mathfrak{A}} \phi$.
\end{customlem}
\end{rem}

\begin{rem}\label{rem333}
The astute reader may notice that Case 3 of the inductive step in the inductive proof of Lemma \ref{uni} seems insubstantial and tagged on. Indeed, removing the need to consider said case in Lemma \ref{uni} does not weaken the lemma. More generally and more precisely, substituting ``$\mathcal{L}^{*}_{\mathfrak{A}}$-$\Sigma_1$'' for ``$\mathcal{L}^{*}_{\mathfrak{A}}$-$\Pi_2$'' in any of
\begin{itemize}
    \item Lemma \ref{uni},
    \item Lemma \ref{lem329p},
    \item Lemma \ref{lem332}, and
    \item Lemma \ref{lem332p}
\end{itemize}
always returns a statement of equivalent strength.
\end{rem}

As a display of reciprocity, Lemma \ref{uni} allows us to simplify our verification procedures for the universality of certain $\mathcal{L}^{*}_{\mathfrak{A}}$-$\Pi_2$ sentences. The proof of this next lemma is spiritually similar to (the most obvious) proofs of the four equivalences highlighted in Remark \ref{rem333}.

\begin{lem}\label{univer}
Let 
\begin{itemize}
    \item $p \in P$, and
    \item $\phi(x)$ be a $\mathcal{L}^{*}_{\mathfrak{A}}$-$\Pi_2$ formula with $x$ as its only free variable. 
\end{itemize}
Suppose for each $a \in A$, $\phi(a)$ is $(\mathfrak{A}, \mathbb{P}, \mathcal{L})$-universal for $p$. Then
\begin{equation*}
    \varphi := \ulcorner \forall x \ \phi(x) \urcorner
\end{equation*}
is a $\mathcal{L}^{*}_{\mathfrak{A}}$-$\Pi_2$ sentence $(\mathfrak{A}, \mathbb{P}, \mathcal{L})$-universal for $p$.
\end{lem}
\begin{proof}
Clearly $\varphi$ is a $\mathcal{L}^{*}_{\mathfrak{A}}$-$\Pi_2$ sentence. Fix $q \leq_P p$. It suffices to find a set $\Sigma$ fulfilling
\begin{itemize}
    \item $\Sigma$ is a $q$-candidate for $(\mathfrak{A}, \mathbb{P}, \mathcal{L})$-universality, and 
    \item $\Sigma \models^{*}_{\mathfrak{A}} \varphi$.
\end{itemize}

Choose any $\mathbb{P}$-$\Sigma_1$-generic filter $g$ over $\mathfrak{A}$ from amongst the weak outer models of $V$, such that $q \in g$. Note that for each $a \in A$, $\phi(a)$ is a $\mathcal{L}^{*}_{\mathfrak{A}}$-$\Pi_2$ sentence $(\mathfrak{A}, \mathbb{P}, \mathcal{L})$-universal for $q$, since $\phi(a)$ is $(\mathfrak{A}, \mathbb{P}, \mathcal{L})$-universal for $p$ and $(\mathfrak{A}, \mathbb{P}, \mathcal{L})$-universality is inherited downwards in $\mathbb{P}$. By Lemma \ref{uni},
\begin{equation*}
    \bigcup g \models^{*}_{\mathfrak{A}} \phi(a) \text{ for all } a \in A \text{.}
\end{equation*}
But this just means 
\begin{equation*}
    \bigcup g \models^{*}_{\mathfrak{A}} \varphi \text{.}
\end{equation*}
We know $\bigcup g$ is $\mathcal{L}$-nice due to Lemma \ref{uni}. That $\bigcup g$ is a $q$-candidate for $(\mathfrak{A}, \mathbb{P}, \mathcal{L})$-universality then follows from the following facts:
\begin{itemize}
    \item $q \in g$, and
    \item $g$ is a filter.
\end{itemize}
All in all, we have shown that $\bigcup g$ is the $\Sigma$ we are looking for.
\end{proof}

A natural strengthening of Lemma \ref{uni} is to have $\phi$ be an arbitrary $\mathcal{L}^{*}_{\mathfrak{A}}$-$\Pi_3$ sentence which is $(\mathfrak{A}, \mathbb{P}, \mathcal{L})$-universal for $p$. As per Case 3 in the proof of Lemma \ref{uni}, we can always get the outermost universal quantification for free, so we only have to prove the strengthened lemma assuming $\phi$ is $\mathcal{L}^{*}_{\mathfrak{A}}$-$\Sigma_2$ instead of $\mathcal{L}^{*}_{\mathfrak{A}}$-$\Pi_3$. However, the nice ``characterisation'' of $\mathcal{L}^{*}_{\mathfrak{A}}$-$\Pi_2$ sentences we will uncover in Section \ref{TCIsec} (brought about by Theorems \ref{genericmodels} and \ref{revgenmodels}) seems to suggest that such a strengthening is impossible. With Remark \ref{endremmain} in mind, it makes sense to ask the following question.

\begin{enumerate}[label=(Q\arabic*)]
    \setcounter{enumi}{-1}
    \item\label{qn0} Are there sets $\mathfrak{A}'$, $\mathcal{L}'$, $\mathbb{P}' = (P', \leq_{\mathbb{P}'})$, $W$, $g$, and $\phi$ such that
    \begin{itemize}
        \item $\mathcal{L}'$ is closed under negation,
        \item $\mathbb{P}'$ is $\Sigma_1$-definable in the language associated with $\mathfrak{A}'$,
        \item $(\mathfrak{A}', \mathbb{P}')$ is good for $\mathcal{L}'$,
        \item $W$ is a weak outer model of $V$,
        \item $g \in W$ is a $\mathbb{P}'$-$\Sigma_1$-generic filter over $\mathfrak{A}'$,
        \item $\phi$ is a $\mathcal{L}^{*}_{\mathfrak{A}'}$-$\Sigma_2$ sentence which is $(\mathfrak{A}', \mathbb{P}', \mathcal{L}')$-universal, and
        \item $\bigcup g \not \models^{*}_{\mathfrak{A}'} \phi$?
    \end{itemize}
\end{enumerate}

We shall defer answering \ref{qn0} until the next subsection. 

\subsection{A Useful Framework}\label{forframe}

Fix a set of $\mathcal{L}^*_{\mathfrak{A}}$-$\Pi_2$ sentences, $\Gamma$, for this subsection.

In the previous subsection, we saw how a forcing notion $\mathbb{P}$ can generate witnesses to certain $\mathcal{L}^*_{\mathfrak{A}}$-$\Pi_2$ sentences when $\mathbb{P}$ is definable in the language associated with $\mathfrak{A}$ and $(\mathfrak{A}, \mathbb{P})$ is good for $\mathcal{L}$. Leveraging on this fact, we shall develop a framework for defining forcing notions that generate witnesses to a given set of $\mathcal{L}^*_{\mathfrak{A}}$-$\Pi_2$ sentences. 

This framework both generalises and is inspired by the forcing construction Asper\'{o} and Schindler carried out in the proof of the main theorem of \cite{schindler}.

\begin{defi}
A set $B$ is $\mathcal{L}$\emph{-closed under finite extensions} iff
\begin{itemize}
    \item $B \subset \mathcal{P}(\mathcal{L})$, and
    \item for all $x \in B$ and all $y \in [\mathcal{L}]^{< \omega}$, $x \cup y \in B$.
\end{itemize}
\end{defi}

\begin{defi}
For any $\Sigma$ and any $p$, we say $\Sigma$ $\Gamma(\mathcal{L}, \mathfrak{A})$\emph{-certifies} $p$ iff 
\begin{enumerate}
    \item $p \subset \Sigma$,
    \item $\Sigma$ is $\mathcal{L}$-nice, and
    \item $\Sigma \models^*_{\mathfrak{A}} \Gamma$.
\end{enumerate}
\end{defi}

It is easy to see that if $\mathcal{L}$, $\mathfrak{A}$, $\Gamma$, $\Sigma$ and $p$ are such that $\Sigma$ $\Gamma(\mathcal{L}, \mathfrak{A})$-certifies $p$, then 
\begin{itemize}
    \item $\Sigma$ $\Gamma(\mathcal{L}, \mathfrak{A})$-certifies $q$ for all $q \subset \Sigma$, and
    \item $\Sigma$ $\Gamma'(\mathcal{L}, \mathfrak{A})$-certifies $p$ for all $\Gamma' \subset \Gamma$.
\end{itemize}
This gives us the following proposition.

\begin{prop}\label{certgood}
Let $|\mathcal{L}| \leq \lambda$ and $B$ be $\mathcal{L}$-closed under finite extensions. If we define $\mathbb{P} := (P, \leq_{\mathbb{P}})$, where
\begin{align*}
    P := \ & \{p \in B : \ \Vdash_{Col(\omega, \lambda)} \exists \Sigma \ (``\Sigma \ \Gamma(\mathcal{L}, \mathfrak{A}) \text{-certifies } p")\}, \text{ and} \\
    \leq_{\mathbb{P}} \ := \ & \{(p, q) \in P \times P : q \subset p\},
\end{align*}
then as long as $P \neq \emptyset$, 
\begin{itemize}
    \item $(\mathfrak{A}, \mathbb{P})$ is good for $\mathcal{L}$, and
    \item whenever $p \in P$ and $\Sigma$ $\Gamma'(\mathcal{L}, \mathfrak{A})$-certifies $p$ for some $\Gamma' \supset \Gamma$, $\Sigma$ is a $p$-candidate for $(\mathfrak{A}, \mathbb{P}, \mathcal{L})$-universality.
\end{itemize}
\end{prop}

\begin{lem}\label{inout}
Let $|trcl(\mathfrak{A})| \leq \lambda$ and $p \subset \mathcal{L}$. Assume there is $\Sigma$ in a weak outer model $W$ of $V$ such that $\Sigma$ $\Gamma(\mathcal{L}, \mathfrak{A})$-certifies $p$. Then $$\Vdash_{Col(\omega, \lambda)} \exists \Sigma \ (``\Sigma \ \Gamma(\mathcal{L}, \mathfrak{A}) \text{-certifies } p").$$
\end{lem}

\begin{proof}
Suppose otherwise, so there is $q \in Col(\omega, \lambda)$ such that
\begin{align}
    \label{eq10} q \Vdash_{Col(\omega, \lambda)} \ \neg \psi,
\end{align}
where $$\psi := \exists \Sigma \ (``\Sigma \ \Gamma(\mathcal{L}, \mathfrak{A}) \text{-certifies } p").$$

Let $g$ be $Col(\omega, \lambda)$-generic over $W$ with $q \in g$, so that $g$ is also $Col(\omega, \lambda)$-generic over $V$. First, that $W \models \psi$ means $W[g] \models \psi$. Next, notice that if $\varphi(\Sigma, y)$ is the conjunction of the statements
\begin{itemize}
    \item $y = \{\phi \in \mathcal{L}^*_{\mathfrak{A}} : \Sigma \models^*_{\mathfrak{A}} \phi\}$,
    \item $\Gamma \subset y$, 
    \item $p \subset \Sigma$, and
    \item $``\Sigma \text{ is } \mathcal{L}\text{-nice}"$, 
\end{itemize}
then $\varphi$ is a $\Sigma_1$ (in fact, $\Delta_1$, although that delineation is unnecessary here) formula in the language of set theory (see Definition \ref{def27}), with parameters among $p, \mathfrak{A}, \mathcal{L}, \mathcal{L}^*_{\mathfrak{A}}, \Gamma$. This is because $\varphi$ is equivalent to the statement of there being a function $f$ with domain $\mathcal{L}^*_{\mathfrak{A}}$ --- a $\Delta_0$-definable subset of $A$ --- and codomain $\{0, 1\}$, such that 
\begin{itemize}
    \item $f$ fulfils the inductive properties of Tarski's definition of the satisfaction relation, applied to the structure $(A; \in, \Vec{R}, \Sigma)$,
    \item $y = \{\phi \in \mathcal{L}^*_{\mathfrak{A}} : f(\phi) = 1\}$, 
    \item $\Gamma \subset y$,
    \item $p \subset \Sigma$, and
    \item $``\Sigma \text{ is } \mathcal{L}\text{-nice}"$,
\end{itemize}
every of which aforementioned points is expressible as $\Delta_0$ formulas (following Definition \ref{def27}) in the language of set theory. Note also that whenever $\varphi(\Sigma, y)$ holds, $\Sigma$, $y$ and any witness $f$ must have transitive closures of cardinalities no larger than $\lambda$. Moreover, we have $$\psi \iff \exists \Sigma \ \exists y \ \varphi(\Sigma, y).$$ As $p, \mathfrak{A}, \mathcal{L}, \mathcal{L}^*_{\mathfrak{A}}, \Gamma$ are subsets of $trcl((p, \mathfrak{A}, \mathcal{L}, \mathcal{L}^*_{\mathfrak{A}}, \Gamma))$ and $|trcl((p, \mathfrak{A}, \mathcal{L}, \mathcal{L}^*_{\mathfrak{A}}, \Gamma))| \leq \lambda$, the structure $$\mathfrak{B} := (trcl((p, \mathfrak{A}, \mathcal{L}, \mathcal{L}^*_{\mathfrak{A}}, \Gamma)); \in, p, \mathfrak{A}, \mathcal{L}, \mathcal{L}^*_{\mathfrak{A}}, \Gamma)$$ can be coded as a real in $V[g]$, by Lemma \ref{setcode}. This means that in all weak outer models of $V[g]$, $\psi$ can be thought of as a $\mathbf{\Sigma^1_1}$ sentence involving a real code of $\mathfrak{B}$ found in $V[g]$. In particular, by Mostowski's absoluteness theorem, $\psi$ is absolute for $V[g]$ and $W[g]$. Now $W[g] \models \psi$ implies $V[g] \models \psi$, contradicting (\ref{eq10}).
\end{proof}

Note that the proof of Lemma \ref{inout} does not require that $\Gamma$ contains only $\mathcal{L}^*_{\mathfrak{A}}$-$\Pi_2$ sentences. Indeed, for a litany of properties $K$, the existence of an object satisfying $K$ is absolute between $V$ and its weak outer models and hence, between $V$ and its forcing extensions. However, it is often useful --- if not integral --- to have a proper handle on such an object. It is towards this end that we are often interested in the existence of a $V$-generic object $k$ such that $k$ satisfies $K$ in $V[k]$.

Specifying ``$k$ satisfies $K$'' to be ``$k \models^{*}_{\mathfrak{A}} \Gamma$'', the following lemma is thus well-motivated. 

\begin{lem}\label{main2}
Let $W$, $\lambda$, $B$, $P$, $\mathbb{P}$ and $g$ be such that
\begin{itemize}
    \item $|trcl(\mathfrak{A})| \leq \lambda$,
    \item $B$ is $\mathcal{L}$-closed under finite extensions,
    \item $P = \{p \in B : \ \Vdash_{Col(\omega, \lambda)} \exists \Sigma \ (``\Sigma \ \Gamma(\mathcal{L}, \mathfrak{A}) \text{-certifies } p")\}$,
    \item $\mathbb{P} = (P, \supset \cap \ P)$, 
    \item $\mathbb{P}$ is $\Sigma_1$-definable in the language associated with $\mathfrak{A}$,
    \item $W$ is a weak outer model of $V$, and
    \item $g \in W$ is a $\mathbb{P}$-$\Sigma_1$-generic filter over $\mathfrak{A}$.
\end{itemize}
If there is $\Sigma$ in a weak outer model $W'$ of $V$ such that $\Sigma$ $\Gamma(\mathcal{L}, \mathfrak{A})$-certifies $\emptyset$, then $\bigcup g$ $\Gamma(\mathcal{L}, \mathfrak{A})$-certifies $\emptyset$.

In particular, if $g$ is $\mathbb{P}$-generic over $V$, then $\bigcup g$ $\Gamma(\mathcal{L}, \mathfrak{A})$-certifies $\emptyset$ in $V[g] = V[\bigcup g]$.
\end{lem}

\begin{proof}
The general statement is clear from Lemma \ref{uni}, Remark \ref{endremmain}, Proposition \ref{certgood} and Lemma \ref{inout}. That $\bigcup g$ $\Gamma(\mathcal{L}, \mathfrak{A})$-certifies $\emptyset$ is absolute for transitive models of $\mathsf{ZFC - Powerset}$, so if $g$ is $\mathbb{P}$-generic over $V$, then  $\bigcup g$ $\Gamma(\mathcal{L}, \mathfrak{A})$-certifies $\emptyset$ in $V[g]$. Moreover, since $g = [\bigcup g]^{< \omega}$, we have $V[g] = V[\bigcup g]$.
\end{proof}

We are now equipped to tackle \ref{qn0}. In the presence of Proposition \ref{certgood}, the next lemma implies an affirmative answer to \ref{qn0}.

\begin{lem}\label{lem339}
There are $W$, $\mathfrak{A}'$, $\mathcal{L}'$, $\Gamma'$, $P'$, $\mathbb{P}'$, $g$ and $\phi$ such that
\begin{enumerate}[label=(\arabic*)]
    \item\label{3391} $\mathcal{L}'$ is closed under negation,
    \item\label{3392} $\mathfrak{A}'$ is $\mathcal{L}'$-suitable,
    \item\label{3393} $\Gamma'$ is a set of $\mathcal{L}^{*}_{\mathfrak{A}'}$-$\Pi_2$ sentences,
    \item\label{3394} $P' = \{p \in [\mathcal{L}']^{< \omega} : \ \Vdash_{Col(\omega, |trcl(\mathfrak{A}')|)} \exists \Sigma \ (``\Sigma \ \Gamma'(\mathcal{L}', \mathfrak{A}') \text{-certifies } p")\} \neq \emptyset$,
    \item\label{3395} $\mathbb{P}' = (P', \supset \cap \ P')$, 
    \item\label{3396} $\mathbb{P}'$ is $\Sigma_1$-definable in the language associated with $\mathfrak{A}'$,
    \item\label{3397} $W$ is a weak outer model of $V$,
    \item\label{3398} $g \in W$ is a $\mathbb{P}'$-generic filter over $\mathfrak{A}'$,
    \item\label{3399} $\phi$ is a $\mathcal{L}^{*}_{\mathfrak{A}'}$-$\Sigma_2$ sentence which is $(\mathfrak{A}', \mathbb{P}', \mathcal{L}')$-universal, and
    \item\label{33910} $\bigcup g \not \models^{*}_{\mathfrak{A}'} \phi$.
\end{enumerate}
\end{lem}

\begin{proof}
Let $\mathcal{L}'$ be the closure under negation of the following set of expressions (strings):
\begin{equation*}
    \{\ulcorner \dot{F} (\alpha) = \beta \urcorner : (\alpha, \beta) \in \omega_1 \times \omega_2\}
\end{equation*}
Clearly, \ref{3391} holds. Set $\mathfrak{A}'$ to be $(H((2^{\omega_2})^+); \in)$, so that \ref{3392} holds. Using a natural notational shorthand for passing parameters into $\mathcal{L}'$ formulas (see Remark \ref{rempp} for more information), define 
\begin{align*}
    \Gamma' := \{ & \ulcorner \forall \alpha < \omega_1 \ \exists \beta \ (E(\ulcorner \dot{F}(\alpha) = \beta \urcorner)) \urcorner, \\ 
    & \ulcorner \forall \alpha \ \forall \beta \ \forall \gamma \ ((E(\ulcorner \dot{F}(\alpha) = \beta \urcorner) \wedge E(\ulcorner \dot{F}(\alpha) = \gamma \urcorner)) \implies \beta = \gamma) \urcorner\} \text{,}
\end{align*}
and observe that \ref{3393} is satisfied (refer again to Remark \ref{rempp} for justification of the shorthand not underselling the complexity of any $\mathcal{L}^{*}_{\mathfrak{A}'}$-$\Pi_2$ formula it abbreviates). 

Next, set 
\begin{gather*}
    P' := \{p \in [\mathcal{L}']^{< \omega} : \ \Vdash_{Col(\omega, |trcl(\mathfrak{A}')|)} \exists \Sigma \ (``\Sigma \ \Gamma'(\mathcal{L}', \mathfrak{A}') \text{-certifies } p")\} \\
    \mathbb{P}' := (P', \supset \cap \ P')
\end{gather*}
to satisfy \ref{3395} and \ref{3396}, since $[\mathcal{L}']^{< \omega} \in H((2^{\omega_2})^+)$ implies $\mathbb{P}' \in H((2^{\omega_2})^+)$.

\begin{defi}
Given a set $X$, let $\Sigma(X)$ denote
\begin{align*}
    & \{\ulcorner \dot{F} (\alpha) = \beta \urcorner : (\alpha, \beta) \in (\omega_1 \times \omega_2) \cap X\} \ \cup \\
    & \{\ulcorner \neg \dot{F} (\alpha) = \beta \urcorner : (\alpha, \beta) \in (\omega_1 \times \omega_2) \setminus X\} \text{.}
\end{align*}
\end{defi}

Note that for any $Z$ in an weak outer model of $V$, 
\begin{equation*}
    Z \cap (\omega_1 \times \omega_2) \text{ is a function from } \omega_1 \text{ into } \omega_2 \iff \Sigma(Z) \ \Gamma'(\mathcal{L}', \mathfrak{A}') \text{-certifies } \emptyset \text{.}
\end{equation*}
As there already exist functions from $\omega_1$ into $\omega_2$ in $V$, $\emptyset \in P'$ by Lemma \ref{inout}, giving us \ref{3394}. It suffices to show that \ref{3397} to \ref{33910} hold for
\begin{itemize}
    \item any $\mathbb{P}'$-generic filter $g$ over $V$ (and thus over $\mathfrak{A}$),
    \item $W := V[g]$, and
    \item $\phi := \ulcorner \exists f \in {^{\omega}{\omega_2}} \ \forall n < \omega \ (E(\ulcorner \dot{F}(n) = f(n) \urcorner)) \urcorner$.
\end{itemize}
Trivially, \ref{3397} and \ref{3398} are done. Further, $\phi$ is a $\mathcal{L}^{*}_{\mathfrak{A}'}$-$\Sigma_2$ sentence, seeing that ${^{\omega}{\omega_2}} \in H((2^{\omega_2})^+)$.

Choose any $p \in P'$. It is not hard to verify that there exists 
\begin{equation*}
    p^* \in Col(\omega_1, \omega_2) = (\bigcup \{{^{\alpha}{\omega_2}} : \alpha < \omega_1\}, \supset)
\end{equation*} 
for which $p \subset \Sigma(q^*)$ whenever $q^*$ extends $p^*$ in $Col(\omega_1, \omega_2)$. Let $h$ be a $Col(\omega_1, \omega_2)$-generic filter over $V$ containing $p^*$. Then in $V[h]$, 
\begin{equation*}
    \Sigma(\bigcup h) \ (\Gamma' \cup \{\phi\})(\mathcal{L}', \mathfrak{A}') \text{-certifies } p \text{.}
\end{equation*}
According to Proposition \ref{certgood}, we arrive at \ref{3399}.

Finally, choose any $f^* \in {^{\omega}{\omega_2}}$. Due to the straightforward observation that 
\begin{align*}
    D_{f^*} := \{p \in P' : \exists n < \omega \ \exists \beta < \omega_2 \ ( & (\ulcorner \dot{F}(n) = \beta \urcorner \in p \wedge f(n) \neq \beta) \ \vee \\
    & (\ulcorner \neg \dot{F}(n) = \beta \urcorner \in p \wedge f(n) = \beta))\}
\end{align*}
is dense in $\mathbb{P}'$, we must have $\bigcup g \models^{*}_{\mathfrak{A}'} \varphi_{f^*}$, where
\begin{equation*}
    \varphi_{f^*} := \ulcorner \neg \forall n < \omega \ (E(\ulcorner \dot{F}(n) = f^*(n) \urcorner)) \urcorner \text{.}
\end{equation*}
Having $f^*$ range over ${^{\omega}{\omega_2}}$ then nets us \ref{33910}.
\end{proof}

This subsection shall be concluded with another absoluteness result. This time, instead of looking for witnesses in forcing extensions, we turn our focus to the forcing notions themselves.

\begin{lem}\label{Pisabs}
The definition of $\mathbb{P}$ from parameters $B$, $\mathcal{L}$, $\mathfrak{A}$, $\Gamma$ in Proposition \ref{certgood}, where $\lambda$ is additionally specified to be $|trcl(\mathfrak{A})|$, is absolute for transitive models of $\mathsf{ZFC}$.
\end{lem}

\begin{proof}
It suffices to show that the set 
\begin{equation}\label{setdef}
    \{p \in B : \ \Vdash_{Col(\omega, |trcl(\mathfrak{A})|)} \exists \Sigma \ (``\Sigma \ \Gamma(\mathcal{L}, \mathfrak{A}) \text{-certifies } p")\}
\end{equation}
is absolute for transitive models of $\mathsf{ZFC}$. 

Let $V'$ and $W$ be transitive models of $\mathsf{ZFC}$ such that $\{B, \mathcal{L}, \mathfrak{A}, \Gamma\} \subset V' \subset W$. Have $P^{V'}$ and $P^W$ denote the versions of the set (\ref{setdef}) defined in $V'$ and $W$ respectively. We want to prove $P^{V'} = P^W$. 

First note that $$ Col(\omega, |trcl(\mathfrak{A})|)^{V'} = Col(\omega, |trcl(\mathfrak{A})|^{V'}) \cong Col(\omega, |trcl(\mathfrak{A})|^W) = Col(\omega, |trcl(\mathfrak{A})|)^W$$ in $W$, so 
\begin{equation}\label{iffpw}
    p \in P^W \iff \ \Vdash_{Col(\omega, |trcl(\mathfrak{A})|)^{V'}} \exists \Sigma \ (``\Sigma \ \Gamma(\mathcal{L}, \mathfrak{A}) \text{-certifies } p")
\end{equation}
in $W$. Since any forcing extension of $W$ is a weak outer model of $V$, a direct application of Lemma \ref{inout} gives us $P^W \subset P^{V'}$. Next, fix any $p \in P^{V'}$ and any $Col(\omega, |trcl(\mathfrak{A})|)^{V'}$-generic filter $g$ over $W$. Now $g$ is also a $Col(\omega, |trcl(\mathfrak{A})|)^{V'}$-generic filter over $V'$, so $V'[g]$ and $W[g]$ are transitive models of $\mathsf{ZFC}$ and moreover, $V'[g] \subset W[g]$. By the definition of $P^{V'}$ in $V'$, $$V'[g] \models ``\Sigma \ \Gamma(\mathcal{L}, \mathfrak{A}) \text{-certifies } p"$$ for some $\Sigma \in V'[g]$. That $\Sigma \ \Gamma(\mathcal{L}, \mathfrak{A})$-certifying $p$ is absolute for transitive models of $\mathsf{ZFC}$ implies $$W[g] \models ``\Sigma \ \Gamma(\mathcal{L}, \mathfrak{A}) \text{-certifies } p".$$ We have thus shown $$W \models (\Vdash_{Col(\omega, |trcl(\mathfrak{A})|)^{V'}} \exists \Sigma \ (``\Sigma \ \Gamma(\mathcal{L}, \mathfrak{A}) \text{-certifies } p")),$$ whence $p \in P^W$ by (\ref{iffpw}). This allows us to conclude $P^{V'} = P^W$.
\end{proof}

\section{Extending Namba Forcing}\label{sect2}

This section illustrates how the framework introduced in Subsection \ref{forframe} can be applied to resolve some problems in set theory.

\subsection{An Extension Problem}

Before stating our problem of interest, we feel obliged to present, at least in brief, the history surrounding it. 

Fix a limit ordinal $\alpha$ and consider the chain of inequalities
\begin{equation}\label{ineq1}
    cof(\alpha) \leq |\alpha| \leq \alpha \text{,}
\end{equation} 
which is provable in $\mathsf{ZFC}$. Set theorists have long investigated the ability to change the signs in (\ref{ineq1}) via forcing. If $|\alpha| < \alpha$ in $V$, then the same must hold in any forcing extension . If $cof(\alpha) < |\alpha|$ in $V$, we can always force $cof(\alpha) = |\alpha|$ by collapsing both $cof(\alpha)$ and $|\alpha|$ to a regular cardinal in $V$ no greater than $cof(\alpha)$.

On the flipside, if $|\alpha| = \alpha$ in $V$, then $\alpha$ is a cardinal there. As long as $\alpha$ is uncountable, a forcing notion that collapses $\alpha$ (to $\omega$, say) exists and necessarily forces $|\alpha| < \alpha$. We are left with the case where $cof(\alpha) = |\alpha|$ in $V$. Note that by swapping $\alpha$ with a smaller ordinal if necessary, we can assume $\alpha$ is regular in $V$ without loss of generality. So assume $\alpha$ is an uncountable regular cardinal in $V$. If there is a singular cardinal $\beta$ below $\alpha$, one can simply collapse $\alpha$ to $\beta$ to achieve $cof(\alpha) < |\alpha|$, since the usual forcing notion for this purpose preserves the cardinality of $\beta$. Otherwise, forcing $cof(\alpha) < |\alpha|$ appears to be highly non-trivial.

In his doctoral dissertation \cite{prikry}, Prikry assumed $\alpha$ is a measurable cardinal, and gave an example of a forcing notion that preserves all cardinalities, yet changes $cof(\alpha)$ to $\omega$. A natural follow-up question to Prikry's result is thus:

\begin{quote}
    can we force the the separation of $cof(\alpha)$ and $|\alpha|$ on an uncountable regular $\alpha$ which provably exists over $\mathsf{ZFC}$?
\end{quote}

As successor cardinals are the only uncountable regular cardinals proven to exist over $\mathsf{ZFC}$, a forcing notion separating $cof(\alpha)$ and $|\alpha|$ for any such $\alpha$ must collapse $\alpha$. But can we ensure $\alpha$ is not collapsed ``too far''? In other words, we want to force $cof(\alpha) < |\alpha|$ while preserving all cardinals below $\alpha$.

The late 1960s saw two independent solutions to this problem in the affirmative, by Bukovsk\'{y} \cite{bukovsky} and Namba \cite{namba}. Both solutions work with $\alpha = \omega_2$, which is the smallest possible value $\alpha$ can take in an affirmative answer. Simplifications were made to the presentation of Bukovsk\'{y}'s and Namba's forcing notions over the years, without losing sight of the goal of their constructions. These simplifications culminated in what is now commonly known as \emph{Namba forcing}. Since the focus of this section is on extending the key effects of Namba forcing, we feel obliged to define the forcing notion for the sake of completeness.

\begin{defi}
We say $(T, \leq)$ is a $\kappa$\emph{-splitting in} $A$ iff
\begin{itemize}
    \item $(T, \leq)$ is a tree, and
    \item for every $s \in T \cap A$, $s$ has $\kappa$ many immediate $\leq$-successors in $T$. 
\end{itemize}
\end{defi}

\begin{defi}
If $(T, \leq)$ is a partial order and $s \in T$, we use $T_s^{\leq}$ to denote the set of $\leq$-successors of $s$ in $T$. More formally, $$T_s^{\leq} := \{t \in T : s \leq t\}.$$ 
\end{defi}

\begin{defi}
Define the order $\leq^{\dagger}$ to be $$\{(s, t) \in \omega_2^{< \omega} \times \omega_2^{< \omega} : dom(s) \subset dom(t) \text{ and } t \! \restriction_{dom(s)} = s\}.$$
\end{defi}

\begin{defi}
A \emph{Namba tree} is a subset $T$ of $\omega_2^{< \omega}$ containing a \emph{root} $s$ such that
\begin{itemize}
    \item $(T, \leq^{\dagger})$ is $\omega_2$-splitting in $T_s^{\leq^{\dagger}}$, and
    \item whenever $t \in T$, either $s \leq^{\dagger} t$ or $t \leq^{\dagger} s$. 
\end{itemize}
\end{defi}

\begin{defi}[Namba]
Define
\begin{align*}
    P_N := \ & \{T \subset \omega_2^{< \omega} : T \text{ is a Namba tree}\} \text{, and} \\
    \leq_N \ := \ & \{(p, q) \in P_N \times P_N : p \subset q\}.
\end{align*}
We call the forcing notion $\mathbb{P}_N := (P_N, \leq_N)$ \emph{Namba forcing}.
\end{defi}

Namba forcing belongs to the class of \emph{uniformly-splitting tree forcings}, one of which earliest-known members is Mathias forcing. A typical condition of a uniformly-splitting tree forcing is a tree, and it can be divided into two components, the stem and the crown. The stem is the main working part of a condition; stems in a generic filter combine to form a function that is the primary generic object we desire. The crowns work as \emph{side conditions}, which in unity, endow the forcing notion with specific regularity properties. These properties are often crucial to the satisfaction of constraints placed on the forcing extension. If $T$ is a Namba tree with root $s$, then its stem is $$\{t \in T : t \subset s\}$$ and its crown is $T_s^{\leq^{\dagger}} \setminus \{s\}$.

By means of tree combinatorics, one can show that Namba forcing gives $\omega_2^V$ a cofinality of $\omega$ without collapsing $\omega_1^V$. In fact, Namba forcing is a textbook example of such a forcing notion. It also has a stronger property than not collapsing $\omega_1^V$, for it is stationary-preserving. In the parlance of the preceding paragraph, the primary generic object here is a cofinal function from $\omega$ into $\omega_2^V$, whereas the regularity property of pertinence is being stationary-preserving. We can then observe the following division of labour: the stems of Namba forcing are in charge of changing the cofinality of $\omega_2^V$ to $\omega$, while the crowns of Namba forcing ensure all stationary subsets of $\omega_1$ in $V$ have their stationarity preserved.

The extended Namba problem, at its most rudimentary, asks (in $V$) for which ordinals $\lambda > \omega_2$ is the statement
\begin{align*}
    Nb_0(\lambda) := \text{`} & \text{there is a stationary-preserving forcing notion } \mathbb{P} \text{ such that} \\ 
    & \Vdash_{\mathbb{P}} ``cof(\alpha) = \omega" \text{ for all regular cardinals } \alpha \text{ satisfying } \omega_2 \leq \alpha < \lambda \text{'}
\end{align*} 
true. It is easy to see that if $\lambda > \omega_2$ is not a regular cardinal, then 
\begin{equation*}
    Nb_0(\lambda) \iff Nb_0(\lambda^+) \text{,}
\end{equation*}
so it suffices to only consider $Nb_0(\lambda)$ for regular cardinals $\lambda > \omega_2$.

\begin{fact}
Namba forcing witnesses $Nb_0(\omega_3)$.
\end{fact}

This formulation of the problem is already non-trivial, because finding witnesses to $Nb_0(\lambda)$ for $\lambda > \omega_3$ turns out to be nearly as difficult as proving $\neg Nb_0(\lambda)$. Perhaps as a sign of this difficulty, the following fact tells us that iterating Namba forcing in the usual way is insufficient to get us $Nb_0(\lambda)$ for any $\lambda > \omega_3$, without assuming a strong failure of $\mathsf{GCH}$. 

\begin{fact}\label{fact26}
Assume $\mathsf{GCH}$ holds below $\omega_2$. Then 
\begin{enumerate}[label=(\arabic*)]
    \item\label{fact261} $\Vdash_{\mathbb{P}_{N}} ``cof(\omega_3^V) = \omega_1"$, and
    \item $\Vdash_{\mathbb{P}_{N}} ``cof(\beta) > \omega"$ for all regular cardinals $\beta$ satisfying $\omega_3 < \beta$.
\end{enumerate}
\end{fact}

Notice that if $\mathsf{GCH}$ holds below $\omega_2$, then by Fact \ref{fact26}, iterating Namba forcing any number of times in the standard sense would not result in $cof(\omega_3^V) = \omega$, not without collapsing $\omega_1^V$.

\begin{rem}\label{rem47}
That any standard iteration of Namba forcing fails to extend $Nb_0(\lambda)$ beyond $\lambda = \omega_3$, assuming the hypothesis of Fact \ref{fact26}, suggests our natural conception of iteration is incompatible with the side conditions of Namba forcing. Indeed, \ref{fact261} of Fact \ref{fact26} is a result of interactions between the hypothesis of said fact, and the behaviour of these side conditions. 

To overcome this incompatibility, it makes sense to consider either
\begin{enumerate}[label=(\alph*)]
    \item\label{s41a} a new kind of iteration, or 
    \item\label{s41b} an overhaul of the side conditions.
\end{enumerate} 

Option \ref{s41a} is almost unfathomable, since the typical intuitions to --- and (informal) definitions of --- iterated forcing necessitates that an iterated forcing notion be a regular extension of each of its initial iterands. In other words, an iterated Namba forcing extension ought to include a Namba forcing extension as a submodel. This forbids an iteration of Namba forcing from witnessing $Nb_0(\lambda)$ for any $\lambda > \omega_3$, should $\mathsf{GCH}$ hold below $\omega_2$. 

On the other hand, option \ref{s41b} could mean a departure from the intuition of uniformly-splitting tree forcings so radical, that the resultant forcing notion has conditions best presented as objects other than trees.
\end{rem}

Indeed, a stronger variant of the extended Namba problem, asking for which regular cardinals $\lambda > \omega_2$ is the statement
\begin{align*}
    Nb_1(\lambda) := \text{`} & \text{there is a stationary-preserving forcing notion } \mathbb{P} \text{ such that} \\ 
    & \Vdash_{\mathbb{P}} ``cof(\alpha) = \omega" \text{ for all regular cardinals } \alpha \text{ satisfying } \omega_2 \leq \alpha < \lambda \text{,} \\
    & \Vdash_{\mathbb{P}} ``cof(\lambda) = \omega_1" \text{, and} \\
    & \Vdash_{\mathbb{P}} ``cof(\beta) > \omega" \text{ for all regular cardinals } \beta \text{ satisfying } \lambda < \beta \text{'}
\end{align*}
true, naturally arises from Fact \ref{fact26}. Indeed, Fact \ref{fact26} is equivalent to --- and can be rewritten as --- the following.

\begin{customfact}{4.7$'$}\label{f46p}
Assume $\mathsf{GCH}$ holds below $\omega_2$. Then Namba forcing witnesses $Nb_1(\omega_3)$.
\end{customfact}

Drawing from the deep and complex theories of subcomplete forcing and $\mathcal{L}$-forcing, Jensen showed in \cite{lforcing} that, modulo weak fragments of $\mathsf{GCH}$, $Nb_1(\lambda)$ holds for all successor and strongly inaccessible cardinals above $\omega_2$. Jensen used very different methods to construct the witnesses $\mathbb{P}$ for different categories of $\lambda$, but in doing so, he also ensured that $\mathbb{P}$ never adds reals. 

But can we have $Nb_1(\lambda)$ hold for a bigger class of cardinals $\lambda$ if we allow $\mathbb{P}$ to add reals?

\subsection{A Conditional Solution}

It turns out there is a somewhat simple proof of $$\text{``} Nb_1(\lambda) \text{ for every regular cardinal } \lambda > \omega_2 \text{''}$$ (in fact, a slightly stronger statement) if we assume a theory of greater consistency strength than $\mathsf{ZFC}$. This proof adopts a novel side-condition technique first employed in \cite{schindler} (cf. Remark \ref{rem47}). It also demonstrates how amenable the forcing framework of Subsection \ref{forframe} is in bolstering natural and obvious forcing conditions with said side conditions. 

\begin{defi}
Let $Nb'_1(\lambda)$ denote the statement
\begin{align*}
    \text{`} & \text{there is a stationary-preserving forcing notion } \mathbb{P} \text{ such that} \\ 
    & \Vdash_{\mathbb{P}} ``cof(\alpha) = \omega" \text{ for all regular cardinals } \alpha \text{ satisfying } \omega_2 \leq \alpha < \lambda \text{,} \\
    & \Vdash_{\mathbb{P}} ``cof(\gamma) = \omega_1" \text{ for all regular cardinals } \gamma \text{ satisfying } \lambda \leq \gamma \leq 2^{< \lambda} \text{, and} \\
    & \Vdash_{\mathbb{P}} ``cof(\beta) = \beta" \text{ for all regular cardinals } \beta \text{ satisfying } 2^{< \lambda} < \beta \text{'.}
\end{align*}
\end{defi}

Immediately, one can see that $Nb'_1(\lambda)$ implies $Nb_1(\lambda)$ whenever $\lambda > \omega_2$ is a regular cardinal. By generalising the (redacted) proof of Fact \ref{fact26} (or equivalently, Fact \ref{f46p}), we can show that statements in the class
\begin{equation*}
    \{Nb'_1(\lambda) : \lambda > \omega_2\}
\end{equation*}
are indeed extensions of a property of Namba forcing, modulo the same mild assumption beyond $\mathsf{ZFC}$.

\begin{fact}
Assume $\mathsf{GCH}$ holds below $\omega_2$. Then Namba forcing witnesses $Nb'_1(\omega_3)$.
\end{fact}

The next theorem, also the main one in this section, tells us $Nb'_1(\lambda)$ holds for all but set many regular cardinals $\lambda$, under significantly stronger assumptions.

\begin{thm}\label{notion1}
Assume $\mathrm{NS}_{\omega_1}$ is precipitous. Then $Nb'_1(\lambda_f)$ holds for all regular cardinals $\lambda_f > 2^{\omega_1}$. In other words, whenever $\lambda_f > 2^{\omega_1}$ is a regular cardinal, there is a stationary-preserving forcing notion $\mathbb{P}'$ such that 
\begin{enumerate}[label=(\arabic*)]
    \item\label{cond1} $\Vdash_{\mathbb{P}'} ``cof(\alpha) = \omega"$ for all regular cardinals $\alpha$ satisfying $\omega_2 \leq \alpha < \lambda_f$,
    \item\label{cond2} $\Vdash_{\mathbb{P}'} ``cof(\gamma) = \omega_1"$ for all regular cardinals $\gamma$ satisfying $\lambda_f \leq \gamma \leq 2^{< \lambda_f}$, and
    \item\label{cond3} $\Vdash_{\mathbb{P}'} ``cof(\beta) = \beta"$ for all regular cardinals $\beta$ satisfying $2^{< \lambda_f} < \beta$.
\end{enumerate}
\end{thm}

\begin{proof}
Assume $\mathrm{NS}_{\omega_1}$ is precipitous, and fix a regular cardinal $\lambda_f > 2^{\omega_1}$.  

Let $h$ be a generic filter on $Col(\lambda_f, \lambda_f)$, so that in $V[h]$, 
\begin{itemize}
    \item $cof(\alpha) = cof^{V}(\alpha)$ for all $ \alpha \leq \lambda_f$,
    \item $\mathrm{NS}_{\omega_1}$ is still precipitous,
    \item $|H(\lambda_f)| = \lambda_f$, and
    \item there is a a $\Diamond_{\lambda_f}$-sequence $(\bar{A}_{\lambda} : \lambda < \lambda_f)$.
\end{itemize}

\begin{lem}\label{lem26}
If in $V[h]$ there is a stationary-preserving forcing notion $\mathbb{P}$ of size $\leq \lambda_f$ fulfilling 
\begin{enumerate}[label=(\arabic*')]
    \item\label{4101} $\Vdash_{\mathbb{P}} ``cof(\alpha) = \omega"$ for all regular cardinals $\alpha$ satisfying $\omega_2 \leq \alpha < \lambda_f$, and
    \item\label{4102} $\Vdash_{\mathbb{P}} ``cof(\lambda_f) = \omega_1"$,
\end{enumerate}
then for some $Col(\lambda_f, \lambda_f)$-name $\dot{\mathbb{P}}$ for $\mathbb{P}$,
\begin{equation*}
    Col(\lambda_f, \lambda_f) * \dot{\mathbb{P}} \text{,}
\end{equation*} 
is a stationary-preserving forcing notion fulfilling \ref{cond1} to \ref{cond3} of Theorem \ref{notion1} in $V$.
\end{lem}

\begin{proof}
Working in $V$, we set $$\mathbb{P}' := Col(\lambda_f, \lambda_f) * \dot{\mathbb{P}} \text{,}$$ and note the following facts:
\begin{enumerate}[label=(\arabic*')]
    \setcounter{enumi}{2}
    \item\label{4103} $|Col(\lambda_f, \lambda_f)| = 2^{< \lambda_f}$,
    \item\label{4104} $\Vdash_{Col(\lambda_f, \lambda_f)} ``2^{< \lambda_f} = \lambda_f"$
\end{enumerate}
$\mathbb{P}'$ is a stationary-preserving forcing notion fulfilling \ref{cond1} of Theorem \ref{notion1} because 
\begin{itemize}
    \item $Col(\lambda_f, \lambda_f)$ is a stationary-preserving forcing notion forcing $\dot{\mathbb{P}}$ to be stationary-preserving, 
    \item $Col(\lambda_f, \lambda_f)$ forces $cof(\alpha) = cof^{V}(\alpha)$ for all $ \alpha < \lambda_f$, and
    \item \ref{4101} holds.
\end{itemize}

In $V^{Col(\lambda_f, \lambda_f)}$, we are given $|\mathbb{P}| \leq \lambda_f$, so $\mathbb{P}$ preserves cofinalities $\geq \lambda_f^+$. But $\lambda_f^+$ in $V^{Col(\lambda_f, \lambda_f)}$ is exactly $(2^{< \lambda_f})^+$ in $V$, by \ref{4103} and \ref{4104}, so $\mathbb{P}'$ preserves cofinalities $\geq (2^{< \lambda_f})^+$ in $V$. This implies \ref{cond3} of Theorem  \ref{notion1}.

Now let $\gamma \geq \lambda_f$ be regular in $V$. Then $cof^{Col(\lambda_f, \lambda_f)} (\gamma) \geq \lambda_f$ because $Col(\lambda_f, \lambda_f)$ is $\lambda_f$-closed. That $\mathbb{P}$ both preserves cofinalities $\geq \lambda_f^+$ and fulfils \ref{4102} implies it forces $cof(\gamma) \geq \omega_1$. We have thus shown
\begin{enumerate}[label=(\arabic*')]
    \setcounter{enumi}{4}
    \item\label{4105} $\Vdash_{\mathbb{P}'} ``cof(\gamma) \geq \omega_1"$ for all regular cardinals $\gamma$ satisfying $\lambda_f \leq \gamma$.
\end{enumerate}
As \ref{4102}, \ref{4104} and \ref{cond1} of Theorem \ref{notion1} give us $$\Vdash_{\mathbb{P}'} ``cof(\gamma) \leq |\gamma| \leq \omega_1"$$ for all ordinals $\gamma$ satisfying $\lambda_f \leq \gamma \leq 2^{< \lambda_f}$, \ref{cond2} of Theorem \ref{notion1} must hold.
\end{proof}

Allow $W$ to denote $V[h]$. Unless otherwise stated, we work in $W$ from now on, towards a forcing notion $\mathbb{P}$ as in Lemma \ref{lem26}. For brevity, we write $\omega_1^V$ as just $\omega_1$. Let 
\begin{align*}
    \kappa & := (2^{\lambda_f})^+ \text{,} \\
    \mathfrak{A} & := (H(\kappa); \in) \text{, and} \\
    R & := \{i < \lambda_f : \omega_2 \leq i \text{ and } i \text{ is regular}\}.
\end{align*}

As $|H(\lambda_f)| = \lambda_f$, we can fix a bijection $c : \lambda_f \longrightarrow H(\lambda_f)$, and define
\begin{itemize}
    \item\label{a3} $Q_{\lambda} := c" \lambda$ and
    \item\label{a4} $A_{\lambda} := c" (\bar{A}_{\lambda})$
\end{itemize}
for each $\lambda < \lambda_f$.

Making use of straightforward closure arguments, we inductively define $C$ such that 
\begin{itemize}
    \item $C$ is a club in $\lambda_f$, and
    \item for all $\lambda \in C$, 
    \begin{itemize}[label=$\circ$]
        \item $Q_{\lambda}$ is transitive,
        \item $(Q_{\lambda}; \in, c \cap Q_{\lambda}) \prec (H(\lambda_f); \in, c)$.
    \end{itemize}
\end{itemize}
We will let $Q_{\lambda_f} = H(\lambda_f)$.

Now, given any $P, B \subset H(\lambda_f)$, the set
\begin{align*}
    \{\lambda \in C : (Q_{\lambda}; \in, P, B) \prec (H(\lambda_f); \in, P, B)\}
\end{align*}
is a club in $\lambda_f$. Moreover, we can derive from $(\bar{A}_{\lambda} : \lambda < \lambda_f)$ being a $\Diamond_{\lambda_f}$-sequence, that the set
\begin{align*}
    \{\lambda \in C : B \cap Q_{\lambda} = A_{\lambda}\}
\end{align*}
is stationary in $\lambda_f$. We thus obtain
\begin{itemize}[label=($\diamond$)]
    \item for all $P, B \subset H(\lambda_f)$, the set
    \begin{align*}
        \{\lambda \in C : (Q_{\lambda}; \in, P, A_{\lambda}) \prec (H(\lambda_f); \in, P, B)\}
    \end{align*}
    is stationary in $\lambda_f$.
\end{itemize}

We want to define $\mathbb{P}$ as a forcing notion comprising finite fragments of some language $\mathcal{L} \subset H(\lambda_f)$, such that $\mathbb{P}$ satisfies the hypothesis of Lemma \ref{lem26}.

Let us first define $\mathcal{L}$.

\begin{defi}\label{defl}
The language $\mathcal{L}$ requires the following distinguished symbols:
\begin{itemize}
    \item $\dot{F}_i$ for $i \in R$, and
    \item $\dot{X}_{\delta, \lambda}$ for $\delta < \omega_1$ and $\lambda \in C$.
\end{itemize}

Now fix $\mathcal{L}$ to be the smallest set closed under negation, that contains expressions (strings) of the following types:
\begin{enumerate}[label=(L\arabic*), leftmargin=40pt]
    \item $\ulcorner \dot{F}_i (n) = \alpha \urcorner$, for $i \in R$, $n < \omega$ and $\alpha < i$, and
    \item $\ulcorner x \in \dot{X}_{\delta, \lambda} \urcorner$, for $\delta < \omega_1$, $\lambda \in C$ and $x \in Q_{\lambda}$.
\end{enumerate}
\end{defi}

Morally, each $\dot{F}_i$ labels an increasing and cofinal partial map from $f(i)$ into $i$, and each $\dot{X}_{\delta}$ labels a side condition. The side conditions will be used to preserve stationary subsets of $\omega_1$. As 
\begin{itemize}
    \item $\mathfrak{A}$ is a transitive model of $\mathsf{ZFC - Powerset}$, and
    \item $\mathcal{L} \subset H(\lambda_f) \in H(\kappa)$,
\end{itemize} 
$\mathfrak{A}$ is $\mathcal{L}$-suitable.

One may argue that the remark below has been a long time coming, considering it was referenced in the proof of Lemma \ref{lem339}.

\begin{rem}\label{rempp}
Sometimes, we want to pass certain parameters of an expression in $\mathcal{L}$ via variables. In such circumstances we are formally passing the parameters through the functions
\begin{enumerate}[label=$\chi_{\arabic*} :$, leftmargin=45pt]
    \item $(i, n, \alpha) \mapsto \ulcorner \dot{F}_i (n) = \alpha \urcorner$,
    \item $(\delta, \lambda, x) \mapsto \ulcorner x \in \dot{X}_{\delta, \lambda} \urcorner$,
    \item $(i, n, \alpha) \mapsto \ulcorner \neg \dot{F}_i (n) = \alpha \urcorner$, and
    \item $(\delta, \lambda, x) \mapsto \ulcorner \neg x \in \dot{X}_{\delta, \lambda} \urcorner$,
\end{enumerate}
with their domains restricted to $H(\lambda_f)$. Note that under this domain restriction, $\chi_1$, $\chi_2$, $\chi_3$ and $\chi_4$ are all members of $H(\kappa)$; in fact, they are all $\Delta_0$-definable functions in $(H(\lambda_f); \in)$ (and thus in $\mathfrak{A}$). For brevity's and clarity's sake, we will abuse notation and suppress mention of these functions, whenever it is clear that we are using variables as placeholders for parameters in our construction of $\mathcal{L}^{*}_{\mathfrak{A}}$ formulas involving the symbol $\ulcorner E \urcorner$ --- see e.g. Definition \ref{defc}. Since 
\begin{itemize}
    \item for $k \in \{1, 2, 3, 4\}$ and $(a, b, c) \in dom(\chi_k)$, $$\text{``} E(\chi_k(a,b,c)) \text{''}$$ can be viewed as a shorthand for both $$\ulcorner \exists z \ (\varphi \wedge E(z)) \urcorner$$ and $$\ulcorner \forall z \ (\varphi \implies E(z)) \urcorner,$$ where $\varphi$ is a $\Delta_0$ formula (going by Definition \ref{def27}) in the language associated with $\mathfrak{A}$ expressing the statement $$\text{``} \chi_k(a, b, c) = z \text{'',}$$ and
    \item $\Delta_0$ formulas in the language associated with $\mathfrak{A}$ is of the lowest complexity class ($\mathcal{L}^{*}_{\mathfrak{A}}$-$\Delta_0$) in our classification of $\mathcal{L}^{*}_{\mathfrak{A}}$ formulas,
\end{itemize} 
by an argument similar to the main one of Remark \ref{shcs}, this method of passing parameters incurs no additional cost to the complexity of any $\mathcal{L}^{*}_{\mathfrak{A}}$-$\Sigma_1$ (and hence also $\mathcal{L}^{*}_{\mathfrak{A}}$-$\Pi_2$) formula thus abbreviated.
\end{rem}

By our choice of $C$, if $\lambda \in C$, then $\mathcal{L} \cap Q_{\lambda}$ is precisely the smallest set closed under negation, that contains expressions of the following types: 
\begin{enumerate}[label=(L\arabic*)$_{\lambda}$, leftmargin=45pt]
    \item\label{l1lamb} $\ulcorner \dot{F}_i (n) = \alpha \urcorner$, for $i \in R \cap \lambda$, $n < \omega$ and $\alpha < i$, and
    \item\label{l2lamb} $\ulcorner x \in \dot{X}_{\delta, \lambda'} \urcorner$, for $\delta < \omega_1$, $\lambda' \in C \cap \lambda$ and $x \in Q_{\lambda'}$.
\end{enumerate}

\begin{defi}\label{def412}
For any $\lambda \in C \cup \{\lambda_f\}$, an object of the form $$\langle \langle F_i : i \in Z \rangle, \langle X_{\delta, \lambda} : \delta < \omega_1, \lambda \in C \rangle \rangle$$ \emph{interprets} $\mathcal{L} \cap Q_{\lambda}$ iff 
\begin{itemize}
    \item $Z = R \cap Q_{\lambda}$,
    \item each $F_i$ is a partial function from $\omega$ into $i$, and 
    \item each $X_{\delta, \lambda}$ is a subset of $Q_{\lambda}$. 
\end{itemize}
\end{defi}

Sometimes it is convenient to talk about interpretations of specific symbols occurring in $\mathcal{L}$.

\begin{defi}
For any pair $(i, \Sigma)$, define $F_i(\Sigma)$ to be the set $$\{(n, \alpha) : \ulcorner \dot{F}_i(n) = \alpha \urcorner \in \Sigma\}.$$
\end{defi}

\begin{defi}
For any triple $(\delta, \lambda, \Sigma)$, define $X_{\delta, \lambda}(\Sigma)$ to be the set $$\{x : \ulcorner x \in \dot{X}_{\delta, \lambda} \urcorner \in \Sigma\}.$$
\end{defi}

\begin{defi}\label{def415}
Given $\lambda \in C \cup \{\lambda_f\}$, $$\mathfrak{C} := \langle \langle F_i : i \in Z \rangle, \langle X_{\delta, \lambda} : \delta < \omega_1, \lambda \in C \rangle \rangle$$
interpreting $\mathcal{L} \cap Q_{\lambda}$ and $\mathcal{L}' \subset \mathcal{L}$, let $\Sigma(\mathfrak{C}, \mathcal{L}')$ denote the union of the following sets:
\begin{itemize}
    \item $\{\ulcorner \dot{F}_i (n) = \alpha \urcorner \in \mathcal{L}' : F_i (n) = \alpha\}$,
    \item $\{\ulcorner \neg \dot{F}_i (n) = \alpha \urcorner \in \mathcal{L}' : F_i (n) \neq \alpha\}$,
    \item $\{\ulcorner x \in \dot{X}_{\delta, \lambda} \urcorner \in \mathcal{L}' : x \in X_{\delta, \lambda}\}$, and
    \item $\{\ulcorner \neg x \in \dot{X}_{\delta, \lambda} \urcorner \in \mathcal{L}' : x \not\in X_{\delta, \lambda}\}$.
\end{itemize}
It is clear that $\Sigma(\mathfrak{C}, \mathcal{L}) \cap Q_{\lambda} = \Sigma(\mathfrak{C}, \mathcal{L} \cap Q_{\lambda})$ is $\mathcal{L} \cap Q_{\lambda}$-nice for all $\lambda \in C \cup \{\lambda_f\}$.
\end{defi}

\begin{rem}\label{rem416}
If 
\begin{itemize}
    \item $\lambda \in C \cup \{\lambda_f\}$, 
    \item $\Sigma$ is $\mathcal{L} \cap Q_{\lambda}$-nice, and
    \item $\mathfrak{C} := \langle \langle F_i(\Sigma) : i \in R \cap \lambda \rangle, \langle X_{\delta, \lambda'}(\Sigma) : \delta < \omega_1, \lambda' \in C \cap \lambda \rangle \rangle$ interprets $\mathcal{L} \cap Q_{\lambda}$,
\end{itemize}
then $\Sigma(\mathfrak{C}, \mathcal{L} \cap Q_{\lambda}) = \Sigma$.
\end{rem}

We will define $\{\mathbb{P}_{\lambda} : \lambda \in C \cup \{\lambda_f\}\}$ by induction on $\lambda$. Assume that $\mathbb{P}_{\lambda'}$ has been defined for all $\lambda' \in \lambda \cap C$. Also, for $\lambda' \in \lambda \cap C$, allow 
\begin{itemize}
    \item $\mathcal{L^*}$ to denote the set of first order formulas over the signature $\{\in, P, B\}$,
    \item $Ef_0^{\lambda}(\lambda')$ to denote the set 
    \begin{align*}
        \{(\phi, \bar{s}, r) : \ & \phi \in \mathcal{L}^* \text{ and} \\
        & r \in Q_{\lambda'} \text{ and} \\
        & \bar{s} \in (Q_{\lambda'})^{< \omega} \text{ and} \\
        & dom(\bar{s}) + 1 = arity(\phi) \text{ and} \\
        & (Q_{\lambda'}; \in, \mathbb{P}_{\lambda'}, A_{\lambda'}) \models \phi(r, \bar{s})\},
    \end{align*}
    \item $Ef_1^{\lambda}(\lambda')$ to denote the set
    \begin{align*}
        \{(\phi, \bar{s}) : \exists r \ ((\phi, \bar{s}, r) \in Ef_0(\lambda'))\},
    \end{align*}
    and
    \item $Df^{\lambda}(\lambda')$ to denote the set 
    \begin{align*}
        \{(\phi, \bar{s}) : \ & \phi \in \mathcal{L}^* \text{ and} \\
        & \bar{s} \in (Q_{\lambda'})^{< \omega} \text{ and} \\
        & dom(\bar{s}) + 1 = arity(\phi) \text{ and} \\
        & \{y \in \mathbb{P}_{\lambda'} : (Q_{\lambda'}; \in, \mathbb{P}_{\lambda'}, A_{\lambda'}) \models \phi(y, \bar{s})\} \text{ is dense in } \mathbb{P}_{\lambda'}\}.
    \end{align*}
\end{itemize}
The functions 
\begin{align*}
    Ef_0^{\lambda} : \ & \lambda \cap C \longrightarrow H(\lambda_f) \ [\lambda' \mapsto Ef_0(\lambda')] \text{,} \\
    Ef_1^{\lambda} : \ & \lambda \cap C \longrightarrow H(\lambda_f) \ [\lambda' \mapsto Ef_1(\lambda')] \text{, and}  \\
    Df^{\lambda} : \ & \lambda \cap C \longrightarrow H(\lambda_f) \ [\lambda' \mapsto Df(\lambda')]
\end{align*}
are clearly members of $H(\kappa)$.

\begin{defi}\label{defc}
Let $\Gamma_{\lambda}$ be the set of $\mathcal{L}^*_{\mathfrak{A}}$-$\Pi_2$ sentences enumerated below.  
\begin{enumerate}[label=(S\arabic*)$_{\lambda}$, leftmargin=40pt]
    \item\label{c1} $\ulcorner \forall i \ \forall n \ \forall \alpha \ \forall \beta \ ((E(\ulcorner \dot{F}_i (n) = \alpha \urcorner) \wedge E(\ulcorner \dot{F}_i (n) = \beta \urcorner)) \implies \alpha = \beta) \urcorner$,
    
    \item\label{c2}
    \!
    $\begin{aligned}[t]
        \ulcorner & \forall i \ \forall \alpha \ \forall m \ \forall \gamma \ \exists n \ \exists \beta \\
        & ((E(\ulcorner \dot{F}_i (m) = \gamma \urcorner) \vee E(\ulcorner \neg \dot{F}_i (m) = \gamma \urcorner)) \wedge \alpha \in i \\
        & \implies (\alpha \in \beta \wedge E(\ulcorner \dot{F}_i (n) = \beta \urcorner))) \urcorner, 
    \end{aligned}$

    \item\label{c3} 
    \!
    $\begin{aligned}[t]
        \ulcorner & \forall \delta \ \forall \lambda_0 \ \forall \lambda_1 \ \forall x \ \forall y \\
        & ((E(\ulcorner x \in \dot{X}_{\delta, \lambda_0} \urcorner) \wedge E(\ulcorner y \in \dot{X}_{\delta, \lambda_1} \urcorner)) \implies \lambda_0 = \lambda_1) \urcorner,
    \end{aligned}$

    \item\label{c4}
    \!
    $\begin{aligned}[t]
        \ulcorner & \forall \delta_0 \ \forall \delta_1 \ \forall \lambda_0 \ \forall \lambda_1 \ \forall x \ \forall y \\
        & ((E(\ulcorner x \in \dot{X}_{\delta_0, \lambda_0} \urcorner) \wedge E(\ulcorner y \in \dot{X}_{\delta_1, \lambda_1} \urcorner) \wedge \delta_0 \in \delta_1) \\
        & \implies \lambda_0 \in \lambda_1) \urcorner
    \end{aligned}$

    \item\label{c5} $\ulcorner \forall \delta \ \forall \lambda' \ \forall \delta' \in \delta \ \forall x \ (E(\ulcorner x \in \dot{X}_{\delta, \lambda'} \urcorner) \implies E(\ulcorner \delta' \in \dot{X}_{\delta, \lambda'} \urcorner)) \urcorner$, 
    
    \item\label{c6} $\ulcorner \forall \delta \ \forall \lambda' \ \forall \delta' \in \omega_1 \ \forall x \ ((E(\ulcorner x \in \dot{X}_{\delta, \lambda'} \urcorner) \wedge \delta \in \delta') \implies \neg E(\ulcorner \delta' \in \dot{X}_{\delta, \lambda'} \urcorner)) \urcorner$, 
    
    \item\label{c7} 
    \!
    $\begin{aligned}[t]
        \ulcorner & \forall \delta \ \forall \lambda' \ \forall \phi \ \forall \bar{s} \ \forall x \ \exists r \\
        & ((E(\ulcorner x \in \dot{X}_{\delta, \lambda'} \urcorner) \ \wedge \ \\
        & (\phi, \bar{s}) \in Ef_1^{\lambda}(\lambda') \ \wedge \\
        & \forall n \in dom(\bar{s}) \ (E(\ulcorner \bar{s}(n) \in \dot{X}_{\delta, \lambda'} \urcorner))) \\
        & \implies (E(\ulcorner r \in \dot{X}_{\delta, \lambda'} \urcorner) \wedge (\phi, \bar{s}, r) \in Ef_0^{\lambda}(\lambda'))) \urcorner,
    \end{aligned}$
    
    \item\label{c8}
    \!
    $\begin{aligned}[t]
        \ulcorner & \forall \delta \ \forall \lambda' \ \forall \phi \ \forall \bar{s} \ \forall x \ \exists p \in \mathbb{P}_{\lambda'} \\
        & ((E(\ulcorner x \in \dot{X}_{\delta, \lambda'} \urcorner) \ \wedge \\
        & (\phi, \bar{s}) \in Df^{\lambda}(\lambda') \ \wedge \\
        & \forall n \in dom(\bar{s}) \ (E(\ulcorner \bar{s}(n) \in \dot{X}_{\delta, \lambda'} \urcorner))) \\
        & \implies ((\phi, \bar{s}, p) \in Ef_0^{\lambda}(\lambda') \ \wedge \\
        & E(\ulcorner p \in \dot{X}_{\delta, \lambda'} \urcorner) \ \wedge \\
        & \forall e \ ((e \in p \wedge \text{``}\emptyset \neq p \text{ and } p \text{ is Dedekind-finite''}) \implies E(e)))) \urcorner.
    \end{aligned}$
\end{enumerate}
\end{defi}

In Definition \ref{defc}, we give a list of constraints on the $\dot{F}_i$'s and the $\dot{X}_{\delta, \lambda'}$'s, that are meant to dictate how the objects interpreting them behave. To be more formal, let $\Sigma$ interpret the unary relation symbol $\ulcorner E \urcorner$ occurring in $\mathcal{L}^*_{\mathfrak{A}}$ (formulas). Here, we are using the term ``interpret'' in the conventional model-theoretic sense. Also, let
\begin{itemize}
    \item $F_i := F_i(\Sigma)$, and
    \item $X_{\delta, \lambda'} := X_{\delta, \lambda'}(\Sigma)$,
\end{itemize}
as $i$, $\delta$ and $\lambda'$ range over their appropriate domains. Then
\begin{itemize}
    \item \ref{c1} and \ref{c3} mean to say that the $F_i$'s and the set $$\{(\delta, \lambda') : X_{\delta, \lambda'} \neq \emptyset\}$$ are functions,
    \item \ref{c2} means to say that the image of each $F_i$ is cofinal in $i$,
    \item \ref{c4} means to say that the function $$\{(\delta, \lambda') : X_{\delta, \lambda'} \neq \emptyset\}$$ is strictly increasing,
    \item \ref{c5} and \ref{c6} mean to say that $X_{\delta, \lambda'} \cap \omega_1 = \delta$ whenever $X_{\delta, \lambda'}$ is non-empty,
    \item \ref{c7} means to tell us that if $X_{\delta, \lambda'}$ is non-empty, then $(X_{\delta, \lambda'}; \in,  \mathbb{P}_{\lambda'}, A_{\lambda'})$ is an elementary submodel of $(Q_{\lambda'}; \in,  \mathbb{P}_{\lambda'}, A_{\lambda'})$, and
    \item \ref{c8} means to tell us that if $X_{\delta, \lambda'}$ is non-empty, then for every dense subset $D$ of $\mathbb{P}_{\lambda'}$ definable over $(Q_{\lambda'}; \in, \mathbb{P}_{\lambda'}, A_{\lambda'})$ with parameters from $X_{\delta, \lambda'}$, $$[\Sigma]^{< \omega} \cap X_{\delta, \lambda'} \cap D \neq \emptyset.$$ 
\end{itemize}

Now we can define $\mathbb{P}_{\lambda} := (P_{\lambda}, \leq_{\lambda})$, where
\begin{align*}
    P_{\lambda} := \ & \{p \in [\mathcal{L} \cap Q_{\lambda}]^{< \omega} : \ \Vdash_{Col(\omega, |H(\kappa)|)} \exists \Sigma \ (``\Sigma \ \Gamma_{\lambda} (\mathcal{L} \cap Q_{\lambda}, \mathfrak{A}) \text{-certifies } p")\}, \text{ and} \\
    \leq_{\lambda} \ := \ & \{(p, q) \in P_{\lambda} \times P_{\lambda} : q \subset p\}.
\end{align*}

We let $\mathbb{P}$ denote $\mathbb{P}_{\lambda_f}$. 

\begin{lem}\label{size}
$|\mathbb{P}| \leq \lambda_f$.
\end{lem}

\begin{proof}
This follows immediately from the observation that 
\begin{align*}
    \mathbb{P} \subset \ & [\mathcal{L}]^{<\omega} \subset H(\lambda_f) \text{, and} \\
    & |H(\lambda_f)| = \lambda_f \text{.} \qedhere
\end{align*}
\end{proof}

By Proposition \ref{certgood} and the lemma below, $(\mathfrak{A}, \mathbb{P})$ is good for $\mathcal{L}$. Obviously, $\mathbb{P}$ is definable in the language associated with $\mathfrak{A}$ because $\mathbb{P} \in H(\kappa)$.

\begin{lem}\label{nonemp}
For all $\lambda \in C \cup \{\lambda_f\}$, $\emptyset \in P_{\lambda}$.
\end{lem}

\begin{proof}
Let $g$ be $Col(\omega, |H(\kappa)|)$-generic over $W$. In $W[g]$, for every $i \in R$, choose a cofinal map from $\omega$ into $i$ and call it $F_i$. For every $\delta < \omega_1$ and every $\lambda \in C$, let $X_{\delta, \lambda}$ be the empty set. Then $$\mathfrak{C} := \langle \langle F_i : i \in R \rangle, \langle X_{\delta, \lambda} : \delta < \omega_1, \lambda \in C \rangle \rangle$$ interprets $\mathcal{L}$ and $\Sigma(\mathfrak{C}, \mathcal{L} \cap Q_{\lambda})$ $\Gamma_{\lambda} (\mathcal{L} \cap Q_{\lambda}, \mathfrak{A})$-certifies $\emptyset$ for all $\lambda \in C \cup \{\lambda_f\}$.  
\end{proof}

Using a argument similar to that in the proof of Lemma \ref{nonemp}, we get the following.

\begin{lem}\label{extcert}
If 
\begin{itemize}
    \item $\lambda_0, \lambda_1 \in C \cup \{\lambda_f\}$, 
    \item $\lambda_0 \leq \lambda_1$, and 
    \item $\Sigma$ $\Gamma_{\lambda_0} (\mathcal{L} \cap Q_{\lambda_0}, \mathfrak{A})$-certify $p$, 
\end{itemize}
then there is $\Sigma' \supset \Sigma$ for which $\Sigma'$ $\Gamma_{\lambda_1} (\mathcal{L} \cap Q_{\lambda_1}, \mathfrak{A})$-certify $p$.
\end{lem}

It can be gleaned from Lemma \ref{extcert} and the definition of the $\mathbb{P}_{\lambda}$'s that
\begin{enumerate}[label=(P\arabic*), leftmargin=40pt]
    \item\label{p1} $P_{\lambda_0} = P_{\lambda_1} \cap Q_{\lambda_0}$ whenever $\lambda_0, \lambda_1 \in C \cup \{\lambda_f\}$ and $\lambda_0 \leq \lambda_1$, and
    \item\label{p2} $P_{\lambda} = \bigcup \{P_{\lambda'} : \lambda' \in C \cap \lambda\}$ whenever $\lambda \in C \cup \{\lambda_f\}$ and $sup(\lambda \cap C) = \lambda$.
\end{enumerate}

\begin{lem}\label{modeldone}
Let 
\begin{itemize}
    \item $\lambda' \in C \cup \{\lambda_f\}$, and
    \item $g$ be a $\mathbb{P}_{\lambda'}$-$\Sigma_1$-generic filter over $W$.
\end{itemize}
Then $\bigcup g \ \Gamma_{\lambda'} (\mathcal{L} \cap Q_{\lambda}, \mathfrak{A}) \text{-certifies } \emptyset$.
\end{lem}

\begin{proof}
We apply Lemma \ref{main2} with 
\begin{itemize}
    \item $\mathfrak{A}$, $\mathcal{L}$ and $g$ as defined or given above,
    \item $P_{\lambda'}$ in place of $P$,
    \item $\mathbb{P}_{\lambda'}$ in place of $\mathbb{P}$, 
    \item $\Gamma_{\lambda'}$ in place of $\Gamma$,
    \item $|H(\kappa)|$ in place of $\lambda$,
    \item $[\mathcal{L} \cap Q_{\lambda'}]^{< \omega}$ in place of $B$,
    \item $W$ in place of $V$, and
    \item $W[g]$ in place of $W$,
\end{itemize}
noting that
\begin{itemize}
    \item $|trcl(\mathfrak{A})| \leq |H(\kappa)|$,
    \item $[\mathcal{L} \cap Q_{\lambda'}]^{< \omega}$ is closed under finite extensions,
    \item the definition of $\mathbb{P}_{\lambda'}$ in relation to the other parameters is faithful to the hypothesis of Lemma \ref{main2},
    \item $g$ satisfies the hypothesis of Lemma \ref{main2} with respect to the other parameters, and
    \item $\mathbb{P}_{\lambda'}$ being non-empty (per Lemma \ref{nonemp}) implies there is $\Sigma$ in some weak outer model $W'$ of $W$ such that $\Sigma$ $\Gamma_{\lambda'} (\mathcal{L} \cap Q_{\lambda'}, \mathfrak{A})$-certifies $\emptyset$,
\end{itemize}
to arrive at $\bigcup g \ \Gamma_{\lambda'} (\mathcal{L} \cap Q_{\lambda}, \mathfrak{A}) \text{-certifies } \emptyset$.
\end{proof}

The proof of Lemma \ref{modeldone} serves as an instructive example of the utility of Lemma \ref{main2}. We shall omit details in subsequent applications of Lemma \ref{main2}, wherever the use cases are deemed similarly straightforward. 

\begin{lem}\label{1done}
$\mathbb{P}$ fulfils \ref{4101} of Lemma \ref{lem26}.
\end{lem}

\begin{proof}
For any $\mathbb{P}_{\lambda'}$-generic filter $g$ over $W$, $\bigcup g \models^*_{\mathfrak{A}} \Gamma_{\lambda_f}$ by Lemma \ref{modeldone}. In particular, $\bigcup g \models^*_{\mathfrak{A}} \text{(S2)}_{\lambda_f}$. That $\mathbb{P}$ fulfils \ref{4101} of Lemma \ref{lem26} follows immediately.
\end{proof}

\begin{defi}\label{def423}
Let 
\begin{itemize}
    \item $S(\omega_1)$ denote the set of all stationary subsets of $\omega_1$, and
    \item $U(C, \lambda_f)$ denote the set of all subsets of $C$ unbounded in $\lambda_f$.
\end{itemize} 
\end{defi}

Check that both $S(\omega_1)$ and $U(C, \lambda_f)$ are members of $H(\kappa)$.

\begin{lem}\label{sideuni}
The $\mathcal{L}^*_{\mathfrak{A}}$-$\Pi_2$ sentence 
\begin{align}\label{c9}
    \ulcorner \forall x \in H(\lambda_f) \ \forall S \in S(\omega_1) \ \forall U \in U(C, \lambda_f) \ \exists \mu \in S \ \exists \nu \in U \ (E(\ulcorner x \in \dot{X}_{\mu, \nu} \urcorner)) \urcorner
\end{align} 
is $(\mathfrak{A}, \mathbb{P}, \mathcal{L})$-universal.
\end{lem}

\begin{proof}
Fix arbitrary
\begin{itemize}
    \item $p \in \mathbb{P}$,
    \item $x \in H(\lambda_f)$,
    \item $S \in S(\omega_1)$, and
    \item $U \in U(C, \lambda_f)$.
\end{itemize}
By Lemma \ref{univer}, it suffices to show that there are 
\begin{itemize}
    \item $\mu \in S$, and
    \item $\nu \in U$
\end{itemize} 
for which $p \cup \{\ulcorner x \in \dot{X}_{\mu, \nu} \urcorner\} \in \mathbb{P}$.

To that end, let $\nu \in U$ be such that $x \in Q_{\nu}$ and $p \in \mathbb{P}_{\nu}$. This is possible by \ref{p2}. Choose $g \times f$ a $\mathbb{P}_{\nu} \times Col(\omega, \nu)$-generic filter over $W$ with $p \in g$, so that $g \in W[g \times f]$ is a $\mathbb{P}_{\nu}$-generic filter over $W$ and $|\nu|^{W[g \times f]} = \omega$. By Lemma \ref{modeldone}, $\bigcup g \models^*_{\mathfrak{A}} \Gamma_{\nu}$.

Since $\mathrm{NS}_{\omega_1}$ is precipitous in $W$, $(W; \in, (\mathrm{NS}_{\omega_1})^W)$ is generically iterable in $W[g \times f]$. Consider a one-step iteration $$\mathfrak{I}_1 = \langle (W; \in, (\mathrm{NS}_{\omega_1})^W), (W_1; \in, I_1) \rangle$$ in $W[g \times f]$, where $(W_1; \in, I_1)$ is the generic ultrapower of $(W; \in, (\mathrm{NS}_{\omega_1})^W)$ via a $W$-generic ultrafilter on $(\mathrm{NS}_{\omega_1})^W$ containing $S$. Extend $\mathfrak{I}_1$ to a generic iteration $\mathfrak{I}$ of length $\omega_1^{W[g \times f]} + 1$ in $W[g \times f]$. Said iteration gives rise to a generic ultrapower map $j : W \longrightarrow M$, where $M$, an inner model of $W[g \times f]$, is the final iterate of $\mathfrak{I}$. Moreover, 
\begin{itemize}
    \item $crit(j) = \omega_1^W \in j(S)$, and
    \item $j(\omega_1^W) = \omega_1^{W[g \times f]}$.
\end{itemize}

Let 
\begin{align*}
    \Sigma' & := j" (\bigcup g) \cup \{\ulcorner j(y) \in \dot{X}_{\omega_1^W \!\! , j(\nu)} \urcorner : y \in Q_{\nu}\} \subset j(\mathcal{L}),  \\
    e  & \text{ be } Col(\omega, j(\lambda_f)) \text{-generic over }  W[g \times f] \text{, and} \\
    W^* & := W[g \times f][e].
\end{align*} 
Working in $W^*$, define $$\mathfrak{C} := \langle \langle F_i : i \in j(R) \rangle, \langle X_{\delta, \lambda} : \delta < \omega_1^M, \lambda \in j(C) \rangle \rangle$$ as follows:
\begin{itemize}
    \item $F_i := F_i(\Sigma')$ whenever $\dot{F}_i$ occurs in $\Sigma'$,
    \item $F_i$ is some (any) strictly increasing cofinal function from $\omega$ into $i$ whenever $\dot{F}_i$ does not occur in $\Sigma'$, and
    \item $x \in X_{\delta, \lambda}$ iff $\ulcorner x \in \dot{X}_{\delta, \lambda} \urcorner \in \Sigma'$.
\end{itemize}
Then $\mathfrak{C}$ interprets $j(\mathcal{L})$, and $\Sigma^* := \Sigma(\mathfrak{C}, j(\mathcal{L}))$ is $j(\mathcal{L})$-nice, noting Remark \ref{rem416}. Obviously $j" \bigcup g \subset \Sigma' \subset \Sigma^*$, so $j(p) = j" p \subset \Sigma^*$. By the definition of $\Sigma'$, we too have $$q^* := j(p) \cup \{\ulcorner j(x) \in \dot{X}_{\omega_1^W \!\! , j(\nu)} \urcorner\} \subset \Sigma^*.$$ In order to conclude that $\Sigma^*$ $j(\Gamma_{\lambda_f})(j(\mathcal{L}), j(\mathfrak{A}))$-certifies $q^*$, we are left with showing $\Sigma^* \models^*_{j(\mathfrak{A})} j(\Gamma_{\lambda_f})$. That $$\Sigma^* \models^*_{j(\mathfrak{A})} j(\text{(S<} k \text{>)}_{\lambda_f})$$ for <$k$> $\in \{1, 3, 4, 5, 6\}$ follows immediately from the construction of $\Sigma^*$, the elementarity of $j$, as well as the fact that $crit(j) = \omega_1^W$.

For <$k$> $\in \{2, 7, 8\}$, we check that $\Sigma^* \models^*_{j(\mathfrak{A})} j(\text{(S<} k \text{>)}_{\lambda_f})$ in greater detail below. 

\begin{enumerate}[label=<$k$> $\eq$ \arabic* :, leftmargin=70pt]
    \addtocounter{enumi}{1}
    \item Let $i \in j(R)$. If $\dot{F}_i$ does not occur in $\Sigma'$, there is nothing to check, because the definition of $\mathfrak{C}$ guarantees $ran(F_i(\Sigma^*))$ is cofinal in $i$. Otherwise, $\dot{F}_i$ occurs in $\Sigma'$, which means $\dot{F}_i$ occurs in $j" \bigcup g$. Then there is $i' \in R$ such that $\dot{F}_i = \dot{F}_{j(i')} = j(\dot{F}_{i'})$. That $\bigcup g \models^*_{\mathfrak{A}} \text{(S2)}_{\nu}$ implies $ran(F_{i'}(\bigcup g))$ is cofinal in $i'$. By a basic property of elementary embeddings associated with generic iterations, we know that for any ordinal $\alpha$ satisfying the inequality $\omega_1^W < cof^{W}(\alpha)$, we must have $j(\alpha) = sup(j" \alpha)$. Thus, $ran(F_i(\Sigma^*)) = j" ran(F_{i'}(\bigcup g))$ is cofinal in $j(i') = i$, and we are done.
    \addtocounter{enumi}{4}
    \item\label{k7m} Let $\delta$, $\lambda$, $\phi$, $\bar{s}$ and $x$ be such that 
    \begin{enumerate}[label=(K7.\arabic*), leftmargin=50pt]
        \item $\ulcorner x \in \dot{X}_{\delta, \lambda} \urcorner \in \Sigma^*$,
        \item\label{k72} $(\phi, \bar{s}) \in j(Ef_1^{\lambda_f})(\lambda)$, and
        \item\label{k73} $\ulcorner \bar{s}(n) \in \dot{X}_{\delta, \lambda} \urcorner \in \Sigma^*$ for all $n \in dom(\bar{s})$.
    \end{enumerate}
    Combining \ref{k72}, \ref{k73} and the definition of $\Sigma^*$ gives us 
    \begin{itemize}
        \item $\lambda = j(\lambda')$ for some $\lambda' \in C$,
        \item $j(\phi) = \phi \in j(\mathcal{L}^*) = \mathcal{L}^*$,
        \item $\bar{s} \in (j" Q_{\lambda'})^{< \omega} = j" (Q_{\lambda'})^{< \omega}$,
        \item $dom(\bar{s}) + 1 = arity(\phi)$, and
        \item $(j(Q_{\lambda'}); \in, j(\mathbb{P}_{\lambda'}), j(A_{\lambda'})) \models \exists r \ \phi(r, \bar{s})$.
    \end{itemize}
    We want to show that $$(X_{\delta, \lambda}(\Sigma^*); \in, j(\mathbb{P}_{\lambda'}), j(A_{\lambda'})) \models \exists r \ \phi(r, \bar{s}).$$ By the elementarity of $j$, 
    \begin{itemize}
        \item $j^{-1}(\bar{s}) \in (Q_{\lambda'})^{< \omega}$, 
        \item $dom(j^{-1}(\bar{s})) + 1 = arity(\phi)$, and
        \item $(Q_{\lambda'}; \in, \mathbb{P}_{\lambda'}, A_{\lambda'}) \models \exists r \ \phi(r, j^{-1}(\bar{s}))\}$,
    \end{itemize}
    so $(\phi, j^{-1}(\bar{s})) \in Ef_1^{\lambda_f}(\lambda')$. Henceforth, there are two possible cases. We will analyse them with reference to the way $\Sigma^*$ is constructed. 
    \begin{enumerate}[label=Case \arabic*:, leftmargin=50pt]
        \item\label{c7c1} $\delta = \omega_1^W$. Then $\lambda' = \nu$ and $X_{\delta, \lambda}(\Sigma^*) = j" Q_{\nu}$. As 
        \begin{align*}
            X_{\delta, \lambda}(\Sigma^*) \cap j(\mathbb{P}_{\lambda'}) = \ & j" \mathbb{P}_{\lambda'} \text{ and} \\
            X_{\delta, \lambda}(\Sigma^*) \cap j(A_{\lambda'}) = \ & j" A_{\lambda'},
        \end{align*}
        we can conclude $$(X_{\delta, \lambda}(\Sigma^*); \in, j(\mathbb{P}_{\lambda'}), j(A_{\lambda'})) \models \exists r \ \phi(r, \bar{s})$$ by invoking the elementarity of $j$ once again.
        \item $\delta \neq \omega_1^W$. Then $\delta < \omega_1^W$, $\lambda' < \nu$, and $$\ulcorner j^{-1}(\bar{s})(n) \in \dot{X}_{\delta, \lambda'} \urcorner \in \bigcup g$$ for all $n \in dom(j^{-1}(\bar{s}))$. Moreover, $X_{\delta, \lambda}(\Sigma^*) = j" X_{\delta, \lambda'}(\bigcup g)$, so $X_{\delta, \lambda}(\Sigma^*)$ being non-empty implies $X_{\delta, \lambda'}(\bigcup g)$ is non-empty as well. Since $\bigcup g \models^*_{\mathfrak{A}} \text{(S7)}_{\nu}$, we have $$(X_{\delta, \lambda'}(\bigcup g); \in, \mathbb{P}_{\lambda'}, A_{\lambda'}) \models \exists r \ \phi(r, j^{-1}(\bar{s})).$$ As 
        \begin{align*}
            X_{\delta, \lambda}(\Sigma^*) \cap j(\mathbb{P}_{\lambda'}) = \ & j" (X_{\delta, \lambda'}(\bigcup g) \cap \mathbb{P}_{\lambda'}) \text{ and} \\
            X_{\delta, \lambda}(\Sigma^*) \cap j(A_{\lambda'}) = \ & j" (X_{\delta, \lambda'}(\bigcup g) \cap A_{\lambda'}),
        \end{align*}
        we can conclude $$(X_{\delta, \lambda}(\Sigma^*); \in, j(\mathbb{P}_{\lambda'}), j(A_{\lambda'})) \models \exists r \ \phi(r, \bar{s})$$ by invoking the elementarity of $j$ yet again.
    \end{enumerate}
    \item This is similar to the argument in the case of <$k$> $\eq$ 7. We provide details for the sake of completeness, and to elucidate the ample similarity. 
    
    Let $\delta$, $\lambda$, $\phi$, $\bar{s}$ and $x$ be such that 
    \begin{enumerate}[label=(K8.\arabic*), leftmargin=50pt]
        \item $\ulcorner x \in \dot{X}_{\delta, \lambda} \urcorner \in \Sigma^*$,
        \item\label{k82} $(\phi, \bar{s}) \in j(Df^{\lambda_f})(\lambda)$, and
        \item\label{k83} $\ulcorner \bar{s}(n) \in \dot{X}_{\delta, \lambda} \urcorner \in \Sigma^*$ for all $n \in dom(\bar{s})$.
    \end{enumerate}
    Combining \ref{k82}, \ref{k83} and the definition of $\Sigma^*$ gives us 
    \begin{itemize}
        \item $\lambda = j(\lambda')$ for some $\lambda' \in C$,
        \item $j(\phi) = \phi \in j(\mathcal{L}^*) = \mathcal{L}^*$,
        \item $\bar{s} \in (j" Q_{\lambda'})^{< \omega} = j" (Q_{\lambda'})^{< \omega}$,
        \item $dom(\bar{s}) + 1 = arity(\phi)$, and
        \item $D := \{y \in j(\mathbb{P}_{\lambda'}) : (j(Q_{\lambda'}); \in, j(\mathbb{P}_{\lambda'}), j(A_{\lambda'})) \models \phi(y, \bar{s})\}$ is dense in $j(\mathbb{P}_{\lambda'})$.
    \end{itemize}
    We want to show that $$[\Sigma^*]^{< \omega} \cap X_{\delta, \lambda}(\Sigma^*) \cap D \neq \emptyset.$$ By the elementarity of $j$, 
    \begin{itemize}
        \item $j^{-1}(\bar{s}) \in (Q_{\lambda'})^{< \omega}$, 
        \item $dom(j^{-1}(\bar{s})) + 1 = arity(\phi)$, and
        \item $j^{-1}(D) = \{y \in \mathbb{P}_{\lambda'} : (Q_{\lambda'}; \in, \mathbb{P}_{\lambda'}, A_{\lambda'}) \models \phi(y, j^{-1}(\bar{s}))\}$ is dense in $\mathbb{P}_{\lambda'}$,
    \end{itemize}
    so $(\phi, j^{-1}(\bar{s})) \in Df^{\lambda_f}(\lambda')$. Henceforth, there are two possible cases. We will analyse them with reference to the way $\Sigma^*$ is constructed. 
    \begin{enumerate}[label=Case \arabic*:, leftmargin=50pt]
        \item\label{c8c1} $\delta = \omega_1^W$. Then $\lambda' = \nu$ and $X_{\delta, \lambda}(\Sigma^*) = j" Q_{\nu}$. Clearly, $$[\bigcup g]^{< \omega} \cap Q_{\nu} \cap j^{-1}(D) = g \cap j^{-1}(D) \neq \emptyset,$$ as $g$ is $\mathbb{P}_{\nu}$-generic over $W$. That $j" \bigcup g \subset \Sigma^*$ means $$j(p) \in [\Sigma^*]^{< \omega} \cap X_{\delta, \lambda}(\Sigma^*) \cap D \neq \emptyset$$ for any $p \in g \cap j^{-1}(D)$.
        \item $\delta \neq \omega_1^W$. Then $\delta < \omega_1^W$, $\lambda' < \nu$, and $$\ulcorner j^{-1}(\bar{s})(n) \in \dot{X}_{\delta, \lambda'} \urcorner \in \bigcup g$$ for all $n \in dom(j^{-1}(\bar{s}))$. Moreover, $X_{\delta, \lambda}(\Sigma^*) = j" X_{\delta, \lambda'}(\bigcup g)$, so $X_{\delta, \lambda}(\Sigma^*)$ being non-empty implies $X_{\delta, \lambda'}(\bigcup g)$ is non-empty too. Since $\bigcup g \models^*_{\mathfrak{A}} \text{(S8)}_{\nu}$, we have $$[\bigcup g]^{< \omega} \cap X_{\delta, \lambda'}(\bigcup g) \cap j^{-1}(D) \neq \emptyset.$$ As in Case 1, we can conclude $$[\Sigma^*]^{< \omega} \cap X_{\delta, \lambda}(\Sigma^*) \cap D \neq \emptyset.$$
    \end{enumerate}
\end{enumerate}

Now that 
\begin{itemize}
    \item $\Sigma^* \in W^*$,
    \item $\Sigma^*$ $j(\Gamma_{\lambda_f})(j(\mathcal{L}), j(\mathfrak{A}))$-certifies $q^*$,
    \item $W^*$ is a weak outer model of $M$, and
    \item $|H(\kappa)^W|^W = |trcl(\mathfrak{A})|^W$,
\end{itemize} 
we can apply Lemma \ref{inout} with 
\begin{itemize}
    \item $M$ in place of $V$,
    \item $W^*$ in place of $W$,
    \item $j(|H(\kappa)^W|^W) = |H(j(\kappa))^M|^M$ in place of $\lambda$,
    \item $\Sigma^*$ in place of $\Sigma$,
    \item $j(\Gamma_{\lambda_f})$ in place of $\Gamma$,
    \item $j(\mathcal{L})$ in place of $\mathcal{L}$,
    \item $j(\mathfrak{A})$ in place of $\mathfrak{A}$, and
    \item $q^*$ in place of $p$,
\end{itemize}
noting that in $M$, $$Col(\omega, |H(j(\kappa))|) = j(Col(\omega, |H(\kappa)^W|^W)).$$ The application yields $$(M; \in) \models ``\Vdash_{j(Col(\omega, |H(\kappa)^W|^W)^W)} \exists \Sigma \ (``\Sigma \ j(\Gamma_{\lambda_f})(j(\mathcal{L}), j(\mathfrak{A})) \text{-certifies } q^*")".$$ But this means $q \in j(\mathbb{P})$, which implies $$(M; \in) \models \exists \mu \in j(S) \ (``j(p) \cup \{\ulcorner j(x) \in \dot{X}_{\mu, j(\nu)} \urcorner\} \in j(\mathbb{P})") \text{.}$$ By the elementarity of $j$, $$(W; \in) \models \exists \mu \in S \ (``p \cup \{\ulcorner x \in \dot{X}_{\mu, \nu} \urcorner\} \in \mathbb{P}") \text{,}$$ completing the proof.
\end{proof}

\begin{lem}\label{statdone}
$\mathbb{P}$ is stationary-preserving.
\end{lem}

\begin{proof}
Let 
\begin{itemize}
    \item $S \in S(\omega_1)$,
    \item $p \in \mathbb{P}$,
    \item $\dot{C}$ be a $\mathbb{P}$-name such that $p \Vdash_{\mathbb{P}} ``\dot{C}$ is a club in $\omega_1^W"$,
    \item $D := \{(q, \eta) \in \mathbb{P} \times \omega_1 : q \Vdash_{\mathbb{P}} \eta \in \dot{C}\}$,
    \item $g$ be a $\mathbb{P}$-generic filter over $W$ with $p \in g$.
\end{itemize}
Applying ($\diamond$) with 
\begin{itemize}
    \item $\mathbb{P}$ in place of $P$, and
    \item $D$ in place of $B$,
\end{itemize}
we get
\begin{align*}
    U := \{\lambda \in C : (Q_{\lambda}; \in, \mathbb{P}, A_{\lambda}) \prec (H(\lambda_f); \in, \mathbb{P}, D)\}
\end{align*}
is stationary in $\lambda_f$, so $U \in U(C, \lambda_f)$.

In $W[g]$, there are $\mu \in S$ and $\nu \in U$ such that 
\begin{align*}
    \emptyset \neq (X_{\mu, \nu}(\bigcup g); \in, \mathbb{P}, A_{\nu}) \prec (Q_{\nu}; \in, \mathbb{P}, A_{\nu}) \prec (H(\lambda_f); \in, \mathbb{P}, D),
\end{align*}
since $$\bigcup g \models^*_{\mathfrak{A}} \Gamma_{\lambda_f} \text{ (in particular } \bigcup g \models^*_{\mathfrak{A}} \text{(S7)}_{\lambda_f} \text{) and } \bigcup g \models^*_{\mathfrak{A}} \text{(}\ref{c9}\text{)}$$ by Lemmas \ref{uni}, \ref{modeldone} and \ref{sideuni}. Now, noting \ref{p1}, we have 
\begin{align}\label{eseq}
    \emptyset \neq (X_{\mu, \nu}(\bigcup g); \in, \mathbb{P}, D) \prec (Q_{\nu}; \in, \mathbb{P}_{\nu}, A_{\nu}) \prec (H(\lambda_f); \in, \mathbb{P}, D).
\end{align}
It suffices to show that $\mu$ is a limit point of $\dot{C}[g]$. We fix $\zeta < \mu$ and seek some $\eta \in \dot{C}[g]$ with $\zeta < \eta < \mu$. 

The set $$E_{\zeta} := \{q \in \mathbb{P} : \exists \eta > \zeta \ ((q, \eta) \in D)\}$$ is dense in $\mathbb{P}$, so (\ref{eseq}) tells us $$E_{\zeta} \cap Q_{\nu} = \{q \in \mathbb{P}_{\nu} : \exists \eta > \zeta \ ((q, \eta) \in A_{\nu})\}$$ is dense in $\mathbb{P}_{\nu}$. As $\bigcup g \models^*_{\mathfrak{A}} \text{(S5)}_{\lambda_f}$, we know $\zeta \in X_{\mu, \nu}(\bigcup g)$. Having $\bigcup g \models^*_{\mathfrak{A}} \text{(S8)}_{\lambda_f}$ then bestows us the existence of some $$q \in [\bigcup g]^{< \omega} \cap X_{\mu, \nu}(\bigcup g) \cap E_{\zeta} \cap Q_{\nu} \neq \emptyset.$$ That $q \in X_{\mu, \nu}(\bigcup g)$ and (\ref{eseq}) holds means $$(X_{\mu, \nu}(\bigcup g); \in, \mathbb{P}, D) \models \exists \eta > \zeta \ ((q, \eta) \in D).$$ Invoking the fact that $\bigcup g \models^*_{\mathfrak{A}} \text{(S6)}_{\lambda_f}$ gives us some $\eta$ such that 
\begin{itemize}
    \item $\zeta < \eta < \mu$, and
    \item $(q, \eta) \in D$.
\end{itemize} 
Recalling the definition of $D$, we conclude $\eta \in \dot{C}[g]$ because $q \in [\bigcup g]^{< \omega} = g$, 
\end{proof}

\begin{lem}\label{2done}
$\Vdash_{\mathbb{P}} ``cof(\lambda_f) = \omega_1"$. That is, $\mathbb{P}$ fulfils \ref{4102} of Lemma \ref{lem26}.
\end{lem}

\begin{proof}
Let $g$ be $\mathbb{P}$-generic over $W$. By Lemma \ref{modeldone}, $\bigcup g \models^*_{\mathfrak{A}} \Gamma_{\lambda_f}$, so that in $W[g]$, $$K := \{(\delta, \lambda) : \exists x \ (\ulcorner x \in \dot{X}_{\delta, \lambda} \urcorner \in g)\}$$ is a strictly increasing function with domain contained in $\omega_1^W$ and range contained in $\lambda_f$. Lemma \ref{sideuni} tells us that $dom(K)$ is cofinal in $\omega_1^W$ and $ran(K)$ is cofinal in $\lambda_f$, hence $$cof^{W[g]}(\lambda_f) = cof^{W[g]}(\omega_1^W).$$ By Lemma \ref{statdone}, we have $\omega_1^W = \omega_1^{W[g]}$, and consequently, 
\begin{equation*}
    cof^{W[g]}(\lambda_f) = \omega_1^{W[g]}.
    \qedhere
\end{equation*}
\end{proof}

In view of Lemma \ref{lem26}, the theorem follows from Lemmas \ref{size}, \ref{1done}, \ref{statdone} and \ref{2done}.
\end{proof}

According to Theorem \ref{notion1}, if $\mathrm{NS}_{\omega_1}$ is precipitous, then there is a uniform way of generating witnesses --- in place of $\mathbb{P}_N$ --- to analogues of Fact \ref{fact26}. To wit, we have the following corollary.

\begin{cor}\label{nambacoro}
Assume 
\begin{itemize}
    \item $\mathrm{NS}_{\omega_1}$ is precipitous, and
    \item $2^{\omega_1} = \omega_2$.
\end{itemize}
Then $Nb'_1(\lambda)$ --- thus also $Nb_1(\lambda)$ --- holds for each regular cardinal $\lambda > \omega_2$. Furthermore, $Nb_0(\alpha)$ holds for each ordinal $\alpha > \omega_2$.
\end{cor}

An advantage of the forcing framework of Subsection \ref{forframe} is that it facilitates modular analyses of the generic object. Adding components to the generic object can be done by extending the language on which the forcing notion is based. Under the right circumstances, that said addition preserves a property of the original forcing notion is readily derived from examining the extended language.

In the next subsection, we will augment the forcing notion $\mathbb{P}$ defined in the proof of Theorem \ref{notion1} while assuming a stronger hypothesis, so that the $\mathbb{P}$-generic object has a generic iteration as one of its components. The reader should notice that there is ample carryover from the proof of Theorem \ref{notion1} in the analysis of the new and augmented $\mathbb{P}$.

\subsection{Incorporating the Asper\'{o}-Schindler Construction}\label{ss43}

Asper\'{o}'s and Schindler's approach to proving ``$\mathsf{MM}^{++}$ implies $(*)$'' in \cite{schindler} goes along the following lines. 
\begin{enumerate}[label=(\arabic*)]
    \item Assume $\mathsf{MM}^{++}$.
    \item Define 
    \begin{align*}
        g_A := \{\bar{N} \in \mathbb{P}_{max} : \ & \text{there is a generic iteration } \\ 
        & \langle \bar{N}_i = (N_i; \in, I_i, a_i), \sigma_{ij} : i \leq j \leq \omega_1 \rangle \\
        & \text{of } \bar{N} \text{ such that} \\
        & a_{\omega_1} = A \text{ and } \mathrm{NS}_{\omega_1} \cap N_{\omega_1} = I_{\omega_1}\} \text{.}
    \end{align*}
    \item Show that whenever $\omega_1^{L[A]} = \omega_1$,
    \begin{itemize}
        \item $g_A$ is a filter, and
        \item if $g_A$ is $\mathbb{P}_{max}$-generic over $L(\mathbb{R})$, then $\mathcal{P}(\omega_1) \subset L(\mathbb{R})[g_A]$.
    \end{itemize}
    \item\label{434} For each dense subset $D \in L(\mathbb{R})$ of $\mathbb{P}_{max}$, find a stationary-preserving forcing notion $\mathbb{P}(D)$ that forces
    \begin{quote}
        ``there are $p \in D^*$ ($D^*$ being the interpretation of $D$ in $V^{\mathbb{P}(D)}$ via some universally Baire encoding) and a generic iteration $$\langle \bar{N}_i = (N_i; \in, I_i, a_i), \sigma_{ij} : i \leq j \leq \omega_1^V \rangle$$ for which
        \begin{itemize}
            \item $p = \bar{N}_0$,
            \item $I_{\omega_1^V} = \mathrm{NS}_{\omega_1} \cap N_{\omega_1^V}$, and
            \item $a_{\omega_1^V} = A$'',
        \end{itemize}
    \end{quote}
    so that $g_A$ is $\mathbb{P}_{max}$-generic over $L(\mathbb{R})$.
\end{enumerate}
Each of the $\mathbb{P}(D)$'s satisfying \ref{434}, as defined in the proof of Lemma 2.14 of \cite{schindler}, possesses curious properties tangential to its chief purpose:
\begin{enumerate}[label=(\roman*)]
    \item\label{43i} its conditions are fragments of a language depending on $D$,
    \item\label{43ii} it forces ``$cof(\omega_2^V) = \omega$'', and
    \item\label{43iii} it forces ``$cof(\omega_3^V) = \omega_1$''.
\end{enumerate}
Since \ref{43ii} and \ref{43iii} make each $\mathbb{P}(D)$ ``Namba-like'', the conjunction of \ref{43i} to \ref{43iii} points to the viability of incorporating the design of each $\mathbb{P}(D)$ into the construction of $\mathbb{P}$ as described in the proof of Theorem \ref{notion1}, so as to strengthen Lemma 2.14 of \cite{schindler}. This incorporation can be thought of both as
\begin{itemize}
    \item an augmentation of $\mathbb{P}$ to serve an expanded agenda, and
    \item a means to extend the ``Namba consequences'' of each $\mathbb{P}(D)$,
\end{itemize}
hence it is sufficiently motivated. We shall spend the rest of this subsection ironing out the details of our (natural) incorporation attempt, with Theorem \ref{notion2} being its upshot.

\begin{defi}\label{defl2}
To prepare for Definition \ref{def433}, let us first set aside the following distinguished symbols:
\begin{itemize}
    \item $\dot{M}_i$, $\dot{N}_i$ for $i < \omega_1$,
    \item $\dot{\pi}_{ij}$ for $i \leq j \leq \omega_1$,
    \item $\dot{\sigma}_{ij}$ for $i \leq j < \omega_1$,
    \item $\dot{n}$ for $n < \omega$, and
    \item $\dot{T}$, $\dot{\Vec{M}}$, $\dot{I}$, $\dot{a}$.
\end{itemize}
Assume, without loss of generality, that none of the distinguished symbols is represented (as a set) by an ordinal.
\end{defi}

\begin{defi}
Let $\hat{\sigma}$ denote the signature
\begin{equation*}
    \{\dot{I}, \dot{a}, \dot{\Vec{M}}\} \cup \{\xi : \xi < \omega_1^W\} \cup \{\dot{M}_j : j < \omega_1^W\} \cup \{\dot{\pi}_{jk} : j \leq k \leq \omega_1^W\} \text{,}
\end{equation*}
in which
\begin{itemize}
    \item $\dot{I}$ is a unary relation symbol, and
    \item every member of $\hat{\sigma} \setminus \{\dot{I}\}$ is a constant symbol.
\end{itemize}
\end{defi}

\begin{defi}\label{def433}
Let $\mathcal{L}^{s_0}$ contain precisely all expressions of the following forms:
\begin{enumerate}[label=(L\arabic*), leftmargin=40pt]
    \setcounter{enumi}{2}
    \item\label{l3} $\ulcorner \dot{N}_i \models \phi(\xi_1, \ldots, \xi_k, \dot{n_1}, \ldots, \dot{n_l}, \dot{I}, \dot{a}, \dot{M}_{j_1}, \ldots, \dot{M}_{j_m}, \dot{\pi}_{q_{1}r_{1}}, \ldots, \dot{\pi}_{q_{s}r_{s}}, \dot{\Vec{M}}) \urcorner$, for 
    \begin{itemize}[leftmargin=10pt]
        \item $i, \xi_1, \ldots, \xi_k, j_i, \ldots, j_m < \omega_1$,
        \item $n_1, \ldots, n_l < \omega$, 
        \item $q_1 \leq r_1 < \omega_1^V, \ldots, q_s \leq r_s < \omega_1$,
        \item $\phi$ a first-order formula in the language of set theory expanded with $\hat{\sigma}$.
    \end{itemize}
    \item\label{l4} $\ulcorner \dot{\pi}_{i\omega_1}(\dot{n}) = x \urcorner$, for $n < \omega$, $i < \omega_1$ and $x \in H(\omega_2)$,
    \item\label{l6} $\ulcorner \dot{\sigma}_{ij}(\dot{m}) = \dot{n} \urcorner$, for $i \leq j < \omega_1$ and $m, n < \omega$, 
    \item\label{l7} $\ulcorner (\Vec{u}, \Vec{\alpha}) \in \dot{T} \urcorner$, for $\Vec{u} \in {^{<\omega}{\omega}}$, $\Vec{\alpha} \in {^{<\omega}{(\omega_2)}}$ and $dom(\Vec{u}) = dom(\Vec{\alpha})$,
\end{enumerate}
\end{defi}

\begin{defi}\label{defsvf}
Given a signature $\sigma$, define $\mathcal{L}^1(\sigma, x)$ to be the set of formulas over $\sigma$ with $x$ as its only variable.
\end{defi}

\begin{defi}\label{subn}
For $1 \leq n \leq 6$, let $\mathcal{L}_n$ denote the set of all expressions of the form (L$n$) in Definitions \ref{defl} and \ref{defl2}. Further, define $\mathcal{L}_0$ such that its members are exactly expressions in $\mathcal{L}$ of the form 
\begin{equation*}
    \ulcorner \dot{N}_i \models \phi(\dot{n_1}, \ldots, \dot{n_l}, \dot{I}, \dot{a}) \urcorner,
\end{equation*}
with $i$ ranging over $\omega_1$ and $l$ ranging over $\omega$. Clearly $\mathcal{L}_0 \subset \mathcal{L}_3$.
\end{defi}

\begin{defi}
For $i < \omega_1$ and $n \in \{0, 3\}$, define 
\begin{equation*}
    \mathcal{L}^i_n := \{x \in \mathcal{L}_n : \ulcorner \dot{N}_i \urcorner \text{ occurs in } x\}.
\end{equation*}
\end{defi}

\begin{con}
If $\varphi = \ulcorner \dot{N}_i \models \phi \urcorner \in \mathcal{L}_1$, denote $\neg(\varphi)$ by $\ulcorner \dot{N}_i \models \neg(\phi) \urcorner$, where $\neg(\phi)$ is resolved as per Definition \ref{neg}. This allows us to conclude that 
\begin{itemize}
    \item $\mathcal{L}_0$, $\mathcal{L}_3$ are closed under negation, and
    \item $\mathcal{L}^i_0$, $\mathcal{L}^i_3$ are closed under negation for any $i < \omega_1$.
\end{itemize}
\end{con}

\begin{defi}
If $\bar{N} = (N; \in, I, a)$ is a countable structure and $f : \omega \longrightarrow N$ is a surjection, then we define the \emph{simple} $\mathcal{L}$\emph{-theory of $\bar{N}$ along $f$}, denoted $Th^{0}_{\mathcal{L}}(\bar{N}, f)$, to be 
\begin{align*}
    \{\ulcorner \dot{N}_0 \models \phi(\dot{n_1}, \ldots, \dot{n_l}, \dot{I}, \dot{a}) \urcorner \in \mathcal{L}^0_0 : \bar{N} \models \phi(f(n_1), \ldots, f(n_l), I, a)\}.
\end{align*}
$Th^{0}_{\mathcal{L}}(\bar{N}, f)$ is obviously $\Delta_0$-definable in $\bar{N}$ and $f$.
\end{defi}

Fix a (recursive) G\"{o}del numbering $Gd$ of $\mathcal{L}^0_0$.

\begin{defi}\label{nota4}
If $s \in {^{A}{(B \times C)}}$, let $pr(s)$ denote the member $t \in {^{A}{B}}$ such that for all $a \in A$, $t(a) = b$ iff there is some $c$ for which $s(a) = (b, c)$.
\end{defi}

We will use the fact below without proof.

\begin{fact}\label{fact434}
Assume
\begin{enumerate}[label=(\roman*), leftmargin=40pt]
    \item $\Gamma  = \bigcup_{1 \leq k < \omega} \mathcal{P}(\mathbb{R}^{k}) \cap L(\Gamma, \mathbb{R})$,
    \item $\Gamma$ is productive,
    \item $\mathrm{NS}_{\omega_1}$ is saturated,
    \item $2^{\omega_1} = \mathbf{\delta^1_2} = \omega_2$, and
    \item $MA(\omega_1)$ holds.
\end{enumerate}
Let $D \in L(\Gamma, \mathbb{R})$ be a dense subset of $\mathbb{P}_{max}$, and $A \subset \omega_1$ such that $\omega_1^{L[A]} = \omega_1$. Then there are
\begin{itemize}
    \item a $\Delta_1$-definable partial map $F^*$ from $^{\omega}{\omega}$ onto the members of $\mathbb{P}_{max}$,
    \item a tree $T$ of size $\aleph_2$ on $\omega \times \omega_2$, and
    \item a $Col(\omega, \omega_2)$-name $\dot{p} \subset H(\omega_2)$ for a member of ${^{\omega}{\omega}}$
\end{itemize}
such that 
\begin{enumerate}[label=(4.34.\arabic*), leftmargin=50pt]
    \item\label{4341} $\Vdash_{Col(\omega, \omega_2)} ``\dot{p} \in p[T] \wedge F^*(\dot{p}) \leq_{\mathbb{P}_{max}} (H(\omega_2)^V; \in, \mathrm{NS}_{\omega_1}^V, A)"$,
    \item\label{4342} $D' := (F^*)^{-1}(D)$ is universally Baire, 
\end{enumerate}
and in every forcing extension of $V$,
\begin{enumerate}[label=(4.34.\arabic*), leftmargin=50pt]
    \setcounter{enumi}{2}
    \item\label{4343} $D'^* \subset dom(F^*)$,
    \item\label{4344} $D^* := (F^*) " (D'^*)$ is a dense subset of $\mathbb{P}_{max}$, 
    \item\label{4345} $F^*(pr(\bigcup S)) \in D^*$ for every $S$ satisfying
    \begin{itemize}
        \item $S \subset T$, and
        \item $\bigcup S \in [T]$, 
    \end{itemize}
    \item\label{4346b} whenever $\bar{M}$, $\bar{N}$, $f$, $S$ fulfil the following:
    \begin{itemize}
        \item $\bar{N}$ is an expansion of some structure of the form $(N; \in, I, a)$, where $I$ interprets $\dot{I}$ and $a$ interprets $\dot{a}$,
        \item $f : \omega \longrightarrow N$ is a surjection,
        \item $S \subset T$,
        \item $\bigcup S \in [T]$, and
        \item $ran(pr(\bigcup S)) = Gd" Th^0_{\mathcal{L}}(\bar{N}, f)$,
    \end{itemize}
    it must be the case that $F^*(pr(\bigcup S)) = \bar{N} \in D^*$, and
    \item\label{4346} whenever $\bar{M}$, $\bar{N}$ fulfil the following:
    \begin{itemize}
        \item $\bar{N}$ is a member of $\mathbb{P}_{max}$, and
        \item $\bar{N} \models ``\bar{M} \text{ is a member of } \mathbb{P}_{max}"$,
    \end{itemize}
    it must be the case that
    \begin{itemize}
        \item $\bar{M} \in \bar{N}$, and
        \item $\bar{M}$ is a member of $\mathbb{P}_{max}$.
    \end{itemize}
\end{enumerate}
\end{fact}

Fix $F^*$, $T$ and $\dot{p}$ as provided by Fact \ref{fact434}. In light of said fact, we can make sense of --- and subsequently prove --- the next theorem.

\begin{thm}\label{notion2}
Assume
\begin{enumerate}[label=(\roman*), leftmargin=40pt]
    \item\label{h1} $\Gamma  = \bigcup_{1 \leq k < \omega} \mathcal{P}(\mathbb{R}^{k}) \cap L(\Gamma, \mathbb{R})$,
    \item\label{h2} $\Gamma$ is productive,
    \item\label{h3} $\mathrm{NS}_{\omega_1}$ is saturated,
    \item\label{h4} $2^{\omega_1} = \mathbf{\delta^1_2} = \omega_2$, and
    \item\label{h5} $MA(\omega_1)$ holds.
\end{enumerate}
Let 
\begin{itemize}
    \item $D \in L(\Gamma, \mathbb{R})$ be a dense subset of $\mathbb{P}_{max}$,
    \item $A \subset \omega_1$ be such that $\omega_1^{L[A]} = \omega_1$, and
    \item $\lambda_f > \omega_2$ be a regular cardinal .
\end{itemize}
Then there is a stationary-preserving forcing notion $\mathbb{P}$ such that in $V^{\mathbb{P}}$,
\begin{itemize}
    \item there is a generic iteration $$\langle \bar{N}_i = (N_i; \in, I_i, a_i), \sigma_{ij} : i \leq j \leq \omega_1^V \rangle$$ satisfying 
    \begin{enumerate}[label=(\arabic*), leftmargin=40pt]
        \item\label{nov1} $\bar{N}_0 \in D^* := (F^*) " ((F^*)^{-1}(D)^*) \subset \mathbb{P}_{max}$,
        \item\label{nov2} $I_{\omega_1^V} = \mathrm{NS}_{\omega_1}^{V^{\mathbb{P}}} \cap N_{\omega_1^V}$, and
        \item\label{nov3} $a_{\omega_1^V} = A$, and
    \end{enumerate}
    \item $Nb'_1(\lambda_f)$ holds.
\end{itemize}
\end{thm}

\begin{proof}
We first import the notation the labelling of definitions from the proof of Theorem \ref{notion1}; they will be reused until subsequent reassignments. 

Move to $W$ as in the proof of Theorem \ref{notion1} via forcing with $Col(\lambda_f, \lambda_f)$, so that 
\begin{itemize}[label=($\diamond$)]
    \item for all $P, B \subset H(\lambda_f)$, the set
    \begin{align*}
        \{\lambda \in C : (Q_{\lambda}; \in, P, A_{\lambda}) \prec (H(\lambda_f); \in, P, B)\}
    \end{align*}
    is stationary in $\lambda_f$.
\end{itemize}
holds in $W$. We further require 
\begin{equation*}
    \mathrm{Par} := \{T, \dot{p}, H(\omega_2)^W = H(\omega_2)^V, \mathrm{NS}_{\omega_1}^W = \mathrm{NS}_{\omega_1}^V, A\} \subset Q_{\lambda}
\end{equation*}
for all $\lambda \in C$. But this is easily done because $\mathrm{Par}$ is small by Fact \ref{fact434} and the hypothesis of the theorem. 

Set
\begin{align*}
    \mathcal{L}^o := \ & \text{the closure of } \mathcal{L}_1 \cup \mathcal{L}_2 \text{ under negation} \\
    \mathcal{L}^s := \ & \text{the closure of } \mathcal{L}^{s_0} \text{ under negation},
\end{align*}
and enlarge $\mathcal{L}$ just enough to include $\mathcal{L}^s$. It is easy to see that
\begin{itemize}
    \item $\mathcal{L}^o$ is the original $\mathcal{L}$ before enlargement,
    \item both $\mathcal{L}^s$ and the newly enlarged $\mathcal{L}$ are closed under negation, 
    \item $\mathcal{L}$ now equals $\mathcal{L}^o \cup \mathcal{L}^s$, and
    \item $\mathcal{L}^s \subset Q_{\lambda}$ for every $\lambda \in C$, so that
    \item $\mathcal{L} \cap Q_{\lambda} = \mathcal{L}^s \sqcup (\mathcal{L}^o \cap Q_{\lambda})$ for every $\lambda \in C \cup \{\lambda_f\}$.
\end{itemize}

\begin{rem}\label{rempp2}
The naturally extended version of Remark \ref{rempp} applies to the updated $\mathcal{L}$.
\end{rem}

Note also that the following hold in $W$:
\begin{itemize}
    \item the hypothesis \ref{h3}, and
    \item the conclusion of Fact \ref{fact434} (i.e. the conjunction of \ref{4341} to \ref{4346}) with our given $D$ and $A$. 
\end{itemize}

The aforementioned truths in $W$ are all we need to proceed, aided by the next fact, which can be viewed as an extension of Lemma \ref{lem26}.

\begin{fact}\label{fact436}
The theorem holds if in $W$, we can define a forcing notion $\mathbb{P}$ with the following properties: 
\begin{enumerate}[label=(K\arabic*), leftmargin=40pt]
    \item\label{K1} $\mathbb{P} \subset H(\lambda_f)$, so that $|\mathbb{P}| \leq \lambda_f$ by the proof of Lemma \ref{size},
    \item\label{K2} $\mathbb{P}$ is stationary-preserving, 
    \item\label{K3} in $V^{\mathbb{P}}$ there is a generic iteration $$\langle \bar{N}_i = (N_i; \in, I_i, a_i), \sigma_{ij} : i \leq j \leq \omega_1^V \rangle$$ satisfying 
    \begin{enumerate}[label=(\arabic*), leftmargin=40pt]
        \item\label{k31} $\bar{N}_0 \in D^* \subset \mathbb{P}_{max}$,
        \item\label{k32} $I_{\omega_1^V} = \mathrm{NS}_{\omega_1}^{V^{\mathbb{P}}} \cap N_{\omega_1^V}$ and
        \item\label{k33} $a_{\omega_1^V} = A$.
    \end{enumerate} 
    \item\label{K4} $\Vdash_{\mathbb{P}} ``cof(\alpha) = \omega"$ for all regular cardinals $\alpha$ satisfying $\omega_2 \leq \alpha < \lambda_f$,
    \item\label{K5} $\Vdash_{\mathbb{P}} ``cof(\lambda_f) = \omega_1"$,
\end{enumerate}
\end{fact}

Going forward, unless otherwise specified, 
\begin{itemize}
    \item we work in $W$ towards a forcing notion $\mathbb{P}$ as in Fact \ref{fact436}, and
    \item every new object (to be) defined in $W$ always denotes its realisation in $W$. 
\end{itemize}
Recall that
\begin{align*}
    \kappa & := (2^{\lambda_f})^+ \text{,} \\
    \mathfrak{A} & := (H(\kappa); \in) \text{, and} \\
    R & := \{i < \lambda_f : \omega_2 \leq i \text{ and } i \text{ is regular}\},
\end{align*}
We shall formally describe the mathematical object that forcing with fragments of $\mathcal{L}$ is supposed to help construct. 

\begin{defi}\label{def437}
Let $\lambda \in C \cup \{\lambda_f\}$. A $\lambda$\emph{-certificate} is a tuple 
\begin{equation*}
    \mathfrak{D} = \langle \langle \bar{M}_i, \pi_{ij},  \bar{N}_i, \sigma_{ij} : i \leq j \leq \omega_1^{W} \rangle, S, \langle e_i : i < \omega_1^W \rangle, \mathfrak{C} \rangle
\end{equation*}
such that in some weak outer model of $W$ containing $\mathfrak{D}$,
\begin{enumerate}[label=(C\arabic*)$_{\lambda}$, leftmargin=40pt]
    \item\label{cd0} for all $i \leq \omega_1^W$, $\bar{N}_i$ is a structure of the form
    \begin{equation*}
        (N_i; \tilde{\in}_i, \Vec{X}_i) \text{,}
    \end{equation*}
    where 
    \begin{itemize}
        \item $\tilde{\in}_i$ interprets the binary relation symbol $\ulcorner \in \urcorner$, and
        \item $\Vec{X}_i$ interprets $\hat{\sigma}$ 
    \end{itemize}
    (we shall use $I_i$ to denote $\dot{I}^{\bar{N}_i}$ and $a_i$ to denote $\dot{a}^{\bar{N}_i}$, both of which are members of $\Vec{X}_i$),
    \item\label{cd1} $\bar{N}_0 \models ``\bar{M}_0$ is a member of $\mathbb{P}_{max}"$,
    \item\label{cd1.5} for all $i < \omega_1^W$, $e_i$ is a bijection $\omega \longrightarrow N_i$,
    \item\label{cd2} $S \subset T$ and $\bigcup S \in [T]$,
    \item\label{cd3} $ran(pr(\bigcup S)) = Gd" Th^0_{\mathcal{L}}(\bar{N}_0, e_0)$,
    \item\label{cd4} $\langle \bar{M}_i, \pi_{ij} : i \leq j \leq \omega_1^{\bar{N}_0} \rangle \in N_0$ is a generic iteration witnessing $(N_0; \tilde{\in}_0, I_0, a_0) < \bar{M}_0$ in $\mathbb{P}_{max}$,
    \item\label{cdnew} $ORD \cap N_i \in \omega_1^W$ for all $i < \omega_1^W$,
    \item\label{cd5} $\langle (N_i; \tilde{\in}_i, I_i, a_i), \sigma_{ij} : i \leq  j \leq \omega_1^{W} \rangle$ is a generic iteration and $\tilde{\in}_i = \ \in$, for $i < \omega_1^W$,
    \item\label{cd6} $\sigma_{0\omega_1^W}(\langle \bar{M}_i, \pi_{ij} : i \leq j \leq \omega_1^{\bar{N}_0} \rangle) = \langle \bar{M}_i, \pi_{ij} : i \leq j \leq \omega_1^W \rangle$, 
    \item\label{cd7} $\bar{M}_{\omega_1^W} = (H(\omega_2)^W; \in, \mathrm{NS}_{\omega_1}^W, A)$,
    \item\label{2ndlast} for all $i \leq \omega_1^W$, 
    \begin{align*}
         \dot{\Vec{M}}^{\bar{N}_i} = \ & \langle \bar{M}_j, \pi_{jk} : j \leq k \leq \omega_1^{\bar{N}_i} \rangle \\
         \xi^{\bar{N}_i} = \ & \xi \text{ for all } \xi \in \omega_1^W \cap N_i \\
         \xi^{\bar{N}_i} = \ & \emptyset \text{ for all } \xi \in \omega_1^W \setminus N_i \\
         \dot{M}_j^{\bar{N}_i} = \ & \bar{M}_j \text{ for all } j \leq \omega_1^{\bar{N}_i} \\
         \dot{M}_j^{\bar{N}_i} = \ & \emptyset \text{ for all } \omega_1^{\bar{N}_i} < j < \omega_1^W \\
         \dot{\pi}_{jk}^{\bar{N}_i} = \ & \pi_{jk} \text{ for all } j \leq k \leq \omega_1^{\bar{N}_0} \\
         \dot{\pi}_{jk}^{\bar{N}_i} = \ & \emptyset \text{ for all } j \leq k \text{ and } \omega_1^{\bar{N}_0} < k < \omega_1^W \text{,}
    \end{align*}
    \item\label{cdlast} $\mathfrak{C}$ interprets $\mathcal{L}^o \cap Q_{\lambda}$ (see Definition \ref{def412}), and
    \item\label{cdrlast} $\Sigma(\mathfrak{C}, \mathcal{L}^o \cap Q_{\lambda})$ $\Gamma_{\lambda} (\mathcal{L}^o \cap Q_{\lambda}, \mathfrak{A}) \text{-certifies } \emptyset$ (see Definition \ref{def415}).
\end{enumerate}
\end{defi}

One can easily verify that being a $\lambda$-certificate, for any $\lambda \in C \cup \{\lambda_f\}$, is absolute for weak outer models of $W$. If a $\lambda$-certificate shows up in some context without reference to the universe it inhabits, we may assume said universe to be any weak outer model of $W$.

\begin{rem}\label{rem438}
\leavevmode
\begin{enumerate}[label=(\arabic*)]
    \item\label{dependl} \ref{cdlast} and \ref{cdrlast} are the only two out of the thirteen conditions --- \ref{cd0} to \ref{cdrlast} --- in Definition \ref{def437} that depend on $\lambda$.
    \item\label{4390} For any two tuples $\mathfrak{C}$ and
    \begin{equation*}
        \mathfrak{D}' = \langle \langle \bar{M}_i, \pi_{ij},  \bar{N}_i, \sigma_{ij} : i \leq j \leq \omega_1^{W} \rangle, S, \langle e_i : i < \omega_1^W \rangle \rangle \text{,}
    \end{equation*}
    if
    \begin{itemize}
        \item $\mathfrak{C}$ satisfies \ref{cdlast} to \ref{cdrlast} of Definition \ref{def437}, and
        \item $\mathfrak{D}'$ satisfies \ref{cd0} to \ref{2ndlast} of Definition \ref{def437},
    \end{itemize}
    then
    \begin{equation*}
        \mathfrak{D} = \langle \langle \bar{M}_i, \pi_{ij},  \bar{N}_i, \sigma_{ij} : i \leq j \leq \omega_1^{W} \rangle, S, \langle e_i : i < \omega_1^W \rangle, \mathfrak{C} \rangle
    \end{equation*}
    is a $\lambda$-certificate.
    \item\label{4391} If a tuple
    \begin{equation*}
        \mathfrak{D}' = \langle \langle \bar{M}_i, \pi_{ij},  \bar{N}_i, \sigma_{ij} : i \leq j \leq \omega_1^{W} \rangle, S, \langle e_i : i < \omega_1^W \rangle \rangle
    \end{equation*}
    satisfies 
    \begin{itemize}
        \item $\langle \bar{N}_i, \sigma_{ij} : i \leq  j \leq \omega_1^{W} \rangle$ is a generic iteration, 
        \item $ORD \cap N_i \in \omega_1^W$ for all $i < \omega_1^W$, where $N_i$ denotes the base set of $\bar{N}_i$, and
        \item \ref{cd1} to \ref{cd4} and \ref{cd6} to \ref{cd7} of Definition \ref{def437},
    \end{itemize}
    then the $\bar{N}_i$'s can be canonically expanded as structures such that $\mathfrak{D}'$ satisfies \ref{cd0} to \ref{2ndlast} of Definition \ref{def437}.
    \item\label{4392} As a result of \ref{4390} and \ref{4391}, if a tuple
    \begin{equation*}
        \mathfrak{D} = \langle \langle \bar{M}_i, \pi_{ij},  \bar{N}_i, \sigma_{ij} : i \leq j \leq \omega_1^{W} \rangle, S, \langle e_i : i < \omega_1^W \rangle, \mathfrak{C} \rangle
    \end{equation*}
    satisfies 
    \begin{itemize}
        \item $\langle \bar{N}_i, \sigma_{ij} : i \leq  j \leq \omega_1^{W} \rangle$ is a generic iteration, 
        \item $ORD \cap N_i \in \omega_1^W$ for all $i < \omega_1^W$, where $N_i$ denotes the base set of $\bar{N}_i$, and
        \item \ref{cd1} to \ref{cd4}, \ref{cd6} to \ref{cd7}, and \ref{cdlast} to \ref{cdrlast} of Definition \ref{def437},
    \end{itemize}
    then the $\bar{N}_i$'s can be canonically expanded as structures such that $\mathfrak{D}$ is a $\lambda$-certificate.
\end{enumerate}

\end{rem}

\begin{defi}
Given
\begin{itemize}
    \item $i < \omega_1^W$,
    \item a structure $\bar{N} = (N; \tilde{\in}, \Vec{X})$ such that
    \begin{itemize}[label=$\circ$]
        \item $\tilde{\in}$ interprets the binary relation symbol $\ulcorner \in \urcorner$, and
        \item $\Vec{X}$ interprets $\hat{\sigma}$,
    \end{itemize}
    and
    \item a function $e$ from $\omega$ into $N$, 
\end{itemize} 
define $Th^{1}_{\mathcal{L}}(\bar{N}, e, i)$ to be 
\begin{align*}
    \{\ulcorner \dot{N}_i \models \phi(\xi_1, \ldots, \xi_k, \dot{n_1}, \ldots, \dot{n_l}, \dot{I}, \dot{a}, \dot{M}_{j_1}, & \ldots, \dot{M}_{j_m}, \dot{\pi}_{q_{1}r_{1}}, \ldots, \dot{\pi}_{q_{s}r_{s}}, \dot{\Vec{M}}) \urcorner \in \mathcal{L}^i_3 : \\ & \bar{N} \models \phi[\dot{n}_1 \mapsto e(n_1), \ldots, \dot{n}_l \mapsto e(n_l)]\}.
\end{align*}
\end{defi}

\begin{defi}\label{def439}
Given 
\begin{itemize}
    \item a tuple
    \begin{equation*}
        \mathfrak{D} = \langle \langle \bar{M}_i, \pi_{ij},  \bar{N}_i, \sigma_{ij} : i \leq j \leq \omega_1^{W} \rangle, S, \langle e_i : i < \omega_1^W \rangle, \mathfrak{C} \rangle
    \end{equation*}
    satisfying 
    \begin{itemize}[label=$\circ$]
        \item \ref{cd0} and \ref{cdlast} of Definition \ref{def437} for some $\lambda \in C \cup \{\lambda_f\}$, 
        \item $e_i$ is a function from $\omega$ into $N_i$ whenever $i < \omega_1^W$, 
        \item $\pi_{i\omega_1^W}$ is a partial function from $N_i$ into $H(\omega_2)^W$ whenever $i \leq \omega_1^W$, and
        \item $\sigma_{ij}$ is a function from $N_i$ into $N_j$ whenever $i \leq j < \omega_1^W$,
    \end{itemize}
    as well as 
    \item a set $\mathcal{L}' \subset \mathcal{L}$,
\end{itemize}
let $\Sigma'(\mathfrak{D}, \mathcal{L}')$ denote the union of the following sets:
\begin{itemize}
    \item $\bigcup \{Th^{1}_{\mathcal{L}}(\bar{N}_i, e_i, i) : i < \omega_1^W\}$,
    \item 
    \!
    $\begin{aligned}[t]
        \{\ulcorner \dot{\pi}_{i\omega_1^W}(\dot{n}) = x \urcorner : \ & \ulcorner \dot{\pi}_{i\omega_1^W}(\dot{n}) = x \urcorner \in \mathcal{L}' \text{, } e_i(n) \in dom(\pi_{i\omega_1^W}) \\ 
        & \text{ and } \pi_{i\omega_1^W}(e_i(n)) = x\} \text{,}
    \end{aligned}$
    \item 
    \!
    $\begin{aligned}[t]
        \{\ulcorner \neg \dot{\pi}_{i\omega_1^W}(\dot{n}) = x \urcorner : \ & \ulcorner \neg \dot{\pi}_{i\omega_1^W}(\dot{n}) = x \urcorner \in \mathcal{L}' \text{, and} \\
        & \text{either } e_i(n) \not\in dom(\pi_{i\omega_1^W}) \text{ or } \pi_{i\omega_1^W}(e_i(n)) \neq x\} \text{,}
    \end{aligned}$
    \item $\{\ulcorner \dot{\sigma}_{ij}(\dot{m}) = \dot{n} \urcorner : \ulcorner \dot{\sigma}_{ij}(\dot{m}) = \dot{n} \urcorner \in \mathcal{L}' \text{ and } \sigma_{ij}(e_i(m)) = e_j(n)\}$,
    \item $\{\ulcorner \neg \dot{\sigma}_{ij}(\dot{m}) = \dot{n} \urcorner : \ulcorner \neg \dot{\sigma}_{ij}(\dot{m}) = \dot{n} \urcorner \in \mathcal{L}' \text{ and } \sigma_{ij}(e_i(m)) \neq e_j(n)\}$,
    \item $\{\ulcorner (\Vec{u}, \Vec{\alpha}) \in \dot{T} \urcorner : \ulcorner (\Vec{u}, \Vec{\alpha}) \in \dot{T} \urcorner \in \mathcal{L}' \text{ and } (\Vec{u}, \Vec{\alpha}) \in S\}$,
    \item $\{\ulcorner \neg (\Vec{u}, \Vec{\alpha}) \in \dot{T} \urcorner : \ulcorner \neg (\Vec{u}, \Vec{\alpha}) \in \dot{T} \urcorner \in \mathcal{L}' \text{ and } (\Vec{u}, \Vec{\alpha}) \not\in S\}$, and
    \item $\Sigma(\mathfrak{C}, \mathcal{L}' \cap \mathcal{L}^o)$.
\end{itemize}
As a result, we can view $\Sigma'(\cdot, \cdot)$ as a function in two variables.
\end{defi}

It is clear that $\Sigma'(\mathfrak{D}, \mathcal{L}) \cap Q_{\lambda} = \Sigma'(\mathfrak{D}, \mathcal{L} \cap Q_{\lambda})$ is $\mathcal{L} \cap Q_{\lambda}$-nice for all $\lambda \in C \cup \{\lambda_f\}$.

\begin{defi}
Let $\Gamma'$ be the following set
\begin{align*}
    \{\phi : \ & \phi \text{ is a } (\mathcal{L}^s)^*_{\mathfrak{A}}\text{-}\Pi_2 \text{ sentence and} \\
    & \Vdash_{Col(\omega, \lambda_f)} \forall \mathfrak{D} \ \forall \lambda \in C \cup \{\lambda_f\} \\
    & \mspace{80mu} (``\mathfrak{D} \text{ is a } \lambda \text{-certificate} \implies \Sigma'(\mathfrak{D}, \mathcal{L} \cap Q_{\lambda}) \models^*_{\mathfrak{A}} \phi")\} \text{.}
\end{align*}
\end{defi}

Notice whenever $\phi$ is a $(\mathcal{L}^s)^*_{\mathfrak{A}}$-$\Pi_2$ sentence, it must be the case that in any $Col(\omega, \lambda_f)$-generic extension of $V$, 
\begin{equation*}
    \varphi(\phi) := \forall \mathfrak{D} \ \forall \lambda \in C \cup \{\lambda_f\} \ (``\mathfrak{D} \text{ is a } \lambda \text{-certificate} \implies \Sigma'(\mathfrak{D}, \mathcal{L} \cap Q_{\lambda}) \models^*_{\mathfrak{A}} \phi")
\end{equation*}
is equivalent to a $\mathbf{\Pi^1_2}$ sentence. We can thus employ an argument akin to that which proved Lemma \ref{inout}, bearing in mind to replace each invocation of Mostowski's absoluteness theorem with an invocation of Shoenfield's absoluteness theorem, to obtain the fact below.

\begin{fact}\label{fact449}
Let $\phi$ be a $(\mathcal{L}^s)^*_{\mathfrak{A}}$-$\Pi_2$ sentence. Then the following are equivalent.
\begin{enumerate}[label=(\arabic*)$_{\phi}$]
    \item\label{4491} $\Vdash_{Col(\omega, \lambda_f)} \varphi(\phi)$.
    \item\label{4492} $\mspace{-3mu}\not \mspace{3mu}\Vdash_{Col(\omega, \lambda_f)} \neg \varphi(\phi)$.
    \item $W \models \varphi(\phi)$ for every outer model $W$ of $V$.
\end{enumerate}
\end{fact}

\begin{rem}
That statements \ref{4491} and \ref{4492} of Fact \ref{fact449} are equivalent can also be derived directly from the homogeneity of $Col(\omega, \lambda_f)$.
\end{rem}

Therefore, $\Gamma'$ is precisely the set 
\begin{align*}
    \{\phi : \phi \text{ is a } (\mathcal{L}^s)^*_{\mathfrak{A}}\text{-}\Pi_2 \text{ sentence and } W \models \varphi(\phi) \text{ for every outer model } W \text{ of } V\} \text{.}
\end{align*}

As in the proof of Theorem \ref{notion1}, we will (re)define $$\{\mathbb{P}_{\lambda} : \lambda \in C \cup \{\lambda_f\}\}$$ by induction on $\lambda$. Indeed, we are inductively modifying the definitions of the $\mathbb{P}_{\lambda}$'s we knew from the proof of Theorem \ref{notion1}. Assume that $\mathbb{P}_{\lambda'}$ has been modified for all $\lambda' \in \lambda \cap C$. 

\begin{defi}\label{def447}
Set
\begin{equation*}
    \Gamma'_{\lambda} := \Gamma' \cup \Gamma_{\lambda} \text{,}
\end{equation*}
where $\Gamma_{\lambda}$ is as in Definition \ref{defc}. 
\end{defi}

\begin{rem}\label{rem443}
\leavevmode
\begin{enumerate}[label=(\arabic*)]
    \item\label{443'1} If $\lambda \in C \cup \{\lambda_f\}$ and
    \begin{equation*}
        \mathfrak{D} = \langle \langle \bar{M}_i, \pi_{ij},  \bar{N}_i, \sigma_{ij} : i \leq j \leq \omega_1^{W} \rangle, S, \langle e_i : i < \omega_1^W \rangle, \mathfrak{C} \rangle
    \end{equation*}
    is a $\lambda$-certificate, then
    \begin{equation*}
        \Sigma'(\mathfrak{D}, \mathcal{L}^o \cap Q_{\lambda}) = \Sigma(\mathfrak{C}, \mathcal{L}^o \cap Q_{\lambda}) \text{.}
    \end{equation*}
    \item\label{443'2} In part due to \ref{4431}, as long as $\lambda \in C \cup \{\lambda_f\}$ and $\mathfrak{D}$ is a $\lambda$-certificate, it must be that $\Sigma'(\mathfrak{D}, \mathcal{L} \cap Q_{\lambda})$ $\Gamma'_{\lambda} (\mathcal{L} \cap Q_{\lambda}, \mathfrak{A})$-certifies $\emptyset$.
    \item\label{443'3} If $\lambda \in C \cup \{\lambda_f\}$ and $\Sigma \ \Gamma'_{\lambda} (\mathcal{L} \cap Q_{\lambda}, \mathfrak{A}) \text{-certifies } \emptyset$, then also
    \begin{equation*}
        (\Sigma \cap \mathcal{L}^o) \ \Gamma_{\lambda} (\mathcal{L}^o \cap Q_{\lambda}, \mathfrak{A}) \text{-certifies } \emptyset \text{.}
    \end{equation*}
\end{enumerate}
\end{rem}

Redefine $\mathbb{P}_{\lambda}$ as follows:
\begin{align*}
    \mathbb{P}_{\lambda} := \ & (P_{\lambda}, \leq_{\lambda}) \text{, where } \\
    P_{\lambda} := \ & \{p \in [\mathcal{L} \cap Q_{\lambda}]^{< \omega} : \ \Vdash_{Col(\omega, |H(\kappa)|)} \exists \Sigma \ (``\Sigma \ \Gamma'_{\lambda} (\mathcal{L} \cap Q_{\lambda}, \mathfrak{A}) \text{-certifies } p")\}, \text{ and} \\
    \leq_{\lambda} \ := \ & \{(p, q) \in P_{\lambda} \times P_{\lambda} : q \subset p\}.
\end{align*}
As before, let $\mathbb{P}$ denote $\mathbb{P}_{\lambda_f}$. 

Since $\Gamma_{\lambda}$ contains references to the set 
\begin{equation*}
    \{\mathbb{P}_{\lambda'} : \lambda' \in \lambda \cap C\} \text{,}
\end{equation*}
so we would expect the semantic value of $\Gamma_{\lambda}$ to be altered when changes are made to the definitions of the $\mathbb{P}_{\lambda'}$'s. In fact, it is through $\Gamma_{\lambda}$ that the definition of $\mathbb{P}_{\lambda}$ gets updated based on the updated definitions of the $\mathbb{P}_{\lambda'}$'s.

\begin{rem}\label{rem447}
\ref{K1} of Fact \ref{fact436} is obvious from the definition of $\mathbb{P}$.
\end{rem}

\begin{con}
Given an elementary embedding $\pi$ of $W$ into some transitive model $M$ of $\mathsf{ZFC}$ such that $ORD^W = ORD^M$, we say a statement (or definition) $\varphi$ holds when relativised to $(\pi, M)$ iff $\varphi$ holds with
\begin{itemize}
    \item every instance therein of each parameter $\zeta$ replaced by $\pi(\zeta)$, and
    \item every evaluation therein, after the replacement of parameters, being done in $M$ instead of $W$.
\end{itemize}
\end{con}

\begin{lem}\label{nonemp2}
For all $\lambda \in C \cup \{\lambda_f\}$, $\emptyset \in P_{\lambda}$.
\end{lem}

\begin{proof}
Fix $\lambda \in C \cup \{\lambda_f\}$ and let 
\begin{itemize}
    \item $h$ be $Col(\omega, \omega_2)$-generic over $W$,
    \item $S \in W[h]$ be a path on $T$ such that $pr(\bigcup S) = p := \dot{p}[h]$, and
    \item $\bar{N}_0 = F^*(p) \in W[h]$ (possible by \ref{4341}).
\end{itemize}
Set $\theta := \omega_1^{W[h]}$. Choose a generic iteration $\langle \bar{M}_i, \pi_{ij} : i \leq j \leq \omega_1^{\bar{N}_0} \rangle$ witnessing $\bar{N}_0 \leq_{\mathbb{P}^{W[h]}_{max}} (H(\omega_2)^W; \in, \mathrm{NS}_{\omega_1}^W, A)$, possible by \ref{4342} to \ref{4345}. Let 
\begin{itemize}
    \item $\langle \bar{N}_i = (N_i; \in, I_i, a_i), \sigma_{ij} : i \leq j \leq \theta \rangle \in W[h]$ be a generic iteration of $\bar{N}_0$,
    \item $\langle \bar{M}_i = (M_i; \in, J_i, b_i), \pi_{ij} : i \leq j \leq \theta \rangle$ denote $\sigma_{0\theta}(\langle \bar{M}_i, \pi_{ij} : i \leq j \leq \omega_1^{\bar{N}_0} \rangle)$, and
    \item $\langle e_i : i < \theta \rangle$ be such that for each $i < \theta$, $e_i$ is a bijection from $\omega$ onto $N_i$.
\end{itemize}
Then by Lemma~\ref{lift}, $\pi_{0\theta}$ lifts to a generic ultrapower map $\pi : W \longrightarrow M$, for some inner model $M$ of $W[h]$. 

Now let $h'$ be $Col(\omega, \pi(\lambda_f))$-generic over $W[h]$, so that in $W[h][h']$, there is 
\begin{equation*}
    \langle F_i : i \in \pi(R) \rangle
\end{equation*}
for which $F_i$ is a strictly increasing cofinal map from $\omega$ into $i$ whenever $i \in \pi(R)$. It is easy to verify that if
\begin{equation*}
    \mathfrak{C} := \langle \langle F_i : i \in \pi(R) \rangle, \langle \rangle \rangle \text{,}
\end{equation*}
then without loss of generality,
\begin{equation*}
    \mathfrak{D} := \langle \langle \bar{M}_i, \pi_{ij},  \bar{N}_i, \sigma_{ij} : i \leq j \leq \theta \rangle, S, \langle e_i : i < \theta \rangle, \mathfrak{C} \rangle \in W[h][h']
\end{equation*}
is a $\pi(\lambda)$-certificate relative to $M$. In other words, $\mathfrak{D}$ fulfils the requirements of Definition \ref{def437} relativised to $(\pi, M)$, bearing in mind \ref{4392} of Remark \ref{rem438}. 

Use $(\Sigma')^{\pi, M}(\cdot, \cdot)$ to denote the function $\Sigma'(\cdot, \cdot)$ relativised to $(\pi, M)$. Following \ref{443'2} of Remark \ref{rem443}, we have that
\begin{itemize}
    \item the hypothesis on $\mathfrak{D}$ in \ref{def439}, relativised to $(\pi, M)$, is satisfied, 
    \item $(\Sigma')^{\pi, M}(\mathfrak{D}, \pi(\mathcal{L} \cap Q_{\lambda}))$ is a set found in some weak outer model of $M$, and
    \item $(\Sigma')^{\pi, M}(\mathfrak{D}, \pi(\mathcal{L} \cap Q_{\lambda}))$ $\pi(\Gamma'_{\lambda}) (\pi(\mathcal{L} \cap Q_{\lambda}), \pi(\mathfrak{A}))$-certifies $\emptyset$. 
\end{itemize}
Applying Lemma \ref{inout} in $M$ gives us the fact that 
\begin{equation*}
    \emptyset \in \pi(P_{\lambda}) \text{,}
\end{equation*}
so also $\emptyset \in \mathbb{P}_{\lambda}$ by the elementarity of $\pi$. 
\end{proof}

By Proposition \ref{certgood}, $(\mathfrak{A}, \mathbb{P})$ is good for $\mathcal{L}$. Obviously, $\mathbb{P}$ is definable in the language associated with $\mathfrak{A}$ because $\mathbb{P} \in H(\kappa)$. Moreover, the following hold as they do in the proof of Theorem \ref{notion1}.
\begin{enumerate}[label=(P\arabic*), leftmargin=40pt]
    \item\label{p1'} $P_{\lambda_0} = P_{\lambda_1} \cap Q_{\lambda_0}$ whenever $\lambda_0, \lambda_1 \in C \cup \{\lambda_f\}$ and $\lambda_0 \leq \lambda_1$, and
    \item\label{p2'} $P_{\lambda} = \bigcup \{P_{\lambda'} : \lambda' \in C \cap \lambda\}$ whenever $\lambda \in C \cup \{\lambda_f\}$ and $sup(\lambda \cap C) = \lambda$.
\end{enumerate}

\begin{lem}\label{modeldone2}
Let 
\begin{itemize}
    \item $\lambda \in C \cup \{\lambda_f\}$, and
    \item $g$ be a $\mathbb{P}_{\lambda}$-$\Sigma_1$-generic filter over $W$.
\end{itemize}
Then $\bigcup g \ \Gamma'_{\lambda} (\mathcal{L} \cap Q_{\lambda}, \mathfrak{A}) \text{-certifies } \emptyset$.
\end{lem}

\begin{proof}
Straightforward, by Lemma \ref{main2} (cf. Lemma \ref{modeldone}).
\end{proof}

\begin{lem}\label{lem446}
There is a definition $\mathfrak{D}(\cdot)$ of a function in one variable such that
\begin{enumerate}[label=(\arabic*)]
    \item\label{4461} $\mathfrak{D}(\cdot)$ is absolute for forcing extensions of $W$, and
    \item\label{4462} whenever
    \begin{itemize}
        \item $\lambda \in C \cup \{\lambda_f\}$, 
        \item $W'$ is a forcing extension of $W$, and
        \item $g \in W'$ is a $\mathbb{P}_{\lambda}$-$\Sigma_1$-generic filter over $W$,
    \end{itemize}
    $\mathfrak{D}(g)$ is a $\lambda$-certificate satisfying 
    \begin{equation*}
        \Sigma'(\mathfrak{D}(g), \mathcal{L}) = \bigcup g \text{.}
    \end{equation*} 
\end{enumerate} 
\end{lem}

\begin{proof}
Let $\lambda$, $W'$ and $g$ fulfil the hypothesis of the lemma. Work in $W'$ for the rest of the proof. We shall unambiguously describe --- constituent by constituent --- the construction of a $\lambda$-certificate
\begin{equation*}
    \mathfrak{D} = \langle \langle \bar{M}_i, \pi_{ij}, \bar{N}_i, \sigma_{ij} : i \leq j \leq \omega_1^{W} \rangle, S, \langle e_i : i < \omega_1^W \rangle, \mathfrak{C} \rangle
\end{equation*}
from $g$, checking that $\mathfrak{D}$ fulfils the conditions of Definition \ref{def437} as we go along. The reader ought to check for themselves that 
\begin{itemize}
    \item every step of the construction, as well as the argument for the purpose it serves, requires only facts which are absolute for forcing extensions (often, even for weak outer models) of $W$, and
    \item at every step of the construction, whatever can be deduced about $\Sigma'(\mathfrak{D}(g), \mathcal{L})$ is consistent with 
    \begin{equation*}
        \Sigma'(\mathfrak{D}(g), \mathcal{L}) = \bigcup g \text{.}
    \end{equation*}
\end{itemize}

First, by Lemma \ref{modeldone2}, 
\begin{equation*}
    \bigcup g \ \Gamma'_{\lambda} (\mathcal{L} \cap Q_{\lambda}, \mathfrak{A}) \text{-certifies } \emptyset \text{,}
\end{equation*}
so \ref{cdlast} of Definition \ref{def437} is satisfied with
\begin{equation*}
    \mathfrak{C} := \langle \langle F_i(\bigcup g \cap \mathcal{L}^o) : i \in R \cap \lambda \rangle, \langle X_{\delta, \lambda'}(\bigcup g \cap \mathcal{L}^o) : \delta < \omega_1, \lambda' \in C \cap \lambda \rangle \rangle \text{.}
\end{equation*}
By Remark \ref{rem416} and \ref{443'3} of Remark \ref{rem443}, we too have \ref{cdrlast} of Definition \ref{def437}. 

Set
\begin{equation*}
    S := \{(\Vec{u}, \Vec{\alpha}) : \ulcorner (\Vec{u}, \Vec{\alpha}) \in \dot{T} \urcorner \in \bigcup g\} 
\end{equation*}
as one would naturally do. For each $i < \omega_1^W$, define a binary relation $\sim_i$ on 
\begin{equation*}
    \hat{\sigma}' := (\hat{\sigma} \cup \{\dot{n} : n < \omega\}) \setminus \{\dot{I}\} 
\end{equation*}
as follows: 
\begin{equation*}
    \tau \sim_i \rho \text{ iff } \ulcorner \dot{N}_i \models \tau = \rho \urcorner \in \bigcup g \text{.}
\end{equation*}
Whenever $i < \omega_1^W$, let $\sim'_i$ be the equivalence closure of $\sim_i$, and
\begin{align*}
    N_i := \ & \{[\tau]_{\sim'_i} : \tau \in \hat{\sigma}'\} \\
    \ulcorner \in \urcorner^{\bar{N}_i} = \tilde{\in}_i := \ & \{([\tau]_{\sim'_i}, [\rho]_{\sim'_i}) : \tau, \rho \in \hat{\sigma}' \text{ and } \ulcorner \dot{N}_i \models \tau \in \rho \urcorner \in \bigcup g\} \\
    I_i := \ & \{[\tau]_{\sim'_i} \in N_i : \ulcorner \dot{N}_i \models \dot{I}(\tau) \urcorner \in \bigcup g\} \\
    \tau^{\bar{N}_i} := \ & [\tau]_{\sim'_i} \text{ for every } \tau \in \hat{\sigma}' \\
    \bar{N}_i := \ & (N_i; \tilde{\in}_i, I_i, \langle \tau^{\bar{N}_i} : \tau \in \hat{\sigma}' \rangle) \\
    e_i : \ & \omega \longrightarrow N_i \ (n \mapsto [\dot{n}]_{\sim'_i}) \text{.}
\end{align*}
Then $Th^{1}_{\mathcal{L}}(\bar{N}_i, e_i, i)$ is well-defined for all $i < \omega_1^W$.

\begin{prop}\label{prop448}
For all
\begin{equation*}
    \varphi = \ulcorner \dot{N}_i \models \phi(\xi_1, \ldots, \xi_k, \dot{n_1}, \ldots, \dot{n_l}, \dot{I}, \dot{a}, \dot{M}_{j_1}, \ldots, \dot{M}_{j_m}, \dot{\pi}_{q_{1}r_{1}}, \ldots, \dot{\pi}_{q_{s}r_{s}}, \dot{\Vec{M}}) \urcorner \in \mathcal{L}_3 \text{,}
\end{equation*}
we have
\begin{gather*}
    \varphi \in \bigcup g \iff \bar{N}_i \models \phi[\dot{n}_1 \mapsto e_i(n_1), \ldots, \dot{n}_l \mapsto e_i(n_l)] \text{.}
\end{gather*}
\end{prop}

\begin{proof}
Fix $i < \omega_1^W$. It suffices to show 
\begin{equation*}
    \phi \in Th^{1}_{\mathcal{L}}(\bar{N}_i, e_i, i) \iff \phi \in \bigcup g
\end{equation*}
for all $\phi \in \mathcal{L}^i_3$. We do this by induction on the length of $\phi$. 

\begin{enumerate}[label=Case \arabic*:, leftmargin=50pt]
    \item $\phi = \ulcorner \dot{N}_i \models \tau = \rho \urcorner$ for some $\tau, \rho \in \hat{\sigma}'$. Then 
    \begin{equation*}
        \phi \in Th^{1}_{\mathcal{L}}(\bar{N}_i, e_i, i) \iff \phi \in \bigcup g
    \end{equation*}
    is implied by $\sim_i$ being an equivalence relation, the latter of which holds because 
    \begin{itemize}
        \item $\ulcorner \forall \tau \in \hat{\sigma}' \ (E(\ulcorner \dot{N}_i \models \tau = \tau \urcorner)) \urcorner$,
        \item $\ulcorner \forall \tau, \rho \in \hat{\sigma}' \ (E(\ulcorner \dot{N}_i \models \tau = \rho \urcorner) \implies E(\ulcorner \dot{N}_i \models \rho = \tau \urcorner)) \urcorner$, and
        \item 
        \!
        $\begin{aligned}[t]\ulcorner \forall \tau, \rho, \zeta \in \hat{\sigma}' \ (&(E(\ulcorner \dot{N}_i \models \tau = \rho \urcorner) \wedge (E(\ulcorner \dot{N}_i \models \rho = \zeta \urcorner)) \\
        & \implies E(\ulcorner \dot{N}_i \models \tau = \zeta \urcorner)) \urcorner
        \end{aligned}$
    \end{itemize}
    are members of $\Gamma'$.
    \item $\phi = \ulcorner \dot{N}_i \models \tau \in \rho \urcorner$ for some $\tau, \rho \in \hat{\sigma}'$. Then 
    \begin{equation*}
        \phi \in Th^{1}_{\mathcal{L}}(\bar{N}_i, e_i, i) \iff \phi \in \bigcup g
    \end{equation*}
    is implied by  
    \begin{align*}
        \ulcorner \forall \tau, \rho, \zeta, \gamma \in \hat{\sigma}' \ (( & E(\ulcorner \dot{N}_i \models \tau = \rho \urcorner) \wedge E(\ulcorner \dot{N}_i \models \zeta = \gamma \urcorner) \\
        & \wedge E(\ulcorner \dot{N}_i \models \tau \in \zeta \urcorner)) \implies E(\ulcorner \dot{N}_i \models \rho \in \gamma \urcorner)) \urcorner,
    \end{align*}
    being a member of $\Gamma'$.
    \item $\phi = \ulcorner \dot{N}_i \models \dot{I}(\tau) \urcorner$ for some $\tau \in \hat{\sigma}'$. Then 
    \begin{equation*}
        \phi \in Th^{1}_{\mathcal{L}}(\bar{N}_i, e_i, i) \iff \phi \in \bigcup g
    \end{equation*}
    is implied by  
    \begin{align*}
        \ulcorner \forall \tau, \rho \in \hat{\sigma}' \ ( & (E(\ulcorner \dot{N}_i \models \tau = \rho \urcorner) \wedge E(\ulcorner \dot{N}_i \models \dot{I}(\tau) \urcorner)) \\
        & \implies E(\ulcorner \dot{N}_i \models \dot{I}(\rho) \urcorner)) \urcorner,
    \end{align*}
    being a member of $\Gamma'$.
    \item $\phi = \ulcorner \dot{N}_i \models \neg \psi \urcorner$ for some $\psi$. Then by the induction hypothesis, we have \begin{equation*}
        \phi' \in Th^{1}_{\mathcal{L}}(\bar{N}_i, e_i, i) \iff \phi' \in \bigcup g \text{,} 
    \end{equation*}
    where
    \begin{align*}
        \phi' := \ulcorner \dot{N}_i \models \psi \urcorner \text{.} 
    \end{align*}
    This means  
    \begin{align*}
        \phi \in Th^{1}_{\mathcal{L}}(\bar{N}_i, e_i, i) \iff \phi' \not \in Th^{1}_{\mathcal{L}}(\bar{N}_i, e_i, i) \iff \phi' \not \in \bigcup g \text{,}
    \end{align*}
    and we are done if 
    \begin{equation*}
        \phi' \not \in \bigcup g \iff \phi \in \bigcup g \text{.}
    \end{equation*}
    But this must hold because $\phi' = \neg(\phi)$ and $\bigcup g$ is $\mathcal{L} \cap Q_{\lambda}$-nice. 
    \item $\phi = \ulcorner \dot{N}_i \models \psi_1 \wedge \psi_2 \urcorner$ for some $\psi_1$ and $\psi_2$. Then by the induction hypothesis, we have 
    \begin{align*}
        \phi' \in Th^{1}_{\mathcal{L}}(\bar{N}_i, e_i, i) & \iff \phi' \in \bigcup g \text{ and } \\
        \phi'' \in Th^{1}_{\mathcal{L}}(\bar{N}_i, e_i, i) & \iff \phi'' \in \bigcup g \text{,} 
    \end{align*}
    where
    \begin{align*}
        \phi' := \ & \ulcorner \dot{N}_i \models \psi_1 \urcorner \text{ and } \\
        \phi'' := \ & \ulcorner \dot{N}_i \models \psi_2 \urcorner.
    \end{align*}
    This means  
    \begin{align*}
        \phi \in Th^{1}_{\mathcal{L}}(\bar{N}_i, e_i, i) \iff \phi', \phi'' \in Th^{1}_{\mathcal{L}}(\bar{N}_i, e_i, i) \iff \phi', \phi'' \in \bigcup g,
    \end{align*}
    and we are done if 
    \begin{equation*}
        \phi', \phi'' \in \bigcup g \iff \phi \in \bigcup g \text{.}
    \end{equation*}
    But this is implied by
    \begin{align*}
        \ulcorner (E(\phi) \iff (E(\phi') \wedge E(\phi'')) \urcorner
    \end{align*}
    being a member of $\Gamma'$.
    \item $\phi = \ulcorner \dot{N}_i \models \exists x \ \psi \urcorner$ for some $\psi$. Then by the induction hypothesis, we have 
    \begin{equation*}
        \phi_{\tau} \in Th^{1}_{\mathcal{L}}(\bar{N}_i, e_i, i) \iff \phi_{\tau} \in \bigcup g \text{,}
    \end{equation*} 
    for all $\tau < \hat{\sigma}'$, where
    \begin{align*}
        \phi_{\tau} := \ulcorner \dot{N}_i \models \psi[x \mapsto \tau] \urcorner \text{.}
    \end{align*}
    This means 
    \begin{align*}
        \phi \in Th^{1}_{\mathcal{L}}(\bar{N}_i, e_i, i) & \iff \exists \tau \in \hat{\sigma}' \ (\phi_{\tau} \in Th^{1}_{\mathcal{L}}(\bar{N}_i, e_i, i)) \\
        & \iff \exists \tau \in \hat{\sigma}' \ (\phi_{\tau} \in \bigcup g) \text{,}
    \end{align*}
    and we are done if 
    \begin{equation*}
        \exists \tau \in \hat{\sigma}' \ (\phi_{\tau} \in \bigcup g) \iff \phi \in \bigcup g \text{.}
    \end{equation*} 
    But this is implied by  
    \begin{align*}
        \ulcorner E(\phi) \iff \exists \tau \in \hat{\sigma}' \ (E(\phi_{\tau})) \urcorner
    \end{align*}
    being a member of $\Gamma'$. \qedhere
\end{enumerate}
\end{proof}

In the rest of the proof, we will apply Proposition \ref{prop448} repeatedly and with great fervour. To minimise annoyance, these applications will be done implicitly as much as possible. 

For every $i < \omega_1^W$ and $\phi \in \mathsf{ZFC}^* + \mathsf{MA}(\omega_1)$, 
\begin{equation*}
    \ulcorner E(\ulcorner \dot{N}_i \models \phi \urcorner) \urcorner \in \Gamma' \text{,}
\end{equation*}
so also 
\begin{equation*}
    \bar{N}_i \models \mathsf{ZFC}^* + \mathsf{MA}(\omega_1) \text{.}
\end{equation*}
Particularly, 
\begin{equation*}
    \bar{N}_i \models ``\text{Axiom of Extensionality}"
\end{equation*}
for $i < \omega_1^W$. In a similar vein, 
\begin{equation*}
    \bar{N}_i \models ``I_i \text{ is a normal uniform ideal on } \omega_1 \text{ and } a_i \subset \omega_1" 
\end{equation*}
because
\begin{equation*}
    \ulcorner E(\ulcorner \dot{N}_i \models ``\dot{I} \text{ is a normal uniform ideal on } \omega_1 \text{ and } \dot{a} \subset \omega_1" \urcorner) \urcorner \in \Gamma' \text{,}
\end{equation*}
as $i$ ranges over $\omega_1^W$. Since
\begin{itemize}
    \item $\ulcorner \forall \xi < \omega_1^W \ (E(\ulcorner \dot{N}_i \models ``\xi \text{ is an ordinal}" \urcorner)) \urcorner$,
    \item $\ulcorner \forall \tau \in \hat{\sigma}' \ \exists \xi < \omega_1^W \ (E(\ulcorner \dot{N}_i \models ``\tau \text{ is an ordinal}" \urcorner) \implies E(\ulcorner \dot{N}_i \models \tau = \xi \urcorner)) \urcorner$,
    \item $\ulcorner \forall \xi_1 < \xi_2 < \omega_1^W \ (E(\ulcorner \dot{N}_i \models ``\xi_2 \neq \emptyset" \urcorner) \implies E(\ulcorner \dot{N}_i \models \xi_1 \in \xi_2 \urcorner)) \urcorner$, and
    \item $\ulcorner \forall \tau \in \hat{\sigma}' \ \forall \xi_1 < \omega_1^W \ \exists \xi_2 < \xi_1 \ (E(\ulcorner \dot{N}_i \models \tau \in \xi_1  \urcorner) \implies E(\ulcorner \dot{N}_i \models \tau = \xi_2 \urcorner)) \urcorner$
\end{itemize}
are members of $\Gamma'$ for each $i < \omega_1^W$, we also have the $\tilde{\in}_i$'s being well-founded. This means the $\bar{N}_i$'s are isomorphic to their respective (well-defined) Mostowski collapse, the latter of which shall be henceforth identified with the former. For $n < \omega$, the $e_i(n)$'s shall also be identified with their respective images under the Mostowski collapse function. As a consequence, for all $i < \omega_1^W$,
\begin{itemize}
    \item $(N_i; \tilde{\in}_i = \ \in)$ is a transitive model of $\mathsf{ZFC}^* + \mathsf{MA}(\omega_1)$ 
    \item $ORD \cap N_i \in \omega_1^W + 1$,
    \item $\xi^{\bar{N}_i} = \xi$ for all $\xi \in \omega_1^W \cap N_i$,
    \item $\xi^{\bar{N}_i} = 0$ for all $\xi \in \omega_1^W \setminus N_i$.
\end{itemize}
That 
\begin{equation*}
\ulcorner \forall i < \omega_1^W \ \exists \xi < \omega_1^W \ (\xi \neq 0 \wedge E(\ulcorner \dot{N}_i \models ``\xi = \emptyset" \urcorner)) \urcorner \in \Gamma'
\end{equation*}
allows us to conclude \ref{cdnew} of Definition \ref{def437}.

Now for each $i < \omega_1^W$, 
\begin{itemize}
    \item $\ulcorner \forall \tau \in \hat{\sigma}' \ \exists n < \omega \ (E(\ulcorner \dot{N}_i \models \tau = \dot{n} \urcorner)) \urcorner$, and
    \item $\ulcorner \forall m, n < \omega \ (m \neq n \implies E(\ulcorner \dot{N}_i \models \neg \ \dot{m} = \dot{n} \urcorner)) \urcorner$
\end{itemize}
being members of $\Gamma'$ tells us that $e_i$ is a bijection from $\omega$ into $N_i$. This settles \ref{cd1.5} of Definition \ref{def437}. Then 
\begin{itemize}
    \item 
    \!
    $\begin{aligned}[t]
        \ulcorner & \forall (\Vec{u}_1, \Vec{\alpha}_1), (\Vec{u}_2, \Vec{\alpha}_2) \in {^{< \omega}{\omega}} \times {^{<\omega}{\omega_2^W}} \\
        & ((E(\ulcorner (\Vec{u}_1, \Vec{\alpha}_1) \in \dot{T} \urcorner) \wedge E(\ulcorner (\Vec{u}_2, \Vec{\alpha}_2) \in \dot{T} \urcorner)) \\
        & \mspace{5mu} \implies ((\Vec{u}_1 \subset \Vec{u}_2 \wedge \Vec{\alpha}_1 \subset \Vec{\alpha}_2) \vee (\Vec{u}_2 \subset \Vec{u}_1 \wedge \Vec{\alpha}_2 \subset \Vec{\alpha}_1))) \urcorner \text{,}
    \end{aligned}$
    \item $\ulcorner \forall n < \omega \ \exists (\Vec{u}, \Vec{\alpha}) \in {^{< \omega}{\omega}} \times {^{<\omega}{\omega_2^W}} \ (E(\ulcorner (\Vec{u}, \Vec{\alpha}) \in \dot{T} \urcorner) \wedge n \in dom(\Vec{u})) \urcorner$, and
    \item $\ulcorner \forall (\Vec{u}, \Vec{\alpha}) \in {^{< \omega}{\omega}} \times {^{< \omega}{\omega_2^W}} \ (E(\ulcorner (\Vec{u}, \Vec{\alpha}) \in \dot{T} \urcorner) \implies (\Vec{u}, \Vec{\alpha}) \in T) \urcorner$
\end{itemize}
being members of $\Gamma'$, and 
\begin{itemize}
    \item
    \!
    $\begin{aligned}[t]
        \ulcorner & \forall x \in \mathcal{L}^0_0 \ \exists (\Vec{u}, \Vec{\alpha}) \in {^{< \omega}{\omega}} \times {^{< \omega}{\omega_2^W}} \\
        & (E(x) \implies (Gd(x) \in ran(\Vec{u}) \wedge E(\ulcorner (\Vec{u}, \Vec{\alpha}) \in \dot{T} \urcorner))) \urcorner \text{, and}
    \end{aligned}$
    \item
    \!
    $\begin{aligned}[t]
        \ulcorner & \forall (\Vec{u}, \Vec{\alpha}) \in {^{< \omega}{\omega}} \times {^{< \omega}{\omega_2^W}} \ \forall n < \omega \ \exists x \in \mathcal{L}^0_0 \\
        & ((n \in ran(\Vec{u}) \wedge E(\ulcorner (\Vec{u}, \Vec{\alpha}) \in \dot{T} \urcorner)) \implies (E(x) \wedge Gd(x) = n)) \urcorner 
    \end{aligned}$
\end{itemize}
being members of $\Gamma'$, respectively give us \ref{cd2} and \ref{cd3} of Definition \ref{def437}. According to \ref{4344} and \ref{4346b}, $(N_0; \in, I_0, a_0)$ must be a member of $\mathbb{P}_{max}$. In particular, $(N_0; \in, I_0)$ is generically iterable.

Whenever $i \leq j < \omega_1^W$, define
\begin{equation*}
    \sigma_{ij} := \{(e_i(m), e_j(n)) : \ulcorner E(\ulcorner \dot{\sigma}_{ij}(\dot{m}) = \dot{n} \urcorner) \urcorner \in \bigcup g\} \subset N_i \times N_j \text{.}
\end{equation*}
We see that $\langle N_i, \sigma_{ij} : i \leq  j < \omega_1^{W} \rangle$ is a directed system as
\begin{itemize}
    \item $\ulcorner \forall i \leq j < \omega_1^W \ \forall m < \omega \ \exists n < \omega \ (E(\ulcorner \dot{\sigma}_{ij}(\dot{m}) = \dot{n} \urcorner)) \urcorner$,
    \item
    \!
    $\begin{aligned}[t]
        \ulcorner \forall i \leq j < \omega_1^W \ \forall m, n_1, n_2 < \omega \ ( & (E(\ulcorner \dot{\sigma}_{ij}(\dot{m}) = \dot{n_1} \urcorner) \wedge E(\ulcorner \dot{\sigma}_{ij}(\dot{m}) = \dot{n_2} \urcorner)) \\
        & \implies n_1 = n_2) \urcorner \text{,}
    \end{aligned}$
    \item
    \!
    $\begin{aligned}[t]
        \ulcorner \forall i \leq j < \omega_1^W \ \forall m_1, m_2, n < \omega \ ( & (E(\ulcorner \dot{\sigma}_{ij}(\dot{m_1}) = \dot{n} \urcorner) \wedge E(\ulcorner \dot{\sigma}_{ij}(\dot{m_2}) = \dot{n} \urcorner)) \\
        & \implies m_1 = m_2) \urcorner \text{, and}
    \end{aligned}$
    \item
    \!
    $\begin{aligned}[t]
        \ulcorner \forall i \leq j \leq k < \omega_1^W \ \forall l, m, n < \omega \ ( & (E(\ulcorner \dot{\sigma}_{ij}(\dot{l}) = \dot{m} \urcorner) \wedge E(\ulcorner \dot{\sigma}_{jk}(\dot{m}) = \dot{n} \urcorner)) \\
        & \implies E(\ulcorner \dot{\sigma}_{ik}(\dot{l}) = \dot{n} \urcorner)) \urcorner
    \end{aligned}$
\end{itemize}
are members of $\Gamma'$. Furthermore, because  
\begin{itemize}
    \item 
    \!
    $\begin{aligned}[t]
        \ulcorner & \forall i, j < \omega_1^W \ \forall \phi \in \mathcal{L}^1(\hat{\sigma} \cup \ulcorner \in \urcorner, x) \ \forall m, n \in \omega \ (E(\ulcorner \dot{\sigma}_{ij}(\dot{m}) = \dot{n} \urcorner) \\ 
        & \implies (E(\ulcorner \dot{N}_i \models \phi[x \mapsto \dot{m}] \urcorner) \iff E(\ulcorner \dot{N}_j \models \phi[x \mapsto \dot{n}] \urcorner))) \urcorner,
    \end{aligned}$
    \\
    \\
    (Recalling Definition \ref{defsvf} and given what we have shown thus far, this means to say that for $i \leq j < \omega_1^W$, $\sigma_{ij}$ is an elementary embedding from $\bar{N}_i$ into $\bar{N}_j$.)
    \item 
    \!
    $\begin{aligned}[t]
        \ulcorner & \forall i < \omega_1^W \ \forall m < \omega \ \exists \xi < \omega_1^W \ \exists n, n' < \omega \ \\
        & (E(\ulcorner \dot{N}_i \models ``\dot{n} \text{ is a function with domain } \omega_1" \urcorner) \wedge E(\ulcorner \dot{N}_i \models ``\xi = \omega_1" \urcorner) \\ 
        & \mspace{6mu} \wedge E(\ulcorner \dot{\sigma}_{i(i+1)}(\dot{n}) = \dot{n'} \urcorner) \wedge E(\ulcorner \dot{N}_{i+1} \models ``\dot{n'}(\xi) = \dot{m}" \urcorner)) \urcorner,
    \end{aligned}$
    \\
    \\
    (This means to say that for all $i < \omega_1^W$, $\bar{N}_{i+1}$ is generated over $\bar{N}_i$ from the ``seed'' $\omega_1^{\bar{N}_i}$.)
    \item 
    \!
    $\begin{aligned}[t]
        \ulcorner & \forall i < \omega_1^W \ \forall m < \omega \ \exists \xi < \omega_1^W \ \exists n, n' < \omega \ (E(\ulcorner \dot{N}_i \models ``\dot{m} \text{ is dense in } \mathcal{P}(\omega_1) \setminus \dot{I}" \urcorner) \\ 
        & \implies (E(\ulcorner \dot{N}_i \models \dot{n} \in \dot{m} \urcorner) \ \wedge E(\ulcorner \dot{N}_i \models ``\xi = \omega_1" \urcorner) \\
        & \mspace{55mu} \wedge E(\ulcorner \dot{\sigma}_{i(i+1)}(\dot{n}) = \dot{n'} \urcorner) \wedge E(\ulcorner \dot{N}_{i+1} \models \xi \in \dot{n'} \urcorner))) \urcorner,
    \end{aligned}$
    \\
    \\
    (This means to say that the set $$\{x \in \mathcal{P}(\omega_1^{\bar{N}_i}) \cap N_i : \omega_1^{\bar{N}_i} \in \sigma_{i(i+1)}(x)\}$$ is $((\mathcal{P}(\omega_1^{\bar{N}_i}) \cap N_i) \setminus I_i)$-generic over $N_i$, for all $i < \omega_1^W$.)
    \item 
    \!
    $\begin{aligned}[t]
        \ulcorner \forall i < \omega_1^W \ \forall m < \omega \ \exists j < \omega_1^W \ \exists n < \omega \ ( & ``i \text{ is a limit ordinal}'' \\ 
        & \implies  (j < i \wedge E(\ulcorner \dot{\sigma}_{ji}(\dot{n}) = \dot{m} \urcorner))) \urcorner,
    \end{aligned}$
    \\
    \\
    (Given what we have shown thus far, this means to say that $(\bar{N}_i, \langle \sigma_{ji} : j < i \rangle)$ is the direct limit of $\langle \bar{N}_j, \sigma_{jj'} : j \leq j' < i \rangle$ for all limit ordinals $i < \omega_1^W$.)
\end{itemize}
are members of $\Gamma'$, 
\begin{equation*}
    \langle (N_i; \in, I_i, a_i), \sigma_{ij} : i \leq j < \omega_1^W \rangle
\end{equation*} 
is a generic iteration. Letting 
\begin{itemize}
    \item $\hat{N} := (\bar{N}_{\omega_1^W}, \langle \sigma_{i\omega_1^W} : i < \omega_1^W \rangle)$ be a direct limit of the directed system
    \begin{equation*}
        \langle \bar{N}_i, \sigma_{ij} : i \leq j < \omega_1^W \rangle \text{,}
    \end{equation*}
    of elementary embeddings, and
    \item $\sigma_{\omega_1^W\omega_1^W}$ be the identity map on the base set of $\bar{N}_{\omega_1^W}$,
\end{itemize}
we arrive at \ref{cd0} of Definition \ref{def437}. The generic iterability of $(N_0; \in, I_0)$ then guarantees the existence of a (unique) Mostowski collapse of $\hat{N}$. Identifying $\hat{N}$ with its Mostowski collapse, we have \ref{cd5} of Definition \ref{def437}.

Finally, set
\begin{align*}
    \bar{M}_i := \ & \dot{M}_i^{\bar{N}_i} \text{ for } i < \omega_1^W \\
    \pi_{ij} := \ & \dot{\pi}_{ij}^{\bar{N}_j} \text{ for } i \leq j < \omega_1^W \\
    \bar{M}_{\omega_1^W} := \ & (H(\omega_2)^W; \in, \mathrm{NS}_{\omega_1}^W, A) \\
    \pi_{i\omega_1^W} := \ & \{(e_i(n), x) : \ulcorner \dot{\pi}_{i\omega_1^W}(\dot{n}) = x \urcorner \in \bigcup g\} \\
    \pi_{\omega_1^W\omega_1^W} := \ & \text{the identity map on } H(\omega_2)^W \text{,}
\end{align*}
so that \ref{cd7} of Definition \ref{def437} clearly holds. Considering 
\begin{itemize}
    \item
    \!
    $\begin{aligned}[t]
        \ulcorner & \forall i, j, k, \xi < \omega_1^W \ ((E(\ulcorner \dot{N}_k \models ``\xi = \omega_1" \urcorner) \wedge i \leq j \leq \xi) \\ 
        & \implies E(\ulcorner \dot{N}_{k} \models ``\dot{\Vec{M}} = \langle \check{M}_{i'}, \check{\pi}_{i'j'} : i' \leq j' \leq \xi \rangle \text{ is a generic iteration} \\
        & \mspace{130mu} \text{with } \check{M}_{i} = \dot{M}_j \text{ and } \check{\pi}_{ij} = \dot{\pi}_{ij}" \urcorner)) \urcorner \text{,}
    \end{aligned}$
    \item
    \!
    $\begin{aligned}[t]
        \ulcorner & \forall i, j, k, \xi < \omega_1^W \ ((E(\ulcorner \dot{N}_k \models ``\xi = \omega_1" \urcorner) \wedge i \leq j \wedge \xi < j) \\ 
        & \implies E(\ulcorner \dot{N}_{k} \models ``\dot{M}_i = \emptyset \text{ and } \dot{\pi}_{ij} = \emptyset" \urcorner)) \urcorner \text{,}
    \end{aligned}$
    \item 
    \!
    $\begin{aligned}[t]
        \ulcorner & \forall i, \xi < \omega_1^W \ (E(\ulcorner \dot{N}_i \models ``\xi = \omega_1" \urcorner) \implies i \leq \xi) \urcorner \text{, and}
    \end{aligned}$
    \item $\ulcorner E(\ulcorner \dot{N}_0 \models ``\dot{M}_0 \text{ is a member of } \mathbb{P}_{max}" \urcorner) \urcorner$
\end{itemize}
are members of $\Gamma'$, 
\begin{equation*}
    \Vec{I} := \langle (N_i; \in, I_i, a_i), \sigma_{ij} : i \leq j \leq \omega_1^W \rangle
\end{equation*}
being a generic iteration then gives us
\begin{itemize}
    \item \ref{cd1} of Definition \ref{def437},
    \item $\langle \bar{M}_i, \pi_{ij} : i \leq j < \omega_1^W \rangle$ is a generic iteration, and
    \item for every $i < \omega_1^W$,
    \begin{align*}
         \dot{\Vec{M}}^{\bar{N}_i} = \ & \langle \bar{M}_j, \pi_{jk} : j \leq k \leq \omega_1^{\bar{N}_i} \rangle \\
         \dot{M}_j^{\bar{N}_i} = \ & \bar{M}_j \text{ for all } j \leq \omega_1^{\bar{N}_i} \\
         \dot{M}_j^{\bar{N}_i} = \ & \emptyset \text{ for all } \omega_1^{\bar{N}_i} < j < \omega_1^W \\
         \dot{\pi}_{jk}^{\bar{N}_i} = \ & \pi_{jk} \text{ for all } j \leq k \leq \omega_1^{\bar{N}_0} \\
         \dot{\pi}_{jk}^{\bar{N}_i} = \ & \emptyset \text{ for all } j \leq k \text{ and } \omega_1^{\bar{N}_0} < k < \omega_1^W \text{,}
    \end{align*}
\end{itemize}
implying that \ref{2ndlast} of Definition \ref{def437} is also fulfilled. Now \ref{cd4} of Definition \ref{def437} is true by \ref{4346} and the fact that 
\begin{align*}
    \ulcorner E(\ulcorner \dot{N}_0 \models `` & \dot{\Vec{M}} = \langle \check{M}_{i}, \check{\pi}_{ij} : i \leq j \leq \omega_1 \rangle \text{ is a generic iteration and} \\ 
    & \text{if } \check{M}_{\omega_1} = (M_{\omega_1}; \in, J_{\omega_1}, b_{\omega_1}) \text{ then} \\
    &  \dot{a} = b_{\omega_1} \text{ and } \dot{I} \cap M_{\omega_1} = J_{\omega_1}" \urcorner) \urcorner
\end{align*}
is a member of of $\Gamma'$.

Among the conditions of Definition \ref{def437} to be verified, we are left with \ref{cd6}. In order to show \ref{cd6} of Definition \ref{def437}, we need only show that
\begin{equation*}
    (\bar{M}_{\omega_1^W}, \langle \pi_{i\omega_1^W} : i < \omega_1^W \rangle)
\end{equation*}
is a direct limit of 
\begin{equation*}
    \langle \bar{M}_i, \pi_{ij} : i \leq j < \omega_1^W \rangle \text{.}
\end{equation*}
But in light of what is known thus far, this is a result of $\Gamma'$ having the following as members:
\begin{enumerate}[label=(\alph*)]
    \item\label{448a}
    \!
    $\begin{aligned}[t]
        \ulcorner \forall i < \omega_1^W \ \forall n < \omega \ \forall x \in H(\omega_2)^W \ ( & E(\ulcorner \dot{\pi}_{i\omega_1^W}(\dot{n}) = x \urcorner) \\
        & \implies E(\ulcorner \dot{N}_i \models ``\dot{n} \in dom(\dot{\pi}_{ii})" \urcorner)) \urcorner \text{,}
    \end{aligned}$
    \item
    \!
    $\begin{aligned}[t]
        \ulcorner \forall i < \omega_1^W \ \forall n < \omega \ \exists x \in H(\omega_2)^W \ ( & E(\ulcorner \dot{N}_i \models ``\dot{n} \in dom(\dot{\pi}_{ii})" \urcorner) \\
        & \implies E(\ulcorner \dot{\pi}_{i\omega_1^W}(\dot{n}) = x \urcorner)) \urcorner \text{,}
    \end{aligned}$
    \item
    \!
    $\begin{aligned}[t]
        \ulcorner & \forall i < \omega_1^W \ \forall n < \omega \ \forall x, y \in H(\omega_2)^W \\
        & ((E(\ulcorner \dot{\pi}_{i\omega_1^W}(\dot{n}) = x \urcorner) \wedge E(\ulcorner \dot{\pi}_{i\omega_1^W}(\dot{n}) = y \urcorner)) \implies x = y) \urcorner \text{,}
    \end{aligned}$
    \item\label{448d}
    \!
    $\begin{aligned}[t]
        \ulcorner & \forall i < \omega_1^W \ \forall m, n < \omega \ \forall x \in H(\omega_2)^W \\
        & ((E(\ulcorner \dot{\pi}_{i\omega_1^W}(\dot{m}) = x \urcorner) \wedge E(\ulcorner \dot{\pi}_{i\omega_1^W}(\dot{n}) = x \urcorner)) \implies m = n) \urcorner \text{,}
    \end{aligned}$
    \\
    \\
    (Points \ref{448a} to \ref{448d} mean to say that $\pi_{i\omega_1^W}$ is an injection from the base set of $\bar{M}_i$ into $H(\omega_2)^W$, for all $i < \omega_1^W$.)
    \item\label{448e}
    \!
    $\begin{aligned}[t]
        \ulcorner & \forall i \leq j < \omega_1^W \ \forall m, n < \omega \ \forall x \in H(\omega_2)^W \ \exists l < \omega \\
        & ((E(\ulcorner \dot{\pi}_{j\omega_1^W}(\dot{n}) = x \urcorner) \wedge E(\ulcorner \dot{N}_j \models ``\dot{m} \in dom(\dot{\pi}_{ij}) \text{ and } \dot{\pi}_{ij}(\dot{m}) = \dot{n}" \urcorner)) \\
        & \mspace{5mu} \implies (E(\ulcorner \dot{\pi}_{i\omega_1^W}(\dot{l}) = x \urcorner) \wedge E(\ulcorner \dot{\sigma}_{ij}(\dot{l}) = \dot{m} \urcorner))) \urcorner
    \end{aligned}$
    \item\label{448f}
    \!
    $\begin{aligned}[t]
        \ulcorner & \forall i \leq j < \omega_1^W \ \forall m, n < \omega \ \forall x \in H(\omega_2)^W \ \exists l < \omega \\
        & ((E(\ulcorner \dot{\pi}_{i\omega_1^W}(\dot{l}) = x \urcorner) \wedge E(\ulcorner \dot{\sigma}_{ij}(\dot{l}) = \dot{m} \urcorner) \\
        & \mspace{5mu} \wedge E(\ulcorner \dot{N}_j \models ``\dot{m} \in dom(\dot{\pi}_{ij}) \text{ and } \dot{\pi}_{ij}(\dot{m}) = \dot{n}" \urcorner)) \\
        & \mspace{5mu} \implies E(\ulcorner \dot{\pi}_{j\omega_1^W}(\dot{n}) = x \urcorner)) \urcorner
    \end{aligned}$
    \\
    \\
    (Points \ref{448e} to \ref{448f} mean to say that 
    \begin{equation*}
        \langle M_i, \pi_{ij} : i \leq  j \leq \omega_1^{W} \rangle
    \end{equation*}
    is a directed system, where $M_i$ denotes the base set of $\bar{M}_i$ for each $i < \omega_1^W$.)
    \item\label{448g}
    \!
    $\begin{aligned}[t]
        \ulcorner \forall x \in H(\omega_2)^W \ \exists i < \omega_1^W \ \exists n < \omega \ (E(\ulcorner \dot{\pi}_{i\omega_1^W}(\dot{n}) = x \urcorner)) \urcorner \text{,}
    \end{aligned}$
    \item
    \!
    $\begin{aligned}[t]
        \ulcorner & \forall i < \omega_1^W \ \forall n < \omega \ \exists x \in \mathrm{NS}_{\omega_1}^W \\
        & (E(\ulcorner \dot{N}_i \models ``\dot{M}_i = (M_i; \in, J_i, b_i) \text{ and } \dot{n} \in J_i" \urcorner) \implies E(\ulcorner \dot{\pi}_{i\omega_1^W}(\dot{n}) = x \urcorner)) \urcorner \text{,}
    \end{aligned}$
    \item
    \!
    $\begin{aligned}[t]
        \ulcorner & \forall x \in \mathrm{NS}_{\omega_1}^W \ \exists i < \omega_1^W \ \exists n < \omega  \\
        & (E(\ulcorner \dot{N}_i \models ``\dot{M}_i = (M_i; \in, J_i, b_i) \text{ and } \dot{n} \in J_i" \urcorner) \wedge E(\ulcorner \dot{\pi}_{i\omega_1^W}(\dot{n}) = x \urcorner)) \urcorner \text{,}
    \end{aligned}$
    \item
    \!
    $\begin{aligned}[t]
        \ulcorner & \forall i < \omega_1^W \ \forall n < \omega \ \exists x \in A \\
        & (E(\ulcorner \dot{N}_i \models ``\dot{M}_i = (M_i; \in, J_i, b_i) \text{ and } \dot{n} \in b_i" \urcorner) \implies E(\ulcorner \dot{\pi}_{i\omega_1^W}(\dot{n}) = x \urcorner)) \urcorner \text{, and}
    \end{aligned}$
    \item\label{448k}
    \!
    $\begin{aligned}[t]
        \ulcorner & \forall x \in A \ \exists i < \omega_1^W \ \exists n < \omega  \\
        & (E(\ulcorner \dot{N}_i \models ``\dot{M}_i = (M_i; \in, J_i, b_i) \text{ and } \dot{n} \in b_i" \urcorner) \wedge E(\ulcorner \dot{\pi}_{i\omega_1^W}(\dot{n}) = x \urcorner)) \urcorner \text{.} 
    \end{aligned}$
    \\
    \\
    (Points \ref{448g} to \ref{448k} mean to say that $H(\omega_2)^W$, $\mathrm{NS}_{\omega_1}^W$, and $A$ behave as they would relative to
    \begin{equation*}
        \langle \bar{M}_i, \pi_{i\omega_1^W} : i < \omega_1^W \rangle \text{,}
    \end{equation*}
    if
    \begin{equation*}
        (\bar{M}_{\omega_1^W}, \langle \pi_{i\omega_1^W} : i < \omega_1^W \rangle)
    \end{equation*}
    is a direct limit of 
    \begin{equation*}
        \langle \bar{M}_i, \pi_{ij} : i \leq j < \omega_1^W \rangle \text{.)} \qedhere
    \end{equation*}
    \end{enumerate}
\end{proof}

We hereby fix a function definition $\mathfrak{D}(\cdot)$ satisfying properties \ref{4461} and \ref{4462} of Lemma \ref{lem446}.

Analogous to how the proof of Theorem \ref{notion1} goes, we want what is essentially a souped up version of Lemma \ref{sideuni}.

\begin{lem}\label{lem452}
The $\mathcal{L}^*_{\mathfrak{A}}$-$\Pi_2$ sentence 
\begin{align}\label{eq4.4}
\begin{split}
    \ulcorner & \forall x \in H(\lambda_f) \ \forall U \in U(C, \lambda_f) \ \forall i < \omega_1^W \ \forall m < \omega \ \exists \delta < \omega_1^W \ \exists n < \omega \ \exists \nu \in U \\ 
    & (E(\ulcorner \dot{N}_i \models ``\dot{m} \in \mathcal{P}(\omega_1) \setminus \dot{I}" \urcorner) \implies (i < \delta \wedge E(\ulcorner \dot{\sigma}_{i(\delta + 1)}(\dot{m}) = \dot{n} \urcorner) \\
    & \mspace{279mu} \wedge E(\ulcorner \dot{N}_{\delta + 1} \models \delta \in \dot{n} \urcorner) \\
    & \mspace{279mu} \wedge E(\ulcorner x \in \dot{X}_{\delta, \nu} \urcorner))) \urcorner \text{,}
\end{split}
\end{align}
where $U(C, \lambda_f)$ is as in Definition \ref{def423}, is $(\mathfrak{A}, \mathbb{P}, \mathcal{L})$-universal.
\end{lem}

\begin{proof}
Fix arbitrary
\begin{itemize}
    \item $p \in \mathbb{P}$,
    \item $x \in H(\lambda_f)$,
    \item $U \in U(C, \lambda_f)$,
    \item $i < \omega_1^W$, and
    \item $m < \omega$.
\end{itemize}
By Lemma \ref{univer}, it suffices to assume 
\begin{equation}\label{eq4.5}
    \ulcorner \dot{N}_i \models ``\dot{m} \in \mathcal{P}(\omega_1) \setminus \dot{I}" \urcorner \in p
\end{equation}
and show that there are 
\begin{itemize}
    \item $i < \delta < \omega_1^W$,
    \item $n < \omega$, and
    \item $\nu \in U$
\end{itemize} 
for which 
\begin{equation*}
    p \cup \{\ulcorner \dot{\sigma}_{i(\delta + 1)}(\dot{m}) = \dot{n} \urcorner, \ulcorner \dot{N}_{\delta + 1} \models \delta \in \dot{n} \urcorner, \ulcorner x \in \dot{X}_{\delta, \nu} \urcorner\} \in \mathbb{P} \text{.}
\end{equation*}

Choose 
\begin{itemize}
    \item $\nu \in U$ such that $x \in Q_{\nu}$ and $p \in \mathbb{P}_{\nu}$ (possible by \ref{p2'}), and
    \item $g \times f$ a $\mathbb{P}_{\nu} \times Col(\omega, \nu)$-generic filter over $W$ with $p \in g$,
\end{itemize}
so that $g \in W[g \times f]$ is a $\mathbb{P}_{\nu}$-generic filter over $W$ and $|\nu|^{W[g \times f]} = \omega$. By Lemma \ref{lem446}, 
\begin{equation*}
    \mathfrak{D}(g) = \langle \langle \bar{M}_i, \pi_{ij}, \bar{N}_i = (N_i; \in, \Vec{X}), \sigma_{ij} : i \leq j \leq \omega_1^{W} \rangle, S, \langle e_i : i < \omega_1^W \rangle, \mathfrak{C} \rangle
\end{equation*}
is a $\nu$-certificate satisfying
\begin{equation*}
    \Sigma'(\mathfrak{D}(g), \mathcal{L}) = \bigcup g \text{.}
\end{equation*}
Set $\theta := \omega_1^{W[g \times f]}$, and find a one-step extension 
\begin{equation*}
    \mathfrak{I}_1 = \langle (N_i; \in, I_i, a_i), \sigma_{ij} : i \leq j \leq \omega_1^{W} + 1 \rangle
\end{equation*}  
of the generic iteration 
\begin{equation*}
    \langle (N_i; \in, I_i, a_i), \sigma_{ij} : i \leq j \leq \omega_1^{W} \rangle
\end{equation*}
in $W[g \times f]$, where $(N_{\omega_1^W + 1}; \in, I_{\omega_1^W + 1})$ is the generic ultrapower of $(N_{\omega_1^W}; \in, I_{\omega_1^W})$ via a $N_{\omega_1^W}$-generic ultrafilter on $I_{\omega_1^W}$ containing $\sigma_{i\omega_1^W}(e_i(m))$. The latter is possible by (\ref{eq4.5}). As a result, $\omega_1^W \in N_{\omega_1^W + 1}$ and
\begin{equation}\label{eq4.6}
    (N_{\omega_1^W + 1}; \in, I_{\omega_1^W + 1}) \models \omega_1^W \in \sigma_{i(\omega_1^W + 1)}(e_i(m)) \text{.}
\end{equation}
Still in $W[g \times f]$, let 
\begin{itemize}
    \item $\mathfrak{I} = \langle \bar{N}_i = (N_i; \in, I_i, a_i), \sigma_{ij} : i \leq j \leq \theta \rangle \in W[g \times f]$ be a generic iteration extending $\mathfrak{I}_1$,
    \item $\langle \bar{M}_i = (M_i; \in, J_i, b_i), \pi_{ij} : i \leq j \leq \theta \rangle$ denote $\sigma_{0\theta}(\langle \bar{M}_i, \pi_{ij} : i \leq j \leq \omega_1^{\bar{N}_0} \rangle)$, and
    \item $\langle e_i : i < \theta \rangle$ extending $\langle e_i : i < \omega_1^W \rangle$ be such that for each $i < \theta$, $e_i$ is a bijection from $\omega$ onto $N_i$.
\end{itemize}
Then by Lemma~\ref{lift}, $\pi_{\omega_1^W\theta}$ lifts to a generic ultrapower map $j : W \longrightarrow M$, for some inner model $M$ of $W[g \times f]$. Moreover, 
\begin{enumerate}[label=(\alph*)]
    \item\label{451a'} $crit(j) = \omega_1^W$, and
    \item\label{451b'} $j(\omega_1^W) = \theta$.
\end{enumerate}

Within a suitable forcing extension $W^*$ of $W[g \times f]$, there is $\mathfrak{C}'$ such that 
\begin{itemize}
    \item $\mathfrak{C}'$ satisfies \ref{cdlast} to \ref{cdrlast} of Definition \ref{def437} relativised to $(j, M)$, and
    \item $\{\ulcorner j(x) \in \dot{X}_{\omega_1^W, j(\nu)} \urcorner\} \in \Sigma(\mathfrak{C}', j(\mathcal{L}^o \cap Q_{\lambda}))$,
\end{itemize} 
where $\lambda := \lambda_f$. Indeed, let us stipulate that $\mathfrak{C}'$ be constructed the same way $\mathfrak{C}$ is in the proof of Lemma \ref{sideuni}. By \ref{4344}, \ref{4346b}, \ref{4346}, \ref{451b'} and the elementarity of $j$, 
\begin{equation*}
    \langle \langle \bar{M}_i, \pi_{ij}, \bar{N}_i, \sigma_{ij} : i \leq j \leq \theta \rangle, j(S), \langle e_i : i < \theta \rangle \rangle
\end{equation*}
satisfies the hypothesis of \ref{4391} of Remark \ref{rem438} relativised to $(j, M)$. Applications of \ref{4390} and \ref{4391} of Remark \ref{rem438} relativised to $(j, M)$ then allow us to conclude, without loss of generality, that
\begin{equation*}
    \mathfrak{D}' := \langle \langle \bar{M}_i, \pi_{ij}, \bar{N}_i, \sigma_{ij} : i \leq j \leq \theta \rangle, j(S), \langle e_i : i < \theta \rangle, \mathfrak{C}' \rangle
\end{equation*}
is a $j(\lambda)$-certificate relative to $M$. In other words, $\mathfrak{D}'$ fulfils the requirements of Definition \ref{def437} relativised to $(j, M)$. By (\ref{eq4.6}), the elementarity of $j$, and \ref{443'1} and \ref{443'2} of Remark \ref{rem443} relativised to $(j, M)$, 
\begin{enumerate}[label=(\alph*)]
    \setcounter{enumi}{2}
    \item\label{451c} $\Sigma'(\mathfrak{D}', j(\mathcal{L}))$ is a set found in some weak outer model of $M$,
    \item\label{451a} $\Sigma'(\mathfrak{D}', j(\mathcal{L}))$ $j(\Gamma'_{\lambda}) (j(\mathcal{L}), j(\mathfrak{A}))$-certifies $\emptyset$, and
    \item\label{451b} there is $n < \omega$ for which 
    \begin{gather*}
        \{\ulcorner \dot{\sigma}_{i(\omega_1^W + 1)}(\dot{m}) = \dot{n} \urcorner, \ulcorner \dot{N}_{\omega_1^W + 1} \models \omega_1^W \in \dot{n} \urcorner, \ulcorner j(x) \in \dot{X}_{\omega_1^W, j(\nu)} \urcorner\} \\
        \subset \Sigma'(\mathfrak{D}', j(\mathcal{L})) \text{.}
    \end{gather*}
\end{enumerate}
Through a routine unfolding of the definition of $\mathfrak{D}'$, with \ref{451a'} in mind, one may ascertain 
\begin{equation*}
    j(p) = j" p \subset \Sigma'(\mathfrak{D}', j(\mathcal{L})) \text{.}
\end{equation*}
a fact which can be combined with \ref{451c} to \ref{451b} and Lemma \ref{inout} to yield
\begin{align*}
    (M; \in) \models \exists \delta < \theta \ \exists n < j(\omega) \ (``\phi \text{ and } j(i) < \delta") \text{,}
\end{align*}
where $\phi$ is a formula in variables $\delta$ and $n$ expressing
\begin{equation*}
    ``j(p) \cup \{\ulcorner \dot{\sigma}_{j(i)(\delta + 1)}(j(\dot{m})) = \dot{n} \urcorner, \ulcorner \dot{N}_{\delta + 1} \models \delta \in \dot{n} \urcorner, \ulcorner j(x) \in \dot{X}_{\delta, j(\nu)} \urcorner\} \in j(\mathbb{P})" \text{.}
\end{equation*}
Appealing once more to \ref{451b'} and the elementarity of $j$, we obtain
\begin{align*}
    (W; \in) \models \ & \exists \delta < \omega_1^W \ \exists n < \omega \\
    & (``p \cup \{\ulcorner \dot{\sigma}_{i(\delta + 1)}(\dot{m}) = \dot{n} \urcorner, \ulcorner \dot{N}_{\delta + 1} \models \delta \in \dot{n} \urcorner, \ulcorner x \in \dot{X}_{\delta, \nu} \urcorner\} \in \mathbb{P} \\
    & \mspace{15mu} \text{and } i < \delta") \text{,}
\end{align*}
and we are done.
\end{proof}

\begin{lem}\label{lem453}
Let $g$ be $\mathbb{P}$-generic over $W$, and
\begin{equation*}
    \mathfrak{D}(g) = \langle \langle \bar{M}_i, \pi_{ij},  \bar{N}_i = (N_i; \in, \Vec{X}), \sigma_{ij} : i \leq j \leq \omega_1^W \rangle, S, \langle e_i : i < \omega_1^W \rangle, \mathfrak{C} \rangle \in W[g] \text{.}
\end{equation*}
Then every member of $\mathcal{P}(\omega_1^W) \cap N_{\omega_1^W} \setminus I_{\omega_1^W}$ is stationary in $W[g]$.
\end{lem}

\begin{proof}
Let 
\begin{itemize}
    \item $p \in \mathbb{P}$,
    \item $\dot{C}$ be a $\mathbb{P}$-name such that $p \Vdash^W_{\mathbb{P}} ``\dot{C}$ is a club in $\omega_1^W"$,
    \item $D := \{(q, \eta) \in \mathbb{P} \times \omega_1 : q \Vdash^W_{\mathbb{P}} \eta \in \dot{C}\}$,
    \item $g$ be a $\mathbb{P}$-generic filter over $W$ with $p \in g$.
\end{itemize}
Applying ($\diamond$) with 
\begin{itemize}
    \item $\mathbb{P}$ in place of $P$, and
    \item $D$ in place of $B$,
\end{itemize}
we get
\begin{align*}
    U := \{\lambda \in C : (Q_{\lambda}; \in, \mathbb{P}, A_{\lambda}) \prec (H(\lambda_f); \in, \mathbb{P}, D)\}
\end{align*}
is stationary in $\lambda_f$, so $U \in U(C, \lambda_f)$.

Move to $W[g]$. There, due to Lemma \ref{lem446}, $\mathfrak{D}(g)$ is a $\lambda_f$-certificate and
\begin{equation*}
    \Sigma'(\mathfrak{D}(g), \mathcal{L}) = \bigcup g \text{.}
\end{equation*}
Choose $Y \in \mathcal{P}(\omega_1^W) \cap N_{\omega_1^W} \setminus I_{\omega_1^W}$. It suffices to show $Y$ has non-trivial intersection with $\dot{C}[g]$. To that end, note that there are $i < \omega_1^W$ and $m < \omega$ for which $\sigma_{i\omega_1^W}(e_i(m)) = Y$. But this means
\begin{equation*}
    \ulcorner \dot{N}_i \models ``\dot{m} \in \mathcal{P}(\omega_1) \setminus \dot{I}" \urcorner \in \bigcup g \text{,}
\end{equation*}
so by Lemma \ref{lem452}, 
\begin{equation*}
    \{\ulcorner \dot{\sigma}_{i(\delta + 1)}(\dot{m}) = \dot{n} \urcorner, \ulcorner \dot{N}_{\delta + 1} \models \delta \in \dot{n} \urcorner, \ulcorner x \in \dot{X}_{\delta, \nu} \urcorner\} \subset \bigcup g
\end{equation*}
for some $\delta < \omega_1^W$ and $n < \omega$. Consequently, 
\begin{equation*}
    \delta \in \sigma_{i(\delta + 1)}(e_i(m)) \subset Y \text{.}
\end{equation*}

Furnished with \ref{443'3} of Remark \ref{rem443}, the argument for $\mu \in \dot{C}[g]$ in the proof of Lemma \ref{statdone} can be reused with 
\begin{itemize}
    \item $\delta$ in place of $\mu$, 
    \item (\ref{eq4.4}) in place of (\ref{c9}),
    \item Lemma \ref{modeldone2} in place of Lemma \ref{modeldone}, and
    \item Lemma \ref{lem452} in place of Lemma \ref{sideuni},
\end{itemize}
to net us $\delta \in \dot{C}[g]$. This completes the proof.
\end{proof}

\begin{lem}\label{lem454}
$\mathbb{P}$ is stationary-preserving. That is, $\mathbb{P}$ fulfils \ref{K2} of Fact \ref{fact436}.
\end{lem}

\begin{proof}
Let 
\begin{itemize}
    \item $g$ be $\mathbb{P}$-generic over $W$,
    \item $\lambda$ denote $\lambda_f$, 
\end{itemize} 
and work in $W[g]$. Consider the $\lambda$-certificate
\begin{equation*}
    \mathfrak{D}(g) = \langle \langle \bar{M}_i, \pi_{ij},  \bar{N}_i = (N_i; \in, \Vec{X}), \sigma_{ij} : i \leq j \leq \omega_1^{W} \rangle, S, \langle e_i : i < \omega_1^W \rangle, \mathfrak{C} \rangle \text{,}
\end{equation*}
so that 
\begin{equation*}
    \mathfrak{I} := \langle (N_i; \in, I_a, a_i), \sigma_{ij} : i \leq j \leq \omega_1^{W} \rangle
\end{equation*}
is a generic iteration. By \ref{cd4}, \ref{cd6} and \ref{cd7} of Definition \ref{def437}, 
\begin{equation*}
    \mathcal{P}(\omega_1^W) \cap H(\omega_2)^W \setminus \mathrm{NS}_{\omega_1}^W \subset \mathcal{P}(\omega_1^W) \cap N_{\omega_1^W} \setminus I_{\omega_1^W} \text{.}
\end{equation*}
Then Lemma \ref{lem453} tells us every member of $\mathcal{P}(\omega_1^W) \cap H(\omega_2)^W \setminus \mathrm{NS}_{\omega_1}^W$ is stationary in $W[g]$. But this is equivalent to $\mathbb{P}$ being stationary-preserving.
\end{proof}

\begin{lem}\label{lem455}
$\mathbb{P}$ fulfils \ref{K3} of Fact \ref{fact436}.
\end{lem}

\begin{proof}
Let 
\begin{itemize}
    \item $g$ be $\mathbb{P}$-generic over $W$,
    \item $\lambda$ denote $\lambda_f$, 
\end{itemize} 
and work in $W[g]$. Consider the $\lambda$-certificate
\begin{equation*}
    \mathfrak{D}(g) = \langle \langle \bar{M}_i, \pi_{ij},  \bar{N}_i = (N_i; \in, \Vec{X}), \sigma_{ij} : i \leq j \leq \omega_1^{W} \rangle, S, \langle e_i : i < \omega_1^W \rangle, \mathfrak{C} \rangle \text{,}
\end{equation*}
so that 
\begin{equation*}
    \mathfrak{I} := \langle (N_i; \in, I_a, a_i), \sigma_{ij} : i \leq j \leq \omega_1^{W} \rangle
\end{equation*}
is a generic iteration. We check that $\mathfrak{I}$ fulfils \ref{k31} to \ref{k33} of \ref{K3}.
\begin{enumerate}[label=That \ref{k3\arabic*} holds:, leftmargin=100pt]
    \item by \ref{4344}, \ref{4346b}, and \ref{cd0}, \ref{cd1.5} to \ref{cd3} of Definition \ref{def437}.
    \item by Lemmas \ref{gisubset} and \ref{lem453}.
    \item by \ref{cd4}, \ref{cd6} and \ref{cd7} of Definition \ref{def437}. \qedhere
\end{enumerate}
\end{proof}

\begin{lem}\label{lem456}
$\mathbb{P}$ fulfils \ref{K4} and \ref{K5} of Fact \ref{fact436}.
\end{lem}

\begin{proof}
Proceed as in the proofs of Lemmas \ref{1done} and \ref{2done}, noting \ref{443'3} of Remark \ref{rem443}.
\end{proof}

In view of Fact \ref{fact436}, the theorem follows from Remark \ref{rem447} and Lemmas \ref{lem454}, \ref{lem455} and \ref{lem456}.
\end{proof}

One may view Theorem \ref{notion2} as a souped-up version of Theorem \ref{notion1}. By doing so, the next corollary naturally becomes a souped-up analogue of Corollary \ref{nambacoro}. 

\begin{cor}\label{nbcoro2}
Assume
\begin{enumerate}[label=(\roman*), leftmargin=40pt]
    \item $\Gamma  = \bigcup_{1 \leq k < \omega} \mathcal{P}(\mathbb{R}^{k}) \cap L(\Gamma, \mathbb{R})$,
    \item $\Gamma$ is productive,
    \item $\mathrm{NS}_{\omega_1}$ is saturated,
    \item $2^{\omega_1} = \mathbf{\delta^1_2} = \omega_2$, and
    \item $MA(\omega_1)$ holds.
\end{enumerate}
Let 
\begin{itemize}
    \item $D \in L(\Gamma, \mathbb{R})$ be a dense subset of $\mathbb{P}_{max}$,
    \item $A \subset \omega_1$ be such that $\omega_1^{L[A]} = \omega_1$, and
    \item $\alpha > \omega_2$ be an ordinal.
\end{itemize}
Then there is a stationary-preserving forcing notion $\mathbb{P}$ such that in $V^{\mathbb{P}}$,
\begin{itemize}
    \item there is a generic iteration $$\langle \bar{N}_i = (N_i; \in, I_i, a_i), \sigma_{ij} : i \leq j \leq \omega_1^V \rangle$$ satisfying 
    \begin{enumerate}[label=(\arabic*), leftmargin=40pt]
        \item $\bar{N}_0 \in D^* := (F^*) " ((F^*)^{-1}(D)^*) \subset \mathbb{P}_{max}$,
        \item $I_{\omega_1^V} = \mathrm{NS}_{\omega_1}^{V^{\mathbb{P}}} \cap N_{\omega_1^V}$, and
        \item $a_{\omega_1^V} = A$, 
    \end{enumerate}
    \item $Nb_0(\alpha)$ holds, and
    \item $Nb'_1(\alpha)$ --- thus also $Nb_1(\alpha)$ --- holds if $\alpha$ is a regular cardinal.
\end{itemize}
\end{cor}

\subsection{Open Questions}

Working in the universe $W$ as defined in the proof of Theorem \ref{notion1} and thinning $C$ if necessary, fix any $\theta$ such that $[Q_{\lambda}]^{< \theta} \subset Q_{\lambda}$ for all $\lambda \in C$. Should we then alter the definition of $\Gamma_{\lambda}$ and $P_{\lambda}$ for each $\lambda \in C \cup \{\lambda_f\}$ as follows:
\begin{align*}
    \Gamma_{\lambda}(\theta) := \ & \text{the set } \Gamma_{\lambda} \text{ defined according to Definition \ref{defc},} \\
    & \text{but with } \mathbb{P}_{\lambda'}(\theta) \text{ in place of the parameter } \mathbb{P}_{\lambda'} \text{ for each } \lambda' \in \lambda \cap C \text{,} \\
    P_{\lambda}(\theta) := \ & \{p \in [\mathcal{L} \cap Q_{\lambda}]^{< \theta} : \\
    & \ \Vdash_{Col(\omega, |H(\kappa)|)} \exists \Sigma \ (``\Sigma \ \Gamma_{\lambda}(\theta) (\mathcal{L} \cap Q_{\lambda}, \mathfrak{A}) \text{-certifies } p")\} \text{, and} \\
    \mathbb{P}_{\lambda}(\theta) := \ & (P_{\lambda}(\theta), \supset) \text{,}
\end{align*}
would all subsequent lemmas in the proof still go through? What notable forcing-theoretic properties can we use to differentiate among the $\mathbb{P}_{\lambda_f}(\theta)$'s that result from varying $\theta$? 

Let $\theta$ be a cardinal greater than $\omega$. Then the most obvious gap in the proof shows up in the semantic interpretation of members of $\Gamma_{\lambda}(\theta)$. More specifically, since there are many dense subsets $D$ of many of the $\mathbb{P}_{\lambda'}$'s for which
\begin{itemize}
    \item $D$ is definable without parameters, and
    \item $D$ contains only infinite sets,
\end{itemize}
\ref{c8} no longer means
\begin{equation}
    \label{eq47}
    \parbox{\dimexpr\linewidth-6em}{
        ``if $X_{\delta, \lambda'}$ is non-empty, then for every dense subset $D$ of $\mathbb{P}_{\lambda'}$ definable over $(Q_{\lambda'}; \in, \mathbb{P}_{\lambda'}, A_{\lambda'})$ with parameters from $X_{\delta, \lambda'}$, $$[\Sigma]^{< \omega} \cap X_{\delta, \lambda'} \cap D \neq \emptyset \text{'',}$$ 
    }
\end{equation}
where 
\begin{itemize}
    \item $\Sigma$ interprets the unary relation symbol $\ulcorner E \urcorner$ occurring in $\mathcal{L}^*_{\mathfrak{A}}$ (formulas), and
    \item $X_{\delta, \lambda'} := X_{\delta, \lambda'}(\Sigma)$.
\end{itemize}
In fact, it now seems impossible to translate (\ref{eq47}) into a set of $\mathcal{L}^*_{\mathfrak{A}}$-$\Pi_2$ sentences, and there is thus an inability to guarantee that the $\mathbb{P}_{\lambda_f}(\theta)$-generic filter over $V$ is sufficiently generic over each non-empty $X_{\delta, \lambda'}$. This throws a wrench into the side condition method so crucial to the main results of this section.  

But is there a way to salvage things to certain reasonable extent, without a complete overhaul of the forcing construction? To be more concrete, we can ask the following question.

\begin{ques}\label{oq1}
Let $W$ and $\lambda_f$ be as in the proof of Theorem \ref{notion1}. In $W$, can $\theta$, $C$ and $$(C \longrightarrow \mathcal{P}(H(\lambda_f))) [\lambda \mapsto Q_{\lambda}]$$ be chosen such that 
\begin{enumerate}[label=(\arabic*)]
    \item $\theta > \omega$ is a cardinal,
    \item $[Q_{\lambda}]^{< \theta} \subset Q_{\lambda}$ for all $\lambda \in C$,
    \item \ref{4101} and \ref{4102} of Lemma \ref{lem26} are fulfilled with $\mathbb{P}_{\lambda_f}(\theta)$ in place of $\mathbb{P}$, and
    \item $\mathbb{P}_{\lambda_f}(\theta)$ does not collapse $\omega_1^W$?
\end{enumerate}
\end{ques}

In the spirit of Jensen's results on the extended Namba problem, we are interested in the possibility of not having the forcing notion $\mathbb{P} = \mathbb{P}_{\lambda_f}$ constructed in the proof of Theorem \ref{notion1} add reals, under the right assumptions. More generally, we can ask the same about the parametrised versions of said forcing notion. 

\begin{ques}\label{addsreals}
Let $V$, $W$ and $\lambda_f$ be as in the proof of Theorem \ref{notion1}. Additionally, assume $GCH$ holds in $V$ (and so also in $W$). Working in $W$, can $\theta$, $C$ and $$(C \longrightarrow \mathcal{P}(H(\lambda_f))) [\lambda \mapsto Q_{\lambda}]$$ be chosen such that
\begin{enumerate}[label=(\arabic*)]
    \item $\theta$ is an infinite cardinal,
    \item $[Q_{\lambda}]^{< \theta} \subset Q_{\lambda}$ for all $\lambda \in C$,
    \item  \ref{4101} and \ref{4102} of Lemma \ref{lem26} are fulfilled with $\mathbb{P}_{\lambda_f}(\theta)$ in place of $\mathbb{P}$, and
    \item $\mathbb{P}_{\lambda_f}(\theta)$ does not add reals?
\end{enumerate}
\end{ques}

Comparing Theorem \ref{notion1} with Theorem \ref{notion2} makes clear the existence of close relatives of Questions \ref{oq1} and \ref{addsreals}, given the following definitions in the universe $W$, where $W$ is as defined in the proof of Theorem \ref{notion2}:
\begin{align*}
    \Gamma'_{\lambda}(\theta) := \ & \text{the set } \Gamma'_{\lambda} \text{ defined according to Definition \ref{def447},} \\
    & \text{but with } \mathbb{P}_{\lambda'}(\theta) \text{ in place of the parameter } \mathbb{P}_{\lambda'} \text{ for each } \lambda' \in \lambda \cap C \text{,} \\
    P'_{\lambda}(\theta) := \ & \{p \in [\mathcal{L} \cap Q_{\lambda}]^{< \theta} : \\
    & \ \Vdash_{Col(\omega, |H(\kappa)|)} \exists \Sigma \ (``\Sigma \ \Gamma'_{\lambda}(\theta) (\mathcal{L} \cap Q_{\lambda}, \mathfrak{A}) \text{-certifies } p")\} \text{, and} \\
    \mathbb{P}'_{\lambda}(\theta) := \ & (P_{\lambda}(\theta), \supset) \text{.}
\end{align*}

\begin{ques}\label{oq2}
Let $W$ and $\lambda_f$ be as in the proof of Theorem \ref{notion2}. In $W$, can $\theta$, $C$ and $$(C \longrightarrow \mathcal{P}(H(\lambda_f))) [\lambda \mapsto Q_{\lambda}]$$ be chosen such that 
\begin{enumerate}[label=(\arabic*)]
    \item $\theta > \omega$ is a cardinal,
    \item $[Q_{\lambda}]^{< \theta} \subset Q_{\lambda}$ for all $\lambda \in C$,
    \item \ref{K1} to \ref{K5} of Fact \ref{fact436} are fulfilled with $\mathbb{P}'_{\lambda_f}(\theta)$ in place of $\mathbb{P}$, and
    \item $\mathbb{P}'_{\lambda_f}(\theta)$ does not collapse $\omega_1^W$?
\end{enumerate}
\end{ques}

\begin{ques}\label{addsreals2}
Let $V$, $W$ and $\lambda_f$ be as in the proof of Theorem \ref{notion2}. Additionally, assume $GCH$ holds in $V$ (and so also in $W$). Working in  $W$, can $\theta$, $C$ and $$(C \longrightarrow \mathcal{P}(H(\lambda_f))) [\lambda \mapsto Q_{\lambda}]$$ be chosen such that
\begin{enumerate}[label=(\arabic*)]
    \item $\theta$ is an infinite cardinal,
    \item $[Q_{\lambda}]^{< \theta} \subset Q_{\lambda}$ for all $\lambda \in C$,
    \item \ref{K1} to \ref{K5} of Fact \ref{fact436} are fulfilled with $\mathbb{P}'_{\lambda_f}(\theta)$ in place of $\mathbb{P}$, and
    \item $\mathbb{P}'_{\lambda_f}(\theta)$ does not add reals?
\end{enumerate}
\end{ques}

In the likely event that the answer to Question \ref{addsreals} is in the negative, it makes sense to consider a more general question.

\begin{enumerate}[label=(Q\arabic*)]
    \item Must $Nb_2(\lambda)$ hold for all $\lambda$ above $\omega_2$, where 
    \begin{align*}
        Nb_2(\lambda) := \text{`} & \text{there is a stationary-preserving forcing notion } \mathbb{P} \text{ such that} \\ 
        & \mathbb{P} \text{ does not add reals,} \\
        & \Vdash_{\mathbb{P}} ``cof(\alpha) = \omega" \text{ for all regular cardinals } \alpha \text{ satisfying } \omega_2 \leq \alpha < \lambda \text{,} \\
        & \Vdash_{\mathbb{P}} ``cof(\lambda) = \omega_1" \text{, and} \\
        & \Vdash_{\mathbb{P}} ``cof(\beta) > \omega" \text{ for all regular cardinals } \beta \text{ satisfying } \lambda < \beta \text{'?}
\end{align*}
\end{enumerate}

As a consequence of Jensen's work, we need only consider the case of $\lambda$ being a weakly inaccessible cardinal without further qualification.

\begin{ques}\label{q428}
Must $Nb_2(\lambda)$ hold for a weakly inaccessible $\lambda$, if $\lambda$ is not strongly inaccessible?
\end{ques}

In another direction, we can ask about the possibility of eschewing the assumptions in Corollary \ref{nambacoro}.

\begin{enumerate}[label=(Q\arabic*)]
    \setcounter{enumi}{1}
    \item\label{newq2} Is it true that $\mathsf{ZFC} \vdash ``Nb_1(\lambda)$ holds for all $\lambda$ above $\omega_2$''?
\end{enumerate}

Very recent results by De Bondt and Veli\v{c}kovi\'{c} (as a part of De Bondt's PhD dissertation, \cite{bendb}) resolved \ref{newq2} in the affirmative. In fact, the class of forcing notions constructed by De Bondt and Veli\v{c}kovi\'{c} assuming only $\mathsf{ZFC}$, bears witness to a theorem significantly stronger than 
\begin{quote}
    ``$Nb_1(\lambda)$ holds for all $\lambda$ above $\omega_2$''.
\end{quote}
However, these forcing notions always add reals, so they cannot be used to answer Question \ref{q428}. Additionally, because the De Bondt-Veli\v{c}kovi\'{c} forcing constructions appear vastly different from the Asper\'{o}-Schindler one, there is no obvious way to integrate one kind with the other towards proving Theorem \ref{notion2}. In fact, it is clear that iterating forcing notions born from these two kinds of constructions would not work, since an Asper\'{o}-Schindler forcing notion must force 
\begin{itemize}
    \item $\omega_2^V$ to have cofinality $\omega$, and
    \item $\omega_3^V$ to have cofinality $\omega_1$.
\end{itemize}

\section{Theories with Constraints in Interpretation (TCIs) and their Models}\label{TCIsec}

In the previous section, we discussed a method of forcing the cofinality of regular cardinals within an interval to be $\omega$. The idea of changing cofinalities via forcing involves extracting a cofinal function from an existing relation $R$ on a structured set $A$. We can couple the structure on $A$ with $R$ to form a new structure $\mathfrak{A}$ that sets the context of the problem. Then a subset of $R$ being a cofinal function becomes a definable property over $\mathfrak{A}$. Compare and contrast this with the notion of a first-order theory, which defines a property over nothing more than a vocabulary; it makes sense that the addition of a structure interpreting said vocabulary would allow us to define more intricate properties. 

Essentially, a structure can be used to provide additional constraints to a first-order theory, and forcing-related questions can often be framed as consistency questions that ask about the existence of models of first-order theories satisfying such constraints. This section is dedicated to studying the aforementioned models, with a focus on their relationship with genericity.

\subsection{Definitions and Basic Properties}\label{subs51}

We first make precise the notion of first-order theories with added constraints, so that we can compare these mathematical objects with the first-order theories we are familiar with. 

\begin{defi}
A \emph{first-order theory with constraints in interpretation} (\emph{first-order TCI}) --- henceforth, just \emph{theory with constraints in interpretation} (\emph{TCI}) --- is a tuple $(T, \sigma, \dot{\mathcal{U}}, \vartheta)$, where
\begin{itemize}
    \item $T$ is a first order theory with signature $\sigma$,
    \item $\dot{\mathcal{U}}$ is a unary relation symbol not in $\sigma$,
    \item $\vartheta$ is a function (the \emph{constraint function}) with domain $\sigma \cup \{\dot{\mathcal{U}}\}$, 
    \item if $x \in ran(\vartheta)$, then there is $y$ such that 
    \begin{itemize}[label=$\circ$]
        \item either $x = (y, 0)$ or $x = (y, 1)$, and
        \item if $\vartheta(\dot{\mathcal{U}}) = (z, a)$, then $y \subset z^n$ for some $n < \omega$, and
    \end{itemize}
    \item if $\vartheta(\dot{\mathcal{U}}) = (z, a)$, then 
    \begin{itemize}[label=$\circ$]
        \item $z \cap z^n = \emptyset$ whenever $1 < n < \omega$, and
        \item $z^m \cap z^n = \emptyset$ whenever $1 < m < n < \omega$.
    \end{itemize}
\end{itemize}
\end{defi}

\begin{defi}
Let $(T, \sigma, \dot{\mathcal{U}}, \vartheta)$ be a TCI. We say $\mathcal{M} := (U; \mathcal{I}) \models^* (T, \sigma, \dot{\mathcal{U}}, \vartheta)$ --- or $\mathcal{M}$ \emph{models} $(T, \sigma, \dot{\mathcal{U}}, \vartheta)$ --- iff all of the following holds:
\begin{itemize}
    \item $\mathcal{M}$ is a structure,
    \item $\sigma$ is the signature of $\mathcal{M}$,
    \item $\mathcal{M} \models T$,
    \item if $\vartheta(\dot{\mathcal{U}}) = (y, 0)$, then $U \subset y$,
    \item if $\vartheta(\dot{\mathcal{U}}) = (y, 1)$, then $U = y$, and
    \item for $\dot{X} \in \sigma$,
    \begin{itemize}[label=$\circ$]
        \item if $\dot{X}$ is a constant symbol and $\vartheta(\dot{X}) = (y, z)$, then $\mathcal{I}(\dot{X}) \in y \cap U$,
        \item if $\dot{X}$ is a $n$-ary relation symbol and $\vartheta(\dot{X}) = (y, 0)$, then $\mathcal{I}(\dot{X}) \subset y \cap U^{n}$,
        \item if $\dot{X}$ is a $n$-ary relation symbol and $\vartheta(\dot{X}) = (y, 1)$, then $\mathcal{I}(\dot{X}) = y \cap U^{n}$,
        \item if $\dot{X}$ is a $n$-ary function symbol and $\vartheta(\dot{X}) = (y, 0)$, then $$\{z \in U^{n+1} : \mathcal{I}(\dot{X})(z \! \restriction_n) = z(n)\} \subset y \cap U^{n+1}, \text{ and}$$
        \item if $\dot{X}$ is a $n$-ary function symbol and $\vartheta(\dot{X}) = (y, 1)$, then $$\{z \in U^{n+1} : \mathcal{I}(\dot{X})(z \! \restriction_n) = z(n)\} = y \cap U^{n+1}.$$
    \end{itemize}
\end{itemize}
We say $(T, \sigma, \dot{\mathcal{U}}, \vartheta)$ has a model if there exists $\mathcal{M}$ for which $\mathcal{M} \models^* (T, \sigma, \dot{\mathcal{U}}, \vartheta)$.
\end{defi}

\begin{ex}
Let $T$ be any first-order theory over the signature $\sigma$, $\dot{\mathcal{U}}$ be a unary relation symbol not in $\sigma$, and $A$ be any set. Define $\vartheta$ on $\sigma \cup \{\dot{\mathcal{U}}\}$ such that
\begin{itemize}
    \item $\vartheta(\dot{\mathcal{U}}) = (A, 1)$,
    \item $\vartheta(\dot{X}) = (A, 0)$ whenever $\dot{X} \in \sigma$ is a constant symbol,
    \item $\vartheta(\dot{X}) = (A^n, 0)$ whenever $\dot{X} \in \sigma$ is a $n$-ary relation symbol, and
    \item  $\vartheta(\dot{X}) = (A^{n+1}, 0)$ whenever $\dot{X} \in \sigma$ is a $n$-ary function symbol.
\end{itemize}
If we set $\mathfrak{T} := (T, \sigma, \dot{\mathcal{U}}, \vartheta)$, then the models of $\mathfrak{T}$ are precisely the models of $T$ with base set $A$.
\end{ex}

\begin{defi}
Let $\mathfrak{A}$ and $T$ be a structure and a first-order theory respectively, over the same signature $\sigma$. Define $$\mathrm{Sub}(\mathfrak{A}, T) := \{\mathfrak{B} : \mathfrak{B} \text{ is a substructure of } \mathfrak{A} \text{ and } \mathfrak{B} \models T\}.$$ Members of $\mathrm{Sub}(\mathfrak{A}, T)$ are called $T$-substructures of $\mathfrak{A}$.
\end{defi}

\begin{ex}\label{2ndex}
Let $T$ be any first-order theory over the signature $\sigma$, and $\mathfrak{A} = (A; \mathcal{I})$ be a structure interpreting $\sigma$. Define $\vartheta$ on $\sigma \cup \{\dot{\mathcal{U}}\}$ such that
\begin{itemize}
    \item $\vartheta(\dot{\mathcal{U}}) = (A, 0)$,
    \item $\vartheta(\dot{X}) = (\{\mathcal{I}(\dot{X})\}, 1)$ whenever $\dot{X} \in \sigma$ is a constant symbol,
    \item $\vartheta(\dot{X}) = (\mathcal{I}(\dot{X}), 1)$ whenever $\dot{X} \in \sigma$ is a $n$-ary relation symbol, and
    \item  $\vartheta(\dot{X}) = (R_{\mathcal{I}}(\dot{X}), 1)$ whenever $\dot{X} \in \sigma$ is a $n$-ary function symbol, where $$R_{\mathcal{I}}(\dot{X}) := \{z \in A^{n+1} : \mathcal{I}(\dot{X})(z \! \restriction_n) = z(n)\}.$$
\end{itemize}
If we set $\mathfrak{T} := (T, \sigma, \dot{\mathcal{U}}, \vartheta)$, then $$\{ \text{models of } \mathfrak{T}\} = \mathrm{Sub}(\mathfrak{A}, T).$$
\end{ex}

In practice, we can view forcing as a technique to refine structures that provably exist in $V$. Often, such refinements cannot be carried out in $V$, for any successful attempt would result in objects that cannot exist in $V$. In each of these cases, forcing can be used to extend $V$ to include an instance of the refined structure. The way we define TCIs allows them to specify --- and act as blueprints for --- refinements of this ilk. If $\mathfrak{T}$ is a TCI specifying a particular refinement, then models of $\mathfrak{T}$ correspond to the results of said refinement. We hope the next example can help illustrate our aforementioned idea of specification.

\begin{ex}\label{3rdex}
Let $\dot{\mathcal{U}}$ be a unary relation and $\dot{R}$ be a binary relation. Define $\vartheta$ on $\{\dot{\mathcal{U}}, \dot{R}\}$ such that
\begin{itemize}
    \item $\vartheta(\dot{\mathcal{U}}) = (\omega_1, 1)$, and
    \item $\vartheta(\dot{R}) = (\omega_1 \times \omega, 0)$.
\end{itemize}
Set $T$ to contain exactly the sentences 
\begin{gather*}
    \ulcorner \forall x \ \exists y \ (\dot{R}(x, y)) \urcorner \text{,} \\
    \ulcorner \forall x \ \forall y \ \forall z \ ((\dot{R}(x, y) \wedge \dot{R}(x, z)) \implies y = z) \urcorner \text{ and} \\
    \ulcorner \forall x \ \forall y \ \forall z \ ((\dot{R}(x, z) \wedge \dot{R}(y, z)) \implies x = y) \urcorner \text{.}
\end{gather*}
Now $\mathfrak{T} := (T, \{\dot{R}\}, \dot{\mathcal{U}}, \vartheta)$ is a TCI that specifies a refinement of the structure $$\mathfrak{A} := (\omega_1; \{(\dot{R}, \omega_1 \times \omega)\})$$ to some $$\mathfrak{A}' := (\omega_1; \{(\dot{R}, F)\}),$$ where $F \subset \omega_1 \times \omega$ is an injection from $\omega_1$ into $\omega$. As an implication, $\mathfrak{T}$ must not have any model in $V$. However, a weak outer model of $V$ in which $\omega_1^V$ is collapsed to $\omega$ necessarily contains models of $\mathfrak{T}$.
\end{ex}

Example \ref{3rdex} reminds us that the existence of models for a TCI is not absolute between $V$ and its (weak) outer models. There is thus a fundamental difference between the model existence of a TCI and that of a first-order theory. This should reflect in our definition of what it means for a TCI to be consistent.

\begin{defi}
A TCI $(T, \sigma, \dot{\mathcal{U}}, \vartheta)$ is \emph{consistent} iff $(T, \sigma, \dot{\mathcal{U}}, \vartheta)$ has a model in some outer model of $V$. 
\end{defi}

\begin{rem}\label{promise}
It might seem at first glance, that the the consistency of a TCI is not a first-order property in the language of set theory, since it involves quantifying over outer models of $V$. This is not a real problem, as we shall see in the next subsection, because said definition is equivalent to a first-order property at the metalevel.
\end{rem}

\begin{defi}
A TCI $(T, \sigma, \dot{\mathcal{U}}, \vartheta)$ is \emph{finitely consistent} iff for all finite $T' \subset T$, $(T', \sigma, \dot{\mathcal{U}}, \vartheta)$ is consistent. 
\end{defi}

\begin{defi}
A TCI $(T, \sigma, \dot{\mathcal{U}}, \vartheta)$ is $\Pi_n$ iff $T$ contains only $\Pi_n$ sentences.

A TCI $(T, \sigma, \dot{\mathcal{U}}, \vartheta)$ is $\Sigma_n$ iff $T$ contains only $\Sigma_n$ sentences.

A TCI $(T, \sigma, \dot{\mathcal{U}}, \vartheta)$ is $\Sigma_n \cup \Pi_n$ iff every sentence in $T$ is either $\Sigma_n$ or $\Pi_n$.
\end{defi}

TCIs allow natural constraints that are not first-order definable to be imposed on the models of a theory. However, they are not a ``true'' generalisation of first-order theories because their models have uppers bounds in size. In fact, we can show that the size limitation of models of TCIs is in some sense, the only shortcoming of TCIs \emph{vis-a-vis} first-order theories. 

\begin{lem}\label{lem511}
Let $T$ be a first-order theory over the signature $\sigma$, and $\dot{\mathcal{U}}$ be a unary relation symbol not in $\sigma$. For every cardinal $\kappa$, there is a TCI $\mathfrak{T}$ such that 
\begin{itemize}
    \item $\mathfrak{T} = (T, \sigma, \dot{\mathcal{U}}, \vartheta)$ for some $\vartheta$, and
    \item every model $\mathfrak{A} = (A; \mathcal{I})$ of $T$ with $|A| \leq \kappa$ is isomorphic to some model of $\mathfrak{T}$.
\end{itemize}
\end{lem}

\begin{proof}
Define 
\begin{itemize}
    \item $\vartheta (\dot{\mathcal{U}}) := (\kappa, 0)$,
    \item $\vartheta (\dot{X}) := (\kappa, 0)$ if $\dot{X}$ is a constant symbol,
    \item $\vartheta (\dot{X}) := (\kappa^n, 0)$ if $\dot{X}$ is a $n$-ary relation symbol,
    \item $\vartheta (\dot{X}) := (\kappa^{n + 1}, 0)$ if $\dot{X}$ is a $n$-ary function symbol.
\end{itemize}
Then $$\mathfrak{T} := (T, \sigma, \dot{\mathcal{U}}, \vartheta)$$ is as required.
\end{proof}

It turns out that there is an analogue of the downward Lowenheim-Skolem theorem for TCIs.

\begin{lem}\label{DLS}
Let 
\begin{itemize}
    \item $\mathfrak{T} = (T, \sigma, \dot{\mathcal{U}}, \vartheta)$ be a TCI,
    \item $y$ be a set, and
    \item $\alpha$ be an infinite ordinal. 
\end{itemize}
Assume that $\vartheta(\dot{\mathcal{U}}) = (y, 0)$ and in some weak outer model $W$ of $V$, there is a pair $(\mathcal{M}, f)$ such that
\begin{itemize}
    \item $\mathcal{M} = (U; \mathcal{I}) \models^* \mathfrak{T}$, and
    \item $f : \alpha \longrightarrow U$ is a bijection.
\end{itemize}
Then for every $\beta$ with $\omega \leq \beta \leq \alpha$, there is a pair $(\mathcal{M}', f')$ in $W$ such that 
\begin{itemize}
    \item $\mathcal{M}' = (U'; \mathcal{I}') \models^* \mathfrak{T}$, and
    \item $f' : \beta \longrightarrow U'$ is a bijection.
\end{itemize}
\end{lem}

\begin{proof}
Let $\beta$ be such that $\omega \leq \beta \leq \alpha$. By the downward Lowenheim-Skolem theorem applied to $\mathcal{M}$ in $W$, there exists a structure $\mathcal{M}' := (U'; \mathcal{I}')$ for which $\mathcal{M}' \prec \mathcal{M}$ and $|U'| = |\beta|$. This means $\mathcal{M'} \models T$. Further, $U' \subset U \subset y$ and whenever $\dot{X} \in \sigma$ is a constant symbol, $\mathcal{I}(\dot{X}) = \mathcal{I}'(\dot{X}) \in U'$. The other criteria for $\mathcal{M}' \models^* \mathfrak{T}$ are easy to check. Fix $f'$ to be any bijection from $\beta$ into $U'$, and we are done.
\end{proof}

If we allow movement among outer models of $V$, we get the following (somewhat trivial) version of the general Lowenheim-Skolem theorem for TCIs.

\begin{lem}\label{GLS}
Let $\mathfrak{T} = (T, \sigma, \dot{\mathcal{U}}, \vartheta)$ be a TCI with an infinite model in some outer model of $V$. Then for every infinite ordinal $\beta$, there is a pair $(\mathcal{M}, f)$ in some outer model of $V$ such that 
\begin{itemize}
    \item $\mathcal{M} = (U; \mathcal{I}) \models^* \mathfrak{T}$, and
    \item $f : \beta \longrightarrow U$ is a bijection.
\end{itemize}
\end{lem}

\begin{proof}
By our assumptions on $\mathfrak{T}$, it has a model $\mathcal{M} = (U, \mathcal{I})$ in some outer model $W$ of $V$, such that $U$ is infinite. Let $g$ be $Col(\omega, |U \cup \beta|^W)$-generic over $W$. In $W[g]$, $\mathcal{M}$ is still a model of $\mathfrak{T}$; moreover, both $U$ and $\beta$ are countably infinite, so there is a bijection $f$ from $\beta$ into $U$. Obviously, $W[g]$ is an outer model of $V$, so $(\mathcal{M}, f)$ is as required.
\end{proof}

On the other hand, we have no good analogue of the compactness theorem for TCIs. Indeed, there are simple examples in which compactness fails. We give one such example below.

\begin{lem}\label{countercom}
There is a $\Sigma_1 \cup \Pi_1$ TCI $\mathfrak{T} := (T, \{\dot{R}\}, \dot{\mathcal{U}}, \vartheta)$ with a countable transitive closure, such that
\begin{itemize}
    \item $\dot{R}$ is binary relation symbol, 
    \item if $x \in ran(\vartheta)$, then $x = (y, 0)$ for some set $y$, and
    \item $\mathfrak{T}$ is finitely consistent but not consistent.
\end{itemize}
\end{lem}

\begin{proof}
Choose $\dot{R}$ and $\dot{\mathcal{U}}$ to be relation symbols of their appropriate arity in $H(\omega)$. We define $\vartheta$ on $\{\dot{\mathcal{U}}, \dot{R}\}$ as follows:
\begin{align*}
    \vartheta(\dot{\mathcal{U}}) := \ & (\omega, 0) \\
    S_n :=  \ & \{(k,l) : 2^n \leq k, l < 2^n + n \text{ and } k < l\} \\
    S := \ & \bigcup \{S_n : n < \omega\} \\
    \vartheta(\dot{R}) := \ & (S, 0).
\end{align*}
Here, $\vartheta$ encodes a set of disjoint finite linear orders of unbounded lengths. Quite clearly, $\mathfrak{T}$ has a countable transitive closure, as all first-order sentences over the signature $\{\dot{R}\}$ are members of $H(\omega)$. Next, we want $T$ to contain the first-order definition of a strict linear ordering, namely the conjunction of the three sentences (properties):
\begin{align*}
    \varphi_1 \text{ (irreflexivity)} : \ & \ulcorner \forall x \ (\neg \dot{R}(x, x)) \urcorner \\
    \varphi_2 \text{ (transitivity)} : \ & \ulcorner \forall x \ \forall y \ \forall z \ ((\dot{R}(x, y) \wedge \dot{R}(y, z)) \implies \dot{R}(x, z)) \urcorner \\
    \varphi_3 \text{ (trichotomy)} : \ & \ulcorner \forall x \ \forall y \ (\dot{R}(x, y) \vee \dot{R}(y, x) \vee y = x) \urcorner.
\end{align*}
Finally, we define $$T' := \{\ulcorner \exists x_1 \ \exists x_2 \ \dots \ \exists x_n \ (\bigwedge_{1 \leq k < n} \dot{R}(x_{k}, x_{k + 1})) \urcorner : 1 < n < \omega\}$$ and let $$T := T' \cup \{\varphi_1, \varphi_2, \varphi_3\}.$$ Note that any finite subset of $T$ can be satisfied by a sufficiently long finite linear order, examples of which $\vartheta$ provides in abundance. However, a model of $T$ is necessarily an infinite linear order, and our definition of $\vartheta$ precisely prohibits all infinite linear orders. We thus have $\mathfrak{T}$ being finitely consistent but not consistent.
\end{proof}

Fix any infinite set $X$. By the Lowenheim-Skolem theorem for first-order logic, the compactness theorem for first-order logic holds even if we require the base set of the models in question to be subsets of $X$. As a result, the failure of compactness in a TCI of the form specified by Lemma \ref{countercom} must come from restrictions imposed by $\vartheta$. In a sense, then, Lemma \ref{countercom} gives one of the simplest examples of such a $\vartheta$, considering it has a singleton as its domain.

\begin{defi}
Given a TCI $\mathfrak{T}$ and any $\mathcal{M}$, we say $\mathcal{M}$ is a \emph{finitely determined model of} $\mathfrak{T}$ iff $\mathcal{M} \models^* \mathfrak{T}$ and for some quantifier-free sentence $\varphi$ in the language associated with $\mathcal{M}$, 
\begin{align*}
    \forall W \ \forall \mathcal{M}' \ (&(W \text{ is an outer model of } V \text{ and } \mathcal{M}' \in W \text{ and } \mathcal{M}' \models^* \mathfrak{T} \text{ and } \mathcal{M}' \models \varphi) \\ 
    & \implies \mathcal{M}' = \mathcal{M}).
\end{align*}
In this case, $\mathcal{M}$ is \emph{finitely determined by} $\varphi$.
\end{defi}

Naturally, all finite models of any TCI are finitely determined. As it turns out, if a TCI is consistent, then all its finitely determined models can be read off a forcing notion associated with it. We will prove this in the next subsection.

We end this subsection with a technical fact.

\begin{fact}
Let $\mathfrak{T}$ be a TCI in $V$. If $\mathcal{M} \models^* \mathfrak{T}$ in an outer model of $V$, then there is a smallest transitive model $W$ of $\mathsf{ZFC}$ such that $V \subset W$ and $\mathcal{M} \in W$. We use $V[\mathcal{M}]$ to denote this $W$.
\end{fact}

\subsection{Forcing Extensions and Models of TCIs}\label{GOCon}

In this subsection and the next one, we investigate how one could ``force'' the existence of models of TCIs, under different restrictions and in a variety of settings. As a starting point, we would like to frame the problem of constructing models of TCIs in the context of Section \ref{setupsec}, just so we can utilise Lemma \ref{main2}, among other things. 

\begin{lem}\label{mcequiv}
There is a formula $\psi_{cert}$ in two free variables, such that in any model of $\mathsf{ZFC}$, 
\begin{itemize}
    \item $\psi_{cert}(\mathfrak{T}, (\mathfrak{A}_{\mathfrak{T}}, \mathcal{L}_{\mathfrak{T}}, \Gamma_{\mathfrak{T}}))$ defines a function $$\mathfrak{T} \mapsto (\mathfrak{A}_{\mathfrak{T}}, \mathcal{L}_{\mathfrak{T}}, \Gamma_{\mathfrak{T}})$$ on the class of all TCIs, wherewith
    \begin{itemize}[label=$\circ$]
        \item $\mathfrak{A}_{\mathfrak{T}} = (H(|trcl(\mathfrak{T})|^+); \in)$,
        \item $\mathcal{L}_{\mathfrak{T}}$ is a set closed under negation, 
        \item $\mathfrak{A}_{\mathfrak{T}}$ is $\mathcal{L}_{\mathfrak{T}}$-suitable, and
        \item $\Gamma_{\mathfrak{T}}$ a set of $({\mathcal{L}_{\mathfrak{T}}})^*_{\mathfrak{A}_{\mathfrak{T}}}$ sentences, and
    \end{itemize}
    \item whenever
    \begin{itemize}[label=$\circ$]
        \item $\mathfrak{T} = (T, \sigma, \dot{\mathcal{U}}, \vartheta)$, 
        \item $\psi_{cert}(\mathfrak{T}, (\mathfrak{A}_{\mathfrak{T}}, \mathcal{L}_{\mathfrak{T}}, \Gamma_{\mathfrak{T}}))$, and
        \item $T$ contains only $\Pi_2$ sentences, 
    \end{itemize} 
    $\Gamma_{\mathfrak{T}}$ must contain only $({\mathcal{L}_{\mathfrak{T}}})^{*}_{\mathfrak{A}_{\mathfrak{T}}}$-$\Pi_2$ sentences.
\end{itemize}
\end{lem}

\begin{proof}
Fix any TCI $\mathfrak{T} = (T, \sigma, \dot{\mathcal{U}}, \vartheta)$. We will constructively define the tuple $(\mathfrak{A}_{\mathfrak{T}}, \mathcal{L}_{\mathfrak{T}}, \Gamma_{\mathfrak{T}})$ based on $\mathfrak{T}$ alone, and in the process, check that the requirements of the lemma are satisfied.

Of course, we have to set $$\mathfrak{A}_{\mathfrak{T}} := (H(|trcl(\mathfrak{T})|^+); \in).$$ Note that $\mathfrak{A}_{\mathfrak{T}} \models \mathsf{ZFC - Powerset}$, so $\mathfrak{A}_{\mathfrak{T}}$ is a transitive model of a sufficiently strong set theory. Next, let
\begin{align*}
    \sigma' := \ & \sigma \cup \{\dot{\mathcal{U}}\} \text{, and} \\
    U := \ & \text{the unique } y \text{ for which there exists } z \text{ such that } \vartheta(\dot{\mathcal{U}}) = (y, z).
\end{align*}
For $\dot{X} \in \sigma'$, define $\mathcal{L}_{\mathfrak{T}}(\dot{X})$ as follows:
\begin{itemize}
    \item if $\dot{X}$ is a constant symbol and $\vartheta(\dot{X}) = (y, z)$, then $$\mathcal{L}_{\mathfrak{T}}(\dot{X}) := \{\ulcorner \dot{X} = x \urcorner : x \in y \cap U\},$$
    \item if $\dot{X}$ is a $n$-ary relation symbol and $\vartheta(\dot{X}) = (y, z)$, then $$\mathcal{L}_{\mathfrak{T}}(\dot{X}) := \{\ulcorner \dot{X}(x) \urcorner : x \in y \cap U^n\}, \text{ and}$$
    \item if $\dot{X}$ is a $n$-ary function symbol and $\vartheta(\dot{X}) = (y, z)$, then $$\mathcal{L}_{\mathfrak{T}}(\dot{X}) := \{\ulcorner \dot{X}(x \! \restriction_n) = x(n) \urcorner : x \in y \cap U^{n+1}\}.$$
\end{itemize}
Then 
\begin{align*}
    \mathcal{L}'_{\mathfrak{T}} := \ & \bigcup \{\mathcal{L}_{\mathfrak{T}}(\dot{X}) : \dot{X} \in \sigma'\} \text{, and} \\
    \mathcal{L}_{\mathfrak{T}} := \ & \text{the closure of }  \mathcal{L}'_{\mathfrak{T}} \text{ under negation.}
\end{align*}
Obviously, $\mathcal{L}_{\mathfrak{T}}$ is both a member and a subset of $H(|trcl(\mathfrak{T})|^+)$, so it is definable in the language associated with $H(|trcl(\mathfrak{T})|^+)$. We thus have that $\mathfrak{A}_{\mathfrak{T}}$ is $\mathcal{L}_{\mathfrak{T}}$-suitable.

Before we get to $\Gamma_{\mathfrak{T}}$, a remark (or rather, a reminder) is imperative.

\begin{rem}\label{subsafe2}
In the same vein as what was elaborated after Definition \ref{defl}, we will use functions to pass parameters of an expression in $\mathcal{L}_{\mathfrak{T}}$ via variables, whenever necessary in the construction of $({\mathcal{L}_{\mathfrak{T}}})^*_{\mathfrak{A}_{\mathfrak{T}}}$ sentences involving the symbol $\ulcorner E \urcorner$. In fact, this can be done uniformly by the universal function $\chi_{\mathfrak{T}}$: 
\begin{align*}
    (S, x_1, \dots, x_n, x_{n+1}) \mapsto 
    \begin{cases}
        \ulcorner S(x_1, \dots, x_n, x_{n+1}) \urcorner & \text{if } S \text{ is a } n+1 \text{-ary relation symbol} \\
        \ulcorner S(x_1, \dots, x_n) = x_{n+1} \urcorner & \text{if } S \text{ is a } n \text{-ary function symbol }\\
        \ulcorner S = x_{n+1} \urcorner & \text{if } S \text{ is a } \text{constant symbol and } n = 0 \\
        \emptyset & \text{otherwise},
    \end{cases}
\end{align*}
which is defined in $V$ by a $\Delta_0$ formula in the language of set theory (as per Definition \ref{def27}).

As in the case of the proof of Lemma \ref{notion1}, we will abuse notation and abbreviate the use of $\chi_{\mathfrak{T}}$ with straightforward substitutions of variables for parameters in the writing of $({\mathcal{L}_{\mathfrak{T}}})^*_{\mathfrak{A}_{\mathfrak{T}}}$ sentences. There are no intrinsic ``hidden costs'' in terms of complexity to such a presentation of $({\mathcal{L}_{\mathfrak{T}}})^*_{\mathfrak{A}_{\mathfrak{T}}}$ sentences.
\end{rem}

Now, we define $\Gamma_{\mathfrak{T}}$ as follows:
\begin{enumerate}[label=(\arabic*)]
    \item\label{gamma1} For each constant symbol $\dot{X} \in \sigma'$,
    \begin{align*}
        \ulcorner \exists x \ (E(\ulcorner \dot{\mathcal{U}}(x) \urcorner) \wedge E(\ulcorner \dot{X} = x \urcorner)) \urcorner & \in \Gamma_{\mathfrak{T}}, \\
        \ulcorner \forall x \ \forall y \ ((E(\ulcorner \dot{X} = x \urcorner) \wedge E(\ulcorner \dot{X} = y \urcorner)) \implies x = y) \urcorner & \in \Gamma_{\mathfrak{T}}.
    \end{align*}
    \item For each $n$-ary relation symbol $\dot{X} \in \sigma'$, $$\ulcorner \forall x_1 \dots \forall x_n \ (E(\ulcorner \dot{X}((x_1, \dots, x_n)) \urcorner) \implies (\bigwedge_{1 \leq k \leq n} E(\ulcorner \dot{\mathcal{U}}(x_k) \urcorner))) \urcorner \in \Gamma_{\mathfrak{T}}.$$
    \item\label{gammaform3} If $\vartheta(\dot{\mathcal{U}}) = (y, 1)$, then $$\ulcorner \forall x \ (x \in y \implies E(\ulcorner \dot{\mathcal{U}}(x) \urcorner)) \urcorner \in \Gamma_{\mathfrak{T}}.$$
    \item\label{gammaform4} For each $n$-ary relation symbol $\dot{X} \in \sigma$ such that $\vartheta(\dot{X}) = (y, 1)$, 
    \begin{align*}
        \ulcorner \forall x_1 \dots \forall x_n \ (((\bigwedge_{1 \leq k \leq n} E(\ulcorner \dot{\mathcal{U}}(x_k) \urcorner)) \wedge \ulcorner \dot{X}((x_1, \dots, x_n)) \urcorner \in \mathcal{L}_{\mathfrak{T}}) & \\
        \implies E(\ulcorner \dot{X}((x_1, \dots, x_n)) \urcorner)) \urcorner & \in \Gamma_{\mathfrak{T}}.
    \end{align*}
    \item For each $n$-ary function symbol $\dot{X} \in \sigma'$, 
    \begin{align*}
        \ulcorner \forall x_1 \dots \forall x_{n+1} \ ( & E(\ulcorner \dot{X} (x_1, \dots, x_n) = x_{n+1} \urcorner) \\
        & \mspace{90mu} \implies (\bigwedge_{1 \leq k \leq n+1} E(\ulcorner \dot{\mathcal{U}}(x_k) \urcorner))) \urcorner \in \Gamma_{\mathfrak{T}}, \\
        \ulcorner \forall x_1 \dots \forall x_n \ \exists y \ ( & (\bigwedge_{1 \leq k \leq n} E(\ulcorner \dot{\mathcal{U}}(x_k) \urcorner)) \\
        & \mspace{70mu} \implies E(\ulcorner \dot{X}(x_1, \dots, x_n) = y \urcorner)) \urcorner \in \Gamma_{\mathfrak{T}}, \\
        \ulcorner \forall x_1 \dots \forall x_n \ \forall y \ \forall z \ ( & (E(\ulcorner \dot{X} (x_1, \dots, x_n) = y \urcorner) \\ 
        & \wedge E(\ulcorner \dot{X} (x_1, \dots, x_n) = z \urcorner)) \implies y = z) \urcorner \in \Gamma_{\mathfrak{T}}.
    \end{align*}
    \item\label{gamma5} For each $n$-ary function symbol $\dot{X} \in \sigma'$ such that $\vartheta(\dot{X}) = (y, 1)$, 
    \begin{align*}
        \ulcorner \forall x_1 \dots \forall x_{n+1} \ (((\bigwedge_{1 \leq k \leq n+1} E(\ulcorner \dot{\mathcal{U}}(x_k) \urcorner)) \wedge \ulcorner \dot{X}(x_1, \dots, x_n) = x_{n+1} \urcorner \in \mathcal{L}_{\mathfrak{T}}) & \\
        \implies E(\ulcorner \dot{X}(x_1, \dots, x_n) = x_{n+1} \urcorner)) \urcorner & \in \Gamma_{\mathfrak{T}}.
    \end{align*}
    \item\label{gamma6} We finally deal with members of $T$. So let $\phi \in T$. We first assume that for every atomic subformula $\varphi$ of $\phi$, 
    \begin{align}\label{norecur}
        \varphi \text{ contains no more than one symbol from } \sigma \text{ (counting recurrences).} 
    \end{align} 
    To see why this assumption can be made without loss of generality, notice that there are canonical algorithms $M_1$ and $M_2$ such that, when given any atomic formula $\varphi'$ over $\sigma$ as input,
    \begin{itemize}
        \item $M_1$ returns a $\Sigma_1$ formula $\varphi$ which is logically equivalent to $\varphi'$ and satisfies (\ref{norecur}), and
        \item $M_2$ returns a $\Pi_1$ formula $\varphi$ which is logically equivalent to $\varphi'$ and satisfies (\ref{norecur}).
    \end{itemize}
    Consequently, by 
    \begin{enumerate}[label=(\roman*)]
        \item\label{stepip} replacing atomic subformulas of $\phi$ via $M_1$ or $M_2$ according to their parities, and then
        \item canonically converting the result of \ref{stepip} to prenex normal form,
    \end{enumerate}
    we can obtain a sentence that
    \begin{itemize}
        \item is logically equivalent to $\phi$,
        \item has each of its atomic subformulas $\varphi$ satisfy (\ref{norecur}), and
        \item is $\Pi_2$ whenever $\phi$ is $\Pi_2$.
    \end{itemize}
    
    Next, we transform $\phi$ into $\phi'$ by first inductively relativising $\phi$ to ``members of $\dot{\mathcal{U}}$'', then simultaneously translating all its atomic subformulas to correspond to membership in $\mathcal{L}'_{\mathfrak{T}}$. In more detail, we carry out the procedure below.
    \begin{enumerate}[label=(\alph*)]
        \item\label{castst} Cast $\phi$ as a string. 
        
        Given any string $A$, we can view $A$ as a sequence of (possibly non-distinct) characters. The ordering of this sequence gives rise to the notion of (relative) \emph{position}. Intuitively, the leftmost character of $A$ marks its first position (position $= 1$), and for any $k$, the character at the ($k+1$)-th position of $A$ necessarily lies to the immediate right of the character at position $k$. Therefore, the positions of $A$ must range from $1$ to the length of $A$.
        
        \item Initialise a pointer $p$ at the first position of the $\phi$. 

        The rationale of having $p$ is to help us traverse the characters of $\phi$ as we modify it. Like any pointer, $p$ occupies exactly one position at any point in time. Specifically, we want $p$ to keep moving rightwards, even though $\phi$ as we now know it might change in length over the run of this procedure.
        
        We will modify $\phi$ in steps, each step being a pass of a numbered stage in the enumeration of this procedure. For clarity of exposition, it is useful to distinguish $\phi$ pre- and post-modification. As we describe the procedure going forward, we shall let $\phi$ refer to the unaltered string: its state right after \ref{castst}. At any particular step, \emph{the current frame} denotes the modified form of $\phi$ at the beginning of said step. 
        
        \item\label{casesub} Let $x$ be the current position occupied by $p$. If there is $\varphi$ such that 
        \begin{itemize}
            \item $\varphi$ is a subformula of $\phi$, 
            \item the leading character of $\varphi$ is a first-order quantifier, and
            \item $\varphi$ is a substring of the current frame starting at position $x$,
        \end{itemize}
        then we let $\varphi'$ be the shortest such string, and proceed according to the cases below. Otherwise, skip to \ref{incpoint}.
        \begin{enumerate}[label=Case \arabic*:, leftmargin=50pt]
            \item $\varphi' = \ulcorner \forall x \ \psi \urcorner$ for some $x$ and $\psi$. Then we replace $\varphi'$ starting at $x$ of the current frame with the string $$\ulcorner \forall x \ (E(\ulcorner \dot{\mathcal{U}}(x) \urcorner) \implies \psi) \urcorner.$$
            \item $\varphi' = \ulcorner \exists x \ \psi \urcorner$ for some $x$ and $\psi$. Then we replace $\varphi'$ starting at $x$ of the current frame with the string $$\ulcorner \exists x \ (E(\ulcorner \dot{\mathcal{U}}(x) \urcorner) \wedge \psi) \urcorner.$$
        \end{enumerate}

        When we speak of replacing a substring $Y$ starting at $x$ of $F$ with another string $Z$, we mean to produce the concatenated string $A^\frown Z^\frown B$, where $A$ and $B$ are the two unique strings for which 
        \begin{itemize}
            \item either $A$ is empty or the last character of $A$ occupies position $x-1$ of $F$, and
            \item $F = A^\frown Y^\frown B$. 
        \end{itemize}
        After the replacement, the position occupied by $p$ remains unchanged --- it should still be at $x$ relative to $A^\frown Z^\frown B$. Note that replacements of this kind make no changes to $F$ at any position less than (to the left of) $x$. 
        
        \item\label{incpoint} If $p$ is not at the rightmost position of the current frame, increment the position it occupies by $1$. Otherwise skip to \ref{atomsub}.
        
        \item\label{endloop} Go to \ref{casesub}.
        
        \item\label{atomsub} Substitute each atomic subformula $\psi$ of $\phi$ occurring in the current frame with $\ulcorner E(\psi) \urcorner$, bearing in mind the abbreviations adopted in Remark \ref{subsafe2}. These substitutions can be done simultaneously because it is impossible to have two distinct substitutable instances occupy overlapping positions of the current frame.
    \end{enumerate}
    The aforementioned procedure produces a sentence $\phi' \in ({\mathcal{L}_{\mathfrak{T}}})^*_{\mathfrak{A}_{\mathfrak{T}}}$ sharing the \emph{quantification structure} of $\phi$. More precisely, this means the existence of a string $A$ such that
    \begin{itemize}
        \item $A$ contains only quantifiers,
        \item $A$ is a subsequence of both $\phi$ and $\phi'$,
        \item if $B$ is a subsequence of $\phi$ containing only quantifiers, then $B$ is a subsequence of $A$, and
        \item if $B'$ is a subsequence of $\phi'$ containing only quantifiers, then $B'$ is a subsequence of $A$.
    \end{itemize}
    Now, convert $\phi'$ to a logically equivalent formula $\phi^*$ in prenex normal form, through an application of the standard conversion algorithm. This algorithm preserves the quantification structure of $\phi'$ --- so that $\phi^*$ and $\phi$ have the same quantification structure --- whenever $\phi$ is in prenex normal form.
    
    Enforce that $\phi^* \in \Gamma_{\mathfrak{T}}$. 
    \item\label{gamma7} Nothing else is in $\Gamma_{\mathfrak{T}}$.
\end{enumerate}

Assume $T$ contains only $\Pi_2$ sentences. Then necessarily every member of $T$ is in prenex normal form. As a consequence, the transformation $$\varpi : \phi \mapsto \phi^*$$ described in \ref{gamma6} takes every member of $T$ to a $({\mathcal{L}_{\mathfrak{T}}})^{*}_{\mathfrak{A}_{\mathfrak{T}}}$ sentence in prenex normal form with the same quantification structure, making $\varpi"T$ a set of $({\mathcal{L}_{\mathfrak{T}}})^{*}_{\mathfrak{A}_{\mathfrak{T}}}$-$\Pi_2$ sentences. Clearly, all additions to $\Gamma_{\mathfrak{T}}$ as per \ref{gamma1} to \ref{gamma5} are $({\mathcal{L}_{\mathfrak{T}}})^{*}_{\mathfrak{A}_{\mathfrak{T}}}$-$\Pi_2$ sentences. By \ref{gamma7}, $\Gamma_{\mathfrak{T}}$ contains only $({\mathcal{L}_{\mathfrak{T}}})^{*}_{\mathfrak{A}_{\mathfrak{T}}}$-$\Pi_2$ sentences.
\end{proof}

Fix $\psi_{cert}$ to be as in Lemma \ref{mcequiv}. We are then justified in our next definition.

\begin{defi}
Let $\mathfrak{T} = (T, \sigma, \dot{\mathcal{U}}, \vartheta)$ be a TCI. Define
\begin{align*}
    \mathfrak{A}_{\mathfrak{T}} := \ & \text{the unique } \mathfrak{A} \text{ for which there are } \mathcal{L} \text{ and } \Gamma \text{ satisfying } \psi_{cert}(\mathfrak{T}, (\mathfrak{A}, \mathcal{L}, \Gamma)) \text{,} \\
    \mathcal{L}_{\mathfrak{T}} := \ & \text{the unique } \mathcal{L} \text{ for which there are } \mathfrak{A} \text{ and } \Gamma \text{ satisfying } \psi_{cert}(\mathfrak{T}, (\mathfrak{A}, \mathcal{L}, \Gamma)) \text{, and} \\
    \Gamma_{\mathfrak{T}} := \ & \text{the unique } \Gamma \text{ for which there are } \mathfrak{A} \text{ and } \mathcal{L} \text{ satisfying } \psi_{cert}(\mathfrak{T}, (\mathfrak{A}, \mathcal{L}, \Gamma)) \text{.}
\end{align*}
\end{defi}

\begin{defi}
Let $\mathfrak{T} = (T, \sigma, \dot{\mathcal{U}}, \vartheta)$ be a TCI. Define \begin{align*}
    P(\mathfrak{T}) := \ & \{p \in [\mathcal{L}_{\mathfrak{T}}]^{< \omega} : \ \Vdash_{Col(\omega, |trcl(\mathfrak{A}_{\mathfrak{T}})|)} \exists \Sigma \ (``\Sigma \ \Gamma_{\mathfrak{T}} (\mathcal{L}_{\mathfrak{T}}, \mathfrak{A}_{\mathfrak{T}})\text{-certifies } p")\} \text{,} \\
    \leq_{\mathbb{P}(\mathfrak{T})} := \ & \{(p, q) \in P(\mathfrak{T}) \times P(\mathfrak{T}) : q \subset p\} \text{, and} \\
    \mathbb{P}(\mathfrak{T}) := \ & (P(\mathfrak{T}), \leq_{\mathbb{P}(\mathfrak{T})}) \text{.}
\end{align*}
\end{defi}

By Lemma \ref{mcequiv}, if $\mathfrak{T}$ is a TCI and $A$ the base set of $\mathfrak{A}_{\mathfrak{T}}$, then 
\begin{itemize}
    \item $\mathbb{P}(\mathfrak{T}) \in A \cap \mathcal{P}(A)$, so $\mathbb{P}(\mathfrak{T})$ is definable in the language associated with $\mathfrak{A}_{\mathfrak{T}}$, and
    \item $(\mathfrak{A}_{\mathfrak{T}}, \mathbb{P}(\mathfrak{T}))$ is good for $\mathcal{L}_{\mathfrak{T}}$.
\end{itemize}

By Lemmas \ref{Pisabs} and \ref{mcequiv}, the definition of $\mathbb{P}(\mathfrak{T})$ from $\mathfrak{T}$ is absolute for transitive models of $\mathsf{ZFC}$.

At this juncture, it is customary for us to revisit the main forcing construction of the previous section.

\begin{rem}\label{rem520}
Consider the sequence of forcing notions $$\{\mathbb{P}_{\lambda} : \lambda \in C \cup \{\lambda_f\}\}$$ constructed within $W$ in the proof of Theorem \ref{notion1}. The adaptation of this inductive construction to the language of TCIs is straightforward: given $\lambda \in C \cup \{\lambda_f\}$ and $\Vec{\mathbb{P}}_{\lambda} := \{\mathbb{P}_{\theta} : \theta \in \lambda \cap C\}$, define $\sigma$ to contain
\begin{enumerate}[label=(\Alph*)]
    \item\label{5201} the variable names of all subsets of $H(\lambda_f)$ germane to the definition of $\mathbb{P}_{\lambda}$ in the original proof (this set includes $\Vec{\mathbb{P}}_{\lambda}$), along with
    \item\label{5202} two other ternary relation symbols, $\dot{F}$ and $\dot{X}$,
\end{enumerate}
and nothing else. If $\dot{Z} \in \sigma$ is of type \ref{5201}, then we interpret $\dot{Z}$ invariably as whichever set it is defined to be in the original proof; for example, when $\dot{Z} = \ulcorner C \urcorner$, we set $\vartheta (\dot{Z}) = (C, 1)$. On the other hand, if $\dot{Z} \in \sigma$ is of type \ref{5202}, then we interpret $\dot{Z}$ according to either \ref{l1lamb} or \ref{l2lamb} based on the identity of $\dot{Z}$; for example, when $\dot{Z} = \dot{F}$,  we set $$\vartheta (\dot{Z}) = (\{(i, n, \alpha) : i \in R \cap \lambda, n < \omega \text{ and } \alpha < i\}, 0).$$ Also set $\vartheta (\dot{\mathcal{U}}) = (H(\lambda_f), 1)$. 

Next, modify \ref{c1} to \ref{c8} such that each subformula of the form $\ulcorner E(\ulcorner \phi \urcorner) \urcorner$ is replaced by $\ulcorner \phi \urcorner$, and let $T$ contain only these formulas (noting and adjusting for the abuse of notation in the original presentation). Finally, letting $\mathfrak{T}_{\lambda} = (T, \sigma, \dot{\mathcal{U}}, \vartheta)$, it takes no more than a routine unfurling and checking of definitions to ascertain that $\mathbb{P}(\mathfrak{T}_{\lambda})$ is isomorphic to the forcing notion $\mathbb{P}_{\lambda}$ defined in the original proof.
\end{rem}

We see in Remark \ref{rem520} that our procedure of associating a partial order with each TCI can be used to generate forcing notions as complex as the ones constructed to solve a difficult problem in set theory. More formal declarations of the power of this procedure will appear --- and be proven --- in the later parts of this section. But before that, let us return to the setting of the ground.

\begin{lem}\label{mcequiv2}
There is a formula $\psi_{trans}$ in three free variables, absolute for transitive models of $\mathsf{ZFC - Powerset}$, such that $\psi_{trans}(\mathfrak{T}, \mathcal{M}, \Sigma)$ defines a bijection from $$\{\mathcal{M} : \mathcal{M} \models^* \mathfrak{T}\}$$ into $$\{\Sigma : \Sigma \ \Gamma_{\mathfrak{T}} (\mathcal{L}_{\mathfrak{T}}, \mathfrak{A}_{\mathfrak{T}})\text{-certifies } \emptyset\}$$ for every fixed TCI $\mathfrak{T} = (T, \sigma, \dot{\mathcal{U}}, \vartheta)$.
\end{lem}

\begin{proof}
Let $U$ be the unique $y$ for which there exists $z$ such that $\vartheta(\dot{\mathcal{U}}) = (y, z)$. Given a model $\mathcal{M} = (M; \mathcal{I})$ of $\mathfrak{T}$, define $$U(\mathcal{M}) := \{\ulcorner \dot{\mathcal{U}}(x) \urcorner : x \in M\} \cup \{\ulcorner \neg \ \dot{\mathcal{U}}(x) \urcorner : x \in U \setminus M\}.$$ Now define $\psi_{trans}$ as follows: $$\psi_{trans}(\mathfrak{T}, \mathcal{M}, \Sigma) \iff (\mathcal{M} \models^* \mathfrak{T} \wedge \Sigma = (U(\mathcal{M}) \cup Diag(\mathcal{M})) \cap \mathcal{L}_{\mathfrak{T}}),$$ where $Diag(\mathcal{M})$ is the atomic diagram of $\mathcal{M}$. Verily, $\psi_{trans}$ is a $\Delta_1$ formula in the language of set theory (according to Definition \ref{def27}), because the binary relation $\models^*$ is $\Delta_1$-definable and the set $\mathcal{L}_{\mathfrak{T}}$ is $\Delta_1$-definable in $\mathfrak{T}$. As such, $\psi_{trans}$ must be absolute for transitive models of $\mathsf{ZFC - Powerset}$. We can then straightforwardly check that $\psi_{trans}$ defines a bijection as required by the lemma for any fixed $\mathfrak{T}$, based on how the triple $(\mathfrak{A}_{\mathfrak{T}}, \mathcal{L}_{\mathfrak{T}}, \Gamma_{\mathfrak{T}})$ is constructed from $\mathfrak{T}$. 
\end{proof}

\begin{rem}\label{swapA}
For any TCI $\mathfrak{T}$ and any structure $\mathfrak{A}$, if $\mathfrak{A}$ is $\mathcal{L}_{\mathfrak{T}}$-suitable and $\Gamma_{\mathfrak{T}}$ is a set of $(\mathcal{L}_{\mathfrak{T}})^*_{\mathfrak{A}}$ sentences, then for all $\Sigma$ and $p$, $$\Sigma \ \Gamma_{\mathfrak{T}} (\mathcal{L}_{\mathfrak{T}}, \mathfrak{A}_{\mathfrak{T}})\text{-certifies } p \iff \Sigma \ \Gamma_{\mathfrak{T}} (\mathcal{L}_{\mathfrak{T}}, \mathfrak{A})\text{-certifies } p.$$ We can therefore replace $\mathfrak{A}_{\mathfrak{T}}$ in Lemma \ref{mcequiv2} with any appropriate $\mathfrak{A}$ and still have the lemma hold true for the same $\psi_{trans}$.
\end{rem}

Fix $\psi_{trans}$ to be as in Lemma \ref{mcequiv2} for the next definition. 

\begin{defi}
Let $\mathfrak{T} = (T, \sigma, \dot{\mathcal{U}}, \vartheta)$ be a TCI in $V$ and $\mathcal{M}$ be a model of $\mathfrak{T}$ in some outer model of $V$. Define \begin{align*}
    \Sigma(\mathfrak{T}, \mathcal{M}) := \text{the unique } \Sigma \text{ for which } \psi_{trans}(\mathfrak{T}, \mathcal{M}, \Sigma).
\end{align*}
\end{defi}

It is time to cash the cheque issued in Remark \ref{promise}.

\begin{lem}\label{conalt}
Let $\mathfrak{T} = (T, \sigma, \dot{\mathcal{U}}, \vartheta)$ be a TCI. Then $\mathfrak{T}$ is consistent iff $$\Vdash_{Col(\omega, \lambda)} \exists \Sigma \ (``\Sigma \ \Gamma_{\mathfrak{T}} (\mathcal{L}_{\mathfrak{T}}, \mathfrak{A}_{\mathfrak{T}})\text{-certifies } \emptyset")$$ iff
$$\Vdash_{Col(\omega, \lambda)} \exists \mathcal{M} \ (``\mathcal{M} \models^* \mathfrak{T}"),$$ where $\lambda \geq |H(|trcl(\mathfrak{T})|^+)|$.
\end{lem}

\begin{proof}
By Lemmas \ref{mcequiv} and \ref{mcequiv2}, we can find a triple $(\mathfrak{A}_{\mathfrak{T}}, \mathcal{L}_{\mathfrak{T}}, \Gamma_{\mathfrak{T}})$ such that 
\begin{itemize}
    \item $\mathfrak{A}_{\mathfrak{T}} = (H(|trcl(\mathfrak{T})|^+); \in)$,
    \item $\mathcal{L}_{\mathfrak{T}}$ is a set closed under negation, 
    \item $\mathfrak{A}_{\mathfrak{T}}$ is $\mathcal{L}_{\mathfrak{T}}$-suitable, 
    \item $\Gamma_{\mathfrak{T}}$ a set of $({\mathcal{L}_{\mathfrak{T}}})^*_{\mathfrak{A}_{\mathfrak{T}}}$ sentences, and
    \item $\mathfrak{T}$ is consistent iff for some outer model $W$ of $V$, $$\exists \Sigma \in W \ (``\Sigma \ \Gamma_{\mathfrak{T}} (\mathcal{L}_{\mathfrak{T}}, \mathfrak{A}_{\mathfrak{T}})\text{-certifies } \emptyset").$$
\end{itemize}
Then the conjunction of
\begin{itemize}
    \item Lemma \ref{inout},
    \item the fact that every forcing extension of $V$ is an outer model of $V$, and
    \item the fact that every outer model of $V$ is a weak outer model of $V$
\end{itemize}   
tells us that $\mathfrak{T}$ is consistent iff $$\Vdash_{Col(\omega, \lambda)} \exists \Sigma \ (``\Sigma \ \Gamma_{\mathfrak{T}} (\mathcal{L}_{\mathfrak{T}}, \mathfrak{A}_{\mathfrak{T}})\text{-certifies } \emptyset"),$$ where $\lambda \geq |H(|trcl(\mathfrak{T})|^+)|$. 

As every $Col(\omega, \lambda)$-generic extension of $V$ is a transitive model of $\mathsf{ZFC - Powerset}$, we can apply Lemma \ref{mcequiv2} again to complete the proof. 
\end{proof}

\begin{cor}\label{cor527}
There is a procedure to uniformly decide in $V$, whether any given $\Pi_2$ TCI is consistent.
\end{cor}

\begin{proof}
By Lemma \ref{conalt}, a $\Pi_2$ TCI $\mathfrak{T}$ is consistent iff $P(\mathfrak{T})$ is non-empty.
\end{proof} 

Intuitively, the consistency of a theory --- however it is defined --- should be absolute in a sufficiently strong sense. This is the case for first-order theories, any of which consistency is absolute for transitive models of set theory. The following Lemma establishes a similar absoluteness property with regards to the consistency of a TCI.

\begin{lem}
Let $\mathfrak{T} = (T, \sigma, \dot{\mathcal{U}}, \vartheta)$ be a TCI. Then $\mathfrak{T}$ being consistent is absolute for transitive models of $\mathsf{ZFC}$ sharing the same ordinals.
\end{lem}

\begin{proof}
This is very much similar to the proof of Lemma \ref{Pisabs}. Nevertheless, we shall provide details. 

Let $V'$ and $W$ be transitive models of $\mathsf{ZFC}$ with $ORD^{V'} = ORD^W$ and $\mathfrak{T} \in V' \subset W$. If $\mathfrak{T}$ is consistent in $W$, then $\mathfrak{T}$ has a model in some outer model of $W$. Said outer model is also an outer model of $V'$, so $\mathfrak{T}$ is consistent in $V'$ as well.

Now assume $\mathfrak{T}$ is consistent in $V'$. Letting $$\lambda := |H(((|trcl(\mathfrak{T})|^{V'})^+)^{V'})^{V'}|^{V'},$$ Lemma \ref{conalt} gives us $$\Vdash_{Col(\omega, \lambda)} \exists \mathcal{M} \ (``\mathcal{M} \models^* \mathfrak{T}")$$ in $V'$. Note that $$\mathbb{P} := Col(\omega, \lambda)^{V'}$$ remains a forcing notion in $W$, so consider $g$ a $\mathbb{P}$-generic filter over $W$. Necessarily, $g$ is also $\mathbb{P}$-generic over $V'$, and further, $V'[g] \subset W[g]$. In $V'[g]$, $\mathfrak{T}$ is forced to have a model --- call it $\mathcal{M}$. Being a model of $\mathfrak{T}$ is absolute for transitive models of $\mathsf{ZFC}$, so $\mathcal{M} \models^* \mathfrak{T}$ holds in $W[g]$ too. Since $W[g]$ is an outer model of $W$, $\mathfrak{T}$ must be consistent in $W$.
\end{proof}

We now define a class of generic objects that manifest as models of TCIs.

\begin{defi}
Let $\mathfrak{T}$ be a consistent $\Pi_2$ TCI. 

If $\mathfrak{A}$ and $\mathbb{P}$ are such that $\mathbb{P}$ is definable in the language associated with $\mathfrak{A}$ and $(\mathfrak{A}, \mathbb{P})$ is good for $\mathcal{L}_{\mathfrak{T}}$, then a $(\mathbb{P}, \mathfrak{A})$\emph{-generic model} of $\mathfrak{T}$ is a model $\mathcal{M}$ of $\mathfrak{T}$ satisfying $$\Sigma(\mathfrak{T}, \mathcal{M}) = (\bigcup g) \cap \mathcal{L}_{\mathfrak{T}}$$ for some $\mathbb{P}$-generic filter $g$ over $\mathfrak{A}$. In this case, we say $g$ \emph{witnesses} $\mathcal{M}$ \emph{is a} $(\mathbb{P}, \mathfrak{A})$-\emph{generic model of} $\mathfrak{T}$. We say $g$ \emph{witnesses} \emph{a} $(\mathbb{P}, \mathfrak{A})$-\emph{generic model of} $\mathfrak{T}$ iff for some $\mathcal{M}$, $g$ witnesses $\mathcal{M}$ is a $(\mathbb{P}, \mathfrak{A})$-generic model of $\mathfrak{T}$.

We call $\mathcal{M}$ a $\mathfrak{A}$\emph{-generic model} of $\mathfrak{T}$ iff for some $\mathbb{P}$ definable in the language associated with $\mathfrak{A}$ such that $(\mathfrak{A}, \mathbb{P})$ is good for $\mathcal{L}_{\mathfrak{T}}$, $\mathcal{M}$ is a $(\mathbb{P}, \mathfrak{A})$-generic model of $\mathfrak{T}$.

We call $\mathcal{M}$ a \emph{generic model} of $\mathfrak{T}$ iff for some $\mathfrak{A}$ and $\mathbb{P}$ such that $\mathbb{P}$ is definable in the language associated with $\mathfrak{A}$ and $(\mathfrak{A}, \mathbb{P})$ is good for $\mathcal{L}_{\mathfrak{T}}$, $\mathcal{M}$ is a $(\mathbb{P}, \mathfrak{A})$-generic model of $\mathfrak{T}$.   
\end{defi}

\begin{ob}\label{smallvgen}
\leavevmode
\begin{enumerate}[label=(\arabic*)]
    \item\label{4401} If $\mathfrak{T}$ is a consistent $\Pi_2$ TCI, and $\mathfrak{A}$ and $\mathbb{P}$ are such that $\mathbb{P}$ is definable in the language associated with $\mathfrak{A}$ and $(\mathfrak{A}, \mathbb{P})$ is good for $\mathcal{L}_{\mathfrak{T}}$, then every $(\mathbb{P}, \mathfrak{A})$-generic model of $\mathfrak{T}$ is a $(\mathbb{P}, \mathfrak{A})$-generic object.
    \item\label{5282} If $g$ witnesses $\mathcal{M}$ is a $(\mathbb{P}, V)$-generic model of $\mathfrak{T}$, and $\bigcup g \subset \mathcal{L}_{\mathfrak{T}}$, then $V[g] = V[\mathcal{M}]$.
    \item\label{2svg} In the same vein as Observation \ref{ob0}, we see that given any consistent $\Pi_2$ TCI $\mathfrak{T}$, 
    \begin{align*}
        \forall x \ ( & x \text{ is a } (\mathbb{P}(\mathfrak{T}), \mathfrak{A}_{\mathfrak{T}}) \text{-generic model of } \mathfrak{T} \\
        & \iff x \text{ is a } (\mathbb{P}(\mathfrak{T}), V) \text{-generic model of } \mathfrak{T})
    \end{align*}
   in every outer model of $V$. As a result, we can thus safely talk about $(\mathbb{P}(\mathfrak{T}), V)$-generic models of $\mathfrak{T}$ without the need to quantify over all sets.
\end{enumerate}
\end{ob}

Our definition of generic models might seem overly restrictive at first glance. The next lemma provides justification that it is not so.

\begin{lem}\label{gmodelsinfe}
Let $\mathfrak{T}$ be a TCI. If $\mathcal{M}$ is a model of $\mathfrak{T}$ in some forcing extension of $V$, then $\mathcal{M}$ is a $V$-generic model of $\mathfrak{T}$.
\end{lem}

\begin{proof}
Let $\mathcal{M}$ be a model of $\mathfrak{T}$ in a forcing extension $W$ of $V$. Without loss of generality, we can assume the existence of $\mathbb{P} = (P, \leq_{\mathbb{P}})$ and $\dot{\Sigma}$ such that
\begin{itemize}
    \item $\mathbb{P}$ is a forcing notion,
    \item $P \cap \mathcal{L}_{\mathfrak{T}} = \emptyset$,
    \item $\dot{\Sigma}$ is a $\mathbb{P}$-name, 
    \item $\Vdash_{\mathbb{P}} ``\exists \mathcal{M}' \ (\mathcal{M}' \models^* \mathfrak{T} \text{ and } \Sigma(\mathfrak{T}, \mathcal{M}') = \dot{\Sigma})"$, and
    \item for some $\mathbb{P}$-generic filter $g_0$ over $V$, $\Sigma(\mathfrak{T}, \mathcal{M}) = \dot{\Sigma}[g_0]$.
\end{itemize}
Define
\begin{align*}
    P^* := \{x \in [P \cup \mathcal{L}_{\mathfrak{T}}]^{< \omega} : \ & x \cap P \text{ has a common extension in } \mathbb{P} \text{ and} \\ 
    & \forall y \ \exists p \ (y \in x \cap \mathcal{L}_{\mathfrak{T}} \implies (p \in x \cap P \text{ and } p \Vdash_{\mathbb{P}} y \in \dot{\Sigma}))\}
\end{align*}
and have $\mathbb{P}^* := (P^*, \supset)$.

Fix $x \in P^*$, and let $p$ be a common extension in $\mathbb{P}$ of the members of $x \cap P$. Then any extension of $\{p\}$ in $\mathbb{P}^*$ is compatible with $x$ in $\mathbb{P}^*$. This means that $$(\pi : P \longrightarrow P^*) \ [p \mapsto \{p\}]$$ is a dense weak embedding from $\mathbb{P}$ into $\mathbb{P}^*$. As a result, $$\sigma : g \mapsto \mathrm{UC}(w(\mathbb{P}^*), \pi" g)$$ is a bijection from $$\{g : g \text{ is a } \mathbb{P} \text{-generic filter over } V\} \cap W$$ into $$\{h : h \text{ is a } \mathbb{P}^* \text{-generic filter over } V\} \cap W$$ with inverse $$\tau : h \mapsto \pi^{-1} h \text{,}$$ in every weak outer model $W$ of $V$. The following fact is easy to see.

\begin{fact}\label{cufs}
If $h$ is a $\mathbb{P}^*$-generic filter over $V$, then $[(\bigcup h) \cap P]^{< \omega} \subset h$.
\end{fact}

\begin{prop}\label{p524}
Let $g$ be a $\mathbb{P}$-generic filter over $V$. Then $(\bigcup \sigma(g)) \cap P = g$.
\end{prop}

\begin{proof}
Denote $(\bigcup \sigma(g)) \cap P$ as $g'$. By the definitions of $\pi$ and $\sigma$, $g \subset g'$ clearly. Choose $p \in g'$. Since $\sigma(g)$ is a $\mathbb{P}^*$-generic filter over $V$, Fact \ref{cufs} tells us that $\{p\} \in \sigma(g)$. As $\tau = \sigma^{-1}$, $p \in g$, and we are done.
\end{proof}

\begin{prop}\label{p525}
Let $g$ be a $\mathbb{P}$-generic filter over $V$. Then $$(\bigcup \sigma(g)) \cap \mathcal{L}_{\mathfrak{T}} = \{y \in \mathcal{L}_{\mathfrak{T}} : \exists p \ (p \in (\bigcup \sigma(g)) \cap P \text{ and } p \Vdash_{\mathbb{P}}^V y \in \dot{\Sigma})\}.$$
\end{prop}

\begin{proof}
Denote $(\bigcup \sigma(g)) \cap \mathcal{L}_{\mathfrak{T}}$ as $\Sigma$ and $$\{y \in \mathcal{L}_{\mathfrak{T}} : \exists p \ (p \in (\bigcup \sigma(g)) \cap P \text{ and } p \Vdash_{\mathbb{P}}^V y \in \dot{\Sigma})\}$$ as $\Sigma'$. By the definition of $P^*$ and the fact that $\sigma(g) \subset P^*$, $\Sigma \subset \Sigma'$. Choose $y \in \Sigma'$, so that there is $p \in (\bigcup \sigma(g)) \cap P$ with $p \Vdash_{\mathbb{P}}^V y \in \dot{\Sigma}$. But that entails the density of $$D_p := \{x \in P^* : y \in x\}$$ below $\{p\}$ in $\mathbb{P}^*$. Since $\{p\} \in \sigma(g)$ by Fact \ref{cufs}, $D_p \cap \sigma(g) \neq \emptyset$, whence $y \in \Sigma$.
\end{proof}

Combining Propositions \ref{p524} and \ref{p525}, we know that whenever $g$ is a $\mathbb{P}$-generic filter over $V$, $(\bigcup \sigma(g)) \cap \mathcal{L}_{\mathfrak{T}} = \dot{\Sigma}[g]$. In particular, $(\bigcup \sigma(g_0)) \cap \mathcal{L}_{\mathfrak{T}} = \Sigma(\mathfrak{T}, \mathcal{M})$. As $\sigma(g_0)$ is a $\mathbb{P}^*$-generic filter over $V$, $\mathcal{M}$ is a $(\mathbb{P}^*, V)$-generic model of $\mathfrak{T}$.
\end{proof}

Henceforth, we will look more closely examine the relationship between genericity and TCIs, while looking into the extent to which generic models of $\Pi_2$ TCIs are abundant. 

First, we link the concept of TCIs and their models back to forcing notions and their generic extensions. The relation $\lessdot$ on forcing notions can be used to define a partial order on the class of all TCIs via the map 
\begin{equation}\label{pmap}
    \hat{\mathbb{P}} : \mathfrak{T} \mapsto \mathbb{P}(\mathfrak{T}) \text{.}
\end{equation}
Given two TCIs $\mathfrak{T}_1$ and $\mathfrak{T}_2$, let $\mathfrak{T}_1 \trianglelefteq \mathfrak{T}_2$ whenever $\mathbb{P}(\mathfrak{T}_1) \lessdot \mathbb{P}(\mathfrak{T}_2)$. The relation $\trianglelefteq$ is a preordering of TCIs because of Fact \ref{forpre}. Have $\mathfrak{T}_1 \sim_T \mathfrak{T}_2$ iff $\mathfrak{T}_1 \trianglelefteq \mathfrak{T}_2$ and $\mathfrak{T}_2 \trianglelefteq \mathfrak{T}_1$. Then $\sim_T$ is an equivalence relation on TCIs. Denoting $\sim_P$ to be the forcing equivalence relation on forcing notions, we can easily verify that $\trianglelefteq / \sim_T$ is a partial order isomorphic to a suborder of $\lessdot / \sim_P$, as witnessed by the map 
\begin{equation}\label{peqmap}
    \Tilde{\mathbb{P}}: [\mathfrak{T}] \mapsto [\mathbb{P}(\mathfrak{T})] \text{.}
\end{equation}

The next two theorems --- the main ones of this subsection --- hint at a strong connection between $\Pi_2$ TCIs and forcing extensions.

\begin{thm}\label{revgenmodels}
Let $\mathbb{P} = (P, \leq_{\mathbb{P}})$ be a partial order. Then there is a consistent $\Pi_2$ TCI $\mathfrak{T} = (T, \sigma, \dot{\mathcal{U}}, \vartheta)$ such that
\begin{itemize}
    \item a dense weak embedding exists from $\mathbb{P}$ into $\mathbb{P}(\mathfrak{T})$, and
    \item for a fixed unary relation symbol $\dot{X} \in \sigma$, every model $\mathcal{M}$ of $\mathfrak{T}$ in any outer model of $V$ satisfies $$``\{p : \mathcal{M} \models \dot{X}(p)\} \text{ is a } \mathbb{P} \text{-generic filter over } V" \text{.}$$
\end{itemize}
\end{thm}

\begin{proof}
Choose $\dot{\mathcal{U}}$, $\dot{\leq}$ and $\dot{G}$ and to be distinct relation symbols of arities $1$, $2$ and $1$ respectively. For each dense subset $D$ of $\mathbb{P}$, choose a fresh unary relation symbol $\dot{D}$. Set $\sigma$ to be $$\{\dot{\leq}, \dot{G}\} \cup \{\dot{D} : D \text{ is a dense subset of } \mathbb{P}\} \text{.}$$ We define $\vartheta$ on $\{\dot{\mathcal{U}}\} \cup \sigma$ as follows:
\begin{align*}
    \vartheta(\dot{\mathcal{U}}) := \ & (P, 1) \\
    \vartheta(\dot{\leq}) := \ & (\leq_{\mathbb{P}}, 1) \\
    \vartheta(\dot{G}) := \ & (P, 0) \\
    \vartheta(\dot{D}) := \ & (D, 1) \text{ for each dense subset } D \text{ of } \mathbb{P} \text{.}
\end{align*}
Now, have $T$ contain only the sentences
\begin{gather*}
    \ulcorner \forall p \ \forall q \ \exists r \ ((\dot{G}(p) \wedge \dot{G}(q)) \implies (\dot{G}(r) \wedge \dot{\leq}(r, p) \wedge \dot{\leq}(r, q))) \urcorner \text{,} \\
    \ulcorner \forall p \ \forall q \ ((\dot{\leq}(p, q) \wedge \dot{G}(p)) \implies \dot{G}(q)) \urcorner \text{, as well as all members of} \\
    \{\ulcorner \exists p \ (\dot{G}(p) \wedge \dot{D}(p)) \urcorner : D \text{ is a dense subset of } \mathbb{P}\} \text{.}
\end{gather*} 
Let $\mathfrak{T} := (T, \sigma, \dot{\mathcal{U}}, \vartheta)$. Then $\mathfrak{T}$ is clearly a consistent $\Pi_2$ TCI, for any $\mathbb{P}$-generic filter over $V$ is an interpretation of $\dot{G}$ satisfying $\mathfrak{T}$. Moreover, it is obvious from our definition of $\mathfrak{T}$ that whenever $\mathcal{M} \models^* \mathfrak{T}$, the set $G(\mathcal{M}) := \{p : \mathcal{M} \models \dot{X}(p)\}$ is a $\mathbb{P}$-generic filter over $V$. We are left to show the existence of a dense weak embedding from $\mathbb{P}$ into $\mathbb{P}(\mathfrak{T})$. Toward that end we note:
\begin{enumerate}[label=(\arabic*)]
    \item\label{4431} for each $p \in P$, $\{\ulcorner \dot{G}(p) \urcorner\} \in P(\mathfrak{T})$, and
    \item\label{4432} for each $x \in P(\mathfrak{T})$, 
    \begin{align*}
        \exists p \ \forall q \ (x \cup \{\ulcorner \dot{G}(p) \urcorner\} \in P(\mathfrak{T}) & \wedge (\ulcorner \dot{G}(q) \urcorner \in x \implies p \leq_{\mathbb{P}} q)) \text{.}
    \end{align*}
\end{enumerate} 
Define $\pi : P \longrightarrow P(\mathfrak{T})$ to be $$p \mapsto \{\ulcorner \dot{G}(p) \urcorner\},$$ which is possible by \ref{4431}. We argue that $\pi$ is a dense weak embedding from $\mathbb{P}$ into $\mathbb{P}(\mathfrak{T})$.

Denote $w(\leq_{\mathbb{P}})$ as $\leq^{\dagger}$, $w(\mathbb{P})$ as $\mathbb{P}^{\dagger}$, $w(\leq_{\mathbb{P}(\mathfrak{T})})$ as $\leq^*$, and $w(\mathbb{P}(\mathfrak{T}))$ as $\mathbb{P}^*$. 

We first show that $\pi$ is a weak embedding. Assume $p \leq^{\dagger} q$ and let $x \leq_{\mathbb{P}(\mathfrak{T})} \pi(p)$. Choose any model $\mathcal{M}$ of $\mathfrak{T}$ in some outer model of $V$ such that $x \subset \Sigma(\mathfrak{T}, \mathcal{M})$. Then $G(\mathcal{M})$ is a $\mathbb{P}$-generic filter over $V$ containing $p$, implying $q \in G(\mathcal{M})$. $\mathcal{M}$ thus witnesses $x \cup \pi(q) \in P(\mathfrak{T})$, so $x \not \! \! \bot_{\mathbb{P}(\mathfrak{T})} \ \pi(q)$. We thus have $\pi(p) \leq^* \pi(q)$. Next, assume $p \not \leq^{\dagger} q$. Then there is $r \leq_{\mathbb{P}} p$ such that $r \ \bot_{\mathbb{P}} \ q$. This means $(\pi(p) \cup \pi(r)) \ \bot_{\mathbb{P}(\mathfrak{T})} \ \pi(q)$, and $\pi(p) \not \leq^* \pi(q)$. Lastly, the observation that
\begin{gather*}
    \pi(p) \ \bot_{\mathbb{P}^*} \ \pi(q) \iff \pi(p) \ \bot_{\mathbb{P}(\mathfrak{T})} \ \pi(q) \text{ and} \\
    p \ \bot_{\mathbb{P}^{\dagger}} \ q \iff p \ \bot_{\mathbb{P}} \ q
\end{gather*}
guarantees $$p \ \bot_{\mathbb{P}^{\dagger}} \ q \implies  \pi(p) \ \bot_{\mathbb{P}^*} \ \pi(q) \text{.}$$

To see that $ran(\pi)$ is dense in $\mathbb{P}^*$, fix any $x \in P(\mathfrak{T})$. By \ref{4432}, there is $p \in P$ for which
\begin{itemize}
    \item $x \cup \{\ulcorner \dot{G}(p) \urcorner\} \in P(\mathfrak{T})$, and
    \item $\forall q \ (\ulcorner \dot{G}(q) \urcorner \in x \implies p \leq_{\mathbb{P}} q)$.
\end{itemize}
What this entails by our definition of $\mathfrak{T}$ is, whenever $\mathcal{M} \models^* \mathfrak{T}$ and $\pi(p) \subset \Sigma(\mathfrak{T}, \mathcal{M})$, we must have $x \subset \Sigma(\mathfrak{T}, \mathcal{M})$. Let $y \leq_{\mathbb{P}(\mathfrak{T})} \pi(p)$. Choose a model $\mathcal{M}$ of $\mathfrak{T}$ in some outer model of $V$ such that $y \subset \Sigma(\mathfrak{T}, \mathcal{M})$. Then $\pi(p) \subset \Sigma(\mathfrak{T}, \mathcal{M})$, so also $x \subset \Sigma(\mathfrak{T}, \mathcal{M})$. As a result, $y \not \! \! \bot_{\mathbb{P}(\mathfrak{T})} \ x$, and we can conclude $\pi(p) \leq^* x$.
\end{proof}

\begin{rem}\label{ramble1}
Theorem \ref{revgenmodels} tells us two things, in view of Remark \ref{dwefor}.
\begin{enumerate}[label=(\arabic*)]
    \item\label{a5331} The map $\Tilde{\mathbb{P}}$ defined in (\ref{peqmap}) is an isomorphism between $\trianglelefteq / \sim_T$ and $\lessdot / \sim_P$.
    \item\label{a5332} Every member of $\trianglelefteq / \sim_T$ contains a $\Pi_2$ TCI $\mathfrak{T}$ for which
    \begin{align*}
        & \{V[\mathcal{M}] : \mathcal{M} \models^* \mathfrak{T} \text{ in an outer model of } V\} \\
        & = \{V[g] : g \text{ is } \mathbb{P}(\mathfrak{T}) \text{-generic over } V\} \text{.}
    \end{align*}
\end{enumerate}
It can be argued that the heart of forcing theory is in comparing the forcing extensions of different forcing notions. In this respect, and especially in the study of iterated forcing, a niceness result very often involves statements of the form 
\begin{quote}
    ``every $\mathbb{Q}$-generic extension over $V$ contains a $\mathbb{P}$-generic extension over $V$, and every $\mathbb{P}$-generic extension over $V$ can be extended to a $\mathbb{Q}$-generic extension over $V$,''
\end{quote}
which is virtually only provable by showing $\mathbb{P} \lessdot \mathbb{Q}$. Therefore, the relation $\lessdot$, and indeed $\lessdot / \sim_P$, encapsulates much of the core content of forcing theory. Points \ref{a5331} and \ref{a5332} can then be viewed as indicators that $\hat{\mathbb{P}}$ (defined in (\ref{pmap})) gives rise to a morally correct correspondence between TCIs and forcing notions. Further, \ref{a5332} suggests that $\Pi_2$ is a natural upper bound to the complexity of objects accessible by the technique of forcing.
\end{rem}

\begin{thm}\label{genericmodels}
Let $\mathfrak{T} = (T, \sigma, \dot{\mathcal{U}}, \vartheta)$ be a $\Pi_2$ TCI. If $\mathfrak{T}$ is consistent, then every $\mathbb{P}(\mathfrak{T})$-generic filter over $V$ witnesses a $(\mathbb{P}(\mathfrak{T}), V)$-generic model of $\mathfrak{T}$.
\end{thm}

\begin{proof}
The theorem follows directly from Lemmas \ref{main2}, \ref{mcequiv} and \ref{mcequiv2}, noting that the hypothesis of Lemma \ref{main2} are satisfied with
\begin{itemize}
    \item $\mathfrak{A}_{\mathfrak{T}}$ in place of $\mathfrak{A}$,
    \item $|\mathfrak{A}_{\mathfrak{T}}| = |trcl(\mathfrak{A}_{\mathfrak{T}})|$ in place of $\lambda$,
    \item $\mathcal{L}_{\mathfrak{T}}$ in place of $\mathcal{L}$,
    \item $[\mathcal{L}_{\mathfrak{T}}]^{< \omega}$ in place of $B$,
    \item $P(\mathfrak{T})$ in place of $P$,
    \item $\mathbb{P}(\mathfrak{T})$ in place of $\mathbb{P}$, 
    \item $\Gamma_{\mathfrak{T}}$ in place of $\Gamma$,
    \item $g$ a $\mathbb{P}(\mathfrak{T})$-generic filter over $V$, and
    \item $V[g]$ in place of $W$. \qedhere
\end{itemize}
\end{proof}

\begin{rem}\label{rem537}
By the proof of Corollary \ref{cor527}, Theorem \ref{genericmodels} is equivalent to, and can be restated as:
\begin{customthm}{5.36$'$}
Let $\mathfrak{T} = (T, \sigma, \dot{\mathcal{U}}, \vartheta)$ be a $\Pi_2$ TCI. If $P(\mathfrak{T})$ is not empty, then every $\mathbb{P}(\mathfrak{T})$-generic filter over $V$ witnesses a $(\mathbb{P}(\mathfrak{T}), V)$-generic model of $\mathfrak{T}$.
\end{customthm}
\end{rem}

\begin{rem}\label{ramble2}
By Theorem \ref{genericmodels} and \ref{5282} of Observation \ref{smallvgen}, we see that for every $\Pi_2$ TCI $\mathfrak{T}$,
\begin{align*}
    & \{V[\mathcal{M}] : \mathcal{M} \models^* \mathfrak{T} \text{ in an outer model of } V\} \\
    & \supset \{V[g] : g \text{ is } \mathbb{P}(\mathfrak{T}) \text{-generic over } V\} \text{.}
\end{align*}
In other words, forcing allows one to construct abundant models of every $\Pi_2$ TCI. This suggests that $\Pi_2$ is a natural lower bound to the complexity of objects accessible by the technique of forcing.
\end{rem}

Remarks \ref{ramble1} and \ref{ramble2} give two different ways of lensing forcing through the study of TCIs and their models. If we measure the power of forcing by the complexity of objects it has access to, then the two perspectives in question posit that $\Pi_2$ is a good classification of said power.  

This interpretation lends credence and weight to the informal thesis (a slogan, rather), 
\begin{quote}
    \centering
    ``Forcing is $\Pi_2$.''
\end{quote}
More importantly, it pitches tantalising prospects for using a complexity class defined on TCIs as a measure of --- or a proxy for --- accessibility within the context of the set-theoretic multiverse. We hope more work can be done in the future to formally establish and justify an approach along these lines of thinking.

\subsection{More Generic Models of \texorpdfstring{$\Pi_2$}{\Pi-2} TCIs}

The remainder of this section concerns itself with finer details regarding the existence of generic models of $\Pi_2$ TCIs. First up is a generic version of Lemma \ref{GLS}.

\begin{lem}\label{genericls}
Let $\mathfrak{T} = (T, \sigma, \dot{\mathcal{U}}, \vartheta)$ be a $\Pi_2$ TCI with an infinite model in some outer model of $V$. Then for every infinite ordinal $\beta$, there is a forcing notion $\mathbb{P}$ such that whenever $g$ is a $\mathbb{P}$-generic filter over $V$, there are sets $\mathcal{M} = (U; \mathcal{I})$ and $f$ in some outer model of $V$ for which  
\begin{enumerate}[label=(\alph*)]
    \item\label{ggls0} $g$ witnesses $(\mathcal{M}, f)$ is a $(\mathbb{P}, V)$-generic object, 
    \item\label{gglsa} $g$ witnesses $\mathcal{M}$ is a $(\mathbb{P}, V)$-generic model of $\mathfrak{T}$, and
    \item\label{gglsb} $f : \beta \longrightarrow U$ is a bijection.
\end{enumerate}
\end{lem}

\begin{proof}
Fix an infinite ordinal $\beta$. We want to modify $\mathfrak{T}$ to get another consistent $\Pi_2$ TCI $\mathfrak{T}^*$ such that from every model $\mathcal{M}^*$ of $\mathfrak{T}^*$ we can read off a structure $\mathcal{M} = (U; \mathcal{I})$ and a function $f$ satisfying both
\begin{enumerate}[label=(\alph*)']
    \setcounter{enumi}{1}
    \item\label{gglsb'} $\mathcal{M} \models^* \mathfrak{T}$ and $\Sigma(\mathfrak{T}^*, \mathcal{M}^*) \cap \mathcal{L}_{\mathfrak{T}} = \Sigma(\mathfrak{T}, \mathcal{M})$
\end{enumerate}
as well as \ref{gglsb} of the lemma. 

Note that we can, without loss of generality, assume $\sigma$ contains only relation symbols and constant symbols. This is because for any function symbol $\dot{X}$ and any $n < \omega$, $\dot{X}$ being a $n$-ary function is definable in a $(n+1)$-ary relation symbol $\dot{Y}$ via the conjunction of the $\Pi_2$ sentences
\begin{gather*}
    \ulcorner \forall x_1 \dots \forall x_n \ \exists y \ (\dot{Y}(x_1, \dots, x_n, y)) \urcorner \text{ and} \\
    \ulcorner \forall x_1 \dots \forall x_n \ \forall y \ \forall z \ ((\dot{Y} (x_1, \dots, x_n, y) \wedge \dot{Y} (x_1, \dots, x_n, z)) \implies y = z) \urcorner \text{,} 
\end{gather*}
if we interpret formulas of the form $\dot{Y} (x_1, \dots, x_n, x_{n+1})$ as $\dot{X} (x_1, \dots, x_n) = x_{n+1}$.

Have $(y, z)$ be $\vartheta(\dot{\mathcal{U}})$ and $\sigma'$ be $\sigma \cup \{\dot{\mathcal{U}}\}$. Choose 
\begin{itemize}
    \item $\dot{F}$ to be a unary function symbol not in $\sigma'$,
    \item $\dot{\mathcal{U}}^*$ and $\dot{V}$ to be distinct unary relation symbols not in $\sigma'$, and
    \item $\dot{c}$ to be a constant symbol not in $\sigma'$, for each $c \in y$, such that $\dot{c} \neq \dot{d}$ if $\{c, d\} \subset y$ and $c \neq d$.
\end{itemize}
Let $$\sigma^* := \sigma \cup \{\dot{F}, \dot{\mathcal{U}}, \dot{V}\} \cup \{\dot{c} : c \in y\}.$$ We specify $\vartheta^*$ by how it acts on members of its domain. Pick a set $b$ of cardinality $|\beta|$ that is disjoint from $y$, and set $\vartheta^*(\dot{\mathcal{U}}^*) := (y \cup b, 1)$. Make the assignments
\begin{align*}
    \vartheta^*(\dot{F}) := \ & (b \times y, 0) \\
    \vartheta^*(\dot{\mathcal{U}}) := \ & (y, z) \\
    \vartheta^*(\dot{V}) := \ & (b, 1) \\
    \vartheta^*(\dot{c}) := \ & (\{c\}, 0) \text{ for each } c \in y \text{.}
\end{align*}
Whenever $\dot{X} \in \sigma$ and $\vartheta(\dot{X}) = (y', z')$, we define $\vartheta^*(\dot{X}) := (y', min\{z, z'\})$.

Now, we modify $\Gamma_{\mathfrak{T}}$ by first removing members of the type described in \ref{gammaform3} and \ref{gammaform4} of Lemma \ref{mcequiv}, and then for each remaining member $\varphi$ of $\Gamma_{\mathfrak{T}}$, replacing every subformula of $\varphi$ of the form $\ulcorner E(\ulcorner x \urcorner) \urcorner$ with $\ulcorner x \urcorner$. 

Call the result of said modification $T'$. Whenever $\dot{X} \in \sigma$ is a $n$-ary relation symbol with $\vartheta(\dot{X}) = (y', 1)$, define $$T(\dot{X}) := \{\ulcorner \bigwedge_{1 \leq k \leq n} \dot{\mathcal{U}}(\dot{c}_k) \implies \dot{X}(\dot{c}_1, \dots, \dot{c}_n) \urcorner : (c_1, \dots, c_n) \in y' \cap y^n\}.$$ Finally, define $T^*$ to be the union of $T'$, $$\bigcup \{T(\dot{X}) : \dot{X} \in \sigma \text{ and } \exists y' \ (\vartheta(\dot{X}) = (y', 1))\},$$ and the finite set of sentences
\begin{align*}
    T^*_0 := \{ & \ulcorner \forall x \ \exists y \ (\dot{V}(x) \implies (\dot{\mathcal{U}}(y) \wedge \dot{F}(x) = y)) \urcorner, \\
    & \ulcorner \forall x \ \exists y \ (\dot{\mathcal{U}}(x) \implies (\dot{V}(y) \wedge \dot{F}(y) = x)) \urcorner, \\
    & \ulcorner \forall x \ \forall y \ (\dot{F}(x) = \dot{F}(y) \implies x = y) \urcorner\} \text{.}
\end{align*}
Clearly $T^*$ is a set of $\Pi_2$ sentences over the vocabulary $\sigma^*$.

A routine verification should enable the reader to see that $$T^*_1 := T' \cup (\bigcup \{T(\dot{X}) : \dot{X} \in \sigma \text{ and } \exists y' \ (\vartheta(\dot{X}) = (y', 1))\})$$ is basically a translation of $\Gamma_{\mathfrak{T}}$ in our expanded vocabulary $\sigma^*$, with the set of constants $\{\dot{c} : c \in y\}$ fulfilling a role similar to that of the parameter $\mathcal{L}_{\mathfrak{T}}$ (say, in the context of \ref{gammaform4} in the proof of Lemma \ref{mcequiv}). On the other hand, $T^*_0$ expresses precisely the requirement that a bijection from $b$ (and thus from $\beta$, in any outer model of $V$) into $\mathcal{I}(\dot{\mathcal{U}})$ exists for every $\mathcal{I}$ satisfying $(y \cup b; \mathcal{I}) \models^* (T^*_0, \sigma^*,  \dot{\mathcal{U}}^*, \vartheta^*)$ --- said bijection is just $\mathcal{I}(\dot{F})$. In fact, it does so in a manner independent of truths over the vocabulary $\sigma'$, so that whenever $\mathcal{M}^*$ is a model of $$\mathfrak{T}^* := (T^*, \sigma^*, \dot{\mathcal{U}}^*, \vartheta^*) = (T^*_0 \cup T^*_1, \sigma^*,  \dot{\mathcal{U}}^*, \vartheta^*),$$ we have $$\Sigma(\mathfrak{T}^*, \mathcal{M}^*) \cap \mathcal{L}_{\mathfrak{T}} = \Sigma(\mathfrak{T}, \mathcal{M})$$ for some model $\mathcal{M}$ of $\mathfrak{T}$. By Lemma \ref{GLS} and our assumptions on $\mathfrak{T}$, $\mathfrak{T}^*$ is consistent. 

We have thus checked that $\mathfrak{T}^*$ possesses the properties we want: it is a consistent $\Pi_2$ TCI, and from every model $\mathcal{M}^*$ of $\mathfrak{T}^*$ we can read off a structure $\mathcal{M} = (U; \mathcal{I})$ and a function $f$ satisfying both \ref{gglsb'} defined at the beginning of the proof as well as \ref{gglsb} of the lemma. An invocation of Theorem \ref{genericmodels} with $\mathfrak{T^*}$ in place of $\mathfrak{T}$ then completes the proof.
\end{proof}

By strengthening the hypothesis on $\mathfrak{T}$ in Lemma \ref{genericls}, we can derive more from our witnesses.

\begin{lem}\label{genericlscor}
Let $\mathfrak{T} = (T, \sigma, \dot{\mathcal{U}}, \vartheta)$ be a $\Pi_2$ TCI with only infinite model(s) across all outer models of $V$. Then for every infinite ordinal $\beta$, there is a forcing notion $\mathbb{P}$ such that whenever $g$ is a $\mathbb{P}$-generic filter over $V$, there are sets $\mathcal{M} = (U; \mathcal{I})$ and $f$ in some outer model of $V$ for which  
\begin{enumerate}[label=(\alph*)]
    \item $g$ witnesses $(\mathcal{M}, f)$ is a $(\mathbb{P}, V)$-generic object, 
    \item $g$ witnesses $\mathcal{M}$ is a $(\mathbb{P}, V)$-generic model of $\mathfrak{T}$,
    \item $g \cap P(\mathfrak{T})$ witnesses $\mathcal{M}$ is a $(\mathbb{P}(\mathfrak{T}), V)$-generic model of $\mathfrak{T}$, and
    \item $f : \beta \longrightarrow U$ is a bijection.
\end{enumerate}
\end{lem}

\begin{proof}
Construct $\mathfrak{T}^*$ from $\mathfrak{T}$ as per the proof of Lemma \ref{genericls}.

\begin{prop}\label{ptsspts}
$\mathbb{P}(\mathfrak{T}) \lessdot \mathbb{P}(\mathfrak{T}^*)$.
\end{prop}

\begin{proof}
Observe that, if $\mathcal{M}$ is a model of $\mathfrak{T}$ is some outer model $W$ of $V$, then $\mathcal{M}$ extends to a model of $\mathfrak{T}^*$ in an outer model $W'$ of $W$. As a result, $\mathbb{P}(\mathfrak{T})$ is a suborder of $\mathbb{P}(\mathfrak{T}^*)$. To show the regularity of $\mathbb{P}(\mathfrak{T})$ as a suborder of $\mathbb{P}(\mathfrak{T}^*)$, let $p \in \mathbb{P}(\mathfrak{T}^*)$. Define $$q_0 := \{\ulcorner \dot{\mathcal{U}}(j) \urcorner : \ulcorner \dot{F}(i) = j \urcorner \in p\}$$ and let $$q := (p \cap \mathcal{L}_{\mathfrak{T}}) \cup q_0.$$ Obviously $q \in \mathbb{P}(\mathfrak{T})$. Consider any $\Sigma$ $\Gamma_{\mathfrak{T}} (\mathcal{L}_{\mathfrak{T}}, \mathfrak{A}_{\mathfrak{T}})$-certifying $q$ in some outer model of $V$. Since the set $U$ defined by $\dot{\mathcal{U}}$ in $\Sigma$ is guaranteed to be infinite following our assumptions on $\mathfrak{T}$, the finitely many restrictions imposed by $p$ on the relationship between (the function interpreting) $\dot{F}$ and $U$ can be circumvented with ease. In other words, $\Sigma$ can be extended to some $\Sigma^*$ $\Gamma_{\mathfrak{T}^*} (\mathcal{L}_{\mathfrak{T}^*}, \mathfrak{A}_{\mathfrak{T}^*})$-certifying $p$. But this means every $q' \leq_{\mathbb{P}(\mathfrak{T})} q$ is compatible with $p$ in $\mathbb{P}(\mathfrak{T}^*)$.
\end{proof}

By Fact \ref{regsubo}, Proposition \ref{ptsspts}, and the identity
\begin{equation*}
    \bigcup (g \cap P(\mathfrak{T})) = (\bigcup g) \cap \mathcal{L}_{\mathfrak{T}}
\end{equation*}
which holds for every $\mathbb{P}(\mathfrak{T}^*)$-generic filter $g$ over $V$, we are done.
\end{proof}

Models of a TCI $\mathfrak{T}$ across all outer models of $V$ can be very complicated. However, when a model of $\mathfrak{T}$ is finitely determined, its atomic diagram can be easily read off $\mathbb{P}(\mathfrak{T})$.

\begin{lem}\label{findetinV}
Let $\mathfrak{T}$ be a TCI and $\mathcal{M}$ be a finitely determined model of $\mathfrak{T}$ in some outer model of $V$. Then for some atom $p$ of $\mathbb{P}(\mathfrak{T})$, $\Sigma(\mathfrak{T}, \mathcal{M}) = g_p (\mathbb{P}(\mathfrak{T}))$. In particular, $\mathcal{M} \in V$.
\end{lem}

\begin{proof}
Let $\mathcal{M}$ be finitely determined by $\varphi$. Without loss of generality, we can assume $\varphi$ is the conjunction of a set of literals $\{l_i : i < n\}$ for some $n < \omega$. This means $$p := \{\ulcorner E(l_i) \urcorner : i < n\}$$ is an atom of $\mathbb{P}(\mathfrak{T})$. Lemma \ref{gpgeneric} tells us that $g_p (\mathbb{P}(\mathfrak{T}))$ is $\mathbb{P}(\mathfrak{T})$-generic over $V$, so necessarily $\Sigma(\mathfrak{T}, \mathcal{M}) = g_p (\mathbb{P}(\mathfrak{T}))$ by Theorem \ref{genericmodels}. Then according to Lemma \ref{mcequiv2}, $\mathcal{M} \in V$ because $g_p (\mathbb{P}(\mathfrak{T})) \in V$.
\end{proof}

It is possible to have an analogue of Lemma \ref{findetinV} for models that are ``close to being finitely determined''.

\begin{defi}\label{CB1}
Let $\mathfrak{T}$ be a TCI. Inductively define $\Gamma_{\mathfrak{T}}^{(\alpha)}$, $P(\mathfrak{T})^{(\alpha)}$ and $\mathbb{P}(\mathfrak{T})^{(\alpha)}$ for all ordinals $\alpha \leq |[\mathcal{L}_{\mathfrak{T}}]^{< \omega}|^+$ as follows:
\begin{align*}
    \Gamma_{\mathfrak{T}}^{(0)} := \ & \Gamma_{\mathfrak{T}} \text{,} \\
    P(\mathfrak{T})^{(0)} := \ & P(\mathfrak{T}) \text{,} \\
    \Gamma_{\mathfrak{T}}^{(\alpha)} := \ & \Gamma_{\mathfrak{T}}^{(\alpha - 1)} \cup \{\ulcorner \bigvee_{x \in p} (\neg E(x)) \urcorner : p \text{ is an atom of } \mathbb{P}(\mathfrak{T})^{(\alpha - 1)}\} \\
    & \text{if } \alpha \text{ is a successor ordinal,} \\
    \Gamma_{\mathfrak{T}}^{(\alpha)} := \ & \bigcup_{\beta < \alpha} \Gamma_{\mathfrak{T}}^{(\beta)} \\
    & \text{if } \alpha \text{ is a limit ordinal,} \\
    P(\mathfrak{T})^{(\alpha)} := \ & \{p \in [\mathcal{L}_{\mathfrak{T}}]^{< \omega} : \ \Vdash_{Col(\omega, |trcl(\mathfrak{A}_{\mathfrak{T}})|)} \exists \Sigma \ (``\Sigma \ \Gamma_{\mathfrak{T}}^{(\alpha)} (\mathcal{L}_{\mathfrak{T}}, \mathfrak{A}_{\mathfrak{T}})\text{-certifies } p")\} \text{,} \\
    \mathbb{P}(\mathfrak{T})^{(\alpha)} := \ & (P(\mathfrak{T})^{(\alpha)}, \leq_{\mathbb{P}(\mathfrak{T})}) \text{.}
\end{align*}
\end{defi}

By a simple cardinality argument, there must exist some $\alpha < |[\mathcal{L}_{\mathfrak{T}}]^{< \omega}|^+$ for which $\Gamma_{\mathfrak{T}}^{(\alpha)} = \Gamma_{\mathfrak{T}}^{(\alpha + 1)}$, whence $\mathbb{P}(\mathfrak{T})^{(\alpha)} = \mathbb{P}(\mathfrak{T})^{(\alpha + 1)}$.

\begin{defi}\label{CB2}
Let $\Gamma_{\mathfrak{T}}^{\top}$ denote the unique $\Gamma$ such that $\Gamma = \Gamma_{\mathfrak{T}}^{(\alpha)} = \Gamma_{\mathfrak{T}}^{(\alpha + 1)}$ for some $\alpha < |[\mathcal{L}_{\mathfrak{T}}]^{< \omega}|^+$. Similarly, $\mathbb{P}(\mathfrak{T})^{\top}$ shall denote the unique $\mathbb{P}$ such that $\mathbb{P} = \mathbb{P}(\mathfrak{T})^{(\alpha)} = \mathbb{P}(\mathfrak{T})^{(\alpha + 1)}$ for some $\alpha < |[\mathcal{L}_{\mathfrak{T}}]^{< \omega}|^+$.
\end{defi}

It is not hard to see that $P(\mathfrak{T})^{\top}$ is an atomless upward closed subset of $\mathbb{P}(\mathfrak{T})$ and $\Gamma_{\mathfrak{T}} \subset \Gamma_{\mathfrak{T}}^{\top}$.

\begin{rem}\label{CBrem}
In constructing the $\mathbb{P}(\mathfrak{T})^{(\alpha)}$'s, we are inductively removing atoms of $\mathbb{P}(\mathfrak{T})$. These atoms are representatives of isolated models of a TCI. By looking at Definition \ref{CB1} in this way, we can draw obvious parallels between $\mathbb{P}(\mathfrak{T})^{(\alpha)}$ and the $\alpha$-th-order Cantor-Bendixson derivative of a set. Such parallels culminate in $\mathbb{P}(\mathfrak{T})^{\top}$ being analogous to the ``perfect core'' of $\mathbb{P}(\mathfrak{T})$.
\end{rem}

\begin{defi}
Given a TCI $\mathfrak{T}$ and any $\mathcal{M}$, we say $\mathcal{M}$ is an \emph{almost finitely determined model of} $\mathfrak{T}$ iff $\mathcal{M} \models^* \mathfrak{T}$ and for some $\alpha < |[\mathcal{L}_{\mathfrak{T}}]^{< \omega}|^+$ and an atom $p$ of $\mathbb{P}(\mathfrak{T})^{(\alpha)}$, $$p \subset \Sigma(\mathfrak{T}, \mathcal{M}).$$
\end{defi}

We have as our next lemma, the promised analogue of Lemma \ref{findetinV}.

\begin{lem}\label{afdinV}
Let $\mathfrak{T}$ be a TCI and $\mathcal{M}$ be an almost finitely determined model of $\mathfrak{T}$ in some outer model of $V$. Then for some $\alpha < |[\mathcal{L}_{\mathfrak{T}}]^{< \omega}|^+$ and some atom $p$ of $\mathbb{P}(\mathfrak{T})^{(\alpha)}$, $\Sigma(\mathfrak{T}, \mathcal{M}) = g_p (\mathbb{P}(\mathfrak{T})^{(\alpha)})$. In particular, $\mathcal{M} \in V$.
\end{lem}

\begin{proof}
Choose any model $\mathcal{M}$ of $\mathfrak{T}$ in an outer model of $V$. It suffices to prove by induction on $\alpha \leq |[\mathcal{L}_{\mathfrak{T}}]^{< \omega}|^+$ that
\begin{align*}
    \forall q \ \exists \beta \leq \alpha \ \exists p \ ( & (q \text{ is an atom of } \mathbb{P}(\mathfrak{T})^{(\alpha)} \text{ and } q \subset \Sigma(\mathfrak{T}, \mathcal{M})) \\
    & \implies (p \text{ is an atom of } \mathbb{P}(\mathfrak{T})^{(\beta)} \text{ and } \Sigma(\mathfrak{T}, \mathcal{M}) = g_p (\mathbb{P}(\mathfrak{T})^{(\beta)}))) \text{.}
\end{align*}
The base case where $\alpha = 0$ is just Lemma \ref{findetinV}. For the inductive case, assume $0 < \alpha \leq |[\mathcal{L}_{\mathfrak{T}}]^{< \omega}|^+$. and let $q$ be an atom of $\mathbb{P}(\mathfrak{T})^{(\alpha)}$ with $q \subset \Sigma(\mathfrak{T}, \mathcal{M})$. Then by Lemma \ref{gpgeneric} and the definition of $\mathbb{P}(\mathfrak{T})^{(\alpha)}$, either $\Sigma(\mathfrak{T}, \mathcal{M}) = g_q (\mathbb{P}(\mathfrak{T})^{(\alpha)})$ or there is $\beta' < \alpha$ and an atom $q'$ of $\mathbb{P}(\mathfrak{T})^{(\beta')}$ such that $q' \subset \Sigma(\mathfrak{T}, \mathcal{M})$. In the latter case, the inductive hypothesis gives us $\beta \leq \beta'$ and an atom $p$ of $\mathbb{P}(\mathfrak{T})^{(\beta)}$ for which $\Sigma(\mathfrak{T}, \mathcal{M}) = g_p (\mathbb{P}(\mathfrak{T})^{(\beta)})$. Either way we are done.
\end{proof}

The way $\mathbb{P}(\mathfrak{T})$ and $\mathbb{P}(\mathfrak{T})^{\top}$ are defined from a TCI $\mathfrak{T}$ allows us to establish a nice dichotomy on the $(\mathbb{P}(\mathfrak{T}), V)$-generic models of $\mathfrak{T}$ when $\mathfrak{T}$ is $\Pi_2$.

\begin{lem}
Let $\mathfrak{T}$ be a $\Pi_2$ TCI and $\mathcal{M}$ be a $(\mathbb{P}(\mathfrak{T}), V)$-generic model of $\mathfrak{T}$. Then one of the following must hold:
\begin{enumerate}[label=(\arabic*)]
    \item $\mathcal{M}$ is almost finitely determined.
    \item $\mathcal{M}$ is a $(\mathbb{P}(\mathfrak{T})^{\top}, V)$-generic model of $\mathfrak{T}$.
\end{enumerate}
\end{lem}

\begin{proof}
Let $g$ be a $\mathbb{P}(\mathfrak{T})$-generic filter over $V$ and assume $\mathcal{A} \cap g = \emptyset$, where $$\mathcal{A} := \{p : \exists \alpha \ (\alpha < |[\mathcal{L}_{\mathfrak{T}}]^{< \omega}|^+ \text{ and } p \text{ is an atom of } \mathbb{P}(\mathfrak{T})^{(\alpha)})\}.$$ This latter assumption is equivalent to saying that the unique model $\mathcal{M}$ of $\mathfrak{T}$ for which $\bigcup g = \Sigma(\mathfrak{T}, \mathcal{M})$ is not almost finitely determined. By Theorem \ref{genericmodels}, it suffices to show that $g$ is a $\mathbb{P}(\mathfrak{T})^{\top}$-generic filter over $V$. Clearly, $\bigcup g$ $\Gamma_{\mathfrak{T}}^{\top} (\mathcal{L}_{\mathfrak{T}}, \mathfrak{A}_{\mathfrak{T}})$-certifies $p$, so $g \subset \mathbb{P}(\mathfrak{T})^{\top}$. That $\mathbb{P}(\mathfrak{T})^{\top}$ is a suborder of $\mathbb{P}(\mathfrak{T})$ means $g$ is a filter on $\mathbb{P}(\mathfrak{T})^{\top}$.

To see $g$ is $\mathbb{P}(\mathfrak{T})^{\top}$-generic over $V$, let $E$ be predense in $\mathbb{P}(\mathfrak{T})^{\top}$. Note that if $p \in \mathbb{P}(\mathfrak{T})$ is incompatible in $\mathbb{P}(\mathfrak{T})$ with every member of $\mathcal{A}$, then $p \in \mathbb{P}(\mathfrak{T})^{\top}$. As such, $E \cup \mathcal{A}$ must be predense in $\mathbb{P}(\mathfrak{T})$. But this implies $E \cap g \neq \emptyset$ because $g$ is $\mathbb{P}(\mathfrak{T})$-generic and $\mathcal{A} \cap g = \emptyset$.
\end{proof}

The following is a stronger version of Theorem \ref{genericmodels}.

\begin{thm}\label{genericdichom}
Let $\mathfrak{T}$ be a $\Pi_2$ TCI. If not all models of $\mathfrak{T}$ are almost finitely determined, then $\mathbb{P}(\mathfrak{T})^{\top}$ is non-empty and every $\mathbb{P}(\mathfrak{T})^{\top}$-generic filter over $V$ witnesses $\mathcal{M}$ is a $(\mathbb{P}(\mathfrak{T})^{\top}, V)$-generic model of $\mathfrak{T}$ for some $\mathcal{M}$.
\end{thm}

\begin{proof}
Assume not all models of $\mathfrak{T}$ are almost finitely determined, and let $\mathcal{M}$ be a model of $\mathfrak{T}$ not almost finitely determined in some outer model of $V$. Then $\Sigma(\mathfrak{T}, \mathcal{M})$ $\Gamma_{\mathfrak{T}}^{\top} (\mathcal{L}_{\mathfrak{T}}, \mathfrak{A}_{\mathfrak{T}})$-certifies $\emptyset$, so $\mathbb{P}(\mathfrak{T})^{\top}$ is non-empty.

Check that the hypothesis of Lemma \ref{main2} are satisfied when we have 
\begin{itemize}
    \item $\mathfrak{A}_{\mathfrak{T}}$ in place of $\mathfrak{A}$,
    \item $|\mathfrak{A}_{\mathfrak{T}}| = |trcl(\mathfrak{A}_{\mathfrak{T}})|$ in place of $\lambda$,
    \item $\mathcal{L}_{\mathfrak{T}}$ in place of $\mathcal{L}$,
    \item $[\mathcal{L}_{\mathfrak{T}}]^{< \omega}$ in place of $B$,
    \item $P(\mathfrak{T})^{\top}$ in place of $P$,
    \item $\mathbb{P}(\mathfrak{T})^{\top}$ in place of $\mathbb{P}$,
    \item $\Gamma_{\mathfrak{T}}^{\top}$ in place of $\Gamma$, 
    \item $g$ a $\mathbb{P}(\mathfrak{T})^{\top}$-generic filter over $V$, and
    \item $V[g]$ in place of $W$.
\end{itemize}
A direct application of said lemma, coupled with the knowledge that $\Gamma_{\mathfrak{T}} \subset \Gamma_{\mathfrak{T}}^{\top}$, completes the proof. 
\end{proof}

Two important, yet perhaps surprising, properties of $\Pi_2$ TCIs follow from the dichotomy in Theorem \ref{genericdichom}. We state these properties in the next corollary.

\begin{cor}\label{cor549}
The following statements hold.
\begin{enumerate}[label=(\arabic*)]
    \item Let $\mathfrak{T}$ be a $\Pi_2$ TCI. Every model of $\mathfrak{T}$ can be found in $V$ iff every model of $\mathfrak{T}$ is almost finitely determined (see also Lemma \ref{afdinV}).
    \item\label{cor5492} There is a procedure to uniformly decide in $V$, whether every model of any given $\Pi_2$ TCI $\mathfrak{T}$ is definable in $V$ with parameters from $$\mathrm{PS}(\mathfrak{T}) := trcl(\mathfrak{T}) \cup |trcl(\mathfrak{T})|^+ \cup H(\omega) \text{,}$$ or equivalently, whether every model of $\mathfrak{T}$ can be found in $V$.
    
    Moreover, for some $n < \omega$ there is a $(n+2)$-ary formula $\psi$ in the language of set theory such that, given any $\Pi_2$ TCI $\mathfrak{T}$,
    \begin{itemize}
        \item whenever $\Vec{z} \in \mathrm{PS}(\mathfrak{T})^n$, if there exists $x$ for which $\psi(x, \mathfrak{T}, \Vec{z})$ holds, then said $x$ is unique, and
        \item if every model of $\mathfrak{T}$ can be found in $V$, then in $V$, 
        \begin{align*}
            \{x : \exists \Vec{z} \ (\Vec{z} \in \mathrm{PS}(\mathfrak{T})^n \wedge \psi(x, \mathfrak{T}, \Vec{z}))\} & = \{\mathcal{M} : \mathcal{M} \models^* \mathfrak{T}\} \\
            & \subset H(|trcl(\mathfrak{T})|^+) \text{.}
        \end{align*}
    \end{itemize}
\end{enumerate}
\end{cor}
Notice that the first sentence to appear in \ref{cor5492} of Corollary \ref{cor549} is similar in form to Corollary \ref{cor527}. Through Remark \ref{rem537}, this points further to Theorem \ref{genericdichom} being an improvement upon Theorem \ref{genericmodels}. 

For a countable TCI $\mathfrak{T}$, the consistency of $\mathfrak{T}$ implies the existence of a model of $\mathfrak{T}$ in $V$. 

\begin{lem}\label{modelinV}
Let $\mathfrak{T} = (T, \sigma, \dot{\mathcal{U}}, \vartheta)$ be a TCI such that $$|\sigma \cup y| \leq \aleph_0$$ whenever $\vartheta(\dot{\mathcal{U}}) = (y, z)$ for some $z$. If $\mathfrak{T}$ is consistent then $\mathfrak{T}$ has a model in $V$.
\end{lem}

\begin{proof}
Fix $f_1$ a bijection between $y$ and $|y|$, and $f_2$ a form-preserving signature embedding from $\sigma \cup \{\dot{\mathcal{U}}\}$ into $H(\omega)$. Then $f_1$ and $f_2$ naturally induce
\begin{itemize}
    \item a TCI $\mathfrak{T}' := (T', \sigma', \dot{\mathcal{U}}, \vartheta')$ in $V$ such that $\mathfrak{T}'$ has a countable transitive closure, and
    \item a bijection between $\{\mathcal{M} : \mathcal{M} \models^* \mathfrak{T}\}$ and $\{\mathcal{M} : \mathcal{M} \models^* \mathfrak{T}'\}$ in every weak outer model of $V$.
\end{itemize}
As a consequence, we can assume $\mathfrak{T}$ has a countable transitive closure without loss of generality. By Lemma \ref{setcode}, $\mathfrak{T}$ can be coded as a real. Besides, if $\mathfrak{T}$ has a model $\mathcal{M}$ in an outer model $W$ of $V$, then $M$ has a real code. By a routine check while unfurling the definition of $\models^*$ (see e.g. the proof of Lemma \ref{inout} for an argument of the satisfaction relation being $\Sigma_1$), we get that the statement $$\exists \mathcal{M} \ (\mathcal{M} \models^* \mathfrak{T})$$ is equivalent to a $\mathbf{\Sigma^1_1}$ sentence involving a real code of $\mathfrak{T}$ found in $V$, so it is absolute for $V$ and any of its weak outer models. If $\mathfrak{T}$ is consistent, it must have a model in some outer model of $V$, whence it has a model in $V$.
\end{proof}

Ideally, in the spirit of Lemmas \ref{genericls} and \ref{genericlscor}, we want to prove a generic version of Lemma \ref{modelinV}. This can be done through a relatively effective version of Theorem \ref{genericdichom} for a certain class of countable TCIs, so as to kill two birds with one stone. Some definitions and facts are prerequisites.

For the rest of this subsection,
\begin{itemize}
    \item fix a bijection 
    \begin{align*}
        f^{\dagger} : \ & \mathrm{Var} \cup \{x : x \text{ is a first-order logical symbol}\} \cup \{\ulcorner \in \urcorner\} \\
        & \longrightarrow \{n < \omega : n \text{ is odd}\} \text{,}
    \end{align*}
    and
    \item interpret $\Delta_n$, $\Pi_n$ and $\Sigma_n$ formulas the way they are defined in Definition \ref{def27}.
\end{itemize}

\begin{defi}
For any countable set $X$, we say $(r, f)$ \emph{witnesses} $(X; \in)$ \emph{is computable} iff
\begin{itemize} 
    \item $f$ is a bijection from $X$ into $\{n < \omega : n \text{ is even}\}$,
    \item
    \!
    $\begin{aligned}[t]
        r = \{\langle (f \cup f^{\dagger})^*(\varphi) \rangle : \varphi \text{ is a member of the } \Delta_0 \text{-elementary diagram of } (X; \in)\}, 
    \end{aligned}$
    \medskip
    \\
    where 
    \begin{itemize}[label=$\circ$]
        \item $\langle \cdot \rangle$ is the standard computable G\"odel numbering of strings over $\omega$, and
        \item $(f \cup f^{\dagger})^*$ is the canonical bijection from the set of finite strings over $dom(f \cup f^{\dagger})$ into the set of finite strings over $\omega$, induced by $f \cup f^{\dagger}$, and
    \end{itemize}
    \item $r$ is computable.
\end{itemize}
We say $r$ is a \emph{nicely computable code of} $(X; \in)$ iff there is $f$ for which $(r, f)$ witnesses $(X; \in)$ is computable.
\end{defi}

\begin{fact}
There is a nicely computable code of $(H(\omega); \in)$. 
\end{fact}

\begin{fact}\label{uniquecode}
If $r$ is a nicely computable code of $(H(\omega); \in)$, then there is a unique $f$ for which $(r, f)$ witnesses $(H(\omega); \in)$ is computable.
\end{fact}

Let $\mathfrak{T} = (T, \sigma, \dot{\mathcal{U}}, \vartheta)$ be a TCI, and $y$ be such that $\vartheta(\dot{\mathcal{U}}) = (y, z)$ for some $z$.

Assume $|\sigma \cup y| \leq \aleph_0$. Then we can find $f_1$ and $f_2$ such that
\begin{itemize}
    \item $f_1$ is a bijection from $y$ into $|y|$, and
    \item $f_2$ is a form-preserving signature embedding from $\sigma \cup \{\dot{\mathcal{U}}\}$ into $H(\omega)$.
\end{itemize}
Together, $f_1$ and $f_2$ naturally induce a TCI $\mathfrak{T}'$ with its associated $\mathcal{L}_{\mathfrak{T}'}$ being a subset of $H(\omega)$. Moreover, in every model $\mathfrak{A}$ of $\mathsf{ZFC - Powerset}$ containing $\{\mathfrak{T}, f_1, f_2\}$, $f_1$ and $f_2$ also induce a bijection $h^{\mathfrak{A}}$ from $$\{\mathcal{M} : \mathcal{M} \models^* \mathfrak{T}\}$$ into $$\{\mathcal{M}' : \mathcal{M}' \models^* \mathfrak{T}'\},$$ such that for all $\mathcal{M} \in dom(h^{\mathfrak{A}})$, $\mathcal{M} \cong h^{\mathfrak{A}}(\mathcal{M})$. 

Hence, if we only care about models of $\mathfrak{T}$ up to isomorphism, we can without loss of generality, assume $y$ is an ordinal at most $\omega$ and $\mathcal{L}_{\mathfrak{T}}$ is a subset of $H(\omega)$.

\begin{defi}
A TCI $\mathfrak{T} = (T, \sigma, \dot{\mathcal{U}}, \vartheta)$ is \emph{code-friendly} iff 
\begin{itemize}
    \item $\vartheta(\dot{\mathcal{U}}) = (y, z) \in (\omega + 1) \times 2$, and 
    \item $\mathcal{L}_{\mathfrak{T}} \subset H(\omega)$.
\end{itemize}
\end{defi}

Code-friendly TCIs are relatively well-behaved and easy to reason about, especially when it comes to things like absoluteness. Notice that given any code-friendly TCI $\mathfrak{T} = (T, \sigma, \dot{\mathcal{U}}, \vartheta)$ and any ordinal $\alpha$,
\begin{enumerate}[leftmargin=40pt, label=(CF\arabic*)]
    \item\label{cf1} $\mathcal{L}_{\mathfrak{T}}$ and $\Gamma_{\mathfrak{T}}$ are definable subsets of $H(\omega)$ over the structure $$\mathfrak{A}^*_{\mathfrak{T}} := \mathfrak(H(\omega); \in, T, \sigma, \dot{\mathcal{U}}, \vartheta),$$
    \item $P(\mathfrak{T})^{(\alpha)} = \{p \in [\mathcal{L}_{\mathfrak{T}}]^{< \omega} : \exists \Sigma \ (``\Sigma \ \Gamma_{\mathfrak{T}}^{(\alpha)} (\mathcal{L}_{\mathfrak{T}}, \mathfrak{A}^*_{\mathfrak{T}})\text{-certifies } p")\}$ by straightforward induction incorporating an argument similar to that which proved Lemma \ref{inout}, and hence
    \item $P(\mathfrak{T})^{\top}$ is $\Delta_1$-definable in $\mathfrak{A}^*_{\mathfrak{T}}$.
\end{enumerate}
This means that the definition of $\mathbb{P}(\mathfrak{T})^{\top}$ from a code-friendly TCI $\mathfrak{T}$ is absolute for transitive models of $\mathsf{ZFC - Powerset}$.

Recall Cohen forcing $\mathbb{C} = (C; \leq_{\mathbb{C}})$. We will use this labelling in the statements and proofs of the subsequent lemmas.

\begin{lem}\label{ctblegeneric}
Let $\mathfrak{T}$ be a code-friendly $\Pi_2$ TCI, and $(r, f)$ witness $(H(\omega); \in)$ is computable. Then one of the following must hold.
\begin{enumerate}[label=(\arabic*)]
    \item All models of $\mathfrak{T}$ are almost finitely determined.
    \item\label{3782} There is an oracle machine $\Psi$ and a countable structure $\mathfrak{A}$ in the language of set theory, such that whenever $g$ is a $\mathbb{C}$-generic filter over $\mathfrak{A}$, there is a unique model $\mathcal{M}_g$ of $\mathfrak{T}$ satisfying 
    \begin{equation*}
        \Psi^{(f" g) \oplus (f" (P(\mathfrak{T})^{\top}))} = f" (\Sigma(\mathfrak{T}, \mathcal{M}_g)) \text{.}
    \end{equation*}
    Moreover, the function $g \mapsto \mathcal{M}_g$ defined as such is injective.
\end{enumerate}
\end{lem}

\begin{proof}
Assume not all models of $\mathfrak{T}$ are almost finitely determined. For brevity, let us write
\begin{align*}
    s := \ & f" (P(\mathfrak{T})^{\top}) \text{ and} \\
    \leq_s \ := \ & f" (\leq_{\mathbb{P}(\mathfrak{T})^{\top}}) \text{.}
\end{align*} 
We shall identify $s$ with $(s, \leq_s)$ whenever contextually necessary. This can be done without loss of generality because $\leq_s$ is computable in $s$.

Going forward, even beyond this proof, we would often argue about things in $H(\omega)$ even though our intended domain of discourse is the set of natural numbers. This is because first-order truths about $(H(\omega); \in)$ are uniformly propagated by $f$ onto its range, so that specific versions of them hold there as well. If one such truth is sufficiently simple, then $r$ knows the version of it on $ran(f)$ and can then relay that to the appropriate machines for further processing. 

Let $\mathfrak{A} = (A; \in)$ be any countable elementary substructure of $\mathfrak{A}_{\mathfrak{T}}$ with $P(\mathfrak{T})^{\top} \in A$. Then $A$ is finitely transitive because $\mathfrak{A}_{\mathfrak{T}}$ is a transitive model of $\mathsf{ZFC - Powerset}$. Since $\mathcal{L}_{\mathfrak{T}}$ is just the closure of $\bigcup P(\mathfrak{T})^{\top}$ under negation in $\mathfrak{A}_{\mathfrak{T}}$, we have $\mathcal{L}_{\mathfrak{T}} \in A$. That $\mathfrak{T}$ is code-friendly and $H(\omega) \subset A$ implies $\mathcal{L}_{\mathfrak{T}} \subset A$ too, so $\mathfrak{A}$ is $\mathcal{L}_{\mathfrak{T}}$-suitable. The members of $\Gamma_{\mathfrak{T}}^{\top}$ are $(\mathcal{L}_{\mathfrak{T}})^*_{\mathfrak{A}_{\mathfrak{T}}}$-$\Pi_2$ sentences with a single parameter $P(\mathfrak{T})^{\top}$ and quantification exclusively over $H(\omega)$, entailing that $\Gamma_{\mathfrak{T}}^{\top}$ is also a set of $(\mathcal{L}_{\mathfrak{T}})^*_{\mathfrak{A}}$-$\Pi_2$ sentences. 

In this vein, similar to what we did in the proof of Theorem \ref{genericdichom}, check that all except the last two points in the hypothesis of Lemma \ref{main2} are satisfied with 
\begin{itemize}
    \item $\mathfrak{A}$ as defined,
    \item $|\mathfrak{A}_{\mathfrak{T}}| = |trcl(\mathfrak{A}_{\mathfrak{T}})|$ in place of $\lambda$,
    \item $\mathcal{L}_{\mathfrak{T}}$ in place of $\mathcal{L}$,
    \item $[\mathcal{L}_{\mathfrak{T}}]^{< \omega}$ in place of $B$,
    \item $P(\mathfrak{T})^{\top}$ in place of $P$,
    \item $\mathbb{P}(\mathfrak{T})^{\top}$ in place of $\mathbb{P}$, and
    \item $\Gamma_{\mathfrak{T}}^{\top}$ in place of $\Gamma$.
\end{itemize} 
Following the proof of Theorem \ref{genericdichom}, while bearing in mind
\begin{itemize}
    \item $\Gamma_{\mathfrak{T}} \subset \Gamma_{\mathfrak{T}}^{\top}$,
    \item Remark \ref{swapA} and how its invocation is justified by the preceding paragraph, as well as
    \item the injectivity of the function $\mathcal{M} \mapsto \Sigma(\mathfrak{T}, \mathcal{M})$, 
\end{itemize}
we apply Lemma \ref{main2} with $V$ in place of $W$ to give us 
\begin{align}\label{marker0}
\begin{split}
    \forall \bar{g} \ \exists ! \mathcal{M}_{\bar{g}} \ ( & \bar{g} \text{ is a } \mathbb{P}(\mathfrak{T})^{\top} \text{-} \Sigma_1 \text{-generic filter over } \mathfrak{A} \\ 
    & \implies (\mathcal{M}_{\bar{g}} \models^* \mathfrak{T} \text{ and } \bigcup \bar{g} = \Sigma(\mathfrak{T}, \mathcal{M}_{\bar{g}}))) \text{.}
\end{split}
\end{align}
In particular,
\begin{align}\label{marker}
\begin{split}
    \forall \bar{g} \ \exists ! \mathcal{M}_{\bar{g}} \ ( & \bar{g} \text{ is a } \mathbb{P}(\mathfrak{T})^{\top} \text{-generic filter over } \mathfrak{A} \\ 
    & \implies (\mathcal{M}_{\bar{g}} \models^* \mathfrak{T} \text{ and } \bigcup \bar{g} = \Sigma(\mathfrak{T}, \mathcal{M}_{\bar{g}}))) \text{.}
\end{split}
\end{align}
Passing ($\ref{marker}$) through $f$ leads us to the presence of an oracle machine $\bar{\Phi}$ fulfilling
\begin{align}\label{marker2}
\begin{split}
    \forall \bar{g} \ \exists ! \mathcal{M}_{\bar{g}} \ ( & \bar{g} \text{ is a } s \text{-generic filter over } \mathfrak{A} \\ 
    & \implies (\mathcal{M}_{\bar{g}} \models^* \mathfrak{T} \text{ and } \bar{\Phi}^{\bar{g}} = f" (\Sigma(\mathfrak{T}, \mathcal{M}_{\bar{g}})))) \text{.}
\end{split}
\end{align}

Next, note that 
\begin{align*}
    u := \ & f" C \text{ and} \\
    \leq_u \ := \ & f" (\leq_{\mathbb{C}})
\end{align*}
are computable subsets of $\omega$. We shall, without loss of generality, identify $u$ with $(u, \leq_u)$ whenever contextually necessary.

\begin{prop}\label{notlastp}
There is a dense embedding $\pi$ of $\mathbb{C}$ into $\mathbb{P}(\mathfrak{T})^{\top}$ such that $\pi$ is $\Delta_1$-definable over the structure $(H(\omega); \in, P(\mathfrak{T})^{\top})$.
\end{prop}

\begin{proof}
First, $C$ is a $\Delta_0$ subset of $H(\omega)$. That $(r, f)$ witnesses $H(\omega)$ is computable means $f$ and $g := f^{-1}$ are functions of which graphs are $\Delta_1$-definable over $$\mathfrak{B} := (H(\omega); \in, P(\mathfrak{T})^{\top}).$$ Also, $\mathcal{L}_{\mathfrak{T}}$ is $\Delta_1$-definable over $\mathfrak{B}$ because $$x \in \mathcal{L}_{\mathfrak{T}} \iff \{x\} \in P(\mathfrak{T})^{\top} \text{ or } \{\neg x\} \in P(\mathfrak{T})^{\top}.$$  

Inductively define sequences $\{a_n : n < \omega\}$ and $\{k_n : n < \omega\}$ as follows:
\begin{align*}
    a_0 := \ & \emptyset \text{,} \\
    k_n := \ & min((f" \mathcal{L}_{\mathfrak{T}}) \setminus (f" a_n)) \text{, and} \\
    a_{n + 1} := \ & a_n \cup \{g(k_n)\} \cup \{\neg(g(k_n))\} \text{.}
\end{align*}
Note that $\{a_n\}$ and $\{k_n\}$ are $\Delta_1$-definable over $\mathfrak{B}$. Next, let $P^*$ be such that
\begin{align*}
    x \in P^* \iff x \in P(\mathfrak{T})^{\top} \text{ and } x \subset a_{|x|} \text{,}
\end{align*}
and have $x$ \emph{split in} $P^*$ iff
\begin{align*}
    x \in P^* \text{ and } \forall y \ (y \in a_{|x| + 1} \setminus a_{|x|} \implies x \cup \{y\} \in P^*) \text{,}
\end{align*}
so that both $P^*$ and the set of all its members that split in $P^*$ are $\Delta_1$-definable over $\mathfrak{B}$. We say $x$ is a $P^*$\emph{-least split above} $z$ iff
\begin{align*}
    z \subset x \text{ and } x \text{ splits in } P^* \text{ and } \forall y \ (z \subset y \subsetneq x \implies y \text{ does not split in } P^*) \text{.}
\end{align*}
Clearly, $P^*$ is dense in $P(\mathfrak{T})^{\top}$, so $P(\mathfrak{T})^{\top}$ being atomless entails $P^*$ is too. This yields the existence of a --- necessarily unique --- $P^*$-least split above $z$ for every $z \in P^*$.

Finally, we can inductively define $\pi$ on $C$ as such:
\begin{align*}
    \pi(\emptyset) := \ & \text{the } P^* \text{-least split above } \emptyset \text{,} \\
    \pi(x^{\frown}\langle 0 \rangle) := \ & \text{the } P^* \text{-least split above } \pi(x) \cup \{\varphi_{|x|, 0}\} \text{, and} \\
    \pi(x^{\frown}\langle 1 \rangle) := \ & \text{the } P^* \text{-least split above } \pi(x) \cup \{\varphi_{|x|, 1}\} \text{,}
\end{align*}
where
\begin{align*}
    \varphi_{n, 0} := \ & \text{the unique member of } a_{n + 1} \setminus a_n \text{ with leading symbol } \ulcorner \neg \urcorner \text{, and} \\
    \varphi_{n, 1} := \ & \text{the unique member of } a_{n + 1} \setminus a_n \text{ with leading symbol not } \ulcorner \neg \urcorner \text{.} 
\end{align*}
It in not difficult to see that $ran(\pi)$ is dense in $P^*$, and thus in $P(\mathfrak{T})^{\top}$. Moreover, since for each $x \in C$, the definition of $\pi(x)$ depends only on the finite set $$\{\pi(x \! \restriction_n) : n < |x|\}$$ and finitely many parameters which are $\Delta_1$-definable over $\mathfrak{B}$, $\pi$ must be $\Delta_1$-definable over $\mathfrak{B}$ as well.
\end{proof}

Proposition \ref{notlastp}, via $f$, implies the existence of a dense embedding $\pi$ of $\mathbb{C}$ into $\mathbb{P}(\mathfrak{T})^{\top}$ with $f" \pi$ computable in $s$, which is all we need to proceed. Fix any such $\pi$. It is not difficult to verify that $f \circ \pi = (f" \pi) \circ f$ on domain $C$ and taking upward closure of a set in a forcing notion commutes with $f$. As such, 
\begin{align}\label{number8}
\begin{split}
    \forall g \ (g \text{ is a } \mathbb{C} \text{-generic filter over } \mathfrak{A} \implies ( & \mathrm{UC}(s, (f" \pi)" (f" g)) \text{ is a filter on } s \text{ and} \\
    & \Phi^{(f" g) \oplus s} = \mathrm{UC}(s, (f" \pi)" (f" g))))
\end{split}
\end{align}
for some oracle machine $\Phi$.

\begin{prop}\label{lastprop0}
Let $g$ be a $\mathbb{C}$-generic subset over $\mathfrak{A}$. Then $\mathrm{UC}(\mathbb{P}(\mathfrak{T})^{\top}, \pi" g)$ is a $\mathbb{P}(\mathfrak{T})^{\top}$-generic subset over $\mathfrak{A}$.
\end{prop}

\begin{proof}
Let $h$ denote $\mathrm{UC}(\mathbb{P}(\mathfrak{T})^{\top}, \pi" g)$, and $D$ be a dense subset of $\mathbb{P}(\mathfrak{T})^{\top}$ definable in the language associated with $\mathfrak{A}$. Then the set
\begin{align*}
    D' := \{p \in C : \exists q \ (q \in D \text{ and } \pi(p) \leq_{\mathbb{P}(\mathfrak{T})^{\top}} q)\}
\end{align*}
is also definable in the language associated with $\mathfrak{A}$. Choose any $p_0 \in C$. By the density of $D$ in $\mathbb{P}(\mathfrak{T})^{\top}$, there is $q_0 \in D$ such that $q_0 \leq_{\mathbb{P}(\mathfrak{T})^{\top}} \pi(p_0)$. That $\pi$ is a dense embedding tells us there exists $q_1 \in ran(\pi)$ with $q_1 \leq_{\mathbb{P}(\mathfrak{T})^{\top}} q_0$. Now, for some $p_1 \leq_{\mathbb{C}} p_0$, $q_1 = \pi(p_1)$ and $p_1 \in D'$. We can therefore conclude that $D'$ is dense in $\mathbb{C}$.

As $g$ is $\mathbb{C}$-generic over $\mathfrak{A}$, we can find $p \in g \cap D'$. Seeing that $\pi(p) \in h$ and $h$ is upward closed, we have by the definition of $D'$, $h \cap D \neq \emptyset$.
\end{proof}

Passing Proposition \ref{lastprop0} through $f$ strengthens (\ref{number8}) to 
\begin{align}\label{number8p}
\begin{split}
    \forall g \ (g \text{ is a } \mathbb{C} \text{-generic filter over } \mathfrak{A} \implies ( & \mathrm{UC}(s, (f" \pi)" (f" g)) \text{ is a } s \text{-generic} \\ 
    & \text{filter over } \mathfrak{A} \text{ and} \\
    & \Phi^{(f" g) \oplus s} = \mathrm{UC}(s, (f" \pi)" (f" g)))) \text{.}
\end{split}
\end{align}
Now (\ref{marker2}) and (\ref{number8p}) in conjunction tells us that we can combine $\bar{\Phi}$ and $\Phi$ into an oracle machine $\Psi$ such that
\begin{align}\label{number10p}
\begin{split}
    \forall g \ \exists ! \mathcal{M}_g \ ( & g \text{ is a } \mathbb{C} \text{-generic filter over } \mathfrak{A} \\ 
    & \implies (\mathcal{M}_g \models^* \mathfrak{T} \text{ and } \Psi^{(f" g) \oplus s} = f" (\Sigma(\mathfrak{T}, \mathcal{M}_g))))
\end{split}
\end{align}
and
\begin{align*}
    F_{\Psi} := ((g \text{ a } \mathbb{C} \text{-generic filter over } \mathfrak{A}) \mapsto \mathcal{M}_g \text{ as per (\ref{number10p})}) = F_{\bar{\Phi}} \circ F_{\Phi} \text{,}
\end{align*}
where
\begin{align*}
    F_{\bar{\Phi}} := \ & (\bar{g} \text{ a } s \text{-generic filter over } \mathfrak{A}) \mapsto \mathcal{M}_{\bar{g}} \text{ as per (\ref{marker2}), and} \\
    F_{\Phi} := \ & (g \text{ a } \mathbb{C} \text{-generic filter over } \mathfrak{A}) \mapsto \mathrm{UC}(s, (f" \pi)" (f" g)) \text{ as per (\ref{number8p})}
\end{align*}
are both injective. 
\end{proof}

\begin{rem}\label{lastrem}
Observe that we derived $\pi$ in a uniform way from the parameters given in Proposition \ref{notlastp}. Turning our attention to the proof of Lemma \ref{ctblegeneric}, said observation passes through $f$ to imply $F_{\Phi}$ is derivable uniformly in $\mathbb{P}(\mathfrak{T})^{\top}$, and thus in $\mathfrak{T}$. Obviously, $F_{\bar{\Phi}}$ is derivable uniformly in $\mathfrak{T}$, so $F_{\Psi}$ is too. As a result, this same $\Psi$ works uniformly in $\mathfrak{T}$ to witness Lemma \ref{ctblegeneric} for all $\mathfrak{T}$ and $g$ as given in said lemma.
\end{rem}

In addition to Remark \ref{lastrem}, we can also strengthen Lemma \ref{ctblegeneric} by lowering the requirement on the genericity of $g$ and omitting $(A; \in)$ altogether. We formulate a strengthened version below in the nomenclature of computability theory.

\begin{lem}\label{ctblegeneric2}
Let $(r, f)$ witness $(H(\omega); \in)$ is computable. Then there is an oracle machine $\Psi$ such that whenever $\mathfrak{T}$ is a code-friendly $\Pi_2$ TCI, one of the following must hold.
\begin{enumerate}[label=(\arabic*)]
    \item All models of $\mathfrak{T}$ are almost finitely determined.
    \item\label{4532} For every $(f" (P(\mathfrak{T})^{\top}))$-$1$-generic real $t$, there is a unique model $\mathcal{M}_t$ of $\mathfrak{T}$ satisfying $$\Psi^{t \oplus (f" (P(\mathfrak{T})^{\top}))} = f" (\Sigma(\mathfrak{T}, \mathcal{M}_t)).$$ Moreover for every $\mathfrak{T}$, the function $t \mapsto \mathcal{M}_t$ defined as such is injective.
\end{enumerate}
\end{lem}

\begin{proof}
Choose an arbitrary a code-friendly $\Pi_2$ TCI $\mathfrak{T}$ with not all models almost finitely determined. Adopt the abbreviations
\begin{align*}
    s := \ & f" (P(\mathfrak{T})^{\top}) \text{ and} \\
    \leq_s \ := \ & f" (\leq_{\mathbb{P}(\mathfrak{T})^{\top}}) \text{,} 
\end{align*}
and identify $s$ with $(s, \leq_s)$ whenever contextually necessary. We will modify the proof of Lemma \ref{ctblegeneric} to get an oracle machine $\Psi$ witnessing \ref{4532}, before checking that a very slightly modified version of Remark \ref{lastrem} applies to $\Psi$. 

Let $\mathfrak{A} = (H(\omega); \in, P(\mathfrak{T})^{\top})$, so that 
\begin{itemize}
    \item $(H(\omega); \in)$ is a transitive model of a sufficiently strong set theory, and
    \item $\mathcal{L}_{\mathfrak{T}}$ is a $\Delta_1$-definable subset of $H(\omega)$ over the $\mathfrak{A}$,
\end{itemize}
whence $\mathfrak{A}$ is $\mathcal{L}_{\mathfrak{T}}$-suitable. We argue as in the proof of Lemma \ref{ctblegeneric} to conclude 
\begin{enumerate}[label=(\arabic*)]
    \item\label{5571} all but the last two points in the hypothesis of Lemma \ref{main2} hold with
    \begin{itemize}[label=$\circ$]
        \item $\mathfrak{A}$ as defined,
        \item $|\mathfrak{A}_{\mathfrak{T}}| = |trcl(\mathfrak{A}_{\mathfrak{T}})|$ in place of $\lambda$,
        \item $\mathcal{L}_{\mathfrak{T}}$ in place of $\mathcal{L}$,
        \item $[\mathcal{L}_{\mathfrak{T}}]^{< \omega}$ in place of $B$,
        \item $P(\mathfrak{T})^{\top}$ in place of $P$,
        \item $\mathbb{P}(\mathfrak{T})^{\top}$ in place of $\mathbb{P}$, and
        \item $\Gamma_{\mathfrak{T}}^{\top}$ in place of $\Gamma$;
    \end{itemize} 
    \item\label{5572} in particular, $\Gamma_{\mathfrak{T}}^{\top}$ is a set of $(\mathcal{L}_{\mathfrak{T}})^*_{\mathfrak{A}}$-$\Pi_2$ sentences.
\end{enumerate}
Now apply Lemma \ref{main2} with $V$ in place of $W$, as well as the substitutions in \ref{5571}, to arrive at (\ref{marker0}).

Notice that
\begin{itemize}
    \item $H(\omega)$ is closed under the function $\chi_{\mathfrak{T}}$ as defined in Remark \ref{subsafe2}, and 
    \item the $f$-image of $\chi_{\mathfrak{T}}$ is $\Delta^0_0$ (i.e. computable),
\end{itemize} 
from which we deduce the following:
\begin{enumerate}[label=(\arabic*)]
    \setcounter{enumi}{2}
    \item\label{5573} the $f$-image of each $\Sigma_1$-definable subset of $P(\mathfrak{T})^{\top}$ over $\mathfrak{A}$ is $\Sigma^{0, s}_1$.
\end{enumerate}

Passing (\ref{marker0}) and \ref{5573} through $f$ the way (\ref{marker}) was passed through $f$ in the proof of Lemma \ref{ctblegeneric}, for some oracle machine $\bar{\Phi}$ we have
\begin{align}\label{number9}
\begin{split}
    \forall \bar{g} \ \exists ! \mathcal{M}_{\bar{g}} \ ( & \bar{g} \text{ is a filter on } s \text{ meeting all } \Sigma^{0, s}_1 \text{ subsets of } s \\
    & \implies (\mathcal{M}_{\bar{g}} \models^* \mathfrak{T} \text{ and } \bar{\Phi}^{\bar{g}} = f" (\Sigma(\mathfrak{T}, \mathcal{M}_{\bar{g}})))) \text{.}
\end{split}
\end{align}

Adopt the abbreviations
\begin{align*}
    u := \ & f" C \text{ and} \\
    \leq_u \ := \ & f" (\leq_{\mathbb{C}}) \text{,}
\end{align*}
and identify $u$ with $(u, \leq_u)$ whenever contextually necessary. 

To deal with generic reals instead of generic filter, we first fix an oracle machine $\ddot{\Phi}$ that computes the set $$f" \{t \! \restriction_n : n < \omega\}$$ when given any real $t$ as oracle. Next, we follow an argument similar to the one used to derive (\ref{number8}) in the proof of Lemma \ref{ctblegeneric}, so that for some 
\begin{itemize}
    \item dense embedding $\pi$ of $\mathbb{C}$ into $\mathbb{P}(\mathfrak{T})^{\top}$ with $f" \pi$ computable in $s$, and 
    \item oracle machine $\Phi$, 
\end{itemize}
we have
\begin{align}\label{number10}
\begin{split}
    \forall t \ \exists ! c \ (t \text{ is a } s \text{-} 1 \text{-generic real} \implies ( & \mathrm{UC}(s, (f" \pi)" c) \text{ is a filter on } s \text{ and} \\
    & \ddot{\Phi}^t = c \text{ and } \Phi^{c \oplus s} = \mathrm{UC}(s, (f" \pi)" c))) \text{.}
\end{split}
\end{align}

The upcoming proposition is an analogue of Proposition \ref{lastprop0}, formulated to restrict the universe of discourse to the (even) natural numbers. 

\begin{prop}\label{lastprop}
Let $c$ be a subset of $u$ meeting all $\Sigma^{0, s}_1$ subsets of $u$. Then \\ $\mathrm{UC}(s, (f" \pi)" c)$ meets all $\Sigma^{0, s}_1$ subsets of $s$.
\end{prop}

\begin{proof}
Let $z$ be the subset of $s$ defined by a $\Sigma^{0, s}_1$ formula $\varphi(x)$ in one free variable. Define $$a_z := \{p \in u : \exists q \ (\varphi(q) \text{ and } (f" \pi)(p) \leq_s q)\} \text{.}$$ That $u$, $f" \pi$ and $\leq_s$ are all computable in $s$ ($u$ is even outright computable) gives us the $\Sigma^{0, s}_1$-definability of $a_z$ as a subset of $u$. Consequently, $c$ must meet $a_z$.

If $c \cap a_z \neq \emptyset$, then by the definition of $a_z$, there are conditions $p \in c$ and $q \in z$ for which $(f" \pi)(p) \leq_s q$, so $\mathrm{UC}(s, (f" \pi)" c)$ meets $z$. Otherwise, there is $p \in c$ that cannot be extended in $u$ to a member of $a_z$. Consider any $q \leq_s (f" \pi)(p)$. As $f" \pi$ densely embeds $u$ into $s$, we can find conditions $p' \in u$ and $q' \in s$ for which $q' \leq_s q$ and $(f" \pi)(p') = q'$. Now $p' \leq_u p$, which according to our choice of $p$, means $p' \not\in a_z$. Unfurling the definition of $a_z$ gives us $q \not\in z$. Having thus shown that $(f" \pi)(p)$ cannot be extended in $s$ to a member of $z$, we are done.
\end{proof}

We can passing the definition of a $s$-$1$-generic real through $f$ to conclude that for every such real $t$, $\ddot{\Phi}^t$ is a subset of $u$ meeting all $\Sigma^{0, s}_1$ subsets of $u$. With (\ref{number9}) and (\ref{number10}) in mind, Proposition \ref{lastprop} then tells us that we can combine $\bar{\Phi}$, $\ddot{\Phi}$ and $\Phi$ to get an oracle machine $\Psi$ fulfilling the requirements 
\begin{align}\label{lastreq}
    \forall t \ \exists ! \mathcal{M}_t \ (t \text{ is a } s \text{-} 1 \text{-generic} \implies (\mathcal{M}_t \models^* \mathfrak{T} \text{ and } \Psi^{t \oplus s} = f" (\Sigma(\mathfrak{T}, \mathcal{M}_t))))
\end{align}
and 
\begin{align*}
    F_{\Psi} := ((t \text{ a } s \text{-} 1 \text{-generic real}) \mapsto \mathcal{M}_t \text{ as per (\ref{lastreq})}) = F_{\bar{\Phi}} \circ F_{\Phi} \circ F_{\ddot{\Phi}} \text{,}
\end{align*}
where all of 
\begin{align*}
    F_{\bar{\Phi}} := \ & (\bar{g} \text{ a filter on } s \text{ meeting all } \Sigma^{0, s}_1 \text{ subsets of } s) \mapsto \mathcal{M}_{\bar{g}} \text{ as per (\ref{number9}),} \\ 
    F_{\Phi} := \ & (c \text{ a filter on } u \text{ meeting all } \Sigma^{0, s}_1 \text{ subsets of } u) \\ 
    & \mapsto \mathrm{UC}(s, (f" \pi)" c) \text{ as per (\ref{number10}), and} \\
    F_{\ddot{\Phi}} := \ & (t \text{ a } s \text{-} 1 \text{-generic real}) \mapsto f" \{t \! \restriction_n : n < \omega\} 
\end{align*}
are injective.

The argument in Remark \ref{lastrem} applies here to net us the uniformity of deriving $F_{\bar{\Phi}}$ and $F_{\Phi}$ from $\mathfrak{T}$, and clearly $F_{\ddot{\Phi}}$ does not depend on $\mathfrak{T}$ at all. As in Remark \ref{lastrem}, we can then conclude that $\Psi$ is the required witness to the lemma.
\end{proof}

Fix any nicely computable code $r$ of $H(\omega)$. Check that
\begin{itemize}
    \item the function $f$ given in Lemma \ref{ctblegeneric2} is a definable subset of $H(\omega)$ over the structure $(H(\omega); \in, r)$, in light of Fact \ref{uniquecode},
    \item the oracle machine $\Psi$ constructed in the proof of Lemma \ref{ctblegeneric2} is a definable element of $H(\omega)$ over the structure $(H(\omega); \in, r, f)$, and
    \item the injective function $t \mapsto f" \Sigma(\mathfrak{T}, \mathcal{M}_t)$ defined in the proof of Lemma \ref{ctblegeneric2} always has a left inverse computable using only $f_r" (P(\mathfrak{T})^{\top})$ as parameter. Further, said left inverse is uniformly computable over all relevant code-friendly $\Pi_2$ TCIs $\mathfrak{T}$.
\end{itemize}
As such, we have actually proven a more general version of Lemma \ref{ctblegeneric2}, which we formally present as our final theorem below.

\begin{thm}\label{lastthm}
There is a formula $\psi_{gmc}$ in two free variables, absolute for transitive models of $\mathsf{ZFC - Powerset}$, such that $\psi_{gmc}(r, (f_r, \Psi_r, \bar{\Psi}_{r}))$ defines a function $$r \mapsto (f_r, \Psi_r, \bar{\Psi}_{r})$$ on the set of all nicely computable codes of $H(\omega)$, wherewith
\begin{enumerate}[label=(\arabic*)]
    \item $(r, f_r)$ witnesses $(H(\omega); \in)$ is computable,
    \item $\Psi_r$ and $\bar{\Psi}_{r}$ are oracle machines, and
    \item whenever $\mathfrak{T}$ is a code-friendly $\Pi_2$ TCI, one of the following must hold:
    \begin{enumerate}[label=(\alph*)]
        \item All models of $\mathfrak{T}$ are almost finitely determined.
        \item For every $(f_r" (P(\mathfrak{T})^{\top}))$-$1$-generic real $t$, there is a unique model $\mathcal{M}_t$ of $\mathfrak{T}$ satisfying 
        \begin{gather*}
            \Psi_r^{t \oplus (f_r" (P(\mathfrak{T})^{\top}))} = f_r" (\Sigma(\mathfrak{T}, \mathcal{M}_t)) \text{ and} \\
            \bar{\Psi}_{r}^{(f_r" (\Sigma(\mathfrak{T}, \mathcal{M}_t))) \oplus (f_r" (P(\mathfrak{T})^{\top}))} = t.
        \end{gather*}
        In particular, for each such pair $(t, \mathcal{M}_t)$, $$t \oplus (f_r" (P(\mathfrak{T})^{\top})) \equiv_{T} (f_r" (\Sigma(\mathfrak{T}, \mathcal{M}_t))) \oplus (f_r" (P(\mathfrak{T})^{\top})).$$
\end{enumerate}
\end{enumerate}
\end{thm}

From Theorem \ref{lastthm}, we can prove that certain $T$-substructures of a countable structure have the prefect set property.

\begin{cor}\label{lastcor}
Let $\mathfrak{A} = (A; \mathcal{I})$ and $T$ be a countable structure and a first-order $\Pi_2$ theory respectively, over the same signature $\sigma$. Then $|\mathrm{Sub}(\mathfrak{A}, T)| \leq \aleph_0$ or $|\mathrm{Sub}(\mathfrak{A}, T)| = 2^{\aleph_0}$.
\end{cor}

\begin{proof}
Clearly $\mathrm{Sub}(\mathfrak{A}, T)$ is invariant under isomorphisms, so without loss of generality, we can assume $A$ is some ordinal $\alpha$ with $\alpha \leq \omega$, and $dom(\mathcal{I}) \subset H(\omega) \setminus \omega$. By Example \ref{2ndex}, there is a $\Pi_2$ TCI $\mathfrak{T}$ such that $$\{ \text{models of } \mathfrak{T}\} = \mathrm{Sub}(\mathfrak{A}, T).$$ Our assumptions on $\mathfrak{A}$ allow us to choose $\mathfrak{T}$ satisfying $\mathcal{L}_{\mathfrak{T}} \subset H(\omega)$, so that $\mathfrak{T}$ is also code-friendly. Note that $|\mathrm{Sub}(\mathfrak{A}, T)| \leq 2^{|A|} \leq 2^{\aleph_0}$.

If all models of $\mathfrak{T}$ are almost finitely generated, then Lemma \ref{afdinV} tells us that the number of models of $\mathfrak{T}$ is bounded above by $$max\{|P(\mathfrak{T})|, |[\mathcal{L}_{\mathfrak{T}}]^{< \omega}|\} \leq \aleph_0,$$ which means $|\mathrm{Sub}(\mathfrak{A}, T)| \leq \aleph_0$. Otherwise, by Theorem \ref{lastthm}, for some real $X$ there is an injection from the set of $X$-$1$-generic reals into the set of models of $\mathfrak{T}$. As there are continuum many $X$-$1$-generic reals, $|\mathrm{Sub}(\mathfrak{A}, T)| = 2^{\aleph_0}$.
\end{proof}

Corollary \ref{lastcor} also follows from a well-known fact in descriptive set theory (see e.g. \cite{kechris}), via the Cantor-Bendixson theorem. 

\begin{fact}\label{lastfact}
Let $\mathfrak{A} = (A; \mathcal{I})$ and $T$ be a countable structure and a first-order theory respectively, over the same signature $\sigma$. Then $\mathrm{Sub}(\mathfrak{A}, T)$ is a closed set in some Polish space.
\end{fact}

Precisely because Theorem \ref{lastthm} establishes a very strong version of perfect set property, it makes sense to treat it like a souped-up variant of the Cantor-Bendixson theorem, subject to restrictions on use cases. Instead of being applicable to arbitrary closed subsets of some Polish space, Theorem \ref{lastthm} only applies to sets which contain exactly all the models of some $\Pi_2$ TCI. This treatment comes off as a natural extension of the parallels we drew in Remark \ref{CBrem}.

\subsection{Open Questions}

The study of how abundant generic models of a TCI are, can be approached from another direction: by comparing them with arbitrary models of the same TCI. As such, the propensity for a model of a TCI to be isomorphic to a generic model becomes of fundamental interest. In view of much of the work done in this section, the following is a most natural question.

\begin{enumerate}[label=(Q\arabic*)]
    \setcounter{enumi}{2}
    \item\label{54qn2} Is there a consistent $\Pi_2$ TCI $\mathfrak{T}$ such that every model of $\mathfrak{T}$ found in some outer model of $V$ is isomorphic to a $V$-generic model of $\mathfrak{T}$?
\end{enumerate}

Consider any consistent first-order $\Pi_2$ theory $T$ with only finite models (there are many such theories with the empty signature). $T$ can be used to define a $\Pi_2$ TCI $\mathfrak{T}$ such that every model of $T$ is isomorphic to some model of $\mathfrak{T}$, and vice versa. This relation between $T$ and $\mathfrak{T}$ remains true in all outer models of $V$. Hence, \ref{54qn2} can be answered in the affirmative. 

Even if we require $\mathfrak{T}$ to have an infinite model in each of these questions, the same answers apply when we choose $\mathfrak{T}$ to be a $\Pi_2$ TCI such that all models of $\mathfrak{T}$ are isomorphic to the unique (up to isomorphism) $\aleph_0$-sized model of a $\aleph_0$-categorical first-order $\Pi_2$ theory $T$ (say, the theory of dense linear orders without endpoints). As the $\aleph_0$-categoricity of a theory is absolute for transitive models of $\mathsf{ZFC}$ with the same ordinals, said relation between $T$ and $\mathfrak{T}$ is preserved across outer models of $V$.

However, the question dual to \ref{54qn2} appears more difficult.

\begin{enumerate}[label=(Q\arabic*)]
    \setcounter{enumi}{3}
    \item\label{54qn3} Is there a $\Pi_2$ TCI $\mathfrak{T}$ with a model $\mathcal{M}$ in some outer model of $V$ such that $\mathcal{M}$ is not isomorphic to any $V$-generic model of $\mathfrak{T}$?
\end{enumerate}

Since our impetus for studying TCIs stems from our interest in uncovering links between forcing/genericity and the semantics of first-order logic, it is perhaps fitting that we ask for a similar example of a first-order theory. 

\begin{enumerate}[label=(Q\arabic*)]
    \setcounter{enumi}{4}
    \item\label{54qn4} Is there a first-order $\Pi_2$ theory $T$ with a model $M$ in some outer model of $V$ such that for no $\Pi_2$ TCI $\mathfrak{T}$ is $M$ isomorphic to a $V$-generic model of $\mathfrak{T}$?
\end{enumerate}

By Lemma \ref{gmodelsinfe}, it seems that \ref{54qn3} and \ref{54qn4} cannot be solved using set forcing alone. Thankfully, class forcing has been developed sufficiently to answer them. Essentially, we ``cheat'' by choosing a close-to-trivial TCI, only possible models of which are of the form $(\omega; \in, A)$, where $A$ can be any real. By way of Jensen's coding-the-universe forcing, we can get to an outer model of $V$ with a new non-generic real $r$. Now the model of our TCI with $r$ as the predicate cannot be isomorphic to \emph{any} member of \emph{any} forcing extension of $V$. 

So, for TCIs with very simple theories, we can construct a non-generic model. We cannot do the same for all $\Pi_2$ TCIs because of Lemma \ref{revgenmodels}. Together, they make us wonder if a clear line can be drawn in $V$. Let 
\begin{align*}
    \mathrm{NG}_1 := \{\mathfrak{T} \in V : \ & \mathfrak{T} \text{ is a } \Pi_2 \text{ TCI and } \exists W \ \exists \mathcal{M} \! \in \! W \ \forall x \! \in \! W \ (W \text{ is an outer model of } V \\
    & \text{and } \mathcal{M} \models^* \mathfrak{T} \text{ and } x \not\cong \mathcal{M} \text{ whenever } x \text{ is a } V \text{-generic model of } \mathfrak{T})\} \\
    \mathrm{NG}_2 := \{T \in V : \ & T \text{ is a } \Pi_2 \text{ theory and } \\
    & \exists W \ \exists \mathcal{M} \! \in \! W \ \forall \mathfrak{T} \! \in \! V \ \forall x \! \in \! W \\
    & (W \text{ is an outer model of } V \text{ and } \mathcal{M} \models T \text{ and } \\
    & x \not\cong \mathcal{M} \text{ whenever } \mathfrak{T} \text{ is a TCI and } x \text{ is a } V \text{-generic model of } \mathfrak{T})\} \text{.}
\end{align*}

\begin{ques}\label{q548}
Is $\mathrm{NG}_1$ definable in $V$?
\end{ques}

\begin{ques}\label{q552}
Is $\mathrm{NG}_2$ definable in $V$?
\end{ques}

Our current line of questioning can be extended to the paradigm of relative effectiveness.

Fix $\psi_{gmc}$ to be as in Theorem \ref{lastthm}. Let $r$ be a nicely computable code of $H(\omega)$. Define
\begin{align*}
    f_r := \ & \text{the unique } f \text{ for which there are } \Psi \text{ and } \bar{\Psi} \text{ satisfying } \psi_{gmc}(r, (f, \Psi, \bar{\Psi})) \text{,} \\
    \Psi_r := \ & \text{the unique } \Psi \text{ for which there are } f \text{ and } \bar{\Psi} \text{ satisfying } \psi_{gmc}(r, (f, \Psi, \bar{\Psi})) \text{, and} \\
    \bar{\Psi}_r := \ & \text{the unique } \bar{\Psi} \text{ for which there are } f \text{ and } \Psi \text{ satisfying } \psi_{gmc}(r, (f, \Psi, \bar{\Psi})) \text{.}
\end{align*}
Analogous to Question \ref{q548}, we want to pick out every code-friendly TCI $\mathfrak{T}$ with a model $M$ that neither almost finitely determined nor isomorphic to any $\mathcal{M}_t$ born from a $(f_r" (P(\mathfrak{T})^{\top}))$-$1$-generic real $t$ \`{a} la Theorem \ref{lastthm}. As turns out, because $f_r \circ f_s^{-1}$ is computable for any other nicely computable code $s$ of $H(\omega)$, the answer to this question is independent of
the choice of $r$.

Let $\mathrm{NGE}'_1$ be the set containing exactly all the code-friendly $\Pi_2$ TCIs $\mathfrak{T}$ with a model $\mathcal{M}$ such that
\begin{enumerate}[label=(\alph*)]
    \item $\mathcal{M}$ is not almost finitely determined, and
    \item for every nicely computable code $r$ of $H(\omega)$ and every $(f_r" (P(\mathfrak{T})^{\top}))$-$1$-generic real $t$, if $\mathcal{M}'$ is a model of $\mathfrak{T}$ satisfying 
    \begin{equation*}
        \Psi_r^{t \oplus (f_r" (P(\mathfrak{T})^{\top}))} = f_r" (\Sigma(\mathfrak{T}, \mathcal{M}')) \text{,}
    \end{equation*}
    then $\mathcal{M} \not \cong \mathcal{M}'$.
\end{enumerate}
We are interested in representing $\mathrm{NGE}'_1$ as a set of reals, so fix a nicely computable code $r$ of $H(\omega)$ and set
\begin{equation*}
    \mathrm{NGE}_1 := \{(f_r" \Gamma_{\mathfrak{T}}) \oplus (f_r" \mathcal{L}_{\mathfrak{T}}) : \mathfrak{T} \in \mathrm{NGE}'_1\} \text{.}
\end{equation*}
Notice that the pair $(\Gamma_{\mathfrak{T}}, \mathcal{L}_{\mathfrak{T}})$ completely determines $\mathfrak{T}$. Further, both $f_r" \Gamma_{\mathfrak{T}}$ and $f_r" \mathcal{L}_{\mathfrak{T}}$ are well-defined by \ref{cf1}.

\begin{ques}\label{q553}
Is $\mathrm{NGE}_1$ a $\mathbf{\Delta^1_0}$ set of reals?
\end{ques}

We can replace $\mathrm{NGE}_1$ Question \ref{q553} with another set to get an analogue of Question \ref{q552} in the same spirit of relative effectiveness. 

Let $\mathrm{NGE}'_2$ be the set containing exactly all the first-order $\Pi_2$ theories $T$ with a model $M$ such that whenever $\mathfrak{T}$ is a code-friendly $\Pi_2$ TCI and $\mathcal{M} \models^* \mathfrak{T}$,
\begin{enumerate}[label=(\alph*)]
    \item if $\mathcal{M}$ is almost finitely determined then $M \not \cong \mathcal{M}$, and
    \item if $r$ is a nicely computable code of $H(\omega)$ and $t$ is a $(f_r" (P(\mathfrak{T})^{\top}))$-$1$-generic real for which
    \begin{equation*}
        \Psi_r^{t \oplus (f_r" (P(\mathfrak{T})^{\top}))} = f_r" (\Sigma(\mathfrak{T}, \mathcal{M})) \text{,}
    \end{equation*}
    then $M \not \cong \mathcal{M}$.
\end{enumerate}
Fix a nicely computable code $r$ of $H(\omega)$ and set
\begin{equation*}
    \mathrm{NGE}_2 := \{(f_r" \Gamma_{\mathfrak{T}}) \oplus (f_r" \mathcal{L}_{\mathfrak{T}}) : \mathfrak{T} \in \mathrm{NGE}'_2\} \text{.}
\end{equation*}

\begin{ques}
Is $\mathrm{NGE}_2$ a $\mathbf{\Delta^1_0}$ set of reals?
\end{ques}

In search of further evidence that $\hat{\mathbb{P}}$ is a useful lens through which one can classify the reach of forcing as a technique, we ask the next few questions following the directions of Remarks \ref{ramble1} and \ref{ramble2}.

\begin{ques}
Let $\mathfrak{T}_1$ and $\mathfrak{T}_2$ be $\Pi_2$ TCIs such that $\mathfrak{T}_1 \trianglelefteq \mathfrak{T}_2$.
\begin{enumerate}[label=(\arabic*)]
    \item If $\mathcal{M} \models^* \mathfrak{T}_2$ in an outer model of $V$, must $V[\mathcal{M}]$ contain a model of $\mathfrak{T}_1$?
    \item If $\mathcal{M} \models^* \mathfrak{T}_1$ in an outer model of $V$, must there be $\mathcal{M}'$ is some outer model of $V$ such that $\mathcal{M}' \models^* \mathfrak{T}_2$ and $V[\mathcal{M}] \subset V[\mathcal{M}']$?
\end{enumerate}
\end{ques}

\begin{ques}
Is there a ``naturally definable'' class $\mathcal{C}$ such that 
\begin{enumerate}[label=(\alph*)]
    \item $\mathcal{C} \subsetneq \{\mathfrak{T} : \mathfrak{T} \text{ is a } \Pi_2 \text{ TCI}\}$, and
    \item every member of $\trianglelefteq / \sim_T$ contains a member $\mathfrak{T}$ of $\mathcal{C}$ for which
    \begin{align*}
        & \{V[\mathcal{M}] : \mathcal{M} \models^* \mathfrak{T} \text{ in an outer model of } V\} \\
        & = \{V[g] : g \text{ is } \mathbb{P}(\mathfrak{T}) \text{-generic over } V\} \text{?}
    \end{align*}
\end{enumerate} 
\end{ques}

\begin{ques}
Is there a ``naturally definable'' class $\mathcal{C}$ of TCIs such that 
\begin{enumerate}[label=(\alph*)]
    \item $\mathcal{C} \supsetneq \{\mathfrak{T} : \mathfrak{T} \text{ is a } \Pi_2 \text{ TCI}\}$, and
    \item for every $\mathfrak{T} \in \mathcal{C}$, 
    \begin{align*}
        & \{V[\mathcal{M}] : \mathcal{M} \models^* \mathfrak{T} \text{ in an outer model of } V\} \\
        & \supset \{V[g] : g \text{ is } \mathbb{P}(\mathfrak{T}) \text{-generic over } V\} \text{?}
    \end{align*}
\end{enumerate} 
\end{ques}

\section{References}
\printbibliography[heading=none]

\end{document}